\documentclass{lmcs-arxiv}


\usepackage{tikz}
\usepackage{tikz-cd} 
\usepackage{mathtools} 
\usepackage{thmtools, thm-restate} 
\usepackage{amssymb} 
\usepackage{mathpartir} 
\usepackage[mathcal]{eucal} 
\usepackage{shuffle} 
\usepackage{hyperref}
\usepackage{mathabx} 
\usepackage{marginnote}
\usepackage{proof}
\usepackage{stmaryrd} 
\usepackage{cmll} 
\usepackage{xspace} 
\usepackage{wrapfig} 
\usepackage{arydshln} 

\definecolor{aureolin}{rgb}{0.99, 0.93, 0.0}
\definecolor{forestgreen}{rgb}{0.13, 0.54, 0.13}
\definecolor{bluebell}{rgb}{0.64, 0.64, 0.82}
\definecolor{brickred}{rgb}{0.8, 0.25, 0.33}
\definecolor{goldenpoppy}{rgb}{0.99, 0.76, 0.0}
\definecolor{grullo}{rgb}{0.66, 0.6, 0.53}
\definecolor{frenchblue}{rgb}{0.0, 0.45, 0.73}
\definecolor{persianred}{rgb}{0.8, 0.2, 0.2}
\definecolor{ginger}{rgb}{0.69, 0.4, 0.0}
\definecolor{goldenbrown}{rgb}{0.6, 0.4, 0.08}
\definecolor{nodeS}{rgb}{1.0, 0.6, 0.11}
\definecolor{nodeT}{rgb}{.58, 0.58, 0.58}

\newcommand{\proofseeappendix}[1]{%
  \begin{proof}
    See Appendix~\ref{#1}.
  \end{proof}
}

\definecolor{officegreen}{rgb}{0.0, 0.5, 0.0}

\newcommand\defeq{\mathrel{:=}}
\newcommand\eqdef{\mathrel{=:}}
\newcommand\defin[1]{\textbf{#1}}

\newcommand{\A}{\mathcal{A}}
\newcommand{\B}{\mathcal{B}}
\newcommand{\C}{\mathcal{C}}
\newcommand{\D}{\mathcal{D}}
\newcommand{\E}{\mathcal{E}}
\newcommand{\K}{\mathcal{K}}
\newcommand{\N}{\mathcal{N}}
\renewcommand{\P}{\mathcal{P}}

\newcommand{\W}{\mathcal{W}}
\newcommand{\Cat}{\mathrm{\mathcal{C}at}}

\newcommand{\Adj}{\mathrm{\mathcal{A}dj}}
\newcommand{\Set}{\mathrm{\mathcal{S}et}}

\newcommand{\BifCat}{\mathrm{\mathcal{B}ifCat}}
\newcommand{\Span}{\mathbb{S}\mathrm{pan}}

\newcommand{\op}{{\mathit{op}}}

\newcommand\vcat[1]{\mathcal{V}(#1)}
\newcommand\hcat[1]{\mathcal{H}(#1)}

\newcommand\ord[1]{{\mathsf{#1}}}
\newcommand{\One}{\ord{1}}
\newcommand{\Two}{\ord{2}}

\newcommand{\Z}{\mathcal{Z}}
\newcommand{\ZZ}{\mathbb{Z}}
\newcommand{\NN}{\mathbb{N}}
\newcommand{\DD}{\mathbb{D}}
\newcommand{\EE}{\mathbb{E}}

\DeclareMathOperator{\cod}{cod}
\DeclareMathOperator{\dom}{dom}
\DeclareMathOperator{\id}{id}

\DeclareMathOperator{\src}{src}
\DeclareMathOperator{\tgt}{tgt}

\newcommand{\hcomp}{\cdot}
\newcommand{\cut}{\mathrm{cut}}

\newcommand\set[1]{\{\,#1\,\}}

\newcommand\atom[1]{#1} 

\newcommand\pull[2][]{\mathop{{#2}^{\!-#1}}}
\newcommand\push[2][]{\mathop{{#2}^{\!+#1}}}

\newcommand{\opcart}[2]{{#1}_{#2}}
\newcommand{\cart}[2]{{\overline{#1}}_{#2}}

\newcommand{\cartcomp}[3]{#1 \mathbin{:} \cart{#2}{#3}}
\newcommand{\opcartcomp}[3]{\opcart{#1}{#2} \mathbin{:} {#3}}

\newcommand{\uLpush}[2]{\opcart{#1}{} \mathbin{\backslash_{#2}}}
\newcommand{\tRpush}[1]{. \mathord{\opcart{#1}{}}}
\newcommand{\tLpull}[1]{\mathord{\cart{#1}{}} .}
\newcommand{\uRpull}[2]{\mathbin{{}_{#1}{\slash}} \cart{#2}{}} 

\newcommand{\tLRpullpush}[3]{\tLpull {#1} {#2} \tRpush {#3}}
\newcommand{\uLRpushpull}[4]{\uLpush {#1} {} \underset{#2}{#3} \uRpull {} {#4}}

\newcommand\Id[1][]{\id_{#1}}

\newcommand\Bifib[1]{\mathcal{B}\mathrm{if}(#1)}
\newcommand\BifibFun[1]{\Lambda_{#1}}
\newcommand\saBifib[1]{sa\mathcal{B}\mathrm{if}(#1)}
\newcommand\nsaBifib[1]{nsa\mathcal{B}\mathrm{if}(#1)}

\newcommand\permeq{\sim}

\newcommand\vdashf[1]{\mathchoice{\underset{#1}{\Longrightarrow}}{\Longrightarrow_{#1}}{}{}}

\newcommand\Rsym{\mathsf{R}}
\newcommand\Lsym{\mathsf{L}}

\newcommand\Rpull[1][f]{\Rsym{\pull{#1}}}
\newcommand\Rpush[1][f]{\Rsym{\push{#1}}}
\newcommand\Lpull[1][f]{\Lsym{\pull{#1}}}
\newcommand\Lpush[1][f]{\Lsym{\push{#1}}}

\newcommand\Rotimes{\Rsym{\otimes}}
\newcommand\Lotimes{\Lsym{\otimes}}
\newcommand\Rlolli{\Rsym{\multimap}}
\newcommand\Llolli{\Lsym{\multimap}}
\newcommand\lolli{\multimap}
\newcommand\Rplus{\Rsym{\oplus}}
\newcommand\Lplus{\Lsym{\oplus}}
\newcommand\Rwith{\Rsym{\&}}
\newcommand\Lwith{\Lsym{\&}}

\newcommand{\istate}{q_\mathsf{i}}
\newcommand{\fstate}{q_\mathsf{f}}

\newcommand{\lockL}{\Lsym}
\newcommand{\lockR}{\Rsym}
\newcommand{\lockNone}{\bot_\mathsf{i}}
\newcommand{\flockNone}{\bot_\mathsf{f}}
\newcommand{\ilockNone}{\lockNone}

\newcommand{\vdashl}[2]{\mathchoice{\mathrel{\underset{#1}{\overset{#2}{\Longrightarrow}}}}{\mathrel{\Longrightarrow_{#1}^{#2}}}{}{}}

\newcommand\Frm[1]{\mathrm{Frm}_{#1}}
\newcommand\Pf[1]{\mathrm{Pf}_{#1}}

\newcommand\Ob[1]{\mathrm{Ob}_{#1}}
\newcommand\Arr[1]{\mathrm{Arr}_{#1}}

\newcommand\zigzag{\leadsto}

\newcommand\refs{\sqsubset}

\newcommand\sem[2]{\left\llbracket#1\right\rrbracket_{#2}}

\newcommand{\PTree}{{\mathsf{PTree}}}
\newcommand{\PForest}{{\mathsf{PForest}}}

\newcommand\collapse[1]{\left[ #1 \right]}
\newcommand\floor[1]{\lfloor #1\rfloor}
\newcommand\ceil[1]{\lceil #1\rceil}

\newcommand{\Strengthensym}{\mathsf{S}}
\newcommand{\Strengthen}[1]{\Strengthensym(#1)}

\newcommand\Seqsym{\mathsf{Seq}}
\newcommand\Seq[1]{\Seqsym(#1)}
\newcommand\FocSeqsym{\mathsf{FocSeq}}
\newcommand\FocSeq[1]{\FocSeqsym(#1)}
\newcommand\InvSeqsym{\mathsf{InvSeq}}
\newcommand\InvSeq[1]{\InvSeqsym(#1)}

\newcommand\seqsym{\mathsf{seq}}
\newcommand\parsym{\mathsf{par}}
\newcommand\grasym{\mathsf{gra}}

\newcommand\rewtoseq{\mathrel{\rightarrow_\seqsym}}
\newcommand\rewtoseqRL{\mathrel{\rightarrow_{\seqsym^\Rsym_\Lsym}}}
\newcommand\rewtoseqLR{\mathrel{\rightarrow_{\seqsym^\Lsym_\Rsym}}}
\newcommand\oprewtoseq{\mathrel{\leftarrow_\seqsym}}

\newcommand\permeqseq{\mathrel{\sim_\seqsym}}

\newcommand\rewtopar{\mathrel{\rightarrow_\parsym}}
\newcommand\rewtoparRL{\mathrel{\rightarrow_{\parsym^\Rsym_\Lsym}}}
\newcommand\rewtoparLR{\mathrel{\rightarrow_{\parsym^\Lsym_\Rsym}}}

\newcommand{\biL}{\mathbb{L}}
\newcommand{\biR}{\mathbb{R}}
\newcommand{\biLR}{\mathbb{LR}}

\newcommand\Conf[1]{\mathsf{Conf}_{#1}}

\newcommand\Free{\mathcal{F}}
\newcommand\G{\mathbb{G}}

\newcommand\enc[1]{\ulcorner{#1}\urcorner}

\newcommand\pullback[2]{\mathbin{\mathstrut_{#1}\times_{#2}}}

\newcommand\rmono{\rightarrowtail}

\newcommand\repi{\twoheadrightarrow}
\newcommand\lepi{\twoheadleftarrow}

\newcommand{\mono}{\mathsf{mono}}
\newcommand{\epi}{\mathsf{epi}}

\newcommand{\fDelta}{\pmb{\Delta}}

\newcommand\BN{\mathbf{B}_\NN}

\newcommand{\FP}{$\mathsf{FP}$\xspace}

\usetikzlibrary{calc,positioning,decorations.markings,decorations.pathreplacing,decorations.pathmorphing,patterns,backgrounds}

\tikzcdset{bifib/.style = {two heads}}

\tikzset{spanmap/.style={
            decoration={markings,
            mark= at position 0.5 with {
                  \node[transform shape] (tempnode) {$/$};
                  }
              },
              postaction={decorate}
}}

\tikzset{vertex/.style =
  {circle,fill=darkgray}}

\tikzset{oedge/.style =
  {line width=1mm,
   blend group=luminosity,
   decoration={
     markings,
     mark=at position 0.5 with {\arrow{Stealth[length=2.5mm, sep=-1.25mm,open]}}},
   postaction={decorate},
 }}

\tikzset{uedge/.style =
  {line width=1mm,
 }}

\tikzset{zedge/.style =
  {line width=1mm,decorate,decoration={zigzag}}}

\tikzset{delta/.style = {circle,fill=red,minimum size=4pt,inner sep=0pt,,font = \footnotesize\sffamily}}
\tikzset{sigma/.style = {circle,fill=blue,minimum size=4pt,inner sep=0pt,text=white,font = \footnotesize\sffamily}}
\tikzset{kappa/.style = {circle,fill=darkgray,minimum size=4pt,inner sep=0pt}}

\tikzstyle{oba}=[bluebell!25]
\tikzstyle{obap}=[green!25]
\tikzstyle{obb}=[orange!25]
\tikzstyle{obc}=[brickred!25]
\tikzstyle{obcp}=[frenchblue!25]
\tikzstyle{obd}=[violet!25]
\tikzstyle{obe}=[goldenpoppy!25]
\tikzstyle{obep}=[grullo!25]
\tikzstyle{arrf}=[bluebell]
\tikzstyle{arrg}=[brickred]
\tikzstyle{arrfp}=[forestgreen]
\tikzstyle{arrgp}=[frenchblue]
\tikzstyle{arrh}=[violet]
\tikzstyle{arrhp}=[ginger]
\tikzstyle{zigS}=[goldenpoppy]
\tikzstyle{zigT}=[persianred]
\tikzstyle{zigU}=[grullo]
\tikzstyle{height0}=[bluebell!25]
\tikzstyle{height1}=[goldenpoppy!25]
\tikzstyle{height2}=[forestgreen!25]
\tikzstyle{height3}=[goldenbrown!25]

\newcommand\Tetrahedron{
\begin{tikzcd}[ampersand replacement=\&]
  \D \arrow[ddrr,"p"']\arrow[rr,"\eta_p"]\arrow[rrrd,"\theta"'] \& \&\Bifib{p}\arrow[dd,"\BifibFun{p}"]\arrow[dr,"\sem{-}\theta"]\&  \\
  \& \&  \& \E\arrow[dl,"q\text{ bifibration}"]\\
  \& \& \C \& 
\end{tikzcd}
}

\begin{document}

\title{The free bifibration on a functor}
\titlecomment{}
\thanks{}	

\author[B.~Clarke]{Bryce Clarke\lmcsorcid{0000-0003-0579-2804}}[a]
\author[G.~Scherer]{Gabriel Scherer\lmcsorcid{0000-0003-1758-3938}}[b]
\author[N.~Zeilberger]{Noam Zeilberger\lmcsorcid{0000-0002-5945-4184}}[c]

\address{Tallinn University of Technology, Tallinn, Estonia}
\email{bryce.clarke@taltech.ee}

\address{Inria, IRIF, Université Paris Cité, Paris, France}
\email{gabriel.scherer@inria.fr}

\address{LIX, CNRS, Inria, École Polytechnique, Institut Polytechnique de Paris, Palaiseau, France}	
\email{noam.zeilberger@polytechnique.edu}  

\begin{abstract}
  We consider the problem of constructing the free bifibration generated by a functor of categories $p : \D \to \C$.
This problem was previously considered by Lamarche, and is closely related to the problem, considered by Dawson, Paré, and Pronk, of ``freely adjoining adjoints'' to a category.
We develop a proof-theoretic approach to the problem, beginning with a construction of the free bifibration $\BifibFun{p} : \Bifib{p}\to\C$ in which objects of $\Bifib{p}$ are formulas of a primitive ``bifibrational logic'', and arrows are derivations in a cut-free sequent calculus modulo a notion of permutation equivalence.
We show that instantiating the construction to the identity functor generates a \emph{zigzag double category} $\ZZ(\C)$, which is also the free double category with companions and conjoints (or fibrant double category) on $\C$.
The approach adapts smoothly to the more general task of building $(\P,\N)$-fibrations, where one only asks for pushforwards along arrows in $\P$ and pullbacks along arrows in $\N$ for some subsets of arrows; this encompasses Kock and Joyal's notion of \emph{ambifibration} when $(\P,\N)$ form a factorization system.
We establish a series of progressively stronger normal forms, guided by ideas of \emph{focusing} from proof theory, and obtain a canonicity result under assumption that the base category is factorization preordered relative to $\P$ and~$\N$.
This canonicity result allows us to decide the word problem and to enumerate relative homsets without duplicates.
Finally, we describe several examples of a combinatorial nature, including a category of plane trees generated as a free bifibration over $\omega$, and a category of increasing forests generated as a free ambifibration over $\Delta$, which contains the lattices of noncrossing partitions as quotients of its fibers by the Beck-Chevalley condition for bicartesian squares.

\end{abstract}

\maketitle

\section*{Introduction}
\label{sec:introduction}

A functor $p : \D \to \C$ between two categories is a \emph{bifibration} when, roughly speaking, objects of $\D$ may be pushed and pulled along arrows of $\C$.
If we visualize the category $\D$ as living over the category $\C$ and mapping downwards via the functor, then the structure of a bifibration on $p$ allows the arrows of $\C$ to be lifted upwards into the category $\D$, in two different kinds of ways.
Formally, for any arrow $f : A \to B$ in $\C$ and any object $S$ in $\D$ such that $p(S) = A$,
there should be an object $\push{f}S$ and an arrow $\opcart{f}{S} : S \to \push{f}S$ of $\D$ such that $p(\opcart f S) = f$,
\begin{equation*}
  \begin{tikzcd}
    \D\ar[d,"p"'] & S \ar[r,"\opcart f S",dashed] & \pmb{\push{f}{S}} \\
    \C & A \ar[r,"f"] & B
  \end{tikzcd}
\end{equation*}
which are universal in the sense that
for any arrow $g : B \to C$ in $\C$ and arrow $\alpha : S \to T$ in $\D$ such that $p(\alpha) = fg$,\footnote{Here and throughout the paper we notate composition of arrows by juxtaposition in diagrammatic order, except in a few places where we explicitly write the composition symbol $\circ$ to mean $g \circ f = fg$.} 
there is a unique arrow $\beta : \push{f}S \to T$ such that $\alpha = \opcart f S \; \beta$. 
\begin{equation}\label{eq:unique-beta-push}
  \begin{tikzcd}
    S \ar[rr,"\alpha"] && T \\
    A \ar[r,"f"] & B \ar[r,"g"] & C
  \end{tikzcd}
  \quad=\quad
  \begin{tikzcd}
    S \ar[r,"\opcart f S"] &\push{f}S\ar[r,"{\beta}",dashed]  & T \\
    A \ar[r,"f"] & B \ar[r,"g"] & C
  \end{tikzcd}
\end{equation}
We refer to the object $\push f S$ as the \emph{pushforward} of $S$ along $f$, and to the arrow $\opcart f S : S \to \push f S$ as a \emph{$+$-cartesian lifting} of $f$ to $S$.
Note that the universal property of $+$-cartesian liftings ensures that the pushforward $\push f S$ is determined up to unique isomorphism lying over the identity arrow $\Id[B]$, which is what justifies speaking of ``the'' pushforward of $S$ along $f$.
Dually, for any arrow $g : B \to C$ in $\C$ and object $T$ in $\D$ such that $p(T) = C$,
there should be an object $\pull{g}T$ and an arrow $\cart g T : \pull{g}T \to T$ of $\D$ such that $p(\cart g T) = g$,
\[
  \begin{tikzcd}
    \pmb{\pull{g}T} \ar[r,"\cart g T",dashed] & T \\
    B \ar[r,"g"] & C
  \end{tikzcd}
\]
again universal in the sense that for any arrow $f : A \to B$ in $\C$ and arrow $\alpha : S \to T$ in $\D$ such that $p(\alpha) = fg$, there is a unique arrow $\beta : S \to \pull{g}T$ such that $\alpha = \beta \; \cart g T$. 
\begin{equation}\label{eq:unique-beta-pull}
  \begin{tikzcd}
    S \ar[rr,"\alpha"] && T \\
    A \ar[r,"f"] & B \ar[r,"g"] & C
  \end{tikzcd}
  \quad=\quad
  \begin{tikzcd}
    S \ar[r,"\beta",dashed] &\pull{g}T\ar[r,"\cart g T"]  & T \\
    A \ar[r,"f"] & B \ar[r,"g"] & C
  \end{tikzcd}
\end{equation}
We refer to the object $\pull g T$ as the \emph{pullback} of $T$ along $g$, and to the arrow $\cart g T : \pull g T \to T$ as a \emph{$-$-cartesian lifting} of $g$ to $T$.
Again, the universal property of $-$-cartesian liftings ensures that the pullback $\pull g T$ is determined up to unique isomorphism lying over $\Id[B]$.

An immediate consequence of the definitions is that if $p : \D \to \C$ is a bifibration
then there is a one-to-one correspondence
\begin{equation}\label{eq:bif-equivs}
  \begin{tikzcd}
    S \ar[r] & \pull{f}T \\
    A \ar[r,"{\Id[A]}"] & A
  \end{tikzcd}
  \quad\iff\quad
  \begin{tikzcd}
    S \ar[r] & T \\
    A \ar[r,"f"] & B
  \end{tikzcd}
  \quad\iff\quad
  \begin{tikzcd}
    \push{f}S \ar[r] & T \\
    B \ar[r,"{\Id[B]}"] & B
  \end{tikzcd}
\end{equation}
of arrows in $\D$ living over the corresponding arrows in $\C$.
As a consequence, every arrow $f : A \to B$ of $\C$ induces an adjunction 
\begin{equation*}
  \begin{tikzcd}[column sep=large]
    \D_A \ar[r,"\push{f}"{name=0},bend left] & \D_B \ar[l,"\pull{f}"{name=1},bend left]
    \ar[from=0,to=1,phantom,"\bot"]
  \end{tikzcd}
\end{equation*}
realized by the operations of pushing and pulling along $f$, where $\D_A$ and $\D_B$ are the \emph{fiber categories} of $A$ and $B$ relative to the functor $p$, defined as the subcategories of $\D$ spanned by the arrows living over the identities $\Id[A]$ and $\Id[B]$ in $\C$.
More precisely, the functors $\push{f}$ and $\pull{f}$ can be defined given any specific choice of cartesian liftings in $\D$, corresponding to a choice of bifibrational structure on $p : \D \to \C$.
Such a bifibrational structure is called a ``cleavage'' or ``cleaving'', and $p$ is said to be a \emph{cloven} bifibration.
We will always be working with bifibrational structures in this paper and so we usually omit the adjective ``cloven''. 

The categorical notion of bifibration was originally introduced by Grothendieck~\cite[VI]{SGA1}, 
together with \emph{fibration} and \emph{cofibration}, the latter nowadays more commonly called \emph{opfibration}.
A functor $p : \D \to \C$ is a fibration (respectively opfibration) when one can pull (respectively, push) objects of $\D$ along arrows of $\C$, so that $p$ is a bifibration iff it is both a fibration and an opfibration.
As established by Grothendieck and which motivated the original definition, the data of a cloven bifibration $p : \D \to \C$ is equivalent to the data of a pseudofunctor $F : \C \to \Adj$ into the 2-category of small categories and adjunctions.
A bifibration over $\C$ is recovered from $F$ by defining the \emph{category of elements} $\int_\C F$ (also called the \emph{Grothendieck construction}) to have objects given by pairs $(A,S)$ of an object $A \in \C$ an an object $S \in F(A)$, and arrows $(A,S) \to (B,T)$ given by pairs $(f,\alpha)$ of an arrow $A \to B \in \C$ and an arrow $\alpha : \push{F(f)}(S) \to T \in F(B)$ or equivalently an arrow $\alpha : S \to \pull{F(f)}(T) \in F(A)$, where we write $\push{F(f)} : F(A) \to F(B)$ and $\pull{F(f)} : F(B) \to F(A)$ respectively for the left and right adjoint associated to the adjunction $F(f)$.
The category of elements is equipped with an evident projection functor $\pi_F : \int_\C F \to \C$, which is a bifibration.

We are interested in the problem of constructing the free bifibration on a functor, meaning the following.
Let $p : \D \to \C$ be an arbitrary functor.
The \emph{free bifibration on $p$} is a category $\Bifib{p}$ equipped with a functor $\BifibFun{p} : \Bifib{p} \to \C$ that is a bifibration, together with a functor $\eta_p : \D \to \Bifib{p}$ commuting with the projections to $\C$:
\begin{equation*}\label{diagram:free-bifibration}
\begin{tikzcd}
  \D \arrow[dr,"p"']\arrow[rr,"\eta_p"] && \Bifib{p}\arrow[dl,"{\BifibFun{p}}"] \\
  &\C& 
\end{tikzcd}
\end{equation*}
Moreover, this data is universal in the sense that for any other bifibration $q : \E \to \C$ equipped with a functor $\theta : \D \to \E$ such that $q\circ\theta = p$,
there is a unique morphism of bifibrations making the diagram below commute:
\begin{equation}\label{eq:free-bifibration-tetra}
\Tetrahedron
\end{equation}
By \emph{morphism of bifibrations,} we mean a functor $\sem{-}{\theta} : \Bifib{p} \to \E$ sending chosen $\pm$-cartesian liftings to chosen $\pm$-cartesian liftings.
A standard argument implies the universal property of $\BifibFun{p} : \Bifib{p} \to \C$ extends to bifibrations over any base category into which $\C$ can be mapped, in the sense that $\sem{-}\theta$ below is uniquely determined from $\theta : \D \to \E$, $q : \E \to \B$, and $\delta : \C \to \B$:
\begin{equation}\label{eq:free-bifibration-pyramid}
\begin{tikzcd}[ampersand replacement=\&]
  \D \arrow[ddrr,"p"']\arrow[rr,"\eta_p"]\arrow[rrrd,"\theta"'] \& \&\Bifib{p}\arrow[dd,"\BifibFun{p}"]\arrow[dr,"\sem{-}\theta"]\&  \\
  \& \&  \& \E\arrow[dd,"q\text{ bifibration}"]\\
  \& \& \C\ar[rd,"\delta"'] \&  \\
  \& \& \&  \B
\end{tikzcd}
\end{equation}
In particular, \eqref{eq:free-bifibration-pyramid} can be reduced to \eqref{eq:free-bifibration-tetra} by first pulling back $q$ along $\delta$ to a bifibration over $\C$, using the fact that bifibrations are closed under pullback along arbitrary functors.

Whereas the free fibration $\C\downarrow p \to \C$ and the free opfibration $p\downarrow\C\to \C$ on a functor have simple and well-known descriptions, the free bifibration has been relatively little studied, and is a more complex object.
As far as we are aware there is only one direct construction in the literature, in an unpublished manuscript by François~Lamarche \cite{Lam13}.
This construction was motivated by his work on a model of type theory in which dependent types were Grothendieck bifibrations \cite{Lam14}, and identity types were modelled as bifibrations $\mathrm{P}(\C) \to \C\times\C$ from a certain ``path category'', of which a detailed analysis was given in the manuscript~\cite{Lam13}.
Otherwise, the problem of building the free bifibration on a functor is also closely related to the problem studied by Dawson, Paré, and Pronk \cite{DPP03,DPP03b} of freely adjoining right adjoints to a category, which may be seen as a weaker version of building the free groupoid over a category.
Their construction, which generalizes Schanuel and Street's definition of the \emph{free adjunction} \cite{SchanuelStreet1986}, takes a category $\C$ and embeds it into a 2-category $\Pi_2(\C)$ in which every arrow of the original category is equipped with a right adjoint.

The connection between these problems is revealed by taking the free bifibration on the identity functor $\Id[\C]$: as we will see, this defines a category $\Z(\C) := \Bifib{\Id[\C]}$ bifibered over $\C$ that may be equipped with the structure of a double category $\ZZ(\C) := \Z(\C) \rightrightarrows \C$.
We call $\ZZ(\C)$ the \emph{zigzag double category of $\C$}, since its vertical arrows (corresponding to the objects of $\Z(\C)$) are zigzags of arrows in $\C$.
In fact, $\Z(\C)$ is equivalent to Lamarche's path category $\mathrm{P}(\C)$, and as he observed \cite[p.22]{Lam13}, any free bifibration may be recovered from this double category by the formula $\BifibFun{p} \equiv \push{\tgt}\pull{\src}p$, or diagrammatically:
\begin{equation}\label{eq:bifibfun-as-pullback-of-zigzags}
\begin{tikzcd}
\D
\arrow[d, "p"']
& \Bifib{p}
\arrow[l]
\arrow[d]
\arrow[rd, "\BifibFun{p}"]
\arrow[ld, phantom, "\llcorner" very near start]
& \\
\C
& \Z(\C)
\arrow[l, "\src"]
\arrow[r, "\tgt"']
& \C
\end{tikzcd}
\end{equation}
In turn, taking the underlying vertical 2-category of $\ZZ(\C)$ recovers a 2-category equivalent to Dawson, Paré, and Pronk's $\Pi_2(\C)$, thus reducing the problem of freely adjoining right adjoints to the problem of understanding the free bifibration on the identity functor.

Our approach to the problem differs from prior work in that we study it from a proof-theoretic perspective.
It is natural to interpret the pushforward $\push{f}$ and pullback $\pull{g}$ operations as unary logical connectives, abstractions of existential and universal quantification as well as of diamond and box modalities.
The rules of this primitive ``bifibrational logic'' may be presented by a simple sequent calculus, leading us to a description of the free bifibration wherein the objects of $\Bifib{p}$ are formulas and its arrows are equivalence classes of sequent calculus derivations.
See Figure~\ref{fig:example-functor-and-derivation} for an example derivation in this sequent calculus parameterized by a specific functor $p : \D \to \C$, representing an arrow $\beta : \pull{f}\push{f} X \to \pull{h}\push{g}Y$ in $\Bifib{p}$ such that $\BifibFun{p}(\beta) = \Id[A]$.
The close connection with the zigzag double category also leads to an isomorphic representation of derivations as stacks of double cells in $\ZZ(\C)$, as well as to a string diagram calculus.

\begin{figure}  [b]
\begin{minipage}{.4\textwidth}
  \begin{tikzpicture}
  \node[draw, thick, inner sep=5pt] {    
  \begin{tikzcd}
    \D\ar[d,"p"] &    X \ar[r,"\alpha"] & Y &  \\
    \C &    A \ar[r,"f"]\ar[rr,"h"'{name=h},bend right] & B \ar[r,"g"] & C
    \ar[from=2-3,to=h,phantom,yshift=1pt,"="]
  \end{tikzcd}
  };
\end{tikzpicture}
\vspace*{1em}

\[
\inferrule*[fraction={---},Right={$\Rpull[h]$}]{
\inferrule*[fraction={===}]{
\inferrule*[fraction={---},Right={$\Lpull[f]$}]{
\inferrule*[Right={$\Lpush[f]$}]{
\inferrule*[Right={$\Rpush[g]$}]{
\inferrule*[Right={$\atom{\alpha}$}]{ }
{\atom X \vdashf{f} \atom Y}}
{\atom X \vdashf{fg} \push{g} \atom Y}}
{\push{f} \atom X \vdashf{g} \push{g} \atom Y}}
{\pull{f} \push{f} \atom X \vdashf{fg} \push{g} \atom Y}}
{\pull{f} \push{f} \atom X \vdashf{h} \push{g} \atom Y}}
{\pull{f} \push{f} \atom X \vdashf{\Id[A]} \pull{h}\push{g} \atom Y}
\]

\end{minipage}\hfill \vrule \hfill
\begin{minipage}{0.3\textwidth}
\begin{tikzcd}[
  execute at end picture = {
    \node (alphanode) [shape=circle,fill=grullo,scale=0.75,opacity=0.50] at (alpha) {};
    \node (eq) [shape=circle,fill=goldenpoppy,scale=0.75,opacity=0.50] at ($(m4.center)!0.65!(m5.center)$) {};
    \draw[arrh, uedge, draw opacity=0.40] ([yshift=-3pt]r6.west) to [out=180,in=-90] (eq);
    \draw[arrf, uedge, draw opacity=0.50] (eq) to [out=110,in=0] (l3.east);
    \draw[arrg, uedge, draw opacity=0.50] (eq) to [out=70,in=200] (r1.west);
    \draw[arrf, uedge, draw opacity=0.50] (l2.east) to [out=10,in=-95] (alphanode);
  }  ]
  X\ar[r,"\alpha"{name=alpha}] & Y \\[-15pt]
  A\ar[r,"f"{name=m1}]\ar[d,equals] & B\ar[d,"g"{name=r1}] \\
  A\ar[r,"fg"{name=m2}]\ar[d,"f"'{name=l2}] & C \ar[d,equals] \\
  B\ar[r,"g"{name=m3}] & C \\
  A\ar[u,"f"{name=l3}]\ar[r,"fg"{name=m4}] & C\ar[u,equals,""'{name=r3}] \\
  A\ar[r,"h"{name=m5}]\ar[u,equals,""{name=l4}] & C\ar[u,equals] \\
  A\ar[u,equals]\ar[r,equals,""{name=m6}] & A\ar[u,"h"'{name=r6}] \\
\end{tikzcd}
\end{minipage}
        \caption{Upper left: an example functor, depicting objects and arrows of the category above over their images below. Lower left: a sequent calculus proof representing an arrow of $\Bifib{p}$. Right: the same proof depicted as a stack of double cells in $\ZZ(\C)$ acting on an arrow in $\D$, overlaid with the corresponding string diagram.}
        \label{fig:example-functor-and-derivation}
\end{figure}
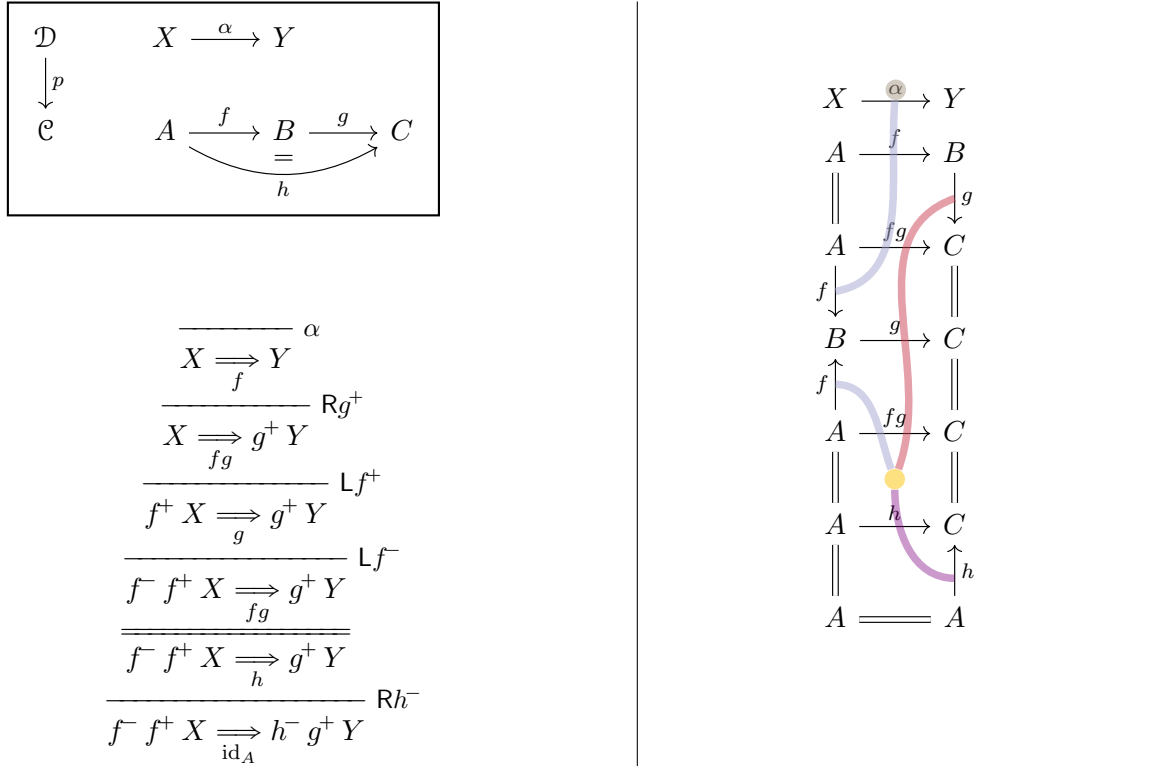

It should be emphasized that we do \emph{not} postulate the Beck-Chevalley condition.
Although bifibrations satisfying the BC condition are commonly used in categorical logic after Lawvere \cite{Lawvere1970} to model extensions of predicate logic with either substitution and existential quantification or substitution and universal quantification (see, e.g., Jacobs~\cite{Jacobs2001CLTT}), we are rather interested in free bifibrations where the BC condition fails.
These model more general situations including ones where the adjoint pair of pushforward and pullback operations $\push{f} \dashv \pull{f}$ themselves represent generalized existential and universal quantifiers $\exists_f \defeq \push{f}$, $\forall_f \defeq \pull{f}$, and hence by alternating these operations one can in some sense define objects of arbitrary quantifier complexity.

One of our original motivations for studying this problem from this perspective was the observation, made independently by Melliès and Zeilberger~\cite{MZ2015popl,MZ2016lawvere,MZ2018isbell} and by Licata, Shulman, and Riley~\cite{LSR2017}, that substructural and modal logics may be naturally modelled in certain bifibrations: specifically bifibrations of monoidal closed categories or multicategories in which the base moreover contains an object with some algebraic structure.
As a typical example, any monoid $(A,m,e)$ in the base of a bifibration of multicategories $p : \E \to \B$ generates a monoidal closed fiber category $\E_A$, with the tensor product and internal hom defined by pushing and pulling along the multiplication map $m : A,A \to A$. 
For more recent work in this spirit see Nicolas~Blanco's PhD thesis~\cite{Blanco2023phd} developing the theory of bifibrations of polycategories as models of classical linear logic, as well as Shulman~\cite{Shulman2023lnl} on LNL~doctrines, which axiomatize a wide variety of categorical structures and type theories.
The problem of constructing free bifibrations may therefore be seen as a natural problem in proof theory (cf.~\cite[\S8]{Shulman2023lnl}), and we were motivated to study the ``pure'' version of the problem for functors of categories as an eventual stepping stone to the more sophisticated case of monoidal closed categories or generalized multicategories and polycategories.
In the present paper we do already consider a simple but practically significant generalization of the problem, namely to free $(\P,\N)$-fibrations where one only asks for $+$-cartesian liftings (resp. $-$-cartesian liftings) of the arrows in $\P$ (resp. $\N$).
In particular when the pair $(\P,\N)$ form a factorization system for the base category this defines what Joachim Kock and André Joyal refer to as an \emph{ambifibration}.

As it turns out, even very simple functors can generate free bifibrations (and free ambifibrations) with surprisingly rich combinatorial structure.
For instance, let $p_\Two : \One \to \Two$ be the functor from the terminal category to the interval category sending the unique object of $\One$ to the initial object of $\Two$:
\[
\begin{tikzcd}
  \One \ar[d,"p_\Two"'] & \ast & \\
  \Two  & 0 \ar[r,"f"] & 1
\end{tikzcd}
\]
Now consider the free bifibration generated by $p_\Two$.
It is easy to see that any object of $\Bifib{p_\Two}$ must be isomorphic to an alternating sequence of pushes and pulls along the arrow $f$ starting at the point $\ast$, with even-length sequences yielding an object in the fiber over 0, and odd-length sequences an object in the fiber over 1.
Let us write $\ord{n} \defeq (\pull{f}\push{f})^n\atom\ast$ and $\ord{n}' \defeq \push{f}(\pull{f}\push{f})^n\atom\ast$ for the objects in the respective fibers (e.g., $\ord{3} = \pull{f}\push{f}\pull{f}\push{f}\pull{f}\push{f}\atom\ast$).
From the fact that the morphisms in the fibers are generated by the unit and counit of the push-pull adjunction $\push{f}\dashv \pull{f}$, with a bit of mental exercise (or familiarity with Schanuel and Street's paper~\cite{SchanuelStreet1986}, building on Auderset's earlier work~\cite{Auderset1974}) one can check that an arrow $\ord{m} \to \ord{n}$ in  $\Bifib{p_\Two}_0$ should correspond to an order-preserving map $\set{1,\dots,m} \to \set{1,\dots,n}$.
Indeed, $\Bifib{p_\Two}_0\cong \Delta$ is equivalent to the (augmented) simplex category  of finite ordinals and order-preserving maps, while $\Bifib{p_\Two}_1\cong \Delta_\bot$ is equivalent to the category  of non-empty finite ordinals and order-and-least-element-preserving maps.
The push-pull adjunction of the free bifibration instantiates to the free / forgetful adjunction
\begin{equation*}
  \begin{tikzcd}[column sep=4em]
   \Delta \ar[r,"L"{name=0},bend left] & \Delta_\bot \ar[l,"R"{name=1},bend left] \ar[from=0,to=1,phantom,"\bot"]
  \end{tikzcd}
\end{equation*}
where $L$ freely adds a least element $\bot$ to turn an order-preserving map $\ord{m} \to \ord{n}$ into an order-and-least-element-preserving map $\ord{m}'\to \ord{n}'$, while the right adjoint $R$ interprets $\ord{m}'$ as $\ord{1{+}m}$ and forgets that a map is $\bot$-preserving.
This much may be unshocking to readers familiar with the free adjunction~$\A$, which can be obtained by freely adjoining adjoints to the interval category, $\A = \Pi_2(\Two)$, as Dawson, Paré, and Pronk already observed~\cite[p.139]{DPP03}---although we emphasize the free bifibration is another way of organizing the data.

More strikingly, this example generalizes in the following way.
Replace the interval category $\Two$ by the ordinal $\omega$ (that is, the total order on the natural numbers) and consider the free bifibration generated by the functor $p_\omega : \One \to \omega$ mapping the point to 0, as depicted below:
\[
\begin{tikzcd}
  \One \ar[d,"p_\omega"'] & \ast & \\
  \omega & 0 \ar[r,"f_0"] & 1 \ar[r,"f_1"] & 2 \ar[r,"f_2"] & \cdots
\end{tikzcd}
\]
What is the fiber of $\Bifib{p_\omega}$ over 0?
Objects are isomorphic to sequences of pushes and pulls along $f_i$'s, such that, reading a formula from right to left, the running total of pulls never exceeds the running total of pushes, and the total numbers of pushes and pulls are equal.
In other words, up to isomorphism the objects of $\Bifib{p_\omega}_0$ may be read as \emph{Dyck words}.
Via a well-known bijection (see Stanley~\cite[V2:169--170]{StanleyEC1+2}), Dyck words are in one-to-one correspondence with rooted plane trees, and thus $\Bifib{p_\omega}_0$ may be interpreted as a category of trees.
Under a natural encoding of plane trees as functors $T : \omega^\op \to \Delta$, it turns out that this fiber category of the free bifibration is equivalent to a full subcategory of $[\omega^\op, \Delta]$, and embeds as a wide subcategory into a category of finite plane trees defined by Joyal \cite{Joyal1997}.

An example of a different nature is obtained by considering the free $(\Delta_\epi,\Delta_\mono)$-fibration generated by the inclusion $i : \NN \to \Delta$ of the set of natural numbers (seen as a discrete category) into the simplex category.
This is a free ambifibration, and its universal property ensures that it is equipped with a canonical morphism into the so-called ``fat~Delta'' category introduced by Kock~\cite{Kock2006weak}.
The category $\Bifib{i,\Delta_\epi,\Delta_\mono}$ appears to have a rich combinatorial structure that we only begin to analyze here.
For example, the \emph{lattices of noncrossing partitions} (see Stanley \cite[V1:515, V2:226]{StanleyEC1+2}) are recovered by quotienting this free ambifibration by the Beck-Chevalley condition for bicartesian squares.

\paragraph{Outline of the paper} After introducing the sequent calculus in Section~\ref{sec:sequent-calculus} and after defining and analyzing the double category of zigzags in Section~\ref{sec:double-categories}, the technical core of the paper is in Section~\ref{sec:focusing}, where we establish a series of progressively stronger normal form results, guided by ideas from proof theory.
The strongest of these (Theorem~\ref{thm:fp-unique-normal-forms}), which holds under the assumption that the base category $\C$ is \emph{factorization preordered} (\FP), leads to an inductive definition of the relative homsets of the free bifibration that does not use any kind of quotient by an equivalence relation, representing arrows by \emph{maximally multifocused} derivations.
The \FP condition also showed up in Dawson, Paré, and Pronk's work on the $\Pi_2(\C)$ construction, where they observed that equality of 2-cells is undecidable in general but is decidable when $\C$ is factorization preordered \cite{DPP03b}.
As a corollary of our canonicity theorem, we obtain both an analogous decidability result for equality of arrows in $\Bifib{p}$ as well as a procedure for enumerating relative homsets without duplicates.
Section~\ref{sec:PN-fibrations} briefly describes how our constructions and results adapt to free $(\P,\N)$-fibrations.
This is technically straightforward but has the useful consequence that one can weaken the assumption needed for canonicity, only requiring that the base category be factorization preordered relative to each class $\P$ and $\N$ independently.
Finally, in Section~\ref{sec:examples} we use the tools developed in the paper to analyze the three examples described above.
Section~\ref{sec:conclusion} concludes and suggests some directions for further exploration.

\paragraph{Remark} This work mixes domains and techniques that will be familiar to different audiences, and we expect most readers to be unfamiliar with some of them. We tried to find a presentation that would be accessible to people only familiar but not expert in category theory, and people unfamiliar with sequent calculus and proof theory. We ask the reader for forgiveness if they find the exposition too basic in places, for the small doses of redundancy this introduces, and for the length of the resulting presentation.

\section{Sequent calculus}
\label{sec:sequent-calculus}

In this section we will explain how the free bifibration $\BifibFun{p} : \Bifib{p} \to \C$
on an arbitrary functor $p : \D \to \C$
may be constructed proof-theoretically, as an inductively-defined sequent calculus for a primitive ``bifibrational logic'', starting from the data of $p$.
The idea will be to take the objects of $\Bifib{p}$ to be formulas of the sequent calculus and its arrows to be proofs modulo a notion of \emph{permutation equivalence} to be defined below.
This sequent calculus is cut-free, but the cut rule is admissible, giving a definition of composition of arrows in the category $\Bifib{p}$.
We will establish a series of properties leading to the main result of this section (Theorem~\ref{thm:bif-is-free}), that the sequent calculus indeed gives a presentation of the free bifibration on $p$.

\subsection{Formulas and unary sequents}
\label{sec:sequent-calculus:frm-seq}

To define the formulas of the logic, we find it helpful to introduce the judgment $S \refs A$ to mean that $S$ is a formula (= object of $\Bifib{p}$) lying over (or ``refining'') the object $A$ of $\C$.
We usually refer to formulas simply as ``formulas'', but sometimes for clarity we refer to them as \emph{bifibrational formulas,} being formulas of bifibrational logic.
Bifibrational formulas are defined inductively by the following rules:
\begin{mathpar}
  \inferrule {X \in \D \\ p(X) = A}
             {\atom{X} \refs A}

  \inferrule {S \refs A \\ f : A \to B}
             {\push{f}S \refs B}

  \inferrule {g : B \to C \\ T \refs C}
             {\pull{g}T \refs B}
\end{mathpar}
Here we are using standard notation for inference rules: the judgment below the horizontal bar is the \emph{conclusion} of the inference rule, and the judgments above the bar are its \emph{premises}. Each rule asserts that if its premises are valid, then its conclusion is valid.
We refer to formulas of the form $\atom{X}$ as \emph{atomic formulas}, and formulas of the form $\push f S$ or $\pull g T$ as \emph{push formulas} and \emph{pull formulas} respectively.
All of the formulas $S'$ appearing in the inductive construction of a bifibrational formula $S$ are said to be \emph{subformulas} of $S$.
We denote the subformula relation $S' \preceq S$, and the strict subformula relation $S' \prec S$.

In set-theoretic terms, each object $A$ of $\C$ determines a set $\Frm{A} = \set{S \mid S \refs A}$ of formulas lying over $A$, and the collection of sets of formulas $(\Frm{A})_{A \in \C}$ is the least family of sets closed under the following three conditions:
\begin{itemize}
\item if $X \in \D$ and $p(X) = A$ then $\atom{X} \in \Frm{A}$;
\item if $S \in \Frm{A}$ and $f : A \to B$ is an arrow of $\C$ then $\push{f}S \in \Frm{B}$;
\item if $T \in \Frm{C}$ and $g : B \to C$ is an arrow of $\C$ then $\pull{g}T \in \Frm{B}$.
\end{itemize}

Besides the judgment $S \refs A$ describing well-formedness of formulas, the main judgment of interest of the calculus is the \emph{unary sequent arrow} of the form
\begin{equation}\label{unary-sequent}
S \vdashf{h} T
\end{equation}
where $S \refs A$ and $T\refs C$ are formulas lying over objects $A$ and $C$ respectively, and where $h : A \to C$ is an arrow of $\C$.
For comparison, it is worth keeping in mind traditional sequents of intuitionistic sequent calculus, which take the form
\[ P_1, \dots, P_n \vdashf{} Q \]
(often written with $\vdash$ in place of an arrow) and which may be read as asserting an entailment from the conjunction of the formulas on the left-hand side to the formula on the right-hand side.
In the bifibrational calculus, unary sequents of the form \eqref{unary-sequent} may be read as asserting an entailment from $S$ to $T$ ``fibered over $h$''.
We refer to $S$ as the \emph{left side,} to $T$ as the \emph{right side,} and to $h$ as the \emph{base} of the sequent.
The definition of the sequent calculus will ensure that a proof of the judgment $S \vdashf{h} T$ corresponds to an arrow $\alpha : S \to T$ in $\Bifib{p}$ such that $\BifibFun{p}(\alpha) = h$.

\subsection{Inference rules and derivations}
\label{sec:sequent-calculus:deriv}

The calculus includes four logical inference rules:
\begin{mathpar}
  \inferrule* [Right={$\Lpush$}]
        {S \vdashf{fg} T}
        {\push{f} S \vdashf{g} T}

  \inferrule* [Right={$\Rpush$}]
        {S' \vdashf{f'} S}
        {S' \vdashf{f'f} \push{f} S}

  \inferrule* [Right={$\Lpull[g]$}]
        {T \vdashf{g'} T'}
        {\pull{g} T \vdashf{gg'} T'}

  \inferrule* [Right={$\Rpull[g]$}]
        {S \vdashf{fg} T}
        {S \vdashf{f} \pull{g} T}
\end{mathpar}
together with an \emph{initial axiom} for every arrow of $\D$:
\[
  \inferrule*[Right={$\atom{\delta}$}]{\delta : X \to Y \in \D \\ p(\delta) = f}{\atom{X} \vdashf{f} \atom{Y}}
\]
In general in sequent calculus, one can define a \emph{derivation} to be a finite rooted tree whose edges are labelled by sequents, and whose nodes are labelled by valid instances of the inference rules.
Moreover, one can distinguish between \emph{open} and \emph{closed} derivations, where open derivations have premises corresponding to open edges of the tree.
Another name for a closed derivation is a \emph{proof}.
A special property of the bifibrational sequent calculus is that every inference rule has at most one premise, so that derivations are completely linear (i.e., they are just \emph{lists} rather than proper trees), and a derivation is closed just in case it is terminated by a unique initial axiom.

The overall format of the rules follows the tradition of sequent calculus with left rules and right rules for every connective, which in this case are the pushforward $\push{f}$ and pullback $\pull{g}$ connectives parameterized by arrows $f$ and $g$ of $\C$.

In set-theoretic terms, the rules may be seen as specifying closure conditions that inductively define a family of sets of proofs of judgments.
For example, the $\Lpush$ rule says that for any formulas $S$ and $T$ and arrows $f$ and $g$, given a proof of $S \vdashf{fg} T$ there should be a proof of $\push{f}S\vdashf{g} T$.
Note that the rules have implicit side-conditions that the judgments are well-formed.
For example, in the $\Lpush$ rule, the arrows $f$ and $g$ should be composable, with $f : A \to B$ and $g : B \to C$ for some objects $A,B,C$ such that $S \refs A$ and $T \refs C$.

We refer the reader back to Figure~\ref{fig:example-functor-and-derivation} from the Introduction, which gives an example derivation in the sequent calculus relative to a simple functor.

Finally, we emphasize that the underlying base arrows are more than a mere annotation on the sequent judgments:
they really do restrict the applicability of the inference rules.
For example, applying the $\Lpush$ rule to a derivation of a judgment $S \vdashf{h} T$ requires factoring the arrow as a composite $h = fg$.

We do not include explicit identity and cut rules in the definition of the sequent calculus, 
\begin{mathpar}
\inferrule*[Right={$\Id[S]$}]
        {S \refs A}
        {S \vdashf{\Id[A]} S}

\inferrule*[Right={$\cut$}]
        {S \vdashf{f} U \\ U \vdashf{g} T}
        {S \vdashf{fg} T}
\end{mathpar}
but these rules are \emph{admissible} in the proof-theoretic sense as we will explain below.
Since derivations are cut-free, they satisfy a variation of the standard subformula property:
\begin{prop}[Subformula property]\label{prop:subformula-property}
A derivation of $S \vdashf{f} T$ involves only sequents of the form $S' \vdashf{g} T'$ where $S' \preceq S$ and $T' \preceq T$ are subformulas of $S$ and $T$.
\end{prop}

We now define closed derivations corresponding to the identity rule below, deferring the treatment of cut to Section~\ref{sec:sequent-calculus:cut}.

\begin{defi}[identity]\label{defi:identity} For any formula $S \refs A$, there is a proof $\id_S$ of $S \vdashf{\id_A} S$ constructed by induction on $S$ as follows:
\[
\inferrule {X \in \D \\ p(X) = A}
           {\atom{X} \refs A}
\qquad\implies\qquad
\Id[\atom{X}]
\quad\defeq\quad
\inferrule* [Right={$\atom{\Id[X]}$}] { } {\atom{X} \vdashf{\Id[A]} \atom{X}}
\]

\[
\inferrule {S \refs A \\ f : A \to B}
           {\push{f}S \refs B}
\qquad\implies\qquad
\Id[{\push{f} S}]
\quad\defeq\quad
\inferrule*[Right={$\Lpush$}]{
\inferrule*[Right={$\Rpush$}]{
\inferrule*[Right={$\id_S$}]{ }
{S \vdashf{\Id[A]} S}}
{S \vdashf{f} \push{f} S}}
{\push{f} S \vdashf{\Id[B]} \push{f} S}
\]

\[
\inferrule {g : B \to C \\ T \refs C}
           {\pull{g}T \refs B}
\qquad\implies\qquad
\Id[\pull{g} T]
\quad\defeq\quad
\inferrule*[Right={$\Rpull[g]$}]{
\inferrule*[Right={$\Lpull[g]$}]{
\inferrule*[Right={$\id_T$}]{ }
{T \vdashf{\Id[C]} T}}
{\pull{g} T \vdashf{g} T}}
{\pull{g} T \vdashf{\Id[B]} \pull{g} T}
\]\qed
\end{defi}

Observe that proofs corresponding to the unit and counit of the canonical family of adjunctions $\push{f} \dashv \pull{f}$ of a bifibration may be constructed as 
\begin{mathpar}
  \inferrule*[Right={$\Rpull$}]{
  \inferrule*[Right={$\Rpush$}]{
  \inferrule*[Right={$\id_S$}]{ }
  {S \vdashf{\Id[A]} S}}
  {S \vdashf{f} \push{f} S}}
  {S \vdashf{\Id[A]} \pull{f}\push{f} S}

  \inferrule*[Right={$\Lpush$}]{
  \inferrule*[Right={$\Lpull$}]{
  \inferrule*[Right={$\id_U$}]{ }
  {U \vdashf{\Id[B]} U}}
  {\pull{f}U \vdashf{f} U}}
  {\push{f}\pull{f}U \vdashf{\Id[B]} U}
\end{mathpar}
where we make use of the identity derivations from Definition~\ref{defi:identity} to terminate the proofs.

\subsection{A term syntax for proofs}
\label{sec:sequent-calculus:term-syntax}

A term syntax is a compact notation to describe derivations that can be convenient for defining some operations and equations.
We introduce here a syntax of terms generated by the following grammar:
\begin{mathpar}
  \begin{array}{rcl@{\qquad}r}
    \alpha & ::= &
      \atom \delta & \text{for $\delta : X \to Y \in \D$}
    \\
    & \mid & \uLpush f g \alpha
    \\
    & \mid & \alpha \tRpush f
    \\
    & \mid & \tLpull g \alpha
    \\
    & \mid & \alpha \uRpull f g
  \end{array}
\end{mathpar}
which are assigned to derivations as shown in Figure~\ref{fig:term-annotated-rules}.

\begin{figure}
\begin{mathpar}
  \inferrule
        {\delta : X \to Y \in \D \\ p(\delta) = f}
        {\atom \delta : X \vdashf{f} Y}
\\
  \inferrule
        {\alpha : S \vdashf{fg} T}
        {\uLpush f g \alpha : \push{f} S \vdashf{g} T}

  \inferrule
        {\alpha : S' \vdashf{f'} S}
        {\alpha \tRpush{f} : S' \vdashf{f'f} \push{f} S}
\\
  \inferrule
        {\alpha : T \vdashf{g'} T'}
        {\tLpull{g} \alpha : \pull{g} T \vdashf{gg'} T'}

  \inferrule
        {\alpha : S \vdashf{fg} T}
        {\alpha \uRpull{f}{g} : S \vdashf{f} \pull{g} T}
\end{mathpar}
\caption{Term-annotated inference rules of the bifibrational calculus.}      
\label{fig:term-annotated-rules}
\end{figure}

Formally, this defines a new judgment $\alpha : S \vdashf{h} T$ whose derivations are in one-to-one correspondence with derivations of $S \vdashf{h} T$. In other words, for a judgment $S \vdashf{h} T$, a derivation of the judgment uniquely determines a term $\alpha$ and a derivation of $\alpha : S \vdashf{h} T$. We will only consider \emph{valid} terms, that do correspond to a derivation, and will use the term $\alpha$ and the underlying derivation interchangeably.

We sometimes call \emph{multiplications} the term-formers
$\alpha \tRpush f$ and $\tLpull g \alpha$ and \emph{divisions} the
term-formers $\uLpush f g \alpha$ and $\alpha \uRpull f g$. The reason why our division notation takes an explicit subscript $f$ or $g$ is that its meaning depends on a choice of decomposition of the arrow $h$ underlying $\alpha$ as $h = fg$; there is no such ambiguity for multiplications.

\paragraph{Examples} The terms for the identity derivations can be defined recursively as follows:
\begin{mathpar}
  \Id[\atom X] \defeq \atom {\Id [X]} \in \D

  \Id[\push f S] \defeq \uLpush f {\Id[B]} (\Id[S] \tRpush f)

  \Id[\pull g T] \defeq (\tLpull g \Id[T]) \uRpull {\Id[B]} g
\end{mathpar}
and the unit and co-unit derivations as follows:
\begin{mathpar}
  \eta_S \defeq (\Id[S] \tRpush f) \uRpull {\Id[A]} f

  \epsilon_T \defeq \uLpush f {\Id[B]} (\tLpull f \Id[U])
\end{mathpar}
These examples use subterms of the form $\Id[S] \tRpush f$ and $\tLpull f \Id[S]$, which we will later show correspond to cartesian liftings in the free bifibration:    
\begin{mathpar}
  \opcart f S : S \vdashf{f} \push f S \quad\defeq\quad \Id[S] \tRpush f

  \cart g T : \pull g T \vdashf{g} T \quad\defeq\quad \tLpull g \Id[T]
\end{mathpar}

\subsection{Interpreting formulas and proofs in any bifibration on $p$}
\label{sec:sequent-calculus:interpretation}

In this section we suppose given an arbitrary bifibration $q : \E \to \C$ equipped with a functor
$\theta : \D \to \E$ such that $q \circ \theta = p$.
As explained in the Introduction, we expect $\theta$ to factor uniquely via the free bifibration
in the diagram below:
\begin{equation*}
\Tetrahedron
\end{equation*}
We describe now how to interpret formulas and proofs of the sequent calculus as objects and arrows of $\E$,
which will determine the functor $\sem{-}\theta$ once we have completed the definition of the category $\Bifib{p}$.

We interpret any formula $S$ as an object $\sem S \theta$ of $\E$, and any proof $\alpha : S \vdashf{h} T$ as an arrow $\sem \alpha \theta : \sem S \theta \to \sem T \theta$ such that $q(\sem \alpha \theta) = h$.

\paragraph{Interpretation of formulas}
Defined by induction on the formula:
\begin{mathpar}
\sem{\atom X} \theta = \theta(X)

\sem{\push f S} \theta = \push {f} {\sem{S} \theta}

\sem{\pull g T} \theta = \pull {g} {\sem{T} \theta}
\end{mathpar}
That is, we interpret the pushforward and pullback connectives using pushforward and pullback in $\E$, and atomic formulas using the action of the functor $\theta : \D \to \E$ on objects.

\paragraph{Interpretation of terms}
By gentle abuse of notation, for any arrow $\alpha : X \to Y \in \E$ such that $q(\alpha) = fg$, let us write $\uLpush f g \alpha$ for the unique $\beta : \push f X \to Y$ such that $\alpha = \opcart f X \; \beta$, and $\alpha \uRpull f g$ for the unique $\beta : X \to \pull g Y$ such that $\alpha = \beta \; \cart g Y$, both of these derived using the respective universal properties \eqref{eq:unique-beta-push} and \eqref{eq:unique-beta-pull} of cartesian liftings.
We thus reuse the notation for division to denote arrows in an arbitrary bifibration.
Let us also allow ourselves to write $\opcart{f}{}$ for $\opcart{f}{X}$ and $\cart{g}{}$ for $\cart{g}{Y}$, when the object can be deduced from the context.

The interpretation of terms is then defined by induction as follows:
\begin{mathpar}
    \sem{\atom \delta} \theta = \theta(\delta)
\\
    \sem{\alpha \tRpush f} \theta = \sem \alpha \theta \; \opcart{f}{}

    \sem{\uLpush f g \alpha} \theta = \uLpush {f} {g} {\sem \alpha \theta}
\\
    \sem{\tLpull g \alpha} \theta = \cart {g} {} \sem{\alpha} \theta

    \sem{\alpha \uRpull f g} \theta = {\sem \alpha \theta} \uRpull {f} {g}
\end{mathpar}
Initial axioms are interpreted using the action of $\theta$ on arrows of $\D$,
multiplications (corresponding to the inference rules $\Lpull[g]$ and $\Rpush$) are interpreted by pre- and post-composition with cartesian liftings in $\E$ (we remind the reader that we write composition by juxtaposition in diagrammatic order),
and divisions ($\Lpush[f]$ and $\Rpull[g]$) are interpreted using division in $\E$, i.e., using the universal properties of the cartesian liftings.

\subsection{Permutation equivalence}
\label{sec:sequent-calculus:permeq}

Let us begin by observing that some algebraic laws relating multiplication and division hold in any bifibration $q : \E \to \B$.

\begin{prop}
  \label{prop:bifib-equations}
  For $\alpha : X \to Y \in \E$ such that $q(\alpha) = fg$,
  and $\beta : X \to Y'$ such that $q(\beta) = fgg'$,
  the following equalities hold:
  \begin{align*}
    \cart {f'} {} \; (\alpha \uRpull f g)
    & =
    (\cart {f'} {} \; \alpha) \uRpull {f'f} g
    \\
    (\uLpush f g \alpha) \; \opcart {g'} {}
    & =
    \uLpush f {gg'} (\alpha \; \opcart {g'} {})
    \\
    (\uLpush f {gg'} \beta) \uRpull g {g'}
    & =
    \uLpush f g (\beta \uRpull {fg} {g'})
  \end{align*}
\end{prop}
\begin{proof}\label{prf:bifib-equations}
  The uniqueness condition on the arrows obtained by division can be expressed
  as an equivalence:
  $\beta = (\alpha \uRpull f g) \iff \beta \; \cart{g}{} = \alpha$
  and
  $\beta = (\uLpush f g \alpha) \iff \opcart{f}{} \; \beta = \alpha$
  . The proofs for the desired equalities are then as follows:

  \begin{mathpar}
    \begin{array}{cl}
    &
    \opcart {f'} {} \; (\alpha \uRpull f g)
    =
    (\opcart {f'} {} \alpha) \uRpull {f'f} g
    \\ \iff &
    \cart{f'}{} \; (\alpha \uRpull f g) \; \cart{g}{}
    =
    \cart{f'}{} \; \alpha
    \\ \Longleftarrow &
    (\alpha \uRpull f g) \; \cart{g}{}
    =
    \alpha
    \\ \iff &
      \alpha \uRpull f g
      =
      \alpha \uRpull f g
    \end{array}

    \begin{array}{cl}
    &
    (\uLpush f g \alpha) \; \opcart {g'} {}
     =
    \uLpush f {gg'} (\alpha \; \opcart {g'} {})
    \\ \iff &
    \opcart{f}{} \; (\uLpush f g \alpha) \; \opcart {g'} {}
    =
    \alpha \; \opcart {g'} {}
    \\ \Longleftarrow &
    \opcart{f}{} \; (\uLpush f g \alpha) {}
    =
    \alpha
    \\ \iff &
      \uLpush f g \alpha
      =
      \uLpush f g \alpha
    \end{array}

    \begin{array}{cl}
    &
    (\uLpush f {gg'} \beta) \uRpull g {g'}
    =
    \uLpush f g (\beta \uRpull {fg} {g'})
    \\ \iff &
    \opcart{f}{} \; (\uLpush f {gg'} \beta) \uRpull g {g'}
    =
    \beta \uRpull {fg} {g'}
    \\ \iff &
    \opcart{f}{} \; ((\uLpush f {gg'} \beta) \uRpull g {g'}) \; \cart{g'}{}
    =
    \beta
    \\ \iff &
    ((\uLpush f {gg'} \beta) \uRpull g {g'}) \; \cart{g'}{}
    =
    \uLpush f {gg'} \beta
    \\ \iff &
    ((\uLpush f {gg'} \beta) \uRpull g {g'})
    =
    (\uLpush f {gg'} \beta) \uRpull g {g'}
    \end{array}
  \end{mathpar}
\end{proof}
\noindent
These laws imply that different derivations of the sequent calculus $\alpha_1,\alpha_2 : S \vdashf{h} T$ will be interpreted by the same arrow $\sem{\alpha_1}\theta = \sem{\alpha_2}\theta : \sem S \theta \to \sem T \theta$ in $\E$ under the bifibrational interpretation described in the previous section.
In order for derivations to form a bifibration themselves, we must quotient them by a \emph{permutation equivalence} relation $(\permeq)$ that relates any such pair of derivations equated in all interpretations.
\begin{defi}
\label{defi:permeq}
We define \defin{permutation equivalence} $(\permeq)$ as the smallest congruence that contains the four generating equations below. A congruence is an equivalence relation that is also closed under the derivation term-formers: if $\alpha_1 \permeq \alpha_2$, then we also have $\alpha_1 \tRpush {g'} \permeq \alpha_2 \tRpush {g'}$, and $\uLpush f g \alpha_1 \permeq \uLpush f g \alpha_2$, etc., whenever both sides of the equation are valid terms. The generating equations of $(\permeq)$ are (whenever both sides of the equation are valid):
\begin{align}
  (\tLpull {f} \alpha) \tRpush {h}
  & \permeq 
  \tLpull {f} (\alpha \tRpush {h})
  & \text{for $\alpha$ over $g$} \label{eq:permgen1}
  \\
  (\uLpush f g \alpha) \tRpush {h}
  & \permeq 
  \uLpush f {gh} (\alpha \tRpush {h})
  & \text{for $\alpha$ over $fg$} \label{eq:permgen2}
  \\
  (\tLpull {f} \alpha) \uRpull {fg} h
  & \permeq 
  \tLpull {f} (\alpha \uRpull g h)
  & \text{for $\alpha$ over $gh$} \label{eq:permgen3}
  \\
  (\uLpush {f} {gh} \alpha) \uRpull g h
  & \permeq 
  \uLpush {f} {g} (\alpha \uRpull {fg} h)
  & \text{for $\alpha$ over $fgh$} \label{eq:permgen4}
\end{align}
\end{defi}
\noindent
The generators of the permutation equivalence relation may be defined equivalently (albeit less compactly) as the following relations on derivations:
\begin{mathpar}
   \inferrule*[Right={$\Rpush[h]$}]{
   \inferrule*[Right={$\Lpull$}]
        {S \vdashf{g} T }
        {\pull{f}S \vdashf{fg} T} }
        {\pull{f}S \vdashf{fgh} \push{h}T}
   \qquad\raisebox{2em}{$\permeq$}\quad
   \inferrule*[Right={$\Lpull$}]{
   \inferrule*[Right={$\Rpush[h]$}]
        {S \vdashf{g} T }
        {S \vdashf{gh} \push{h}T} }
        {\pull{f}S \vdashf{fgh} \push{h}T}

   \inferrule*[Right={$\Rpush[h]$}]{
   \inferrule*[Right={$\Lpush$}]
        {S \vdashf{fg} T }
        {\push{f}S \vdashf{g} T} }
        {\push{f}S \vdashf{gh} \push{h}T}
   \qquad\raisebox{2em}{$\permeq$}\quad
   \inferrule*[Right={$\Lpush$}]{
   \inferrule*[Right={$\Rpush[h]$}]
        {S \vdashf{fg} T }
        {S \vdashf{fgh} \push{h}T} }
        {\push{f}S \vdashf{gh} \push{h}T}

   \inferrule*[Right={$\Rpull[h]$}]{
   \inferrule*[Right={$\Lpull$}]
        {S \vdashf{gh} T }
        {\pull{f}S \vdashf{fgh} T} }
        {\pull{f}S \vdashf{fg} \pull{h}T}
   \qquad\raisebox{2em}{$\permeq$}\quad
   \inferrule*[Right={$\Lpull$}]{
   \inferrule*[Right={$\Rpull[h]$}]
        {S \vdashf{gh} T }
        {S \vdashf{g} \pull{h}T} }
        {\pull{f}S \vdashf{fg} \pull{h}T}

   \inferrule*[Right={$\Rpull[h]$}]{
   \inferrule*[Right={$\Lpush$}]
        {S \vdashf{fgh} T }
        {\push{f}S \vdashf{gh} T} }
        {\push{f}S \vdashf{g} \pull{h}T}
   \qquad\raisebox{2em}{$\permeq$}\quad
   \inferrule*[Right={$\Lpush$}]{
   \inferrule*[Right={$\Rpull[h]$}]
        {S \vdashf{fgh} T }
        {S \vdashf{fg} \pull{h}T} }
        {\push{f}S \vdashf{g} \pull{h}T}
\end{mathpar}
Observe that each of the relations permutes a left rule with a right rule.
\begin{restatable}{lem}{semrespectspermeq}
  \label{eq:sem-respects-permeq}
  Intepretation in any bifibration respects permutation
  equivalence: 
  if $\alpha_1 \permeq \alpha_2$, then
  $\sem {\alpha_1} \theta = \sem {\alpha_2} \theta$.
\end{restatable}

\begin{proof}
  The equivalence and congruence rules are obviously respected by the interpretation. It suffices to check that the four generating equations are also respected.
  The equation with two multiplications holds by associativity of composition:
  \begin{mathpar}
    \sem {(\tLpull {f'} \alpha) \tRpush {g'}} \theta
    \quad=\quad
    \cart {f'} {S'} \; \alpha \; \opcart {g'} {T}
    \quad=\quad
    \sem {\tLpull {f'} (\alpha \tRpush {g'})} \theta
  \end{mathpar}
  The three other equations exactly correspond to the equalities in
  Proposition~\ref{prop:bifib-equations}, which hold in any
  bifibration.
\end{proof}

\begin{rem}\label{rem:orienting-permeq-rules}
The first permutation rule \eqref{eq:permgen1} is really an associativity property for left and right multiplication: a term outline $\tLpull f \alpha \tRpush h$ can be read in two ways depending on the parenthesis placement, the two terms are valid (if $\alpha$ is) and are equivalent.

The second permutation rule \eqref{eq:permgen2} can be oriented from left to right: if the term $(\uLpush f g \alpha) \tRpush h$ is valid, then one can ``push the multiplication inside'', in the sense that the term $\uLpush{f}{gh} (\alpha \tRpush {h})$ is also valid and is equivalent. But going in the other direction, the validity of a term of the form $\uLpush{f}{k} (\alpha \tRpush {h})$ does not uniquely determine a corresponding term of the form $(\uLpush{f}{l} \alpha) \tRpush{h}$. This depends on the solutions $l$ of the equation $k = lh$: it may be that no such $l$ exists, and there is no valid term of this shape, or that several $l$ exist giving different possible rewrites.
For example, we saw that for any $S \refs A$ and $f : A \to B$ the identity derivation $\Id[\push f S] : \push{f} S \vdashf{\Id B} \push{f} S$ is of the form $\Id[\push f S] = \uLpush{f} {\Id B} ({\Id S} \tRpush{f})$, but there are not necessarily any derivations of the form $(\uLpush{f}{g} \Id S) \tRpush{f}$.
Indeed, $(\uLpush{f}{g} \Id S) \tRpush{f} \permeq \Id[\push f S]$ is an instance of permutation rule \eqref{eq:permgen2} just in case $g : B \to A$ exists and is inverse to $f$.

Similarly, the third permutation rule \eqref{eq:permgen3} can be oriented from right to left: one can always ``push the multiplication inside'' by moving from a term of the shape $\tLpull f (\alpha \uRpull {} h)$ to an equivalent term of the shape $(\tLpull f \alpha) \uRpull {} h$, but the other direction is not always possible or is non-deterministic.

Finally, the fourth permutation rule \eqref{eq:permgen4} is also an associativity-like property that can be read in both directions: the two terms of the form $\uLpush f {} \alpha \uRpull {} h$ are equally valid (not always valid, but in the same situations) and equivalent. For example, let us assume that a term of the form $(\uLpush f {} \alpha) \uRpull {} h$ is valid and is over some morphism $g$; then the term must be of the form $(\uLpush f {} \alpha) \uRpull g h$, so $\uLpush f {} \alpha$ must be over the morphism $gh$: our original term is uniquely determined to be $(\uLpush f {gh} \alpha) \uRpull g h$, and $\alpha$ must be over $fgh$. By the same reasoning, any valid term of the form $\uLpush f {} (\alpha \uRpull {} h)$ over $g$ must be $\uLpush f g (\alpha \uRpull {fg} h)$, this term is always valid (assuming that the other term is) and equivalent.
\end{rem}

\begin{rem}
This situation, where multiplications can be pushed inside divisions but not necessarily outside, is reminiscent of Lambek calculus~\cite{Lambek1958}, which inspires our notation.
Indeed, the generators of the permutation equivalence relation are analogous to the natural transformations
\begin{equation}\label{eq:lambek-isos}
  \begin{tikzcd}
  (A \bullet B) \bullet C \ar[r,"\sim"] & A \bullet (B \bullet C) &
  A\mathbin{\backslash} (B \mathbin{/} C) \ar[r,"\sim"] & (A \mathbin{\backslash} B) \mathbin{/} C
  \end{tikzcd}
\end{equation}
\begin{equation}\label{eq:lambek-maps}
  \begin{tikzcd}
  (A \mathbin{\backslash} B)\bullet C \ar[r] & A \mathbin{\backslash}(B \bullet C) &
  A\bullet (B \mathbin{/} C) \ar[r] & (A \bullet B)\mathbin{/} C
  \end{tikzcd}
\end{equation}
that exist in any monoidal biclosed category, where the transformations \eqref{eq:lambek-isos} are always invertible, but the transformations \eqref{eq:lambek-maps} in general are not.
For example, considering the special case of cartesian closed categories, the former correspond to the natural isomorphisms $(A \times B)\times C \cong A \times (B \times C)$ and $(B^C)^A \cong (B^A)^C$, while the latter correspond to the natural transformations 
$B^A \times C \longrightarrow (B\times C)^A$
and 
$A \times B^C \longrightarrow (A\times B)^C$
that may be derived in any ccc, but are very rarely invertible.
\end{rem}

\subsection{Cut}
\label{sec:sequent-calculus:cut}

From a logical perspective, cut is an expected rule for any logic, that ensures that sequences of reasoning can be chained as expected.
The surprising insight behind Gentzen's invention of sequent calculus was that the cut rule is in a sense logically redundant.
The original approach for demonstrating this was to include a cut rule directly in the inference system and then show \emph{cut elimination:} any closed proof can be rewritten without using the cut rule.
We use the alternative but nearly equivalent approach of first defining a cut-free system and then showing that the cut rule is admissible.
Recall that a reasoning rule is \emph{admissible} if, given closed derivations for its premises, there is a closed derivation for its conclusion.

\begin{defi}
\label{defi:cut}
The following \defin{cut rule} is admissible
\[
\inferrule
        {\alpha : S \vdashf{g} T \\
         \beta : T \vdashf{h} U}
        {{\alpha \hcomp \beta} : S \vdashf{gh} U}
\]
via a composition function $\alpha \hcomp \beta$ defined up to permutation equivalence $(\permeq)$.
\end{defi}

\begin{proof}
The cut, or horizontal composition, written $\alpha \hcomp \beta$, is
defined by induction in the usual style of cut-elimination arguments: we reason by case analysis on the derivations $\alpha$ and $\beta$, and we may use instances of the cut rule on strictly smaller arguments (more precisely: in each case both arguments are at least as small and one of them is strictly smaller).

Principal cuts:
\[
\inferrule*
{\inferrule* {\delta : X \to Y \in \D}
             {\atom \delta : \atom X \vdashf{g} \atom Y}
 \\
 \inferrule* {\epsilon : X \to Y \in \D}
             {\atom \epsilon : \atom Y \vdashf{h} \atom Z}
}
{\atom \delta \hcomp \atom \epsilon : \atom X \vdashf{gh} \atom Z}
\qquad\defeq\qquad
\inferrule*
        {\delta \; \epsilon : X \to Z \in D}
        {\atom {\delta \; \epsilon} : \atom X \vdashf{gh} \atom Z}
\]

\[
\inferrule*
{\inferrule*
  {\alpha : S \vdashf{g} T}
  {\alpha \tRpush{f} : S \vdashf{gf} \push{f} T}
 \\
 \inferrule*
  {\beta : T \vdashf{fh} U}
  {\uLpush{f}{h} \beta : \push{f} T \vdashf{h} U}
}
{\alpha \tRpush{f} \hcomp \uLpush{f}{h} \beta : S \vdashf{gfh} U}
\qquad\defeq\qquad
\inferrule*
        {\alpha : S \vdashf{g} T \\
         \beta : T \vdashf{fh} U}
        {\alpha \hcomp \beta : S \vdashf{gfh} U}
\]

\[
\inferrule*
{\inferrule*
  {\alpha : S \vdashf{gf} T}
  {\alpha \uRpull{g}{f} : S \vdashf{g} \pull{f} T}
 \\
 \inferrule*
 {\beta : T \vdashf{h} U}
 {\tLpull{f} \beta : \pull{f} T \vdashf{fh} U}
}
{\alpha \uRpull{g}{f} \hcomp \tLpull{f} \beta : S \vdashf{gfh} U}
\qquad\defeq\qquad
\inferrule*
{\alpha : S \vdashf{gf} T
 \\
 \beta : T \vdashf{h} U}
{\alpha \hcomp \beta : S \vdashf{gfh} U}
\]

Commutative cuts:
\[
\inferrule*
{\inferrule*{
 \alpha : S \vdashf{g} T}{\tLpull{f} \alpha : \pull{f} S \vdashf{fg} T}
 \\
 \beta : T \vdashf{h} U
 }
{(\tLpull{f} \alpha) \hcomp \beta : \pull{f} S \vdashf{fgh} U}
\qquad\defeq\qquad
\inferrule*
{\inferrule*
 {\alpha : S \vdashf{g} T
  \\
  \beta : T \vdashf{h} U}
 {\alpha \hcomp \beta : S \vdashf{gh} U}}
{\tLpull{f} (\alpha \hcomp \beta) : \pull{f} S \vdashf{fgh} U}
\]

\[
\inferrule*
{
\inferrule*
 {\alpha : S \vdashf{fg} T}
 {\uLpush{f}{g} \alpha : \push{f} S \vdashf{g} T}
 \\
 \beta : T \vdashf{h} U
}
{(\uLpush{f}{g} \alpha) \hcomp \beta : \push{f} S \vdashf{gh} U}
\qquad\defeq\qquad
\inferrule*
{\inferrule*
 {\alpha : S \vdashf{fg} T
  \\
  \beta : T \vdashf{h} U}
 {\alpha \hcomp \beta : S \vdashf{fgh} U}}
{\uLpush{f}{gh} (\alpha \hcomp \beta) : \push{f} S \vdashf{gh} U}
\]

\[
\inferrule*
{
\alpha : S \vdashf{g} T
\\
\inferrule*
{\beta : T \vdashf{h} U}
{\beta \tRpush{f} : T \vdashf{hf} \push{f} U}
}
{\alpha \hcomp {\beta \tRpush{f}} : S \vdashf{ghf} \push{f} U}
\qquad\defeq\qquad
\inferrule*
{\inferrule*
 {\alpha : S \vdashf{g} T
  \\
  \beta : T \vdashf{h} U}
 {\alpha \hcomp \beta : S \vdashf{gh} U}}
{(\alpha \hcomp \beta) \tRpush{f} : S \vdashf{ghf} \push{f} U}
\]

\[
\inferrule*
{
\alpha : S \vdashf{g} T
\\
\inferrule*
{\beta : T \vdashf{hf} U}
{\beta \uRpull{h}{f} : T \vdashf{h} U}
}
{\alpha \hcomp {\beta \uRpull{h}{f}} : S \vdashf{gh} \pull{f} U}
\qquad\defeq\qquad
\inferrule*
{\inferrule*
 {\alpha : S \vdashf{g} T
  \\
  \beta : T \vdashf{hf} U}
 {\alpha \hcomp \beta : S \vdashf{ghf} U}}
{(\alpha \hcomp \beta) \uRpull{gh}{f} : S \vdashf{gh} \pull{f} U}
\]

These rules cover all possible cases for a pair of derivations $\alpha, \beta$ whose judgments are composable judgments---the codomain of $\alpha$ is the domain of $\beta$. In particular, notice that when $\alpha$ starts with a right rule, for example a multiplication $\alpha = \alpha' \tRpush f$, $\beta$ must start with the opposite rule on the \emph{same} arrow, here a division by $f$: $\beta = \uLpush f h \beta'$. This is the only possibility because we know that the two judgments must compose at a type of the form $\push f {T'}$, and this formula admits only one left (respectively right) rule.

The definition exhibits some non-determinism on the raw
syntax of derivations, as it may be the case that two commutative
cut rules are applicable. We consider it well-defined as a function \emph{up to} permutation equivalence $(\permeq)$: whenever several of the cases above would apply, all derivations that can be defined from the definition are permutation equivalent.

Note that this happens only in the cases where the derivation on the left starts with the left rule, and the derivation on the right starts with a right rule. There are four such critical pairs (shapes of derivations that can be composed in two different ways), corresponding to all four permutations of generators. Two pairs are symmetric to each other, so we only need to check three.

\begin{mathpar}
  \begin{tikzcd}
      & (\tLpull f \alpha) \hcomp (\beta \tRpush i) &  \\
    \tLpull f (\alpha \hcomp (\beta \tRpush i))   &   & ((\tLpull f \alpha) \hcomp \beta) \tRpush i \\
    \tLpull f ((\alpha \hcomp \beta) \tRpush i)   &   & (\tLpull f (\alpha \hcomp \beta)) \tRpush i
    \arrow["\mathsf{def}"',from=1-2,to=2-1]
    \arrow["\mathsf{def}",from=1-2,to=2-3]
    \arrow["\mathsf{def}"',from=2-1,to=3-1]
    \arrow["\mathsf{def}",from=2-3,to=3-3]
    \arrow[-, "\sim", from=3-1,to=3-3]
  \end{tikzcd}
\end{mathpar}
\begin{mathpar}
  \begin{tikzcd}
      & (\tLpull f \alpha) \hcomp (\beta \uRpull h i) &  \\
    \tLpull f (\alpha \hcomp (\beta \uRpull h i))   &   & ((\tLpull f \alpha) \hcomp \beta) \uRpull {fgh} i \\
    \tLpull f ((\alpha \hcomp \beta) \uRpull {gh} i)   &   & (\tLpull f (\alpha \hcomp \beta)) \uRpull {fgh} i
    \arrow["\mathsf{def}"',from=1-2,to=2-1]
    \arrow["\mathsf{def}",from=1-2,to=2-3]
    \arrow["\mathsf{def}"',from=2-1,to=3-1]
    \arrow["\mathsf{def}",from=2-3,to=3-3]
    \arrow[-, "\sim", from=3-1,to=3-3]
  \end{tikzcd}
\end{mathpar}
\begin{mathpar}
  \begin{tikzcd}
      & (\uLpush f g \alpha) \hcomp (\beta \uRpull h i) &  \\
    \uLpush f {gh} (\alpha \hcomp (\beta \uRpull h i))   &   & ((\uLpush f g \alpha) \hcomp \beta) \uRpull {gh} i \\
    \uLpush f {gh} ((\alpha \hcomp \beta) \uRpull {fgh} i)   &   & (\uLpush f {ghi} (\alpha \hcomp \beta)) \uRpull {gh} i
    \arrow["\mathsf{def}"',from=1-2,to=2-1]
    \arrow["\mathsf{def}",from=1-2,to=2-3]
    \arrow["\mathsf{def}"',from=2-1,to=3-1]
    \arrow["\mathsf{def}",from=2-3,to=3-3]
    \arrow[-, "\sim", from=3-1,to=3-3]
  \end{tikzcd}
\end{mathpar}
\end{proof}
\noindent
We will refer to Definition~\ref{defi:cut} as a ``cut-elimination procedure'', even if technically speaking it is a proof of cut-admissibility.
\begin{restatable}{lem}{cutrespectspermeq}
  \label{lem:cut-respects-permeq}
  Cut is compatible with permutation equivalence.
  If $\alpha_1 \permeq \alpha_2$ then
  $\alpha_1 \hcomp \beta \permeq \alpha_2 \hcomp \beta$,
  and
  $\beta \hcomp \alpha_1 \permeq \beta \hcomp \alpha_2$,
\end{restatable}

\begin{proof}[Proof sketch]
  The proof (detailed in Appendix~\ref{prf:cut-respects-permeq}) proceeds by equational
  reasoning on proof terms. The following piece of equational
  reasoning is representative of the rest of the proof, and many other
  proofs in this section:

  \begin{mathpar}
    \begin{array}{cl@{\qquad}l}
      & (\tLpull f (\alpha \tRpush h)) \hcomp \beta
      & \\ \permeq
      & \tLpull f ((\alpha \tRpush h) \hcomp \beta)
      & \\ =
      & \tLpull f ((\alpha \tRpush h) \hcomp (\uLpush h i \beta'))
      & \\ =
      & \tLpull f (\alpha \hcomp \beta')
      & \\ \permeq
      & (\tLpull f \alpha) \hcomp \beta'
      & \\ =
      & ((\tLpull f \alpha) \tRpush h) \hcomp (\uLpush h i \beta')
      & \\ =
      & ((\tLpull f \alpha) \tRpush h) \hcomp \beta
      &
    \end{array}
  \end{mathpar}

  Each reasoning step is by definition of cut-elimination. Principal cuts are defined unambiguously, for example we have $(\alpha \tRpush h) \hcomp \beta = (\alpha \tRpush h) \hcomp (\uLpush h i \beta')$ by definition. Commutative cuts (left rule in the left derivation, or right rule in the right derivation) are only defined up to permutation equivalence, for example we only have $\tLpull f (\alpha \hcomp \beta') \permeq (\tLpull f \alpha) \hcomp \beta'$.
\end{proof}

\begin{restatable}[identity is neutral]{lem}{idneutral}
  \label{lem:id-neutral}
  The identity derivations are neutral elements for cut:
  \begin{mathpar}
    \Id[S] \hcomp \alpha
    \quad
    \permeq
    \quad
    \alpha
    \quad
    \permeq
    \quad
    \alpha \hcomp \Id[S]
  \end{mathpar}
\end{restatable}
\proofseeappendix{prf:id-neutral}

\begin{restatable}[$\eta$-expansion]{lem}{etaexpansionpushpull}~\\
  \label{lem:eta-expansion-pushpull}
  For any $\alpha : \push f S \vdashf{g} T$,
  for any $\beta : S \vdashf{f} \pull g T$,
  we have:
  \begin{mathpar}
    \alpha \permeq \uLpush f g (\opcart f S \hcomp \alpha)

    \beta \permeq (\beta \hcomp \cart g T) \uRpull f g
  \end{mathpar}
\end{restatable}
\proofseeappendix{prf:eta-expansion-pushpull}

\begin{restatable}{lem}{cutassociative}
  \label{lem:cut-associative}
  Cut is associative: $(\alpha \hcomp \beta) \hcomp \gamma \permeq \alpha \hcomp (\beta \hcomp \gamma)$.
\end{restatable}
\proofseeappendix{prf:cut-associative}

\subsection{The free bifibration}
\label{sec:sequent-calculus:free-bifibration}

We can now formally define $\Bifib p$ as a category where:
\begin{itemize}
\item The objects are the valid formulas $S \refs A$,
  (Section~\ref{sec:sequent-calculus:frm-seq}).
\item The arrows between $S \refs A$ and $T \refs B$ are the valid (closed) derivations of the judgment $S \vdashf{f} T$ for any $f : A \to B \in \C$ (Section~\ref{sec:sequent-calculus:deriv}), quotiented by the permutation equivalence relation $(\permeq)$ (Section~\ref{sec:sequent-calculus:permeq}).
\item The identity arrow for $S \refs A$ is (the equivalence class of) the derivation $\Id[S]$ (Defn.~\ref{defi:identity}).
\item The composition of $\alpha : S \vdashf{g} T$ and
  $\beta : T \vdashf{h} U$
  is the cut $\alpha \hcomp \beta : S \vdashf{gh} U$
  (Defn.~\ref{defi:cut}).
\end{itemize}

\begin{restatable}{lem}{bifibisacategory}
  \label{lem:bifib-is-a-category}
  $\Bifib p$ is a category.
\end{restatable}

\begin{proof}
  We proved that composition is well-defined on derivations quotiented
  by permutation equivalence (Lemma~\ref{lem:cut-respects-permeq}),
  and that identity and cut satisfy neutrality
  (Lemma~\ref{lem:id-neutral}) and associativity
  (Lemma~\ref{lem:cut-associative}).
\end{proof}
Next we show that $\Bifib{p}$ is the domain of a bifibration $\BifibFun p : \Bifib{p} \to \C$ equipped with a
functor $\eta_p : \D \to \Bifib{p}$ such that $\BifibFun p \circ \eta_p = p$.
We define $\BifibFun p$ by
  \begin{enumerate}
  \item $\BifibFun p(S \refs A) \defeq A$
  \item
    $\BifibFun p(\alpha : S \vdashf{f} T)
    \defeq
    f$
  \end{enumerate}

\begin{restatable}{lem}{bifibfunbifibration}
  \label{lem:bifibfun-bifibration}
  $\BifibFun p : \Bifib{p} \to \C $ is a bifibration.
\end{restatable}

\begin{proof}
  For any arrow $f : A \to B$ in $\C$ and formula $S \refs A$, the pushforward
  of $S$ along $f$ is precisely the formula $\push f S \refs B$, and we have previously defined
  the arrow $\opcart f S : S \to \push f S$ as (the equivalence class of)
  $\Id[S] \tRpush f$. Let us show universality, i.e., that $\opcart f S$ is $+$-cartesian.

  For any $g : B \to C$ and
  $\alpha : S \vdashf{fg} T$, we can take
  $\beta \defeq \uLpush f g \alpha : \push f S \vdashf{g} T$. We have to show $\alpha \permeq \opcart f S \hcomp \beta$:
  \begin{mathpar}
    \begin{array}{cl@{\qquad}l}
      & \opcart f S \hcomp \beta
      & \\ =
      & (\Id[S] \tRpush f) \hcomp (\uLpush f g \alpha)
      & \\ =
      & \Id[S] \hcomp \alpha
      & \\ \permeq
      & \alpha
    \end{array}
  \end{mathpar}
  as well as unicity: if $\alpha \permeq \opcart f S \hcomp \beta'$ then
  \begin{mathpar}
    \begin{array}{cl@{\qquad}l}
      & \beta'
      & \\ \permeq
      & \uLpush f g (\opcart f S \hcomp \beta')
      & \text{by $\eta$-expansion (Lemma~\ref{lem:eta-expansion-pushpull})}
      \\ \permeq
      & \uLpush f g \alpha
      & \\ =
      & \beta
    \end{array}
  \end{mathpar}

  The case for $\cart g T : \pull g T \to T$ is symmetrical.
\end{proof}

\noindent
We define the functor $\eta_p  : \D \to \Bifib{p}$ by interpreting every object $X$ of $\D$ as an atomic formula and every arrow $\delta : X \to Y$ as an initial axiom, and it is immediate that $\BifibFun p \circ \eta_p = p$.

Lastly, we need to show that $\BifibFun{p} : \Bifib{p} \to \C$ is universal in the sense explained before, that for any bifibration $q : \E \to \C$ equipped with a functor $\theta : \D \to \E$ such that $q \circ \theta = p$, there exists a unique morphism of bifibrations $\sem{-}\theta : \BifibFun{p} \to q$ making the diagram below commute:
\begin{equation*}
\Tetrahedron
\end{equation*}
We already defined the interpretation $\sem - \theta$ of formulas and proof terms in Section~\ref{sec:sequent-calculus:interpretation}, and showed that it respects permutation equivalence (Lemma~\ref{eq:sem-respects-permeq}).
We now establish that it is a functor $\sem - \theta : \Bifib{p} \to \E$, and that it is is the unique morphism of bifibrations making the diagram above commute.

\begin{lem}
  \label{lem:sem-is-a-functor}
  $\sem \_ \theta : \Bifib p \to \E$ is a functor.
\end{lem}
\begin{proof}
  \label{prf:sem-is-a-functor}
  We prove the interpretation respects identities $\sem {\Id [S]} \theta = \Id [\sem S \theta]$
  by induction on $S$:

  \begin{mathpar}
    \sem {\Id [\atom X]} \theta
    = \theta(\Id [X])
    = \Id [\theta (X)]
    = \Id [\sem {\atom X} \theta]
    \\
    \begin{array}{cl@{\qquad}l}
      & \sem {\Id [\push f S]} \theta
      & \\ =
      & \sem {\uLpush f {\Id [B]} {(\Id[S] \tRpush f)}} \theta
      & \\ =
      & \uLpush {f} {\Id [B]}
        \left(\sem {\Id[S]} \theta \; \opcart {f} {\sem S \theta}\right)
      & \\ =
      & \uLpush {f} {\Id [B]}
        \left(\Id [\sem S \theta] \; \opcart {f} {\sem S \theta}\right)
      & \text{by induction on $S$}
      \\ =
      & \uLpush {f} {\Id [B]} \opcart {f} {\sem S \theta}
      & \\ =
      & \Id [\sem S \theta]
      & \text{property of cartesian liftings}
    \end{array}
  \end{mathpar}
  (The case of the identity on $\pull g T$ is symmetrical.)

  We prove that the interpretation respects composition $\sem {\alpha \hcomp \beta} \theta = \sem \alpha \theta \; \sem \beta \theta$
  by simultaneous induction on the derivations $\alpha$ and $\beta$.
  Representative cases follow:
  \begin{mathpar}
     \sem {\atom \delta \hcomp \atom \epsilon} \theta
     = \sem {\atom {\delta \; \epsilon}} \theta
     = \theta (\delta \; \epsilon)
     = \sem \delta \theta \; \sem \epsilon \theta

    \begin{array}{cl@{\qquad}l}
      & \sem {\alpha \tRpush f} \theta \; \sem {\uLpush f h \beta} \theta
      & \\ =
      & \sem \alpha \theta \; \opcart {f} {\sem S \theta} \; (\uLpush {f} {h} \sem \beta \theta)
      & \\ =
      & \sem \alpha \theta \; \sem \beta \theta
      & \text{by }\eqref{eq:unique-beta-push} 
      \\ =
      & \sem {\alpha \hcomp \beta} \theta
      & \\ =
      & \sem {(\alpha \tRpush f) \hcomp (\uLpush f h \beta)} \theta
    \end{array}

    \begin{array}{cl@{\qquad}l}
      & \sem {\alpha \hcomp (\beta \tRpush f)} \theta
      & \\ =
      & \sem {(\alpha \hcomp \beta) \tRpush f} \theta
      & \\ =
      & \sem {\alpha \hcomp \beta} \theta \opcart {f} {\sem U \theta}
      & \\ =
      & \sem {\alpha} \theta \; \sem \beta \theta \opcart {f} {\sem U \theta}
      & \\ =
      & \sem {\alpha} \theta \; \sem {\beta \tRpush f} \theta
    \end{array}

    \begin{array}{cl@{\qquad}l}
      & \sem {\alpha \hcomp (\beta \uRpull h f)} \theta
      & \\ =
      & \sem {(\alpha \hcomp \beta) \uRpull {gh} f)} \theta
      & \\ =
      & \sem {\alpha \hcomp \beta} \theta \uRpull {gh} {f}
      & \\ =
      & (\sem \alpha \theta \; \sem \beta \theta) \uRpull {gh} {f}
      & \\ =
      & \sem \alpha \theta \; (\sem \beta \theta \uRpull {h} {f})
      & \text{by Proposition~\ref{prop:bifib-equations}}
      \\ =
      & \sem \alpha \theta \; \sem {\beta \uRpull h f} \theta
    \end{array}
  \end{mathpar}
\end{proof}

\begin{lem}\label{lem:free-bifib-unicity}
$\sem{-} \theta : \BifibFun{p} \to q$ is the unique morphism of bifibrations such that $\sem{-}\theta \circ \eta_p = \theta$.
\end{lem}
\begin{proof}
Recall that the interpretation was defined as follows on terms:
\begin{mathpar}
    \sem{\atom \delta} \theta = \theta(\delta)
\\
    \sem{\alpha \tRpush f} \theta = \sem \alpha \theta \; \opcart{f}{}

    \sem{\uLpush f g \alpha} \theta = \uLpush {f} {g} {\sem \alpha \theta}
\\
    \sem{\tLpull g \alpha} \theta = \cart {g} {} \sem{\alpha} \theta

    \sem{\alpha \uRpull f g} \theta = {\sem \alpha \theta} \uRpull {f} {g}
\end{mathpar}
The equation $\sem{-}\theta \circ \eta_p = \theta$ holds immediately by construction.

To be a morphism of bifibrations, the interpretation must send cartesian arrows to cartesian arrows
of the same type.
This follows from the definition and the previous lemma.
For $+$-cartesian arrows for example, we have:
\begin{mathpar}
\sem{\opcart f S}\theta
 = \sem{\Id[S] \tRpush f}\theta
 = \sem{\Id[S]} \theta \; \opcart {f} {\sem{S}{\theta}}
 = \Id[\sem S \theta] \; \opcart {f} {\sem{S}\theta}
 = \opcart {f} {\sem{S}\theta}
\end{mathpar}
Finally, for unicity we remark that each of the five equations defining the interpretation $\sem{-}\theta$ must hold for any morphism of bifibrations $\tau : \BifibFun{p} \to q$ such that $\tau \circ \eta_p = \theta$, so necessarily $\tau = \sem{-}\theta$.
\end{proof}
\noindent
Putting it all together, we conclude with the main theorem of this section.
\begin{thm}\label{thm:bif-is-free}$\BifibFun{p} : \Bifib{p} \to \C$ is the free bifibration on $p : \D \to \C$.
\end{thm}
\noindent
An immediate consequence---applying one of the equivalent characterizations of adjunctions \cite[Theorem IV.1.2(ii), p.83]{MacLaneCWM}---is that we can construct a functor
\[ \BifibFun{} : \Cat/\C \longrightarrow \BifCat(\C) \]
from the slice of the category of categories over $\C$ into the category of cloven bifibrations over $\C$, which is left adjoint to the evident forgetful functor:
\begin{equation}\label{eq:free-bifibration-adjunction}
\begin{tikzcd}
\Cat/\C \ar[rr,bend left,"\BifibFun{}"{name=L}] && \BifCat(\C)\ar[ll,bend left,"U"{name=R}]
\ar[phantom,from = L, to = R,"\dashv" rotate = -90]
\end{tikzcd}
\end{equation}
This induces a monad on $\Cat/\C$ whose algebras are the cloven bifibrations over $\C$.

\subsection{Some proof-theoretic consequences}
\label{sec:sequent-calculus:consequences}


Initiality of the $\Bifib{p}$ construction has as another direct consequence that the sequent calculus can be used to reason about
both existence and equality of morphisms in arbitrary bifibrations.
We state this as a corollary of Theorem~\ref{thm:bif-is-free}.

\begin{cor}
Suppose given a functor $p : \D \to \C$, an arrow $f : A \to B$ of $\C$, and a pair of bifibrational formulas $S \refs A$ and $T \refs B$.
Then:
\begin{enumerate}
\item the sequent $S \vdashf{f} T$ is derivable iff for every bifibration $q : \E \to \C$ and functor $\theta : \D \to \E$ such that $q \circ \theta = p$, there is a morphism $\alpha : \sem S \theta \to \sem T \theta$ in $\E$ such that $q(\alpha) = f$ ;
\item two derivations $\alpha_1,\alpha_2 : S \vdashf{f} T$ are permutation equivalent $\alpha_1 \permeq \alpha_2$ iff for every bifibration $q : \E \to \C$ and functor $\theta : \D \to \E$ such that $q \circ \theta = p$, they have the same interpretation $\sem {\alpha_1} \theta = \sem {\alpha_2} \theta$ as morphisms in $\E$.
\end{enumerate}
\end{cor}

As recalled in the Introduction, any bifibration $\E \to \C$ determines and is determined by the data of a pseudofunctor $\C \to \Adj$ into the category of categories and adjunctions.
This pseudofunctor sends every object $A$ of $\C$ to its fiber category $\E_A$ consisting of the arrows lying over $\Id[A]$, and every arrow $f : A \to B$ of $\C$ to the push/pull adjunction $\push f : \E_A \to \E_B \dashv \pull{f} : \E_B \to \E_A$ between fiber categories.
In sequent calculus terms, the action of the push/pull functors on arrows is described by the following open derivations:
\begin{equation}\label{eq:functor-derivations}
\inferrule*[Right={$\Lpush$}]{\inferrule*[Right={$\Rpush$}]{S_1 \vdashf{\Id[A]} S_2}{S_1 \vdashf{f} \push f {S_2}}}
{\push f {S_1} \vdashf{\Id[B]} \push f {S_2}}
\qquad\qquad
\inferrule*[Right={$\Rpull[g]$}]{\inferrule*[Right={$\Lpull[g]$}]{T_1 \vdashf{\Id[B]} T_2}{\pull g {T_1} \vdashf{g} T_2}}
{\pull g {T_1} \vdashf{\Id[A]} \pull g {T_2}}
\end{equation}
The fact that the fiber functor $\C \to \Adj$ associated to a bifibration $\E \to \C$ is only a \emph{pseudo-}functor in general means that the identities $\push {(g \circ f)} \equiv \push g \circ \push f$ and $\pull {(g \circ f)} \equiv \pull f \circ \pull g$ hold only up to natural isomorphism, as do the identities $\push {\Id[A]} \equiv \Id[{\E_A}] \equiv \pull {\Id[A]}$.
To describe this in the sequent calculus, let us first define \emph{logical equivalence} of formulas (corresponding to what is often called ``vertical isomorphism'' in the categorical literature on fibrations).
\begin{defi}\label{defi:logical-equivalence} Let $S_1,S_2 \refs A$ be two bifibrational formulas over the same object $A$.
We say that $S_1$ and $S_2$ are \defin{logically equivalent,} written $S_1 \equiv_A S_2$ or simply $S_1 \equiv S_2$, if there are a pair of derivations $\alpha : S_1 \vdashf{\Id[A]} S_2$ and $\beta : S_2 \vdashf{\Id[A]} S_1$ such that $\alpha \hcomp \beta \permeq \Id[S_1]$ and $\beta\hcomp \alpha \permeq \Id[S_2]$.
\end{defi}
\begin{prop}\label{prop:equivalence-is-vertical-iso}
Two formulas $S_1,S_2 \refs A$ are logically equivalent $S_1 \equiv_A S_2$ iff they are isomorphic in the fiber category $\Bifib{p}_A$.
\end{prop}
\noindent
Logical equivalence should be distinguished from the weaker notion of isomorphism $S \cong T$ given by isomorphism in the total category $\Bifib{p}$, i.e., a pair of derivations
$\alpha : S \vdashf{f} T$ and $\beta : T \vdashf{g} S$ such that $\alpha \hcomp \beta \permeq \Id[S]$ and $\beta\hcomp \alpha \permeq \Id[T]$ (where one does not require $f$ and $g$ to be identity arrows or that $S \refs A$ and $T \refs B$ lie over the same object).
\begin{prop}[Pseudofunctoriality] \label{prop:pseudofunctoriality}
The following logical equivalences hold
\begin{align}
\push{(g\circ f)}S &\equiv \push g \push f S \label{eq:pushpush} \\
\pull {(f \; g)} T & \equiv \pull f \pull g T \label{eq:pullpull}  \\
\push {\Id} U & \equiv U \equiv \pull {\Id} U \label{eq:pushidpull} 
\end{align}
for all $f, g, S, T, U$ of the appropriate type.
\end{prop}
\begin{proof}
By Proposition~\ref{prop:equivalence-is-vertical-iso}, this is a corollary of Lemma~\ref{lem:bifibfun-bifibration}, since these isomorphisms hold in the fiber categories of any bifibration.
Working out the derivation explicitly, the equivalence \eqref{eq:pushpush}, for instance, is witnessed by the following derivations (parameterized by $S \refs A$, $f : A \to B$, and $g : B \to C$), whose compositions are permutation equivalent to the identity derivations on $\push{(g\circ f)}S$ and $\push g \push f S$:
\begin{mathpar}
\inferrule*{\inferrule*{\inferrule*{\inferrule*{ }{S \vdashf{\Id[A]} S}}{S \vdashf{f} \push f S}}{S \vdashf{f\; g} \push g \push f S}}{\push{(g\circ f)}S \vdashf{\Id[C]} \push g \push f S}

\inferrule*{\inferrule*{\inferrule*{\inferrule*{ }{S \vdashf{\Id[A]} S}}{S \vdashf{f\; g} \push {(g\circ f)} S}}{\push f S \vdashf{g} \push{(g\circ f)}S}}{\push g \push f S \vdashf{\Id[C]} \push{(g\circ f)}S}
\end{mathpar}
\end{proof}
\noindent
More generally, by Theorem~\ref{thm:bif-is-free} we have
\begin{cor}
\label{cor:equiv-sound-and-complete}
Given a functor $p : \D \to \C$, two bifibrational formulas $S_1,S_2 \refs A$ are logically equivalent $S_1 \equiv S_2$ iff their interpretations are isomorphic $\sem {S_1} \theta \cong \sem {S_2} \theta$ in every fiber category $\E_A$ for every bifibration $q : \E \to \C$ and functor $\theta : \D \to \E$ such that $q \circ \theta = p$.
\end{cor}
\noindent
Finally, let us observe that the free bifibration construction is \emph{conservative}.
\begin{prop}[Conservativity]
\label{prop:conservativity} The functor $\eta_p : \D \to \Bifib{p}$ is full and faithful.
\end{prop}
\noindent
Conservativity is obvious by inspection of the sequent calculus, since the only way to prove an atomic sequent $\atom X \vdashf{f} \atom Y$ is by providing an arrow $\delta : X \to Y$ in $\D$ such that $p(\delta) = f$.
Faithfulness of $\eta_p$ is also a corollary of the following stronger statement, whose proof only uses the universal property of the free bifibration, and which will serve us in the next section.
\begin{prop}
  \label{prop:axiom-functor} The functor $\eta_p : \D \to \Bifib{p}$ is equipped with a left inverse $\chi_p : \Bifib{p} \to \D$, such that the pairing $(\chi_p,\BifibFun{p}) : \Bifib{p} \to \D\times \C$ is a morphism of bifibrations from $\BifibFun{p}$ to the trivial bifibration $\pi_2 : \D\times \C \to \C$.
\end{prop}
\begin{proof}
  Projection functors $\pi_2 : \D \times \C \to \C$ are always bifibrations, with cartesian liftings defined by
  \[
    \begin{tikzcd}
      (X,A) \ar[r,"{(\Id[X],f)}"] & (X,B) \eqdef \push{f}(X,A) & \pull{g}(X,C) \defeq (X,B) \ar[r,"{(\Id[X],g)}"] & (X,C) \\
      A \ar[r,"f"] & B & B \ar[r,"g"] & C
    \end{tikzcd}
  \]
  By the universal property of the free bifibration,  we therefore obtain a unique morphism of bifibrations $\sem{-}\theta : \BifibFun{p} \to \pi_2$
  \[
\begin{tikzcd}[ampersand replacement=\&]
  \D \arrow[ddrr,"p"']\arrow[rr,"\eta_p"]\arrow[rrrd,"{\theta}"'] \& \&\Bifib{p}\arrow[dd,"\BifibFun{p}"]\arrow[dr,"\sem{-}\theta"]\&  \\
  \& \&  \& \D\times\C\arrow[dl,"\pi_2"]\\
  \& \& \C \& 
\end{tikzcd}
  \]
where $\theta \defeq (\Id[\D],p)$.
This proves the proposition, taking $\chi_p \defeq \pi_1 \circ \sem{-}\theta$ and using the universal property of the product of categories. 
\end{proof}
\noindent
In terms of our concrete description of the free bifibration, $\chi_p : \Bifib{p} \to \D$ is the functor sending a bifibrational formula to its atomic source and sending an equivalence class of derivations to the necessarily unique initial axiom with which all of the derivations begin.

\subsection{Related work}

There is a well-known construction, originally due to Gray \cite{Gray1966}, of the free fibration on a functor (or dually the free opfibration) as a comma category.
The standard construction actually defines a \emph{split} (op)fibration, in the sense that the pseudofunctoriality laws \eqref{eq:pushpush}--\eqref{eq:pushidpull} collapse to strict equalities, although one can represent non-split fibrations and opfibrations as pseudo-algebras for the corresponding pseudomonads (see the account by Anders Kock~\cite{Kock2013fibrations}, as well as recent work of Emmenegger, Mesiti, Rosolini, and Streicher on an alternative construction of fibrations~\cite{EMRStreicher2024}).
Conversely, our construction produces a non-split bifibration, but one can recover a split bifibration by restricting to \emph{nonidentity strictly alternating formulas} (see Section~\ref{sec:focusing:weak-focusing} below, and in particular Theorem~\ref{thm:splitting}).
It is worth mentioning that since a functor is a bifibration just in case it is both a fibration and an opfibration, the free bifibration monad is formally the coproduct of the free fibration and the free opfibration monads---but this characterization does not give much help in constructing it explicitly.

We will discuss Lamarche's construction of the free bifibration on a functor along with Dawson, Paré, and Pronk's closely related $\Pi_2(\C)$ construction, both already referenced in the Introduction, further in Section~\ref{sec:double-categories}.

The notation $S \refs A$ to denote an object $S \in \Bifib{p}$ such that $\BifibFun{p}(S) = A$ is taken from Melliès and Zeilberger's work on functors as type refinement systems \cite{MZ2015popl}, and more broadly our proof-theoretic construction is inspired in part from their suggestive type-theoretic reading of bifibrations \cite{MZ2016lawvere,MZ2018isbell}, as well as from the bifibrational calculus of Licata, Shulman, and Riley~\cite{LSR2017}.
This approach of course owes much to the tradition of categorical proof theory as pioneered by Lambek~\cite{Lambek1968}.

\section{The double category of zigzags}
\label{sec:double-categories}

As mentioned in the Introduction, Lamarche~\cite{Lam13,Lam14} defined a \emph{path category} $\mathrm{P}(\C)$ for each category $\C$, showed that it admits the structure of a double category, and explained how to use this structure to construct the free bifibration on a functor.
In this section we will start the other way around, defining a category $\Z(\C)$ equivalent to Lamarche's path category as the free bifibration on the identity functor, and showing that it admits a strict double category structure.
The objects and tight morphisms of $\ZZ(\C)$
are the objects and arrows in~$\C$, while the loose morphisms are zigzags (or paths) of arrows in $\C$,
and thus $\ZZ(\C)$ is called the \emph{zigzag double category}.
Its underlying 2-category of loose morphisms is equivalent to $\Pi_2(\C)$, the 2-category studied by Dawson, Paré, and Pronk~\cite{DPP03} that is obtained by freely adjoining right adjoints to $\C$.
To prove this equivalence, we establish that $\ZZ(\C)$ has another characterization as the free \emph{fibrant double category} on $\C$.

Whereas Lamarche and Dawson, Paré, and Pronk gave rather intricate
combinatorial descriptions of what it means to be a morphism of
zigzags (i.e., to be a double cell in $\ZZ(\C)$, or a 2-cell in $\Pi_2(\C)$), we will extract a simple presentation of $\ZZ(\C)$ by translating the sequent calculus definition to double categorical syntax.
In turn we will see how that presentation can be translated backwards to obtain a topologically meaningful string diagrammatic representation of proofs in the bifibrational calculus.

\subsection{Deriving a double category structure from a right action}
\label{sec:double-categories:double-cat}

For the rest of the paper we define $\Z(\C) \defeq \Bifib{\Id[\C]}$ to be the total category of the free bifibration on the identity functor for $\C$.
Observe that by Proposition~\ref{prop:axiom-functor}, $\Z(\C)$ is automatically equipped with the structure of a reflexive graph object in $\Cat$
\[
\begin{tikzcd}
  \Z(\C) \ar[d,shift left=1em,"\tgt"] \ar[d,shift right=1em,"\src"'] \\
  \C\ar[u,"U" description]
\end{tikzcd}
\]
taking $\src \defeq \chi_p$, $U \defeq \eta_p$, and $\tgt \defeq \BifibFun{p}$, for $p = \Id[\C]$.
We use the letter $Z$ to range over the objects of $\Z(\C)$, which we refer to as \emph{zigzag formulas} or simply \emph{zigzags,} and write $Z : A \zigzag B$ to indicate that $\src(Z) = A$ and $\tgt(Z) = B$.
We use the letter $\zeta$ to range over arrows $Z \to Z'$ of $\Z(\C)$.
Such an arrow may be depicted as a \emph{cell}
\[
  \begin{tikzcd}A \ar[d,rightsquigarrow,"Z"']\ar[r,"f"]    \arrow[rd, phantom, "\zeta"]
 & B\ar[d,rightsquigarrow,"Z'"] \\ C \ar[r,"g"'] & D \end{tikzcd}
\]
where $f = \src(\zeta)$ and $g = \tgt(\zeta)$.
We always depict the tight morphisms (corresponding to the underlying arrows of $\C$) horizontally and the loose morphisms (corresponding to objects of $\Z(\C)$, i.e.,~zigzags of arrows in $\C$) vertically.
We will develop this diagrammatic picture further in Section~\ref{sec:double-categories:inference-rules-double-cells}, but for now we work with the purely formal definition building on the results of Section~\ref{sec:sequent-calculus}.

To extend this reflexive graph structure on $\Z(\C)$ to a double category (i.e., to an internal category in $\Cat$), we need to exhibit a vertical composition functor 
\[
\odot : \Z(\C) \pullback{\tgt}{\src} \Z(\C) \to \Z(\C)
\]
and prove that it is associative and neutral relative to $U$.
We will recover vertical composition $\odot$ as the special case $p = \Id[\C]$ of a more general action 
\[
\oast_p : \Bifib{p} \pullback{\BifibFun{p}}{\src} \Z(\C) \to \Bifib{p}
\]
of $\Z(\C)$ on any free bifibration $\Bifib{p}$.

To define this action, let us begin by observing that the sequent calculus for $\Z(\C)$ shares the same logical inference rules as any free bifibration $\Bifib{p}$ over the same base category: their sequent calculi only differ in the choice of atomic formulas and initial axioms.
Moreover, both atomic formulas and initial axioms of $\Z(\C)$ 
carry no information, since they are uniquely determined by the underlying object/arrow of the base category.
These observations make clear that we can define the action of $\Z(\C)$ on $\Bifib{p}$ as follows:
\begin{itemize}
\item at the level of formulas, given a formula $S \refs A$ and a zigzag $Z : A \zigzag B$, the formula $S \oast Z \refs B$ is obtained by substituting $S$ for the unique atom $\atom{A}$ in $Z$;
\item at the level of proofs (i.e., representatives of arrows), given $\alpha \in \Bifib{p}$ such that $\BifibFun{p}(\alpha) = f$ and $\zeta \in \Z(\C)$ such that $\src(\zeta) = f$ and $\tgt(\zeta) = g$, the proof $\alpha \oast_p \zeta \in \Bifib{p}$ such that $\BifibFun{p}(\alpha\oast_p\zeta) = g$ is obtained by substituting $\alpha$ for the unique initial axiom $\atom{f}$ in $\zeta$.
\end{itemize}
As a visual aid, let us write ``$\Box$'' for the uniquely determined atoms and axioms in a zigzag formula/proof.
Formally, the action can then be expressed inductively as follows:
\begin{mathpar}
  S \oast_p \Box = S

  S \oast_p (\push{f} Z) = \push{f} (S \oast_p Z)

  S \oast_p (\pull{g} Z) = \pull{g} (S \oast_p Z) \\

  \alpha \oast_p \Box = \alpha 

  \alpha \oast_p (\zeta \tRpush f) = (\alpha \oast_p \zeta) \tRpush f

  \alpha \oast_p (\uLpush f g \zeta) = \uLpush {f} {g} {(\alpha \oast_p \zeta)}
\\
  \alpha \oast_p (\tLpull g \zeta) = \tLpull g (\alpha \oast_p \zeta)

  \alpha \oast_p (\zeta \uRpull f g) = {(\alpha\oast_p \zeta)} \uRpull {f} {g}
\end{mathpar}
\begin{prop}
  The action extends to a functor $\oast_p : \Bifib{p} \pullback{\BifibFun{p}}{\src} \Z(\C) \to \Bifib{p}$.
\end{prop}
\begin{proof}
  Since $\oast_p$ is defined as a substitution operator, it is immediate that it respects permutation equivalence of proofs in the sense that we can define $[\alpha] \oast_p [\zeta] \defeq [\alpha \oast_p \zeta]$ without ambiguity.
  To show that it is functorial, we need to verify the permutation equivalence
  \begin{equation}
    \label{eq:action-functorial}
    (\alpha \hcomp \alpha')\oast_p (\zeta\hcomp \zeta') \permeq (\alpha\oast_p \zeta) \hcomp (\alpha'\oast_p \zeta')
  \end{equation}
  for all $\alpha,\alpha' \in \Bifib{p}$ and $\zeta,\zeta' \in \Z(\C)$ for which these operations are defined.
  This follows from the definition of cut (Definition~\ref{defi:cut}) by a straightforward induction on $\zeta$ and $\zeta'$.
\end{proof}

We can now define the vertical composition functor $\odot : \Z(\C) \pullback{\tgt}{\src} \Z(\C) \to \Z(\C)$ as the action of $\Z(\C)$ on itself, $\odot \defeq \oast_{\Id[\C]}$, and use that to define the double category of zigzags.
\begin{prop}
  Vertical composition $\odot$ is strictly associative and neutral relative to $U$.
\end{prop}
\begin{proof}
  Immediate from the definition of $\oast_p$ as a substitution operator.
\end{proof}
\begin{defi}\label{defi:zigzag-double-category}
  The \defin{zigzag double category} of $\C$, notated $\ZZ(\C)$, is defined as a strict double category by the data $\ZZ(\C) \defeq (\C, \Z(\C),\src,\tgt,U,\odot)$, where
\[
  \Z(\C) \defeq \Bifib{p}\quad
  \src \defeq \chi_p\quad
  \tgt \defeq \BifibFun{p}\quad
  U \defeq \eta_p\quad
  \odot = \oast_p
\]
as explained above, for $p = \Id[\C]$.
\end{defi}
\noindent
Notice that with this definition, equation \eqref{eq:action-functorial} specializes to the interchange axiom for the double category $\ZZ(\C)$.
We already saw in Section~\ref{sec:sequent-calculus:permeq} that the permutation equivalence relations were necessary for obtaining a bifibration, but it is worth mentioning that this axiom also forces the definition of permutation equivalence.
For example, the permutation relation \eqref{eq:permgen2} may be rederived as follows:
\begin{align*}
  (\uLpush f g \alpha) \tRpush h
  &= (\alpha \oast (\uLpush f g \atom{fg})) \hcomp (\Id[T] \tRpush h \oast \atom{h}) & \text{defn. of }\oast\\
  &\permeq (\alpha \hcomp (\Id[T] \tRpush h)) \oast ((\uLpush f g \atom{fg}) \hcomp \atom{h}) & \text{by } \eqref{eq:action-functorial}\\
  &\permeq (\alpha \tRpush h) \oast (\uLpush f {gh} \atom{fgh}) & \text{properties of cut and id}\\
  &= \uLpush f {gh} {(\alpha \tRpush h)}& \text{defn. of }\oast
\end{align*}
Similar observations were made by Dawson, Paré, and Pronk~\cite[\S1.5]{DPP03} and by Lamarche~\cite[p.12]{Lam13}.

\subsection{From inference rules to double cells}
\label{sec:double-categories:inference-rules-double-cells}

Having established that the free bifibration on the identity functor $\Id[\C] : \C \to \C$ generates a double category of zigzags in $\C$, at this point let us develop a more diagrammatic presentation of $\ZZ(\C)$ that is isomorphic to Definition~\ref{defi:zigzag-double-category} but will help us to better visualize it as a double category.

An easy but important observation is that what we have called ``zigzag formulas'' (that is, objects of the free bifibration $\Z(\C) \defeq \Bifib{\Id[\C]}$) really are in one-to-one correspondence with zigzags of arrows in the category $\C$.
Of course we have not defined what a ``zigzag of arrows'' is formally, but one possible definition is, say, as a graph homomorphism $Z : \tilde L \to \C$ from some orientation $\tilde L$ of the line graph $L_n$ with $n$ edges and $n+1$ vertices into the underlying graph of $\C$, for some $n\ge 0$.
It is clear that such a graph homomorphism can be represented as a zigzag formula with $n$ total pushes and pulls, and vice versa.
For example, if $\C$ is the category on the left and $Z$ is the (preformally defined!) zigzag on the right, 
\[
  \begin{tikzcd}[row sep=1em]
    && C\\
    A \ar[r,"f"] & B \ar[ur,"g"]\ar[dr,"h"] & && A \ar[r,"f"] & B\ar[r,"h"] & D\ar[r,<-,"h"] & B\ar[r,"g"] & C\\
    && D
  \end{tikzcd}
\]
then $Z$ can be encoded isomorphically either as a graph homomorphism $Z : \tilde L \to \C$ where
\begin{mathpar}
\tilde L =  \begin{tikzcd}[row sep=1em, ampersand replacement=\&]
v_0 \ar[r,"e_1"] \& v_1\ar[r,"e_2"] \& v_2\ar[r,<-,"e_3"] \& v_3\ar[r,"e_4"] \& v_4
\end{tikzcd}

\begin{array}{cccc}
Z(v_0) = A
&
Z(v_1) = Z(v_3) = B
&
Z(v_2) = D
&
Z(v_4) = C
\end{array}

\begin{array}{cccc}
Z(e_1) = f
&
Z(e_2) = h
&
Z(e_3) = g
&
Z(e_4) = g
\end{array}
\end{mathpar}
or by the zigzag formula $Z = \push{g}\pull{h}\push{h}\push{f}A$.
One important point is that both of these representations allow us to represent the \emph{empty zigzag,} which corresponds to a choice of object in $\C$, and should not be confused with a one-step zigzag along an identity arrow.

Now turning to double cells, observe that each logical inference rule of the sequent calculus, recalled below for reference, may be interpreted as a generating double cell of $\ZZ(\C)$ (to make the notation lighter, we omit arrows from cell names):
\begin{equation}
\label{eq:generating-cells}
\begin{array}{c}
  \inferrule* [Right={$\Lpush$}]
        {S \vdashf{fg} T}
        {\push{f} S \vdashf{g} T} \\\rotatebox{-90}{$\leadsto$} \\[1.5em]
\begin{tikzcd}
    A
    \arrow[d, "f"']
    \arrow[r, "fg"]
    \arrow[rd, phantom, "\push \Lsym"]
    &
    C
    \arrow[d, equal]
    \\
    B
    \arrow[r, "g"']
    &
    C
  \end{tikzcd}
\end{array}
\quad\qquad
\begin{array}{c}
  \inferrule* [Right={$\Rpush$}]
        {S' \vdashf{f'} S}
        {S' \vdashf{f'f} \push{f} S} \\ \rotatebox{-90}{$\leadsto$} \\[1.5em]
	\begin{tikzcd}
    A'
    \arrow[d, equal]
    \arrow[r, "f'"]
    \arrow[rd, phantom, "\push \Rsym"]
    &
    A
    \arrow[d, "f"]
    \\
    A'
    \arrow[r, "f'f"']
    &
    B
    \end{tikzcd}
\end{array}
\quad\qquad
\begin{array}{c}
  \inferrule* [Right={$\Lpull[g]$}]
        {T \vdashf{g'} T'}
        {\pull{g} T \vdashf{gg'} T'} \\ \rotatebox{-90}{$\leadsto$} \\[1.5em]
	\begin{tikzcd}
    C
    \arrow[from=d, "g"]
    \arrow[r, "g'"]
    \arrow[rd, phantom, "\pull \Lsym"]
    &
    C'
    \arrow[d, equal]
    \\
    B
    \arrow[r, "gg'"']
    &
    C'
    \end{tikzcd}
\end{array}
\quad\qquad
\begin{array}{c}
  \inferrule* [Right={$\Rpull[g]$}]
        {S \vdashf{fg} T}
        {S \vdashf{f} \pull{g} T} \\ \rotatebox{-90}{$\leadsto$} \\[1.5em]
	\begin{tikzcd}
    A
    \arrow[d, equal]
    \arrow[r, "fg"]
    \arrow[rd, phantom, "\pull \Rsym"]
    &
    C
    \arrow[from=d, "g"']
    \\
    A
    \arrow[r, "f"']
    &
    B
    \end{tikzcd}
\end{array}
\end{equation}
Formally, since we defined $\ZZ(\C)$ as an internal category, each of these cells corresponds to an arrow in $\Z(\C)$ with given source and target arrows in $\C$.
For example, the cell we have named $\push\Lsym$ corresponds to an arrow $\gamma : \push{f}\atom{A} \to \atom{C}$ with $\src(\gamma) = fg$ and $\tgt(\gamma) = g$.
These four cells (really, a family of cells parameterized by the arrows $f, g, f', g'$) \emph{generate} the double category $\ZZ(\C)$ in the sense that every representative $\zeta$ of an arrow $\Z(\C)$ may be uniquely decomposed as a vertical composition of generators $\zeta = \gamma_1 \odot \dots \odot \gamma_n$.
Indeed, it is immediate from the inductive definition of the sequent calculus that any proof in any free bifibration over $\C$ may be uniquely decomposed as the action of a list of generators on an initial axiom.
\begin{prop}\label{prop:unique-generator-decomposition}
  For any derivation $\alpha : S \vdashf{f} T$ there exists a unique series of generators $\gamma_1,\dots,\gamma_n \in \ZZ(\C)$ of the form \eqref{eq:generating-cells} such that $\alpha = \atom \delta \oast_p \gamma_1 \oast_p \cdots \oast_p \gamma_n$, where $\delta = \chi_p(\alpha)$.
\end{prop}
\noindent
It is possible to go in the other direction and extract a direct presentation of $\ZZ(\C)$, considering its double cells as formal vertical compositions of generators---we will refer to such lists of generators as ``stacks''---modulo an equivalence relation.
Our definition of permutation equivalence for the sequent calculus (Definition~\ref{defi:permeq}) translates directly to the following relations on stacks:
\begin{equation}\label{eq:permeq-on-double-cells}
  \begin{aligned}
    \begin{tikzcd}[sep = 1cm]
	\cdot 
	\arrow[r, "g"]
	\arrow[rd, phantom, "\pull\Lsym"]
	& \cdot
	\arrow[d, equals]
	\\
	\cdot
	\arrow[u, "f"]
	\arrow[r, "fg" description]
	\arrow[d, equals]
	\arrow[rd, phantom, "\push\Rsym"]
	& \cdot
	\arrow[d, "h"]
	\\
	\cdot
	\arrow[r, "fgh"']
	& \cdot
	\end{tikzcd}
\quad
\permeq
\quad
	\begin{tikzcd}[sep = 1cm]
	\cdot
	\arrow[d, equals]
	\arrow[r, "g"]
	\arrow[rd, phantom, "\push\Rsym"]
	& \cdot
	\arrow[d, "h"]
	\\
	\cdot
	\arrow[r, "gh" description]
	\arrow[rd, phantom, "\pull\Lsym"]
	& \cdot
	\arrow[d, equals]
	\\
	\cdot
	\arrow[u, "f"]
	\arrow[r, "fgh"']
	& \cdot
      \end{tikzcd} 
\qquad\qquad
    \begin{tikzcd}[sep = 1cm]
    \cdot 
    \arrow[r,"fg"]
    \arrow[d,"f"']
    \arrow[phantom,dr,"\push\Lsym"] 
    & \cdot 
    \arrow[d,equals] 
    \\
    \cdot 
    \arrow[r,"g" description]
    \arrow[d,equals]
    \arrow[phantom,dr,"\push\Rsym"] 
    & \cdot 
    \arrow[d,"h"] 
    \\
    \cdot 
    \arrow[r,"gh"'] 
    & \cdot
    \end{tikzcd}
\quad
\permeq
\quad
    \begin{tikzcd}[sep = 1cm]
    \cdot 
    \arrow[r,"fg"]
    \arrow[d,equals]
    \arrow[phantom,dr,"\push\Rsym"] 
    & \cdot 
    \arrow[d,"h"] 
    \\
    \cdot 
    \arrow[r,"fgh" description]
    \arrow[d,"f"']
    \arrow[phantom,dr,"\push\Lsym"] 
    & \cdot 
    \arrow[d,equals] 
    \\
    \cdot \arrow[r,"gh"'] 
    & \cdot
  \end{tikzcd}
\\[1em]
	\begin{tikzcd}[sep = 1cm]
	\cdot
	\arrow[r, "gh"]
	\arrow[rd, phantom, "\pull\Lsym"]
	& \cdot
	\arrow[d, equals]
	\\
	\cdot
	\arrow[u, "f"]
	\arrow[r, "fgh" description]
	\arrow[d, equals]
	\arrow[rd, phantom, "\pull\Rsym"]
	& \cdot
	\\
	\cdot
	\arrow[r, "fg"']
	& \cdot
	\arrow[u, "h"']
	\end{tikzcd}
\quad
\permeq
\quad
	\begin{tikzcd}[sep = 1cm]
	\cdot
	\arrow[r, "gh"]
	\arrow[d, equals]
	\arrow[rd, phantom, "\pull\Rsym"]
	& \cdot
	\\
	\cdot
	\arrow[r, "g" description]
	\arrow[rd, phantom, "\pull\Lsym"]
	& \cdot
	\arrow[u, "h"']
	\arrow[d, equals]
	\\
	\cdot
	\arrow[u, "f"]
	\arrow[r, "fg"']
	& \cdot
      \end{tikzcd}
\qquad \qquad
	\begin{tikzcd}[sep = 1cm]
	\cdot
	\arrow[r, "fgh"]
	\arrow[d, "f"']
	\arrow[rd, phantom, "\push\Lsym"]
	& \cdot
	\arrow[d, equals]
	\\
	\cdot
	\arrow[d, equals]
	\arrow[r, "gh" description]
	\arrow[rd, phantom, "\pull\Rsym"]
	& \cdot
	\\
	\cdot
	\arrow[r, "g"']
	& \cdot
	\arrow[u, "h"']
	\end{tikzcd}
\quad
\permeq
\quad
	\begin{tikzcd}[sep = 1cm]
	\cdot 
	\arrow[r, "fgh"]
	\arrow[d, equals]
	\arrow[rd, phantom, "\pull\Rsym"]
	& \cdot
	\\
	\cdot
	\arrow[d, "f"']
	\arrow[r, "fg" description]
	\arrow[rd, phantom, "\push\Lsym"]
	& \cdot
	\arrow[u, "h"']
	\arrow[d, equals]
	\\
	\cdot
	\arrow[r, "g"']
	& \cdot
      \end{tikzcd}
\end{aligned}
\end{equation}
Our inductive definition of the identity derivation $\Id[S] : S \vdashf{\Id[A]} S$ for every bifibrational formula $S\refs A$ (Definition~\ref{defi:identity}) likewise translates to the following definition of identity double cells for every one-step zigzag, which can be vertically stacked to define the identity cell for an arbitrary length zigzag $Z : A \zigzag B$:
\begin{equation}
  \label{eq:id-double-cells}
\begin{tikzcd}
  A \arrow[r,equals]\arrow[phantom,dr,"\push\Rsym"]\arrow[d,equals] & A \arrow[d,"f"] \\
  A \arrow[r,"f" description]\arrow[phantom,dr,"\push\Lsym"]\arrow[d,"f"'] & B\arrow[d,equals] \\
  B \arrow[r,equals] & B
\end{tikzcd}
\qquad\qquad
\begin{tikzcd}
  B \arrow[r,equals]\arrow[phantom,dr,"\pull\Lsym"]\arrow[d,<-,"g"'] & B \arrow[d,equals] \\
  A \arrow[r,"g" description]\arrow[phantom,dr,"\pull\Rsym"]\arrow[d,equals] & B\arrow[d,<-,"g"] \\
  A \arrow[r,equals] & A
\end{tikzcd}
\end{equation}
We elide the explicit definition of horizontal composition for $\ZZ(\C)$ in terms of stacks, but suffice it to say that our cut-elimination procedure (Definition~\ref{defi:cut}) can be repeated.
We should mention that horizontal composition $\alpha \hcomp \beta$ is an operation taking two stacks with matching zigzags along their boundaries $\tgt(\alpha) = \src(\beta)$, where the left and right boundaries of a stack can be computed by concatenating all of the one-step zigzags (i.e., up or down arrows) that appear on that side of the generators, ignoring empty zigzags.

Now, a corollary of Proposition~\ref{prop:unique-generator-decomposition} is that any free bifibration may be easily reconstructed from the double category of zigzags, as observed by Lamarche~\cite[p.22]{Lam13}.
\begin{cor}\label{cor:bifib-as-pullback}
$\Bifib{p} \cong \D \pullback{p}{\src} \Z(\C)$ and $\BifibFun{p} \equiv \push{\tgt}\pull{\src}p$, or diagrammatically:
\begin{equation*}
\begin{tikzcd}
\D
\arrow[d, "p"']
& \Bifib{p}
\arrow[l]
\arrow[d]
\arrow[rd, "\BifibFun{p}"]
\arrow[ld, phantom, "\llcorner" very near start]
& \\
\C
& \Z(\C)
\arrow[l, "\src"]
\arrow[r, "\tgt"']
& \C
\end{tikzcd}
\end{equation*}  
\end{cor}
\noindent
This suggests an isomorphic representation of sequent calculus proofs, as a stack of generators in $\ZZ(\C)$ paired with a single compatible arrow of $\D$.
We refer the reader again back to the Introduction for an example of this representation on the right side of Figure~\ref{fig:example-functor-and-derivation} (we will explain the overlaid string diagram further below).

\subsection{Aside: displayed bifibrations}
\label{sec:displayed-bifibrations}

To better understand the relationship between the zigzag double category and the bifibrational sequent calculus, we find it revealing to pursue a more conceptual analysis of the action defined in Section~\ref{sec:double-categories:double-cat}.
This analysis will not be needed in the rest of the paper and so may be safely skipped, but could interest readers familiar with Bénabou's work \cite[\S7]{Benabou2000} and so-called \emph{displayed categories} \cite{AhrensLumsdaine2019}.

Recall that we wrote $\Frm{A}$ for the set of bifibrational formulas (that is, objects of $\Bifib{p}$) lying over an object $A\in\C$.
Let us similarly write
\[
 \Pf{h} \defeq \set{ \alpha : S \vdashf{h} T \mid S \in \Frm{A}, T \in \Frm{C}} / \permeq
\]
for the set of equivalence classes of proofs (that is, arrows of $\Bifib{p}$) lying over an arrow $h : A \to C$.
This set is equipped with evident projection functions 
$\Frm{A} \leftarrow \Pf{h} \rightarrow \Frm{C}$
returning the domain and codomain of a proof.
The action $\oast_p : \Bifib{p} \pullback{\BifibFun{p}}{\src} \Z(\C) \to \Bifib{p}$ now permits the definition of a \emph{double functor}
\[ \partial {\BifibFun{p}} : \begin{tikzcd}\ZZ(\C)^\top\ar[r] & \Span(\Set)\end{tikzcd} \]
from the transpose of the zigzag double category into the double category of spans of sets, by the following mappings:
\begin{eqnarray*} 
  A &\mapsto & \Frm{A} \\
  f : A \to B &\mapsto & \Frm{A} \leftarrow \Pf{f}\rightarrow \Frm{B} \\
  Z : A \zigzag B &\mapsto & -\oast Z : \Frm{A} \to \Frm{B} \\ \\
  \begin{tikzcd}[ampersand replacement=\&]A\ar[r,"f"]\ar[d,rightsquigarrow,"Z"']\ar[rd,phantom,"\zeta"]  \& B\ar[d,rightsquigarrow,"Z'"] \\ C\ar[r,"g"'] \& D\end{tikzcd} & \mapsto & \begin{tikzcd}[ampersand replacement=\&]\Frm{A}\ar[d,"-\oast Z"']  \& \Pf{f}\ar[l]\ar[r]\ar[d,"-\oast\zeta"] \& \Frm{B}\ar[d,"-\oast Z'"] \\ \Frm{C} \& \Pf{g}\ar[l]\ar[r] \& \Frm{D}\end{tikzcd}
\end{eqnarray*}
Note the transpose is necessary because the tight morphisms of $\ZZ(\C)$ are sent to loose morphisms of $\Span(\Set)$.

In particular, $\partial\BifibFun{p}$ maps each generator of $\ZZ(\C)$ to a double cell of spans:
\begin{equation}\label{eq:rules-as-span-cells}
\begin{split}
\begin{tikzcd}[column sep=1.5em,ampersand replacement=\&]
    \Frm{A}\ar[d,"\push{f}"'] \& \Pf{fg}\ar[l]\ar[r]\ar[d,"\push\Lsym"] \& \Frm{C}\ar[d,equals] \\
    \Frm{B} \& \Pf{g}\ar[l]\ar[r] \& \Frm{C} 
\end{tikzcd}
\qquad
\begin{tikzcd}[column sep=1.5em,ampersand replacement=\&]
    \Frm{A'}\ar[d,equals] \& \Pf{f'}\ar[l]\ar[r]\ar[d,"\push\Rsym"] \& \Frm{A}\ar[d,"\push{f}"] \\
    \Frm{A'} \& \Pf{f'f}\ar[l]\ar[r] \& \Frm{B} 
\end{tikzcd}
\\
\begin{tikzcd}[column sep=1.5em,ampersand replacement=\&]
    \Frm{C}\ar[d,"\pull{g}"'] \& \Pf{g'}\ar[l]\ar[r]\ar[d,"\pull\Lsym"] \& \Frm{C'}\ar[d,equals] \\
    \Frm{B} \& \Pf{gg'}\ar[l]\ar[r] \& \Frm{C'} 
\end{tikzcd}
\qquad
\begin{tikzcd}[column sep=1.5em,ampersand replacement=\&]
    \Frm{A}\ar[d,equals] \& \Pf{fg}\ar[l]\ar[r]\ar[d,"\pull\Rsym"] \& \Frm{C}\ar[d,"\pull{g}"] \\
    \Frm{A} \& \Pf{f}\ar[l]\ar[r] \& \Frm{B} 
\end{tikzcd}
\end{split}
\end{equation}
One might think of this interpretation as internalizing the idea of an inference rule as a transformation from a proof of the premise(s) into a proof of the conclusion.
Notice that this double functor $\partial\BifibFun{p} : \ZZ(\C)^\top \to \Span(\Set)$
\begin{itemize}
\item \emph{strictly preserves composition of zigzags:} since $\oast_p$ is a strict action, with the equation $x \oast (a \odot b) = (x \oast a) \oast b$ corresponding to simple chaining of logical connectives and inference rules;
\item \emph{laxly preserves composition of arrows in $\C$:} since the span morphisms $\Pf{f} \hcomp \Pf{g} \to \Pf{fg}$, which correspond to cut-elimination, are typically not invertible.
\end{itemize}
In fact, this construction is very general and can be applied to \emph{any} bifibration $q : \E \to \C$ to derive a lax double functor $\partial q : \ZZ(\C)^\top \to \Span(\Set)$.
On the horizontal category $\hcat{\ZZ(\C)} = \C$, this functor sends every object $A \in \C$ to the set $\Ob A = \set{S \in \E \mid q(S) = A}$ of objects lying over it, and similarly every arrow $f : A \to B$ to the set $\Arr f = \set{\alpha \in \E \mid q(\alpha) = f}$ of arrows lying over it, which is equipped with evident projections $\Ob A \leftarrow \Arr f \rightarrow \Ob B$.
This extends to the generators of $\ZZ(\C)$ by interpreting the span morphisms
\begin{equation}\label{eq:bifibrations-as-span-cells}
\begin{split}
\begin{tikzcd}[column sep=1.5em,ampersand replacement=\&]
    \Ob{A}\ar[d,"\push{f}"'] \& \Arr{fg}\ar[l]\ar[r]\ar[d,"\push\Lsym"] \& \Ob{C}\ar[d,equals] \\
    \Ob{B} \& \Arr{g}\ar[l]\ar[r] \& \Ob{C} 
\end{tikzcd}
\qquad
\begin{tikzcd}[column sep=1.5em,ampersand replacement=\&]
    \Ob{A'}\ar[d,equals] \& \Arr{f'}\ar[l]\ar[r]\ar[d,"\push\Rsym"] \& \Ob{A}\ar[d,"\push{f}"] \\
    \Ob{A'} \& \Arr{f'f}\ar[l]\ar[r] \& \Ob{B} 
\end{tikzcd}
\\
\begin{tikzcd}[column sep=1.5em,ampersand replacement=\&]
    \Ob{C}\ar[d,"\pull{g}"'] \& \Arr{g'}\ar[l]\ar[r]\ar[d,"\pull\Lsym"] \& \Ob{C'}\ar[d,equals] \\
    \Ob{B} \& \Arr{gg'}\ar[l]\ar[r] \& \Ob{C'} 
\end{tikzcd}
\qquad
\begin{tikzcd}[column sep=1.5em,ampersand replacement=\&]
    \Ob{A}\ar[d,equals] \& \Arr{fg}\ar[l]\ar[r]\ar[d,"\pull\Rsym"] \& \Ob{C}\ar[d,"\pull{g}"] \\
    \Ob{A} \& \Arr{f}\ar[l]\ar[r] \& \Ob{B} 
\end{tikzcd}
\end{split}
\end{equation}
using the bifibrational structure of $q$, very similarly to how we defined the interpretation functor $\sem{-}{\theta} : \BifibFun{p} \to \E$ from the free bifibration in Section~\ref{sec:sequent-calculus:interpretation}.
For example, the function $\push{f} : \Ob{A} \to \Ob{B}$ maps any object $S\in \E$ lying over $A$ to (the object component of) its pushforward $\push{f}S$, the span morphism $\push\Lsym : \Arr{fg} \to \Arr{g}$ is interpreted using the universal property of this pushforward, and $\push\Rsym : \Arr{f'} \to \Arr{f'f}$ by postcomposition with the associated $+$-cartesian arrow $S \to \push{f}S$.

Conversely, any such lax double functor $F : \ZZ(\C)^\top \to \Span(\Set)$ determines a category $\int_\C F$ together with a functor $\pi_\C : \int_\C F \to \C$ equipped with the structure of a bifibration.
Indeed, the category $\int_\C F$ is constructed in the usual way as a displayed category: objects of $\int_\C F$ are given by pairs $(A,S)$ of an object $A \in \C$ and an element $S \in F(A)$, arrows are given by pairs $(f,\alpha)$ of an arrow $f \in \C$ and an element $\alpha \in F(f)$ (meaning an element of the apex of the span $F(A) \leftarrow F(f) \to F(B)$), and composition and identity arrows are defined using the lax structure of $F$.

In turn, the tight component of the lax double functor $F : \ZZ(\C)^\top \to \Span(\Set)$ enables us to go further and equip the projection functor $\pi_\C : \int_\C F \to \C$ with the structure of a bifibration.
For example, the pushforward along an arrow $f : A \to B$ is defined using the functions $F(\push{f}A) : F(A) \to F(B)$ to compute pushforward objects, and applying $F(\push\Rsym) : F(\Id[A]) \to F(f)$ to an identity arrow to compute the associated cartesian arrow.
The universal property of this pushforward is derived from the interpretation of $F(\push\Lsym)$ and the defining equations of $\ZZ(\C)$.

The constructions $\partial$ and $\int$ witness an equivalence
\[
  \begin{tikzcd}
    \D \ar[d, "\text{bifibration}"] \\ \C
  \end{tikzcd}
  \qquad\iff\qquad
  \begin{tikzcd}
    \ZZ(\C)^\top \ar[r,"\text{lax}"] & \Span(\Set)
  \end{tikzcd}
\]
and one might therefore reasonably refer to a lax double functor $F : \ZZ(\C)^\top \to \Span(\Set)$ as a ``displayed bifibration''.

We will not pursue these reflections further in this paper, although we think they could eventually be helpful for developing a more algebraic perspective on the sequent calculus.

\subsection{A dagger structure on zigzags}
\label{sec:double-categories:dagger}

One advantage of considering free bifibrations in terms of the double category of zigzags is that the latter clearly has a large degree of symmetry.
Of particular relevance is a dagger-type involution
\[ (-)^\dagger : \ZZ(\C) \to \ZZ(\C)^{v}\]
that reverses the direction of vertical arrows and double cells.
Formally, this corresponds to an endofunctor on $\Z(\C)$ swapping source and target while leaving the underlying arrows in $\C$ invariant:
\[
\begin{tikzcd}[column sep=1em,row sep=1.2cm]
\Z(\C)\ar[rd,"\src"']\ar[rr,bend left,"(-)^\dagger"{name=0}] & & \Z(\C)\ar[ld,"\tgt"]\ar[ll,bend left,"(-)^\dagger"{name=1}] \\
&\C&
\ar[from=0,to=1,phantom,"\sim"]
\end{tikzcd}
\]
At the level of formulas, $(-)^\dagger$ turns pushforwards into pullbacks and vice versa while reversing the sequence of connectives applied (e.g., $(\push{h}\pull{g}\push{f}\Box)^\dagger = \pull{f}\push{g}\pull{h}\Box$).
More geometrically, the transformation may be visualized as reflecting the generating cells \eqref{eq:generating-cells} across the $x$-axis (or equivalently reading them from bottom to top), which swaps $\Lpush[f]$ with $\Lpull[f]$ and $\Rpush[g]$ with $\Rpull[g]$.
The $(-)^\dagger$ transformation extends to arbitrary stacks of double cells by $(\gamma_1 \odot \dots \odot \gamma_n)^\dagger = \gamma_n^\dagger \odot \dots \odot \gamma_1^\dagger$.
The following is immediate.
\begin{prop}\label{prop:dagger-involution}
  $(-)^\dagger$ is an involutive isomorphism.
\end{prop}
\noindent
As a corollary, we can derive that the pairing of the source and target functors is a bifibration. 
\begin{thm}[{Lamarche~\cite[Theorem 3]{Lam14}}]\label{thm:srctgt-bifibration}
$(\src,\tgt) : \Z(\C) \to \C\times \C$ is a bifibration.
\end{thm}
\begin{proof}
  By Proposition~\ref{prop:dagger-involution}, $\src : \Z(\C) \to \C$ is a bifibration since it inherits a bifibrational structure from $\tgt$ via the $(-)^\dagger$ involution.
  As we saw earlier by Proposition~\ref{prop:axiom-functor}, the pairing $(\src,\tgt) : \Z(\C) \to \C\times \C$ is a morphism of bifibrations from $\tgt$ to $\pi_2$, which means that $\tgt$-cartesian arrows in $\Z(\C)$ are mapped by $\src$ to identity arrows in $\C$.
  Symmetrically, applying the $(-)^\dagger$ involution, it is also a morphism of bifibrations from $\src$ to $\pi_1$.
  This ensures that the pairing is a bifibration $(\src,\tgt) : \Z(\C) \to \C \times \C$, since the two bifibrational structures act independently.  
\end{proof}

\subsection{Recovering the free adjoint construction}

Given a double category $\EE$ seen as an internal category $\EE = (\C,\E,\src_\E,\tgt_\E,U_\E,\odot_\E)$, we write $\hcat{\EE} = \C$ for its underlying \emph{horizontal category}.
Any double category $\EE$ also induces a \emph{vertical 2-category} denoted $\vcat{\EE}$, with same set of objects, 1-cells given by vertical arrows $F : \begin{tikzcd}A \ar[r,"\circ" marking] & B\end{tikzcd}$, and 2-cells $\alpha : F \Longrightarrow G$ given by globular double cells
\[
  \begin{tikzcd}
    A \ar[r,equals]\ar[d,"F"', "\circ" marking]
    \arrow[rd, phantom, "\alpha"]
    & A\ar[d,"G", "\circ" marking] \\ B \ar[r,equals] & B\end{tikzcd}
\]
that is, cells whose source and target are identities.

A double category $\EE$ is said to be \emph{fibrant} (or to be a \emph{framed bicategory} \cite{Shu08}) just in case the pairing of its source and target functors $(\src_\E,\tgt_\E) : \E \to \C\times\C$ is a fibration, or equivalently an opfibration, and therefore a bifibration \cite[Theorem~4.1]{Shu08}.
This is equivalent to asking that the double category has all \emph{companions and conjoints} meaning that for every horizontal arrow $f : A \to B$ there is both a vertical arrow $F_f: \begin{tikzcd}A \ar[r,"\circ" marking] & B\end{tikzcd}$ and a vertical arrow $G_f: \begin{tikzcd}B \ar[r,"\circ" marking] & A\end{tikzcd}$ equipped with a pair of cells
\[
  	\begin{tikzcd}
    A
    \arrow[d, equal]
    \arrow[r, equal]
    \arrow[rd, phantom, "\alpha"]
    &
    A
    \arrow[d, "F_f", "\circ" marking]
    \\
    A
    \arrow[r, "f"']
    &
    B
    \end{tikzcd}
    \qquad
	\begin{tikzcd}
    B
    \arrow[d, "G_f"', "\circ" marking]
    \arrow[r, equals]
    \arrow[rd, phantom, "\beta"]
    &
    B
    \arrow[d, equal]
    \\
    A
    \arrow[r, "f"']
    &
    B
    \end{tikzcd}
  \]
satisfying a few equations (see \cite{DPP2010Span} and \cite[Theorem~A.2]{Shu08}; note that the conventions for the roles of horizontal and vertical arrows are often swapped in the literature).

Fibrant double categories are moreover essentially equivalent to \emph{proarrow equipments} \cite[Appendix~C]{Shu08}.
In particular, any fibrant double category $\EE$ canonically determines an identity-on-objects pseudofunctor $\push{(-)} : \C \to \K$ from its underlying horizontal category $\C = \hcat{\EE}$ to its vertical 2-category $\K = \vcat{\EE}$ that sends every horizontal arrow $f$ to its vertical companion $F_f$, which has a right adjoint $F_f \dashv G_f$ given by the vertical conjoint.
Conversely, given a 2-category $\K$ equipped with a pseudofunctor $I : \C \to \K$ such that every arrow $f \in \C$ is mapped to an arrow $If \in \K$ with a right adjoint, one can define an associated double category $\EE_\K$ whose objects and horizontal arrows are the objects and arrows of $\C$, whose vertical arrows $M : A \to B$ are 1-cells $M : IA \to IB$ in $\K$, and whose double cells
\[
\begin{tikzcd}
    A
    \arrow[d, "M"', "\circ" marking]
    \arrow[r, "f"]
    \arrow[rd, phantom, "\alpha"]
    &
    C
    \arrow[d, "N", "\circ" marking]
    \\
    B
    \arrow[r, "g"']
    &
    D
    \end{tikzcd}
\]
are 2-cells $\alpha : M \cdot Ig \Rightarrow If \cdot N$ in $\K$.
Then $\EE_\K$ is a fibrant double category \cite[Prop.~C.3]{Shu08}.

The correspondence between fibrant double categories and proarrow equipments means that we can recover Dawson, Paré, and Pronk's free adjoint construction $\Pi_2(\C)$ up to equivalence as the vertical 2-category of $\ZZ(\C)$.

\begin{thm}\label{thm:free-fibrant-double-category}
  $\ZZ(\C)$ is the free fibrant double category over $\C$.
\end{thm}
\begin{proof}
By Theorem~\ref{thm:srctgt-bifibration}, $\ZZ(\C)$ is a fibrant double category.
Moreover, given any other fibrant double category $\EE = (\C,\E,\src_\E,\tgt_\E,U_\E)$ with underlying horizontal category $\hcat{\EE} = \C$, there exists a unique double functor $\ZZ(\C) \to \EE$ whose underlying functor $\Z(\C) \to \E$ is a morphism of bifibrations from $(\src,\tgt)$ to $(\src_\E,\tgt_\E)$.
Indeed, any double functor $\ZZ(\C) \to \EE$ is determined by the image of the generators \eqref{eq:generating-cells}, and the image of the generators is uniquely determined by the bifibrational structure of $\src_\E : \E \to \C$ and $\tgt_\E : \E \to \C$.
\end{proof}
\noindent
In particular, there is an identity-on-objects pseudofunctor $\push{(-)} : \C \to \vcat{\ZZ(\C)}$ mapping every arrow $f : A \to B$ of $\C$ to the forward zigzag $\push{f} A : A \zigzag B$, which has a right adjoint given by the backward zigzag $\pull{f} B : B \zigzag A$.
Observe that $\push{(-)}$ is only a pseudofunctor rather than a strict functor because, as we already observed, the free bifibration is \emph{not} split in the sense that the canonical identities $\push{(g\circ f)} S \equiv \push g \push f S$ and $\push {\Id} U \equiv U$ only hold as isomorphisms in the fiber categories rather than as strict equalities.
In contrast, Dawson, Paré, and Pronk constructed $\Pi_2(\C)$ in a way that admits a strict embedding functor $\push{(-)} : \C \to \Pi_2(\C)$.
Nevertheless, up to splitting of the free bifibration (see Theorem~\ref{thm:splitting} below), the two constructions are equivalent.
\begin{cor}\label{cor:Pi2C-is-VZZC}
  $\vcat{\ZZ(\C)} \simeq \Pi_2(\C)$.
\end{cor}

\subsection{From double cells to string diagrams}
\label{sec:double-categories:string-diagrams}

Another reward of the double categorical analysis is that it leads to a natural graphical calculus, derived by applying standard conventions of string diagrams to the generators of $\ZZ(\C)$.
The idea is simple: we begin by dualizing the generators \eqref{eq:generating-cells} geometrically, so that cells become vertices and objects become regions, while arrows become edges separating the regions which meet at a vertex:
\begin{equation*}
\begin{tikzpicture}[scale=0.5,baseline=(tl)]
\path coordinate (tl) 
+(0,-4) coordinate (bl)
+(0,-2) coordinate[label=left:$f$] (lf)
+(.5,-1.5) coordinate[label=$A$] (A)
+(1.75,0) coordinate[label=above:$f$] (tf)
+(0.75,-3.5) coordinate[label=$B$] (B)
+(2,-2) coordinate (middle)
+(2.5,0) coordinate[label=above:$g$] (tg)
+(3.5,-2.5) coordinate[label=$C$] (C)
+(2.5,-4) coordinate[label=below:$g$] (bg)
+(4,-4) coordinate (lr);
\node (Lpush) [vertex] at (middle.center) [text=white] {$\pmb{\push\Lsym}$};
\draw[oedge,arrf] (lf) to (Lpush);
\draw[oedge,arrf] (Lpush) to (tf);
\draw[oedge,arrg] (bg) to (Lpush);
\draw[oedge,arrg] (Lpush) to (tg);
\begin{scope}[on background layer]
\fill[obc] (tl) rectangle (lr);
\fill[obb] (bl) to (lf) to (Lpush.center) to (bg) -- cycle;
\fill[obb] (tf) to (Lpush.center) to (tg) -- cycle;
\fill[oba] (tl) to (lf) to (Lpush.center) to (tf) -- cycle;
\end{scope}
\end{tikzpicture}
\quad \qquad
\begin{tikzpicture}[scale=0.5,baseline=(tl)]
\path coordinate (tl) 
+(0,-4) coordinate (bl)
+(1.5,0) coordinate[label=above:$f'$] (tfp)
+(.5,-2.5) coordinate[label=$A'$] (Ap)
+(4,-2) coordinate[label=right:$f$] (rf)
+(2,-2) coordinate (middle)
+(3.25,-1.5) coordinate[label=$A$] (A)
+(1.5,-4) coordinate[label={[yshift=1pt]below:$f'$}] (bfp)
+(3.5,-3.5) coordinate[label=$B$] (C)
+(2.5,-4) coordinate[label=below:$f$] (bf)
+(4,-4) coordinate (lr);
\node (Rpush) [vertex] at (middle.center) [text=white] {$\pmb{\push\Rsym}$};
\draw[oedge,arrfp] (bfp) to (Rpush);
\draw[oedge,arrfp] (Rpush) to (tfp);
\draw[oedge,arrf] (bf) to (Rpush);
\draw[oedge,arrf] (Rpush) to (rf);
\begin{scope}[on background layer]
\fill[oba] (tl) rectangle (lr);
\fill[obap] (tl) to (tfp) to (Rpush.center) to (bfp) to (bl) to (tl) -- cycle;
\fill[obb] (bf) to (Rpush.center) to (rf) to (lr) -- cycle;
\end{scope}
\end{tikzpicture}
\quad
\begin{tikzpicture}[scale=0.5,baseline=(tl)]
\path coordinate (tl) 
+(0,-4) coordinate (bl)
+(-1.5,0) coordinate[label=above:$g'$] (tgp)
+(-.5,-2.5) coordinate[label=$C'$] (Cp)
+(-2,-2) coordinate (middle)
+(-4,-2) coordinate[label=left:$g$] (rg)
+(-3.25,-1.5) coordinate[label=$C$] (C)
+(-1.5,-4) coordinate[label={[yshift=3pt]below:$g'$}] (bgp)
+(-3.5,-3.5) coordinate[label=$B$] (C)
+(-2.5,-4) coordinate[label=below:$g$] (bg)
+(-4,-4) coordinate (lr);
\node (Lpull) [vertex] at (middle.center) [text=white] {$\pmb{\pull\Lsym}$};
\draw[oedge,arrgp] (bgp) to (Lpull);
\draw[oedge,arrgp] (Lpull) to (tgp);
\draw[oedge,arrg] (bg) to (Lpull);
\draw[oedge,arrg] (Lpull) to (rg);
\begin{scope}[on background layer]
\fill[obc] (tl) rectangle (lr);
\fill[obb] (bg) to (Lpull.center) to (rg) to (lr) -- cycle;
\fill[obcp] (bgp) to (Lpull.center) to (tgp) to (tl) to (bl);
\end{scope}
\end{tikzpicture}
\quad\qquad
\begin{tikzpicture}[scale=0.5,baseline=(tl)]
\path coordinate (tl) 
+(0,-4) coordinate (bl)
+(0,-2) coordinate[label=right:$g$] (rg)
+(-.5,-1.5) coordinate[label=$C$] (C)
+(-1.75,0) coordinate[label=above:$g$] (tg)
+(-0.75,-3.5) coordinate[label=$B$] (B)
+(-2,-2) coordinate (middle)
+(-2.5,0) coordinate[label=above:$f$] (tf)
+(-3.5,-2.5) coordinate[label=$A$] (A)
+(-2.5,-4) coordinate[label=below:$f$] (bf)
+(-4,-4) coordinate (lr);
\node (Rpull) [vertex] at (middle.center) [text=white] {$\pmb{\pull\Rsym}$};
\draw[oedge,arrg] (rg) to (Rpull);
\draw[oedge,arrg] (Rpull) to (tg);
\draw[oedge,arrf] (bf) to (Rpull);
\draw[oedge,arrf] (Rpull) to (tf);
\begin{scope}[on background layer]
\fill[oba] (tl) rectangle (lr);
\fill[obb] (bl) to (rg) to (Rpull.center) to (bf) -- cycle;
\fill[obb] (tf) to (Rpull.center) to (tg) -- cycle;
\fill[obc] (tl) to (rg) to (Rpull.center) to (tg) -- cycle;
\end{scope}
\end{tikzpicture}
\end{equation*}
Then we remove the vertices, depicting the cells by simply bending edges appropriately:
\begin{equation*}
\begin{tikzpicture}[scale=0.5,baseline=(tl)]
\path coordinate (tl) 
+(0,-2) coordinate[label=left:$f$] (lf)
+(.5,-1.5) coordinate[label=$A$] (A)
+(1.75,0) coordinate[label=above:$f$] (tf)
+(1.75,-2.5) coordinate[label=$B$] (B)
+(2.5,0) coordinate[label=above:$g$] (tg)
+(3.5,-2.5) coordinate[label=$C$] (C)
+(2.5,-4) coordinate[label=below:$g$] (bg)
+(4,-4) coordinate (lr);
\draw[oedge,arrf] (lf) to[out=0, in=-90] (tf);
\draw[oedge,arrg] (bg) to (tg);
\begin{scope}[on background layer]
\fill[obc] (tl) rectangle (lr);
\fill[obb] (tl) rectangle (bg);
\fill[oba] (tl) to (lf) to[out=0,in=-90] (tf) --
cycle;
\end{scope}
\end{tikzpicture}
\quad \qquad
\begin{tikzpicture}[scale=0.5,baseline=(tl)]
\path coordinate (tl) 
+(1.5,0) coordinate[label=above:$f'$] (tfp)
+(.75,-2.5) coordinate[label=$A'$] (Ap)
+(4,-2) coordinate[label=right:$f$] (rf)
+(2.25,-2.5) coordinate[label=$A$] (A)
+(1.5,-4) coordinate[label={[yshift=1pt]below:$f'$}] (bfp)
+(3.5,-3.5) coordinate[label=$B$] (C)
+(2.5,-4) coordinate[label=below:$f$] (bf)
+(4,-4) coordinate (lr);
\draw[oedge,arrfp] (bfp) to (tfp);
\draw[oedge,arrf] (bf) to[out=90, in=180] (rf);
\begin{scope}[on background layer]
\fill[oba] (tl) rectangle (lr);
\fill[obap] (tl) rectangle (bfp);
\fill[obb] (bf) to[out=90,in=180] (rf) to (lr) --
cycle;
\end{scope}
\end{tikzpicture}
\quad
\begin{tikzpicture}[scale=0.5,baseline=(tl)]
\path coordinate (tl) 
+(-1.5,0) coordinate[label=above:$g'$] (tfp)
+(-.75,-2.5) coordinate[label=$C'$] (Ap)
+(-4,-2) coordinate[label=left:$g$] (rg)
+(-2.25,-2.5) coordinate[label=$C$] (A)
+(-1.5,-4) coordinate[label={[yshift=3pt]below:$g'$}] (bgp)
+(-3.5,-3.5) coordinate[label=$B$] (C)
+(-2.5,-4) coordinate[label=below:$g$] (bg)
+(-4,-4) coordinate (lr);
\draw[oedge,arrgp] (bgp) to (tgp);
\draw[oedge,arrg] (bg) to[out=90, in=0] (rg);
\begin{scope}[on background layer]
\fill[obc] (tl) rectangle (lr);
\fill[obcp] (tl) rectangle (bgp);
\fill[obb] (bg) to[out=90,in=0] (rg) to (lr) --
cycle;
\end{scope}
\end{tikzpicture}
\quad\qquad
\begin{tikzpicture}[scale=0.5,baseline=(tl)]
\path coordinate (tl) 
+(0,-2) coordinate[label=right:$g$] (lf)
+(-.5,-1.5) coordinate[label=$C$] (C)
+(-1.75,0) coordinate[label=above:$g$] (tf)
+(-1.75,-2.5) coordinate[label=$B$] (B)
+(-2.5,0) coordinate[label=above:$f$] (tg)
+(-3.5,-2.5) coordinate[label=$A$] (A)
+(-2.5,-4) coordinate[label=below:$f$] (bg)
+(-4,-4) coordinate (lr);
\draw[oedge,arrg] (lf) to[out=180, in=-90] (tf);
\draw[oedge,arrf] (bg) to (tg);
\begin{scope}[on background layer]
\fill[oba] (tl) rectangle (lr);
\fill[obb] (tl) rectangle (bg);
\fill[obc] (tl) to (lf) to[out=180,in=-90] (tf) --
cycle;
\end{scope}
\end{tikzpicture}
\end{equation*}
The permutation equivalences \eqref{eq:permgen1}--\eqref{eq:permgen4} corresponding to the stack relations \eqref{eq:permeq-on-double-cells} now have natural interpretations as planar isotopies.
For example, \eqref{eq:permgen1} and \eqref{eq:permgen2} (which permute $\push\Rsym$ with $\pull\Lsym$ and $\push\Lsym$ respectively) may be depicted as follows:
\begin{equation*}
\begin{tikzpicture}[scale=0.5,baseline=($(tl)!0.5!(lr)$)]
\path coordinate (tl) 
+(5,0) coordinate (tr)
+(0,-2) coordinate (lf)
+(1.75,-8) coordinate (bf)
+(2.5,0) coordinate (tg)
+(2.5,-8) coordinate (bg)
+(3.25,-8) coordinate (bh)
+(0,-8) coordinate (bl)
+(5,-8) coordinate (lr)
+(5,-6) coordinate (rh);
\draw[oedge,arrf] (bf) to[out=90, in=0] (lf);
\draw[oedge,arrh] (bh) to[out=90, in=180] (rh);
\draw[oedge,arrg] (bg) to (tg);
\begin{scope}[on background layer]
\fill[obd] (tg) rectangle (lr);
\fill[obb] (tl) rectangle (bg);
\fill[oba] (bl) to (bf) to[out=90,in=0] (lf) --
cycle;
\fill[obc] (bg) to (tg) to (tr) to (rh) to[out=180,in=90] (lr) -- cycle;
\fill[obd] (bh) to[out=90,in=190] (rh) to (lr) --
cycle;
\end{scope}
\end{tikzpicture}
\quad \permeq \quad
\begin{tikzpicture}[scale=0.5,baseline=($(tl)!0.5!(lr)$)]
\path coordinate (tl) 
+(5,0) coordinate (tr)
+(0,-6) coordinate (lf)
+(1.75,-8) coordinate (bf)
+(2.5,0) coordinate (tg)
+(2.5,-8) coordinate (bg)
+(3.25,-8) coordinate (bh)
+(0,-8) coordinate (bl)
+(5,-8) coordinate (lr)
+(5,-2) coordinate (rh);
\draw[oedge,arrf] (bf) to[out=90, in=0] (lf);
\draw[oedge,arrh] (bh) to[out=90, in=180] (rh);
\draw[oedge,arrg] (bg) to (tg);
\begin{scope}[on background layer]
\fill[obd] (tg) rectangle (lr);
\fill[obb] (tl) rectangle (bg);
\fill[oba] (bl) to (bf) to[out=90,in=0] (lf) --
cycle;
\fill[obc] (bg) to (tg) to (tr) to (rh) to[out=180,in=90] (lr) -- cycle;
\fill[obd] (bh) to[out=90,in=190] (rh) to (lr) --
cycle;
\end{scope}
\end{tikzpicture}
\qquad\qquad
\begin{tikzpicture}[scale=0.5,baseline=($(tl)!0.5!(lr)$)]
\path coordinate (tl) 
+(0,-2) coordinate (lf)
+(1.75,0) coordinate (tf)
+(2.5,0) coordinate (tg)
+(2.5,-8) coordinate (bg)
+(3.25,-8) coordinate (bh)
+(5,-8) coordinate (lr)
+(5,-6) coordinate (rh);
\draw[oedge,arrf] (lf) to[out=0, in=-90] (tf);
\draw[oedge,arrh] (bh) to[out=90, in=180] (rh);
\draw[oedge,arrg] (bg) to (tg);
\begin{scope}[on background layer]
\fill[obc] (tl) rectangle (lr);
\fill[obb] (tl) rectangle (bg);
\fill[oba] (tl) to (lf) to[out=0,in=-90] (tf) --
cycle;
\fill[obd] (bh) to[out=80,in=190] (rh) to (lr) --
cycle;
\end{scope}
\end{tikzpicture}
\quad \permeq \quad
\begin{tikzpicture}[scale=0.5,baseline=($(tl)!0.5!(lr)$)]
\path coordinate (tl) 
+(0,-6) coordinate (lf)
+(1.75,0) coordinate (tf)
+(2.5,0) coordinate (tg)
+(2.5,-8) coordinate (bg)
+(3.25,-8) coordinate (bh)
+(5,-8) coordinate (lr)
+(5,-2) coordinate (rh);
\draw[oedge,arrf] (lf) to[out=0, in=-90] (tf);
\draw[oedge,arrh] (bh) to[out=90, in=180] (rh);
\draw[oedge,arrg] (bg) to (tg);
\begin{scope}[on background layer]
\fill[obc] (tl) rectangle (lr);
\fill[obb] (tl) rectangle (bg);
\fill[oba] (tl) to (lf) to[out=0,in=-90] (tf) --
cycle;
\fill[obd] (bh) to[out=90,in=190] (rh) to (lr) --
cycle;
\end{scope}
\end{tikzpicture}
\end{equation*}
The equations defining cut-elimination (Definition~\ref{defi:cut}) also have natural topological interpretations.
Below we illustrate one principal cut case and one commutative cut case:
\begin{mathpar}
\begin{tikzpicture}[scale=0.38,baseline=($(tl)!0.5!(lr)$)]
\path coordinate (tl) 
+(1.5,0) coordinate (tfp)
+(.75,-2.5) coordinate (Ap)
+(4,-2) coordinate (rf)
+(2.25,-2.5) coordinate (A)
+(1.5,-4) coordinate (bfp)
+(3.5,-3.5) coordinate (C)
+(2.5,-4) coordinate (bf)
+(4,-4) coordinate (lr);
\draw[oedge,arrfp] (bfp) to (tfp);
\draw[oedge,arrf] (bf) to[out=90, in=180] (rf);
\begin{scope}[on background layer]
\fill[oba] (tl) rectangle (lr);
\fill[obap] (tl) rectangle (bfp);
\fill[obb] (bf) to[out=90,in=180] (rf) to (lr) --
cycle;
\end{scope}
\end{tikzpicture}
  \ \hcomp\ 
\begin{tikzpicture}[scale=0.38,baseline=($(tl)!0.5!(lr)$)]
\path coordinate (tl) 
+(0,-2) coordinate (lf)
+(1.75,0) coordinate (tf)
+(2.5,0) coordinate (tg)
+(3.5,-2.5) coordinate (C)
+(2.5,-4) coordinate (bg)
+(4,-4) coordinate (lr);
\draw[oedge,arrf] (lf) to[out=0, in=-90] (tf);
\draw[oedge,arrg] (bg) to (tg);
\begin{scope}[on background layer]
\fill[obc] (tl) rectangle (lr);
\fill[obb] (tl) rectangle (bg);
\fill[oba] (tl) to (lf) to[out=0,in=-90] (tf) --
cycle;
\end{scope}
\end{tikzpicture}
\ = \
\begin{tikzpicture}[scale=0.38,baseline=($(tl)!0.5!(lr)$)]
\path coordinate (tl)
+(1,0) coordinate (tf)
+(2,0) coordinate (tg)
+(3,0) coordinate (th)
+(1,-4) coordinate (bf)
+(2,-4) coordinate (bg)
+(3,-4) coordinate (bh)
+(4,-4) coordinate (lr);
\draw[oedge,arrfp] (bf) to (tf);
\draw[oedge,arrf] (bg) to (tg);
\draw[oedge,arrg] (bh) to (th);
\begin{scope}[on background layer]
\fill[obap] (tl) rectangle (lr);
\fill[oba] (tf) rectangle (lr);
\fill[obb] (tg) rectangle (lr);
\fill[obc] (th) rectangle (lr);
\end{scope}
\end{tikzpicture}

\begin{tikzpicture}[scale=0.38,baseline=($(lf)!0.5!(lr)$)]
\path coordinate (tl) 
+(0,-2) coordinate (lf)
+(2,0) coordinate (tf)
+(4,-2) coordinate (rh)
+(2,-2) coordinate (middle);
\path (4,-4) coordinate (tlx) 
+(0,-2) coordinate (llx) 
+(-4,-0) coordinate (rgx)
+(-1.25,-2) coordinate (bgpx)
+(-2.5,-2) coordinate (bgx)
+(-4,-2) coordinate (lrx);
\draw[oedge,arrh] (bgpx) to[out=90,in=-90] (middle.center);
\draw[oedge,arrg] (bgx) to[out=90, in=0] (rgx);
\begin{scope}[on background layer]
\fill[obb] (bgx) to[out=90,in=0] (rgx) to (lrx) --
cycle;
\fill[obc] (bgpx) to[out=90,in=-90] (middle.center) to (lf) to (rgx) to[out=0,in=90] (bgx) -- cycle;
\fill[obd] (llx) to (bgpx) to[out=90,in=-90] (middle.center) to (rh) -- cycle;
\end{scope}
\node (cut) at (4.75,-4) {$\hcomp$};
\path (5.5,0) coordinate (tl) 
+(0,-6) coordinate (bl)
+(0,-2) coordinate (lf)
+(.5,-1.5) coordinate (A)
+(4,-2) coordinate (rh)
+(0.75,-3.5) coordinate (B)
+(2,-2) coordinate (middle)
+(3.5,-2.5) coordinate (C)
+(2,-6) coordinate (bg)
+(4,-6) coordinate (lr);
\draw[oedge,arrf] (bg) to (middle.center);
\begin{scope}[on background layer]
\fill[obep] (lf) rectangle (lr);
\fill[obd] (bl) to (lf) to (middle.center) to (bg) -- cycle;
\end{scope}
\end{tikzpicture}
\ = \
\begin{tikzpicture}[scale=0.38,baseline=($(tl)!0.5!(llrb)$)]
\path coordinate (x)
+(0,-2) coordinate (tl) 
+(0,-4) coordinate (bl)
+(0,-2) coordinate (lf)
+(.5,-1.5) coordinate (A)
+(4,-2) coordinate (rh)
+(2,-2) coordinate (middle);
\path (4,-4) coordinate (tlx) 
+(0,-2) coordinate (llx) 
+(-4,0) coordinate (rgx)
+(-4,-2) coordinate (lcc)
+(-0.75,-2) coordinate (bgpx)
+(-2.5,-2) coordinate (bgx)
+(-4,-4) coordinate (lrx);
\draw[oedge,arrh] (bgpx) to[out=90,in=-90] (middle);
\draw[oedge,arrg] (bgx) to[out=90, in=0] (rgx);
\begin{scope}[on background layer]
\fill[obb] (bgx) to[out=90,in=0] (rgx) to (lcc) --
cycle;
\end{scope}
\path (3,0) coordinate (tlb) 
+(0,-2) coordinate (lfb)
+(.5,-1.5) coordinate (Ab)
+(1.5,-2) coordinate (middleb)
+(2,-2) coordinate (bgb)
+(4,-2) coordinate (lrb)
+(4,-6) coordinate (llrb)
+(2,-6) coordinate (bbgb)
+(0,-6) coordinate (lllb);
\draw[oedge,arrf] (bbgb) to [out=90,in=-90] (middleb.center);
\begin{scope}[on background layer]
\fill[obep] (lfb) rectangle (llrb);
\fill[obd] (bbgb) to (middleb.center) to (middle.center) to [out=-90,in=90] (bgpx) -- cycle;
\fill[obep] (bbgb.center) to (llrb) to (lrb) to (middleb.center) to [out=-90,in=90] (bbgb) -- cycle;
\fill[obc] (bgpx) to[out=90,in=-90] (middle) to (lf) to (rgx) to[out=0,in=90] (bgx) -- cycle;
\end{scope}
\end{tikzpicture}
\end{mathpar}
We recover in this way the calculus of planar arc diagrams that was used by Dawson, Paré, and Pronk~\cite[\S6]{DPP03} to represent 2-cells in $\Pi_2(\C)$ for the special case that $\C$ is a free category.
More generally, when $\C$ is not free we also need to include vertices representing equations in the base category.
For example, an equation $h = fg$ that is applied implicitly within a sequent proof, as depicted on the left side of Figure~\ref{fig:example-functor-and-derivation}, may be represented by a trivalent vertex in the corresponding string diagrammic proof, as depicted on the right side of that figure.
As also depicted there, we also need to include terminating vertices corresponding to the initial axioms in $\D$.

The graphical calculus that we thus obtain is comparable to the string diagram calculus for proarrow equipments of David Jaz Myers~\cite{DJMyers2018}.

Given the isomorphisms between the three different ways of representing arrows of free bifibrations---as sequent calculus proofs, as stacks of double cells, or as string diagrams---in the sequel we will freely interweave them without comment.










\subsection{Related work}

Since $\Bifib{p}$ and $\Pi_2(\C)$ are determined by their universal properties, a detailed comparison with the constructions of Dawson, Paré, and Pronk~(DPP)~\cite{DPP03} and of Lamarche~\cite{Lam13} feels a bit superfluous, and the reader interested in the details of those constructions is referred to those papers.
Still, a few points are worth highlighting.
Zigzags and morphisms between them are essential in all three constructions, but whereas Lamarche (like us) considers zigzags as not necessarily alternating sequences of signed arrows, DPP restrict to strictly alternating zigzags beginning and ending with a forward arrow.
This is motivated by the fact that they wanted the embedding of $\C$ into $\Pi_2(\C)$ to define a strict functor $\push{(-)} : \C \to \Pi_2(\C)$ rather than a pseudofunctor.
Although the free bifibration is not split and hence $\push{(-)} : \C \to \vcat{\ZZ(\C)}$ defines a proper pseudofunctor, we will also have motivations to consider strictly alternating zigzags beginning in Section~\ref{sec:focusing:weak-focusing}.

All three constructions likewise involve first introducing some notion of ``pre-morphism'' between zigzags and then defining zigzag morphisms as equivalence classes of pre-morphisms modulo an equivalence relation.
Lamarche notes that ``these classes can be quite complex, and there is no hope of finding normal forms in general'' \cite[p.16]{Lam13}.
In the next section we will establish that it is possible to obtain normal forms in the case that the base category is \emph{factorization preordered,} a condition that was also identified by DPP as a sufficient condition for decidability.

Interestingly, categories and double categories of zigzags have also been introduced in other contexts.
A category of zigzags and zigzag morphisms was defined by Heidemann, Reutter, and Vicary \cite{Heidemann+2022} motivated by the formalization of higher-dimensional diagrams and $n$-categories.
Their definition of zigzag map \cite[Def.~2.2]{Heidemann+2022} is nearly identical to DPP's definition of \emph{fence} \cite[Def.~1.2]{DPP03}, without condition (c) which enforces globularity.
However, Heidemann, Reutter, and Vicary do not quotient zigzag maps by permutation equivalence.
Hence their zigzag category is \emph{not} equivalent to $\Z(\C)$, but rather to Lamarche's category $\mathrm{P\!\!e}(\C)$, from which $\mathrm{P}(\C) \simeq \Z(\C)$ is obtained as the quotient by permutation equivalence.
Our zigzag double category $\ZZ(\C)$ is equivalent to the ``weave double category'' by Williams~\cite[Def.~22]{Williams2023phd}, as follows from Williams' definition as a fibered coproduct of arrow and oparrow double categories, which freely add companions and conjoints respectively.
As mentioned, Lamarche first identified the double category structure on $\mathrm{P}(\C) \simeq \Z(\C)$, although it was not emphasized in his papers \cite{Lam13,Lam14}.\footnote{
  According to personal communication from François~Lamarche (November~2025), 
  he noticed the close connection between this double category and the 2-category $\Pi_2(\C)$, discussing it in October~2013 by email with Paré, Dawson, and Pronk.
  Although the originally submitted version of the manuscript~\cite{Lam13} did not mention DPP's work, Lamarche says that this was an oversight and that he had originally intended to include a comparison.
  The manuscript~\cite{Lam13} was submitted for publication but unfortunately ran into some negative referee reports and never made it to print.}
An alternative double categorical perspective on the $\Pi_2(\C)$ construction was also developed by DPP at the end of their paper on the Span construction~\cite{DPP2010Span}.

Perhaps one of the most interesting open questions in double category theory is an explicit construction that freely adds companions and conjoints to a double category $\DD$. For each category $\C$, there is a double category $\mathbb{H}(\C)$ whose objects and tight morphisms are the objects and morphism of $\C$, and whose loose morphisms and cells are identities.
In Theorem~\ref{thm:free-fibrant-double-category}, we established $\Z(\C)$ as the free companion and conjoint completion of the double category $\D = \mathbb{H}(\C)$.
The free cornering of a strict monoidal category, introduced by Nester \cite{Nester2021} (see also Section 1 of \cite{BoisseauNesterRoman2022}) is another example of a free companion and conjoint completion, in this case, of a strict monoidal category viewed as a double category.
Coulembier, Street, and van den Bergh likewise considered the closely related problem of freely adjoining duals to specified objects of a monoidal category~\cite{coulembier2021freely}

\section{Focused normal forms}
\label{sec:focusing}

In Section~\ref{sec:sequent-calculus}, we described how to construct $\Bifib{p}$, the total category of the free bifibration on a functor $p : \D \to \C$, by taking its arrows to be equivalence classes of derivations in a sequent calculus modulo a notion of permutation equivalence, with composition defined by the admissible cut rule.
In this section we will use additional ideas from proof theory---in particular the notion of \emph{focusing} originating in the proof theory of linear logic---to get a more precise characterization of the relative homsets $\Bifib{p}_f(S,T)$, via a series of sequent calculi with progressively more rigid structure.
In the case that $\C$ is \emph{factorization preordered} in the sense of Johnstone~\cite{Johnstone1999}, we derive a normal form theorem that leads to a complete inductive characterization of these relative homsets, without any quotienting.
As a corollary, we obtain a simple procedure for deciding permutation equivalence of derivations, as well as for enumerating  relative homsets without duplicates.

Lurking in the background is the fact that in the general case, permutation equivalence is undecidable.
Dawson, Paré, and Pronk~\cite{DPP03b} established an analogue of this undecidability result for their freely adjoining adjoints construction $\Pi_2(\C)$, and their proof translates readily to our presentation of the zigzag double category $\ZZ(\C)$.
Nevertheless, we will see in Sections~\ref{sec:PN-fibrations} and ~\ref{sec:examples} that the canonicity result has wide applications.

\subsection{Background: invertibility, polarity, and synthetic connectives}
\label{sec:focusing:background}

Focusing, introduced by Andreoli~\cite{Andreoli1992} as an optimization for linear logic proof search and subsequently applied in other settings, combines several observations about the rules of sequent calculus and the way that they determine the behavior of the connectives.
To provide some context and motivation for our approach, we take a moment here to briefly recall these observations in the setting of linear logic \cite{Girard1987}, before explaining how to transfer these ideas to analyze free bifibrations.

We consider a fragment of the sequent calculus for intuitionistic linear logic, containing the inference rules for tensor $\otimes$ (also called \emph{multiplicative conjunction}) and internal hom $\lolli$ (also called \emph{linear implication}):
\begin{mathpar}
\inferrule*[Right={$\Lotimes$}]{A,B,\Gamma \vdash C}{A\otimes B,\Gamma\vdash C}

\inferrule*[Right={$\Rotimes$}]{\Gamma \vdash A \\ \Delta \vdash B}{\Gamma,\Delta\vdash A\otimes B}

\inferrule*[Right={$\Llolli$}]{\Gamma \vdash A \\ B,\Delta\vdash C}{A\lolli B,\Gamma,\Delta \vdash C}

\inferrule*[Right={$\Rlolli$}]{\Gamma,A \vdash B}{\Gamma\vdash A\lolli B}
\end{mathpar}
Here the basic judgment $\Gamma \vdash B$ takes a single formula $B$ on the right and a list of formulas $\Gamma = A_1,\dots,A_n$ on the left considered up to reordering, and we write a comma for concatenation of lists.
The above inference rules are combined with initial axioms of the form $X \vdash X$ for every atomic formula $X$.
Identity derivations for arbitrary formulas as well as the cut rule 
\begin{mathpar}
\inferrule*[Right={$\Id[A]$}]
        { }
        {A \vdash A}

\inferrule*[Right={$\cut$}]
        {\Gamma \vdash A \\ A,\Delta \vdash B}
        {\Gamma,\Delta \vdash B}
\end{mathpar}
are admissible,
similarly to how we saw in Section~\ref{sec:sequent-calculus}.

A basic observation about the inference rules for the logical connectives is that the rules $\Lotimes$ and $\Rlolli$ are \emph{invertible} in the proof-theoretic sense, meaning that if the conclusion is derivable then the premise is derivable.
Indeed, we can derive inverses of $\Lotimes$ and $\Rlolli$ using the admissible cut and identity rules, combined with instances of $\Rotimes$ and $\Llolli$:

\begin{mathpar}
\inferrule*{\inferrule*{\inferrule*{ }{A \vdash A} \\ \inferrule*{ }{B \vdash B} }{A,B \vdash A\otimes B} \\
A\otimes B,\Gamma\vdash C}
{A,B,\Gamma \vdash C}

\inferrule*{\Gamma\vdash A\lolli B \\ \inferrule*{\inferrule*{ }{A \vdash A} \\ \inferrule*{ }{B \vdash B} }{A\lolli B,A \vdash B}}
{\Gamma,A \vdash B}
\end{mathpar}
One interest of invertible inference rules is that they may be safely applied as \emph{goal transformers} during proof search.
For example, if our goal is to prove a sequent of the form $A\otimes B,\Gamma \vdash C$, we can immediately apply the $\Lotimes$ rule and try to prove $A, B, \Gamma \vdash C$, knowing that if there is any proof of the original goal then there is one that goes via $\Lotimes$.
Moreover, if a goal matches the conclusion of more than one invertible rule, then the order in which these rules are invoked is irrelevant.
For example, a goal of the form $A \otimes B,\Gamma \vdash C \lolli D$ can be transformed into the equivalent goal $A,B,\Gamma,C \vdash D$ by applying the $\Lotimes$ and $\Rlolli$ rules in either order.

On the other hand, the rules $\Rotimes$ and $\Llolli$ are not invertible in this sense.
Intuitively, this is due to the fact that list concatenation is not an invertible operation, meaning that passing from the conclusion to the premises requires finding an appropriate splitting of the formulas on the left-hand side of the turnstile.
As a simple counterexample to invertibility of $\Rotimes$, consider the sequent $X\otimes Y \vdash X\otimes Y$. It is derivable, but attempting to prove it using the rule $\Rotimes$ produces two unprovable premises---no matter how we decide to distribute the unique formula on the left of the sequent:
\begin{mathpar}
  \infer
  {X \otimes Y \vdash X \otimes Y}
  {X \otimes Y \vdash X \\ \cdot \vdash Y}

  \infer
  {X \otimes Y \vdash X \otimes Y}
  {\cdot \vdash X \\ X \otimes Y \vdash Y}
\end{mathpar}


Focusing originally arose from the idea that proof search could be structured to alternate between an \emph{inversion phase,} during which invertible rules are applied eagerly as goal transformers, and a \emph{focus phase,} during which non-invertible rules are applied to decompose a single formula and its subformulas.
It was found that this greatly reduces the overall non-determinism of the sequent calculus, leading to more efficient proof search.
Moreover, an interesting duality emerges, with half of the connectives being invertible on the \emph{left} and half being invertible on the \emph{right.}

This left or right bias was referred to as ``polarity''---with $\otimes$ being a \emph{positive} connective and $\lolli$ a \emph{negative} connective---and helped to explain some of the distinctive associativity, distributivity, and non-distributivity properties of linear logic.
For example,
like tensor, additive disjunction $\oplus$ is also invertible on the left, and distributes through $\otimes$ by the equivalence $A\otimes (B\oplus C) \equiv (A \otimes B)\oplus (A\otimes C)$.
Similarly, additive conjunction $\with$ is invertible on the right, like internal hom, and distributes through the codomain of $\lolli$ by the equivalence $A \lolli (B\with C) \equiv (A \lolli B)\with (A \lolli C)$.

Related to these properties is the possibility of considering certain compositions of the logical connectives as ``synthetic connectives'' \cite{Andreoli2001}.
For example, consider the composite formula $A_1 \lolli (A_2 \lolli B)$.
By applying the left rule twice and the right rule twice, we can build the following open derivations:

\begin{mathpar}
\inferrule*[Right={$\Llolli$}]{\Gamma_1 \vdash A_1 \\
 \inferrule*[Right={$\Llolli$}]{\Gamma_2 \vdash A_2 \\ B,\Delta \vdash C }{A_2 \lolli B, \Gamma_2, \Delta \vdash C}}{A_1 \lolli (A_2 \lolli B),\Gamma_1,\Gamma_2, \Delta \vdash C}

\inferrule*[Right={$\Rlolli$}]{
  \inferrule*[Right={$\Rlolli$}]{
     \Gamma,A_1, A_2 \vdash B
  }{
    \Gamma,A_1  \vdash A_2\lolli B
  }
}{
  \Gamma \vdash A_1 \lolli (A_2 \lolli B)
}
\end{mathpar}
%
If we remove the intermediate steps, then these derivations could be seen as inference rules for a new ternary connective $(A_1,A_2) \lolli' B$:
\begin{mathpar}
\inferrule*[Right={$\Llolli'$}]{
  \Gamma_1 \vdash A_1 \\
  \Gamma_2 \vdash A_2 \\
  B,\Delta \vdash C}{(A_1, A_2) \lolli' B,\Gamma_1,\Gamma_2,\Delta \vdash C}

\inferrule*[Right={$\Rlolli'$}]{\Gamma, A_1, A_2 \vdash B}{\Gamma \vdash (A_1, A_2) \lolli' B}
\end{mathpar}
%
%
%
Adding this new connective and these inference rules preserves all the important properties of the sequent calculus, in particular admissibility of identity and cut.
%
%
On the other hand, not every combination of basic connectives can be viewed as a first-class, synthetic connective.
For example, there is no good way to define a right rule for the composite connective $A \lolli (B_1 \otimes B_2)$, intuitively because after applying the invertible rule $\Rlolli$ to decompose the outer implication, we may need to examine the formula $A$ on the left before applying the non-invertible rule $\Rotimes$ to decompose the tensor.

\subsection{Invertibility and polarity in the bifibrational calculus}
\label{sec:focusing:invertibility-bifibrations}
        
Returning to bifibrational logic, we begin by observing that pushforward behaves like a positive connective and pullback like a negative connective, since $\push{f}$ has an invertible left rule and non-invertible right rule, while $\pull{g}$ has an invertible right rule and non-invertible left rule.
We already established the invertibility of $\Lpush$ and $\Rpull[g]$ with Lemma~\ref{lem:eta-expansion-pushpull}, in the strong sense that they define bijections between proofs of their premises and proofs of their conclusions up to permutation equivalence.
Indeed, for any derivation $\alpha : \push{f} S \vdashf{g} T$ there is a permutation equivalent derivation ending in the $\Lpush$ rule, namely the one represented by the term $\uLpush f g (\opcart f S \hcomp \alpha)$, which is $\Lpush$ applied to the derivation of $S \vdashf{fg} T$ obtained by cutting $\alpha$ with the cartesian lifting $\opcart f S : S \vdashf{f} \push{f} S$.
Dually, for any derivation $\beta : S \vdashf{f} \pull{g} T$ there is a permutation equivalent derivation ending in the $\Rpull[g]$ rule, namely $(\beta \hcomp \cart g T) \uRpull f g$.

To see why the $\Rpush$ rule is \emph{not} invertible in general, consider the free bifibration on the functor depicted below
\[
  \begin{tikzcd}
\D\ar[d,"p"] &    X \ar[r,"\delta",bend left] & Y &  \\
\C &    E \ar[r,"e_1",bend left]\ar[r,"e_2"',bend right] & A \ar[r,"f"] & B
  \end{tikzcd}
\]
where the base category $\C$ contains two parallel arrows $e_1,e_2$ and a third arrow $f$ such that $e_1 f = e_2 f = h$, and where $\D$ is generated by a single arrow $\delta$ such that $p(\delta) = e_1$.
Observe that the sequent $X \vdashf{h} \push{f} Y$ is derivable by applying $\Rpush$ to the atomic derivation $\atom{\delta}:X \vdashf{e_1} Y$.
Since $h = e_2 f$ is another valid factorization, if the $\Rpush$ rule were invertible then there should also be some derivation of $X \vdashf{e_2} Y$.
But by conservativity (Prop.~\ref{prop:conservativity}) we can immediately tell that there isn't, contradicting invertibility of $\Rpush$.
Dualizing the example establishes non-invertibility of the $\Lpull[g]$ rule.
Of course these counterexamples do not preclude the possibility of $\Rpush[f]$ and $\Lpull[g]$ being invertible in specific cases---for example, they will be when $\C$ is a groupoid.
Nevertheless, the behavior of the connectives in the general case justifies our adopting the polarity terminology to bifibrational formulas, and from now on we will refer to pushforwards as positive and pullbacks as negative, with atomic formulas having neutral polarity.

Continuing the analogy with linear logic, compositions of like-polarity connectives may be grouped into synthetic connectives, in the sense that $\push g \push f$ can be represented by $\push{(g\circ f)}$ and $\pull f \pull g$ can be represented by $\pull {(f \; g)}$, corresponding to the pseudofunctoriality laws of a bifibration.
We already explained how these laws are witnessed by logical equivalences in the bifibrational calculus (Prop.~\ref{prop:pseudofunctoriality}).
Here we simply point out that, similarly to the linear logic situation discussed in Section~\ref{sec:focusing:background}, the inference rules for the synthetic connectives are operationally equivalent to compositions of basic inference rules, in the sense that, for example, by composing the left rules and right rules for pushforward along $f$ and $g$,
\begin{mathpar} 
  \inferrule* [Right={$\Lpush[g]$}]
        {\inferrule* [Right={$\Lpush[f]$}]{S \vdashf{fgh} T}{\push{f} S\vdashf{gh} T}}
        {\push{g}\push{f} S \vdashf{h} T}

  \inferrule* [Right={$\Rpush[g]$}]
        {\inferrule* [Right={$\Rpush[f]$}]{S' \vdashf{f'} S}{S' \vdashf{f'f} \push{f} S}}
        {S' \vdashf{f'fg} \push{g}\push{f} S}
\end{mathpar}
we obtain derivations with identical premises to the rules for the pushforward along $g\circ f$:
\begin{mathpar} 
  \inferrule* [Right={$\Lpush[(g\circ f)]$}]
        {S \vdashf{fgh} T}
        {\push{(g \circ f)} S \vdashf{h} T}

  \inferrule* [Right={$\Rpush[(g\circ f)]$}]
        {S' \vdashf{f'} S}
        {S' \vdashf{f'fg} \push{(g\circ f)} S}
\end{mathpar}
A difference with the linear logic situation is that such ``synthetic'' connectives are already contained in the logic, since we assume pushforward and pullback along arbitrary arrows, and the arrows of a category are closed under composition.
On the other hand, in general it is not possible to reduce opposite-polarity composites like $\push{f}\pull{g}$ or $\pull{g}\push{f}$ to a single pushforward or pullback, just like alternating quantifiers $\exists\forall$ or $\forall\exists$ in predicate logic cannot be reduced to a single quantifier in general.

\subsection{Strictly alternating formulas, splittings, and weak focusing}
\label{sec:focusing:weak-focusing}

Motivated by the preceding observations, in this section we introduce a restriction of the sequent calculus in which sequences of invertible rules and sequences of non-invertible rules must be performed all at once, on the left or right, but without (yet) placing any restrictions on their interleavings.
This calculus may be said to be ``weakly focused'', 
and will serve as an intermediate point before our introduction of a strongly focused calculus in the next section.

\subsubsection{Strictly alternating formulas}

Before defining the calculus itself, we motivate it logically and categorically by observing that every bifibrational formula is canonically isomorphic to a \emph{strictly alternating formula.}
\begin{defi}    
A bifibrational formula is \defin{strictly alternating} if it contains no subformula of the form $\push{g}\push{f} S$ or $\pull{f}\pull{g} T$.
\end{defi}
\noindent
It is easy to see that by repeatedly applying the equivalences $\push{g}\push{f} S \equiv \push{(g\circ f)} S$ and $\pull{f}\pull{g} T \equiv \pull{(f\;g)}T$ (Prop.~\ref{prop:pseudofunctoriality}), any formula $S$ may be normalized to a logically equivalent strictly alternating formula $\collapse S$.
Let us write $\theta_S : S \equiv \collapse{S}$ for this family of equivalences, and define a function on derivations
\[
  \begin{array}{lcccc}
  \collapse - & : & (S \vdashf{f} T) & \to &  (\collapse S \vdashf{f} \collapse T) \\
  \collapse - & \defeq & \alpha & \mapsto & \theta_S^{-1} \hcomp \alpha \hcomp \theta_T \\
  \end{array}
\]
which we call the \emph{strictification} map.
Observe that by the subformula property (Prop.~\ref{prop:subformula-property}), the strictification map transforms any derivation into a derivation involving only strictly alternating formulas, since every subformula of a strictly alternating formula is strictly alternating.

Now let $\saBifib{p}$ denote the full subcategory of $\Bifib{p}$ spanned by the strictly alternating formulas, let $i : \saBifib{p} \to \Bifib{p}$ be the full and faithful inclusion functor, and let us write $\BifibFun{p}|_{sa} = \BifibFun{p} \circ i : \saBifib{p} \to \C$ for the restriction of the domain of the free bifibration to strictly alternating formulas.

\begin{thm}\label{thm:strictify-equivalence}
Strictification extends to a functor $\collapse{-} : \Bifib{p} \to \saBifib{p}$ realizing an equivalence of categories over $\C$:
\[
\begin{tikzcd}
\Bifib{p}\ar[ddr,"\BifibFun{p}"']
\ar[rr,bend left,"\collapse{-}"{name=L}] && \saBifib{p} \ar[ldd,"\BifibFun{p}|_{sa}"] 
\ar[ll,bend left,"i"'{name=R}]
\ar[phantom,from = L, to = R,"\sim"] \\\\
&\C&
\end{tikzcd}
\]
\end{thm}
\begin{proof}
It is immediate from the definition of the $\collapse{-}$ map on derivations that
\begin{enumerate}
\item it respects permutation equivalence (since, by Lemma~\ref{lem:cut-respects-permeq}, cut does so);
\item it satisfies the functor laws $\collapse{\alpha\hcomp \beta} \permeq \collapse{\alpha}\hcomp\collapse{\beta}$ and $\collapse{\Id[S]} \permeq \Id[\collapse{S}]$ (since the derivations $\theta_S$ and $\theta_S^{-1}$ cancel out);
\item it leaves the underlying objects and arrows in $\C$ unchanged;
\item it is invariant on derivations in the strictly alternating fragment.
\end{enumerate}
Hence strictification defines a functor $\collapse{-} : \Bifib{p} \to \saBifib{p}$ that commutes with the projection functors and such that $\collapse{-}\circ i = \Id[\saBifib{p}]$.
Finally, the family of logical equivalences $\theta_S : S \equiv \collapse{S}$ realize an isomorphism $i \circ \collapse{-} \cong \Id[\Bifib{p}]$ that is natural by construction.
Indeed, it suffices to show that the square
\[
\begin{tikzcd}  
S \arrow[r,"\alpha"]\arrow[d,"\theta_S"'] & T\arrow[d,"\theta_T"] \\
i\collapse{S} \arrow[r,"i\collapse{\alpha}"'] & i\collapse{T} 
\end{tikzcd}
\]
commutes for every arrow $\alpha : S \to T$ in $\Bifib{p}$, which is equivalent to asserting that every derivation $\alpha : S \vdashf{f} T$ may be factored up to permutation equivalence as 
\begin{equation}
\begin{tikzcd}  
S \arrow[r,Rightarrow,"f"',"\alpha"] & T 
\end{tikzcd}
\permeq
\begin{tikzcd}  
S \stackrel{\theta_S}{\equiv} \collapse{S} \arrow[r,Rightarrow,"f"',"\collapse{\alpha}"] & \collapse{T} \stackrel{\theta_T^{-1}}{\equiv} T        
\end{tikzcd}.
\end{equation}
This holds by definition of the strictification map $\collapse -$, as we have
\[
\theta_S \hcomp \collapse \alpha \hcomp \theta_T
\quad=\quad
\theta_S \hcomp (\theta_S^{-1} \hcomp \alpha \hcomp \theta_T^{-1}) \hcomp \theta_T
\quad\permeq\quad
\alpha
\]
\end{proof}
\begin{cor}       
The functor $\BifibFun{p}|_{sa} : \saBifib{p} \to \C$ is a bifibration.
\end{cor}
\begin{rem}\label{rem:not-quite-split-bifibration}
The bifibrational structure of $\BifibFun{p}|_{sa}$ may be described explicitly, with the pushforward and pullback of strictly alternating formulas in $\saBifib{p}$ along arrows of $\C$ defined by cases depending on the polarity of the formulas, so as to ensure strict alternation.
For instance, the pushforward of a strictly alternating formula $S \refs A$ along an arrow $f : A \to B$ is defined as follows:
\begin{itemize}
\item If $S$ is negative or atomic, then the pushforward is given by $\push{f} S$ equipped with the ``free'' cartesian lifting  $\opcart f S : S \to \push f S$.
\item If $S$ is positive, i.e., of the form $S = \push{f'} S'$ for some $S' \refs A'$ and $f' : A' \to A$, then $\push{f}S$ is defined as the pushforward along the composite $\push{f}S := \push{(f \circ f')}S'$, equipped with the cartesian arrow given by composing $\opcart f S : \push{f'}S' \to \push{f}\push{f'}S'$ with the logical equivalence $\push{f}\push{f'}S'\equiv \push{(f \circ f')}S'$.
\end{itemize}
The pullback is defined symmetrically. \qed 
\end{rem}
\subsubsection{Aside: splitting bifibrations}\label{sec:focusing:weak-focusing:splitting}

It is worth observing that the bifibrational structure on the subcategory of strictly alternating formulas described in Remark~\ref{rem:not-quite-split-bifibration} is \emph{nearly} split: the pushforward and pullback functors preserve composition on the nose, but identity is only preserved up to isomorphism witnessed by the logical equivalences $\push{\Id} U \equiv U$ and $\pull{\Id} U \equiv U$.
Thus the associated pseudofunctor $\C \to \Adj$ representing the bifibration $\saBifib{p} \to \C$ is not fully strict (it is not ``normalized''), although it is a strict functor of semicategories.

It is clear that a split bifibration could be obtained, if desired, by adjusting the construction to deal explicitly with identity arrows. 
Indeed, it suffices to restrict further to the full subcategory of $\Bifib{p}$ spanned by the formulas that strictly alternate along nonidentity arrows, which we denote $\nsaBifib{p}$.
The equivalence $\Bifib{p} \simeq \nsaBifib{p}$ can be established just like Theorem~\ref{thm:strictify-equivalence}, applying the full suite of pseudofunctoriality laws \eqref{eq:pushpush}--\eqref{eq:pushidpull} to normalize any formula to an equivalent nonidentity strictly alternating formula.
The split bifibrational structure on $\nsaBifib{p}$ is then almost identical to what was described for $\saBifib{p}$ in Remark~\ref{rem:not-quite-split-bifibration}, but taking care to never push or pull along an identity arrow.
For example, when defining the pushforward of $\push{f'}S'$ along $f$, if $f \circ f' = \Id$ then we choose the pushforward object to be $S'$.
We can derive the following as a consequence.
\begin{thm}[{cf.~\cite[Theorem 7.1]{Kock2013fibrations}}]\label{thm:splitting}
Every bifibration is equivalent to a split bifibration.
\end{thm}
\begin{proof}
Let $q : \E \to \C$ be a bifibration.
The free bifibration $\BifibFun{q} : \Bifib{q} \to \C$ is equipped with a canonical morphism of bifibrations $\sem{-}{\E} : \Bifib{q} \to \E$, which restricts to a functor $\nsaBifib{q} \to \E$.
Now consider the \emph{full image} of this functor.
This defines a category $\sem{\nsaBifib{q}}{\E}$ with same objects as $\nsaBifib{q}$ (that is, whose objects are nonidentity strictly alternating formulas) and whose arrows are given by arbitrary arrows $\sem{S}{\E} \to \sem{T}{\E}$ in $\E$.
Since every object of $\E$ is equal to its interpretation as an atomic formula $\sem{X}{\E} = X$, the full and faithful functor $\sem{\nsaBifib{q}}{\E} \to \E$ is surjective on objects and hence an equivalence of categories.
Finally, the composite $q \circ \sem{-}{\E} : \sem{\nsaBifib{q}}{\E} \to \C$ may be equipped with the structure of a split bifibration by transferring the structure of the free split bifibration $\BifibFun{q}|_{nsa} : \nsaBifib{q} \to \C$ along the morphism $\sem{-}{\E}$.
\end{proof}
\noindent
The above proof is similar to the proof of the analogous fact for fibrations described by Anders Kock~\cite[Theorem 7.1]{Kock2013fibrations} as Giraud's ``left adjoint splitting'' construction, which is also analyzed (together with the ``right adjoint splitting'') by Emmenegger, Mesiti, Rosolini, and Streicher \cite{EMRStreicher2024}.

In the rest of the paper we are not particularly concerned with obtaining fully split bifibrations, and so unless otherwise stated we do not forbid pushing and pulling along identity arrows.

\subsubsection{A syntax of alternating formulas}

We now introduce a new sequent calculus, whose purpose is to force strict alternation at the level of derivations without restricting the language of formulas.
It does this by performing multiple inference steps at once, in one large step, precisely where the $\collapse{-}$ translation would force them to be grouped together.

Strictly alternating formulas $\collapse S$ admit a simple inductive characterization: a formula is strictly alternating just in case it is either atomic $\atom X$, or the pushforward $\push f \collapse N$ of a non-positive strictly alternating formula, or the pullback $\pull g \collapse P$ of a non-negative strictly alternating formula.

\begin{mathpar}
  \begin{array}{lcl}
  \collapse S & ::= & \collapse P \ \mid\ \collapse N \\
  \collapse P & ::= & X \ \mid\ \push f \collapse N \\
  \collapse N & ::= & X \ \mid\ \pull g \collapse P \\
  \end{array}
\end{mathpar}

Here we use ``non-positive'' and ``non-negative'' in the expected sense, to mean that $N$ is either atomic or a pullback formula, and $P$ is either atomic or a pushforward formula.
Note that our definition of strictly alternating allows for pushing and pulling along identity arrows (cf.~\S\ref{sec:focusing:weak-focusing:splitting}).

In fact, \emph{any} bifibrational formula can be described in a similar alternating way, if instead of pushing or pulling applying a single arrow $f$ or $g$, we push or pull along a \emph{sequence} $\pi$, where $\pi$ denotes a non-empty sequence of composable arrows $(f_0,\dots,f_n)$. We introduce an explicit strictly-alternating grammar for arbitrary formulas $S_{sa}$, non-positive formulas $N_{sa}$ and non-negative formulas $P_{sa}$.

\begin{mathpar}
  \begin{array}{lcl}
  S_{sa} & ::= & P_{sa} \ \mid\ N_{sa} \\
  P_{sa} & ::= & X \ \mid\ \push \pi N_{sa} \\
  N_{sa} & ::= & X \ \mid\ \pull \pi P_{sa} \\
  \end{array}
\end{mathpar}
Formulas described in this way can be interpreted in two different ways, depending on the interpretation of sequences $\pi = (f_0, \dots, f_n)$:
\begin{enumerate}
\item If we interpret $\push \pi$ as an \emph{iterated} push
  $\push {f_n} \cdots \push {f_0}$, a formula $S_{sa}$ in alternating
  syntax is interpreted as an arbitrary formula
  (not strictly alternating) that we denote $\ceil {S_{sa}}$.
\item If we interpret $\push \pi$ as a push along the composite arrow
  $\collapse \pi \defeq f_0 \cdots f_n$, a formula $S_{sa}$ is interpreted as a strictly alternating formula that we denote $\floor {S_{sa}}$.
\end{enumerate}
For example if we take $N_{sa} = \push{(f, g)} \pull{(h,i)} X$, we have $\ceil {N_{sa}} = \push g \push f \pull h \pull i X$ and $\floor {N_{sa}} = \push {(f \; g)} \pull {(h \; i)} X$. All bifibrational formulas are in the image of $\ceil -$, all strictly-alternating bifibrational formulas are in the image of $\floor -$, and for any $S_{sa}$ we have $\floor {S_{sa}} = \collapse {\ceil {S_{sa}}}$.

To keep notations lighter, we write $S$ to mean either a bifibrational formula $S$ or its alternating presentation (the unique $S_{sa}$ determined by $\ceil {S_{sa}} = S$).

\subsubsection{Weak focusing as a syntax of alternating derivations}
\label{subsubsec:weakly-focused-derivations}

We now introduce a weakly focused calculus as a sequent calculus on formulas in the alternating syntax.
We present the inference rules in Figure~\ref{fig:weak-focusing} at the same time as a more compact term syntax for derivations, similar to one that we introduced in \S\ref{sec:sequent-calculus:term-syntax}.
\begin{figure}
  \begin{align*}
  \pi &::= (f_0,\dots,f_n) \\   
  \alpha_w &::=
    \atom{\delta}
    \mid \tLpull \pi \alpha_w
    \mid \uLpush \pi g \alpha_w
    \mid \alpha_w \tRpush \pi
    \mid \alpha_w \uRpull f \pi
  \end{align*}
  
 \begin{mathpar}
  \inferrule*[Right={$\atom \delta$}]
  {\delta : X \to Y \in \D \\ p(\delta) = f}
  {\atom \delta : \atom X \vdashf{f} \atom Y}
\\
  \inferrule*[Right={$\Lpush[\pi]$}]
        {\alpha_w : N \vdashf{\pi g} T}
        {\uLpush \pi g \alpha_w : \push{\pi} N \vdashf{g} T}

  \inferrule*[Right={$\Rpush[\pi]$}]
        {\alpha_w : S \vdashf{f} N}
        {\alpha_w \tRpush{\pi} : S \vdashf{f\pi} \push{\pi} N}
  \\
  \inferrule*[Right={$\Lpull[\pi]$}]
        {\alpha_w : P \vdashf{g} T}
        {\tLpull{\pi} \alpha_w : \pull{\pi} P \vdashf{\pi g} T}

  \inferrule*[Right={$\Rpull[\pi]$}]
        {\alpha_w : S \vdashf{f\pi} P}
        {\alpha_w \uRpull{f}{\pi} : S \vdashf{f} \pull{\pi} P}
\end{mathpar}
\caption{Rules of the weakly focused calculus.}
\label{fig:weak-focusing}
\end{figure}

Just as formulas $S_{sa}$ in alternating syntax admit two natural interpretations, each derivation $\alpha_w : S_{sa} \vdashf{f} T_{sa}$ in the weakly focused calculus admits two additional interpretations besides its formal interpretation as a sequence of applications of the above rules:
\begin{enumerate}
\item As a derivation $\ceil{\alpha_w} : \ceil {S_{sa}} \vdashf{f} \ceil {T_{sa}}$ in the unrestricted sequent calculus, by expanding the ``big-step'' rules $\Lpush[\pi]$, $\Rpush[\pi]$, $\Lpull[\pi]$, and $\Rpull[\pi]$ appropriately in terms of the ``small-step'' rules $\Lpush$, $\Rpush$, $\Lpull[g]$, and $\Rpull[g]$.
This corresponds to translating $\uLpush \pi g \alpha_w$ to $\ceil{\uLpush \pi g \alpha_w} = \uLpush \pi g {\ceil{\alpha_w}}$, where $\uLpush \pi {g} \alpha$ is defined as the iterated division $\uLpush {f_n} {g} \cdots \uLpush {f_0} {f_1\cdots f_n g} \alpha$
translating $\alpha_w \tRpush{\pi}$ to $\ceil{\alpha_w \tRpush{\pi}} = \ceil{\alpha_w}\tRpush{\pi}$, where $\alpha \tRpush {\pi}$ is defined as $\alpha \tRpush {f_0} \cdots \tRpush{f_n}$, and so on.
\item As a derivation $\floor{\alpha_w} : \floor {S_{sa}} \vdashf{f} \floor {T_{sa}}$ in the strictly alternating fragment of the unrestricted sequent calculus, by interpreting each instance of $\Lpush[\pi], \Rpush[\pi]$, etc., as a single instance of $\Lpush,\Rpush$, etc., on the composite arrow $\collapse \pi = f_0 \cdots f_n$.
This corresponds to translating $\uLpush \pi g \alpha_w$ to a single left division $\floor{\uLpush \pi g \alpha_w} = \uLpush {\collapse\pi} g {\floor{\alpha_w}}$, translating $\alpha_w \tRpush{\pi}$ to a single right multiplication
$\floor{\alpha_w \tRpush{\pi}} = \floor{\alpha_w}\tRpush{\collapse\pi}$, and so on.
\end{enumerate}
When we remain ambiguous about whether $S$ denotes an arbitrary bifibrational formula or its alternating syntax, we can write that for any $\alpha_w : S \vdashf{f} T$, we have $\ceil {\alpha_w} : S \vdashf{f} T$ and $\floor {\alpha_w} : \collapse S \vdashf{f} \collapse T$.

\subsubsection{Completeness of weak focusing}

We will use the following properties of iterated divisions and iterated multiplications on unrestricted derivations.

\begin{lem}[Iterated identity]\label{lem:iterated-id}
$\uLpush \pi {} (\Id[S] \tRpush \pi) \permeq \Id[\push \pi S]$ and
$(\tLpull \pi \Id[T]) \uRpull {} \pi {}  \permeq \Id[\pull \pi T]$ for all compatible formulas $S,T$ and sequences of arrows $\pi$.
\end{lem}
\begin{lem}[Iterated cut]\label{lem:iterated-cut}
$(\alpha \tRpush \pi)\hcomp (\uLpush \pi {h} \beta) \permeq \alpha \hcomp \beta$ and
$(\alpha \uRpull {g} \pi)\hcomp (\tLpull \pi \beta) \permeq \alpha \hcomp \beta$
for all compatible derivations $\alpha,\beta$ and sequences of arrows $\pi$.
\end{lem}
\begin{proof}
Both straightforward by induction on $\pi$.
\end{proof}
\noindent
We can now give explicit derivations for the equivalences $\theta_S : S \equiv \collapse S$, by defining for every $S \refs A$ a pair of (unrestricted) derivations
\[
\theta_S : S \vdashf{\Id[A]} \collapse{S}
\qquad
\theta_S^{-1} : \collapse{S} \vdashf{\Id[A]} S
\]
as follows. (We omit the subscript arrow on the division signs below, which is always $\Id[A]$.)
\[
\theta_{\push{\pi}N} = \uLpush \pi {} (\theta_N \tRpush {\collapse{\pi}})
\qquad
\theta^{-1}_{\push{\pi}N} = \uLpush {\collapse\pi} {} (\theta^{-1}_N \tRpush {\pi})
\]
\[
\theta_{\pull{\pi}P} = (\tLpull {\collapse{\pi}} \theta_P) \uRpull {} \pi
\qquad
\theta^{-1}_{\pull{\pi}P} = (\tLpull {\pi} \theta^{-1}_P) \uRpull {} {\collapse\pi}
\]
\[
\theta_X = \theta^{-1}_X = \Id[X]
\]
By appealing to Lemmas~\ref{lem:iterated-id} and \ref{lem:iterated-cut}, we have:

\begin{prop}
  \label{prop:theta-inverses}
  $\theta_S$ and $\theta_S^{-1}$ are inverses up to permutation equivalence.
\end{prop}
\begin{proof}
\label{prf:theta-inverses}
For example in the case $S = \push \pi N$ we have:
\begin{align*}
\theta_{\push{\pi}N} \hcomp \theta_{\push{\pi}N}^{-1}
  &= (\uLpush \pi {} (\theta_N \tRpush {\collapse{\pi}}))\hcomp (\uLpush {\collapse\pi} {} (\theta^{-1}_N \tRpush {\pi})) \\
  &\permeq \uLpush \pi {} (\theta_N \hcomp \theta^{-1}_N \tRpush \pi) & \text{definition of cut} \\
  &\permeq \uLpush \pi {} (\Id[N] \tRpush \pi) & \text{induction hypothesis} \\
  &\permeq \Id[\push \pi N] & \text{Lemma~\ref{lem:iterated-id}}
\end{align*}
\begin{align*}
\theta_{\push{\pi}N}^{-1} \hcomp \theta_{\push{\pi}N}
  &= (\uLpush {\collapse\pi} {} (\theta^{-1}_N \tRpush {\pi}))\hcomp (\uLpush \pi {} (\theta_N \tRpush {\collapse{\pi}})) \\
  &\permeq \uLpush {\collapse\pi} {} (\theta_N^{-1} \hcomp \theta_N \tRpush {\collapse\pi}) & \text{definition of cut and Lemma~\ref{lem:iterated-cut}} \\
  &\permeq \uLpush {\collapse\pi} {} (\Id[N] \tRpush {\collapse\pi}) & \text{induction hypothesis} \\
  &\permeq \Id[\push {\collapse\pi} N] & \text{definition of identity}
\end{align*}
\end{proof}
\noindent
On formulas we remarked that we have the relation $\floor S = \collapse {\ceil S}$. There is a similar relation at the level of derivations.

\begin{prop}
\label{prop:floor-ceil-derivation}
For any $\alpha_w : S \vdashf{f} T$ we have $\floor{\alpha_w} \sim \collapse {\ceil{\alpha_w}}$.
\end{prop}
\begin{proof}
\label{prf:floor-ceil-derivation}
By induction on the weakly focused derivation $\alpha_w$.
For example, in the case that $\alpha_w$ is of the form $\uLpush \pi g {\alpha_w'} : \push{\pi}N \vdashf{g} T$, we have
\begin{align*}
\collapse{\ceil{\uLpush \pi g {\alpha_w'}}} &= (\uLpush {\collapse{\pi}} {} (\theta_N^{-1}\tRpush \pi)) \hcomp (\uLpush \pi g \ceil{\alpha_w'}) \hcomp \theta_T & \text{by definition of }\collapse{-}\text{ and }\ceil{-}\\
&\permeq \uLpush {\collapse{\pi}} {} (\theta_N^{-1}\hcomp \ceil{\alpha_w} \hcomp \theta_T) & \text{definition of cut and Lemma~\ref{lem:iterated-cut}}\\ 
&\permeq \uLpush {\collapse{\pi}} {} \floor{\alpha_w} & \text{induction hypothesis}\\
&= \floor{\uLpush \pi g \alpha_w} & \text{definition of $\floor{-}$}
\end{align*}
\end{proof}

We remarked above that every bifibrational formula is of the form $\ceil {S_{sa}}$, and every strictly alternating formula is of the form $\floor {S_{sa}}$, where $S_{sa}$ is a formula in alternating syntax.
The situation for derivations is different: all strictly alternating derivations are of the form $\floor {\alpha_w}$, but \emph{not} all unrestricted derivations are of the form $\ceil {\alpha_w}$---it is really a more restricted syntax. For example, the derivation on the left below is not in the image of $\ceil -$, but it is permutation equivalent to the derivation on the right, which is:
\begin{mathpar}
  \inferrule* [Right={$\Lpush[g]$}]
        {\inferrule* [Right={$\Rpull[i]$}]
          {\inferrule* [Right={$\Lpush[f]$}]
            {S \vdashf{fghi} T}
            {\push{f} S \vdashf{ghi} T}}
          {\push{f} S\vdashf{gh} \pull{i} T}}
        {\push{g}\push{f} S \vdashf{h} \pull{i} T}

  \inferrule* [Right={$\Lpush[g]$}]
        {\inferrule* [Right={$\Lpush[f]$}]
          {\inferrule* [Right={$\Rpull[i]$}]
            {S \vdashf{fghi} T}
            {S \vdashf{fgh} \pull{i} T}}
          {\push{f} S\vdashf{gh} \pull{i} T}}
        {\push{g}\push{f} S \vdashf{h} \pull{i} T}
\end{mathpar}

\begin{prop}
The mapping $\floor -$ is a one-to-one correspondence between weakly focused derivations of $S_{sa} \vdashf{f} T_{sa}$ and unrestricted derivations of strictly alternating sequents $\floor{S_{sa}} \vdashf{f} \floor{T_{sa}}$.
\end{prop}
\begin{proof}
  We prove that for any judgment $S_{sa} \vdashf{f} T_{sa}$,
  if $\alpha$ is an unrestricted derivation of $\floor{S_{sa}} \vdashf{f} \floor{T_{sa}}$, then there exists a \emph{unique} weakly focused derivation $\alpha_w$ such that $\alpha = \floor {\alpha_w}$. The proof proceeds by induction on the product ordering of the two sides of the sequent $(S_{sa}, T_{sa})$ with respect to the subformula ordering.
  We will have more occasions to use this style of induction in the rest of this section, which we refer to as \emph{induction on judgments.}
  The associated well-founded ordering on judgments is defined explicitly by
  \[ (S' \vdashf{h'} T') \prec (S \vdashf{h} T) \]
  just in case $S'$ is a subformula of $S$, $T'$ is a subformula of $T$, and either $S'$ is a strict subformula of $S$ or $T'$ is a strict subformula of $T$.
  We now proceed to the proof.

  If $\alpha$ is an axiom rule then it is already a weakly focused proof.

  If $\alpha$ is of the form $\uLpush f g {\alpha'}$:
  \begin{mathpar}
    \inferrule
    {\alpha' : {S'} \vdashf{{f} g} \floor{T_{sa}}}
    {\uLpush {f} g {\alpha'} : \push {f} {S'} \vdashf{g} \floor{T_{sa}}}
  \end{mathpar}
  then from $\push {f} {S'} = \floor S_{sa}$ we can deduce that $S_{sa}$ must be of the form $\push \pi N$, so $f$ is uniquely determined to be $\collapse \pi$ and $S'$ is $\collapse N$. So in fact the derivation of $\alpha$ is of the form
  \begin{mathpar}
    \inferrule
    {\alpha' : {\collapse N} \vdashf{\pi g} \floor {T_{sa}}}
    {\uLpush {\collapse \pi} g {\alpha'} : \push {\collapse \pi} {\collapse N} \vdashf{g} \floor{T_{sa}}}
  \end{mathpar}
By induction hypothesis, there exists a unique $\alpha'_w$ such that $\alpha' = \floor {\alpha'_w}$, so $\alpha$ is necessarily equal to $\uLpush {\collapse \pi} g {\floor {\alpha'_w}}$, and $\alpha_w \defeq \uLpush \pi g {\alpha'_w}$ is uniquely determined.

  The reasoning is symmetric when if $\alpha$ is of the form $\tLpull f {\alpha'}$.
\end{proof}

\begin{cor}[Completeness of weak focusing]
For any derivation $\alpha : S \vdashf{f} T$, there is a weakly focused derivation $\alpha_w : S \vdashf{f} T$ such that $\alpha \permeq \ceil{\alpha_w}$.
\end{cor}
\begin{proof}
Let $\collapse \alpha : \collapse S \vdashf{f} \collapse T$ be the strictification of $\alpha$, and let $\alpha_w : S \vdashf{f} T$ be the unique weakly focused derivation such that $\collapse{\alpha} = \floor{\alpha_w}$, which exists by the previous proposition.
In particular, $\collapse \alpha$ and $\collapse {\ceil {\alpha_w}}$ are equal as $\saBifib{p}$ morphisms, so by Theorem~\ref{thm:strictify-equivalence} we have that $\alpha$ and $\ceil {\alpha_w}$ are equal as $\Bifib{p}$ morphisms, that is, $\alpha \permeq \ceil {\alpha_w}$.
\end{proof}

\subsubsection{Equivalence of weakly focused proofs} There are two reasonable notions of equivalence on weakly focused proofs $\alpha_w, \beta_w$:
\begin{itemize}
\item One can consider the relation $\ceil {\alpha_w} \permeq \ceil {\beta_w}$, which views $\alpha_w, \beta_w$ as unrestricted derivations and considers permutation equivalence between them. Such a proof of equivalence permutes inference rules one at a time, and in particular it may contain intermediate steps that are not weakly focused derivations.
\item One can consider the relation $\floor {\alpha_w} \permeq \floor {\beta_w}$, which views $\alpha_w$ and $\beta_w$ as derivations on strictly alternating formulas. Such proofs of equivalence respect the strictly-alternating structure in the sense that they cannot break sequences of pushes or sequence of pulls apart. They can be seen as permutation equivalences on ``synthetic'' connectives.
\end{itemize}

\begin{cor}
  These two notions of equivalence coincide.
\end{cor}
\begin{proof}
  \begin{mathpar}
\floor {\alpha_w} \permeq \floor {\beta_w}

\stackrel{\text{Proposition~\ref{prop:floor-ceil-derivation}}}{\iff}

\collapse {\ceil {\alpha_w}} \permeq \collapse {\ceil {\beta_w}}

\stackrel{\text{Theorem~\ref{thm:strictify-equivalence}}}{\iff}

\ceil {\alpha_w} \permeq \ceil {\beta_w}\qedhere
  \end{mathpar}
\end{proof}

\subsubsection{Presentation} The results of this section can be reformulated as an alternative presentation of the arrows of $\Bifib{p}$ as triples $(S, \hat \alpha_w, T)$, where $\hat {\alpha}_w \in \saBifib{p}(S, T)$ is a permutation-equivalence class of weakly focused derivations $\alpha_w : S \vdashf{f} T$.
Weak focusing removes some of the inessential non-determinism of the sequent calculus, but it is possible to go further and obtain a still more canonical presentation, as we will now examine.

\subsection{Multifocusing}
\label{subsec:multi-focusing}
\subsubsection{Introduction}

The previous calculus was only ``weakly'' focused in the sense that it allows applying non-invertible rules to derive a judgment even where an invertible rule is applicable. It would be more canonical to enforce that invertible rules are applied eagerly. For instance, if we have a positive formula on the left then we can always apply an invertible left rule, and if we have a negative formula on the right then we can always apply an invertible right rule.
In the case of a sequent of the form $\push \pi N \vdashf{f} \pull \rho P$ with \emph{both} a positive formula on the left and a negative formula, it seems we have a useless choice to make of whether to work on the left or right first:
\begin{mathpar}
  \inferrule*[Right={$\Rpull[\rho]$}]
  {\inferrule*[Right={$\Lpush[\pi]$}]
    {N \vdashf{\pi f \rho} P}
    {\push \pi N \vdashf{f \rho} P}}
  {\push \pi N \vdashf{f} \pull \rho P}

  \inferrule*[Right={$\Lpush[\pi]$}]
  {\inferrule*[Right={$\Rpull[\rho]$}]
    {N \vdashf{\pi f \rho} P}
    {N \vdashf{f \rho} \pull \rho P}}
  {\push {\pi} N \vdashf{f} \pull \rho P}
\end{mathpar}
The choice of whether to apply $\Rpull[\pi]$ or $\Lpush[\rho]$ first is irrelevant, since the two derivations above are permutation equivalent.
However, a natural way of addressing this issue is to add a rule that can perform \emph{both} left and right inversion at once.
By combining such a bi-inversion rule with polarity-restricted instances of the left- and right-inversion rules,
\begin{equation}\label{rules:strong-inversion}
  \infer[{\Lpush[\pi]}]{\push{\pi}N \vdashf{f} P}{N \vdashf{\pi f} P}
  \qquad
  \infer[{\Lpush[\pi]\Rpull[\rho]}]{\push{\pi}N \vdashf{f} \pull{\rho}P}{N \vdashf{\pi f \rho} P}
  \qquad
  \infer[{\Rpull[\rho]}]{N \vdashf{f} \pull{\rho}P}{N \vdashf{f \rho} P}
\end{equation}
we can enforce that non-invertible rules are only used on polarized sequents of the form $N \vdashf{f} P$, with a uniquely determined invertible rule applicable in all other cases. We say that such a system, with a polarity restriction on non-invertible phases, is \emph{strongly} focused.

Once we arrive at a sequent of the form $N \vdashf{f} P$ we have to make a potentially significant choice as to whether to try applying a non-invertible rule on the left or right first.
However, here a similar phenomenon arises in considering permutable applications of non-invertible rules.
%
%
Again it is natural to introduce a rule that allows to perform two non-invertible rules in parallel, written in the middle below:
\begin{equation}\label{rules:strong-focus}
  \infer[{\Lpull[\sigma]}]{\pull{\sigma}P \vdashf{\sigma f} Q}{P \vdashf{f} Q}
  \qquad
  \infer[{\Lpull[\sigma]\Rpush[\tau]}]{\pull{\sigma}P \vdashf{\sigma f\tau} \push{\tau}N}{P \vdashf{f} N}
  \qquad
  \infer[{\Rpush[\tau]}]{N \vdashf{f \tau} \push{\tau}M}{N \vdashf{f} M}
\end{equation}
The collection of six polarized rules \eqref{rules:strong-inversion} and \eqref{rules:strong-focus}, together with the ordinary axiom rule, together define a system that we call \emph{strongly multifocused}.

Strong focusing and multi-focusing are standard notions from the literature~\cite{ChaudhuriMillerSaurin2008,Simmons2014}.
Going from a weakly-focused to a strongly-focused system (enforcing that invertible rules are applied eagerly) makes the system more canonical.
On the other hand, introducing multifocusing (the bi-focusing rule ${\Lpull[\sigma]\Rpush[\tau]}$ in our case) does not automatically make the system more canonical, since it actually allows more proofs: every time a bi-focused rule is used, it is also possible to use two mono-focused rules in sequence. In Section~\ref{subsec:maximal-multi-focusing} below, we will explain how to restrict this system by limiting the usage of mono-focused rules, resulting in a calculus that is canonical under an assumption on the base category.


In Figure~\ref{fig:strong-multifocusing} we recapitulate the rules of the strongly multifocused system while introducing a term syntax. 

\begin{figure}
\begin{mathpar}
  \begin{array}{lr lclcl}
  \alpha_m & ::= & \atom{\delta} &&&& \\
    & \mid & \tLpull \sigma \alpha_m
    & \mid & \alpha_m \tRpush \tau
    & \mid & \tLRpullpush \sigma {\alpha_m} \tau
    \\
    & \mid & \uLpush \pi g \alpha_m
    & \mid & \alpha_m \uRpull g \rho
    & \mid & \uLRpushpull \pi g {\alpha_m} \rho
  \end{array}
  \\
  \inferrule*[Right={$\atom \delta$}]
  {\delta : X \to Y \in \D \\ p(\delta) = f}
  {\atom \delta : \atom X \vdashf{f} \atom Y}
  \\
  \inferrule*[Right={$\Lpush[\pi]$}]
        {\alpha_m : N \vdashf{\pi f} P}
        {\uLpush \pi f \alpha_m : \push \pi N \vdashf{f} P}
  \qquad
  \inferrule*[Right={$\Lpush[\pi]\Rpull[\rho]$}]
        {\alpha_m : N \vdashf{\pi f \rho} P}
        {\uLRpushpull \pi f {\alpha_m} \rho : \push \pi N \vdashf{f} \pull \rho P}
  \qquad\qquad
  \inferrule*[Right={$\Rpull[\rho]$}]
        {\alpha_m : N \vdashf{f \rho} P}
        {\alpha_m \uRpull f \rho : N \vdashf{f} \pull \rho P}
\\
  \inferrule*[Right={$\Lpull[\sigma]$}]
        {\alpha_m : P \vdashf{f} Q}
        {\tLpull \sigma \alpha_m : \pull \sigma P \vdashf{\sigma f} Q}
  \qquad
  \inferrule*[Right={$\Lpull[\sigma]\Rpush[\tau]$}]
        {\alpha_m : P \vdashf{f} N}
        {\tLRpullpush \sigma {\alpha_m} \tau : \pull \sigma P \vdashf{\sigma f \tau} \push \tau N}
  \qquad\qquad
  \inferrule*[Right={$\Rpush[\tau]$}]
        {\alpha_m : N \vdashf{f} M}
        {\alpha_m \tRpush \tau : N \vdashf{f \tau} \push \tau M}
\end{mathpar}
\caption{Rules of the strongly multifocused calculus.}
\label{fig:strong-multifocusing}
\end{figure}

\subsubsection{Neutral sequents, bipoles} Following the focusing
tradition, we call \emph{neutral} a sequent of the form
$N \vdashf{f} P$. A proof of a (non-atomic) neutral sequent has to start with
a focusing step, either on the left, right or both sides of the sequent.

Following the same tradition, we call \emph{bipole} a two-step (open) derivation of a neutral sequent
beginning with a focusing step followed by an inversion step.
We use the notation $\biL$ for a bipole of $\push \Lsym \ \pull \Lsym$ rules, $\biR$ for a bipole of $\pull \Rsym \ \push \Rsym$ rules, and $\biLR$ for a bi-focused bipole $({\push \Lsym}{\pull \Rsym}) \ ({\pull \Lsym}{\push \Rsym})$.
In order to avoid having to treat atomic formulas as a special case, we will expand the interpretation of
the $\Lpush[\pi]$, $\Rpull[\rho]$, and $\Lpush[\pi]\Rpull[\rho]$ rules by allowing $\pi$ (respectively, $\rho$) to be empty when the left (respectively right) side of the sequent is atomic.
With that convention, every strongly focused proof of a neutral sequent decomposes as a sequence of $\biL$~/~$\biR$~/~$\biLR$ bipoles ended by an axiom.

\subsubsection{Sequentialization equivalence of strongly multi-focused proofs}
\label{subsubsec:multi-focused-equivalence}

The permutation equivalence that we introduced in Section~\ref{sec:sequent-calculus:permeq} and extended to weakly focused proofs is not directly applicable for strongly focused proofs. Indeed, permutation equivalence breaks the polarity invariants that we described above.
However, there is a natural notion of equivalence for strongly multifocused derivations that arises from considering both possible \emph{sequentializations} of a $\biLR$ bipole:

\begin{equation*}
\begin{tikzcd}
&
\begin{array}{c}
  \inferrule*[Right={$\Lpull[\sigma]\Rpush[\tau]$}]
  {\inferrule*[Right={$\Lpush[\pi]\Rpull[\rho]$}]
    {N \vdashf{\pi f\rho} P}
    {\push{\pi} N\vdashf{f} \pull{\rho}P}}
  {\pull{\sigma}\push{\pi} N \vdashf{\sigma f\tau} \push{\tau}\pull{\rho}P}
\end{array}
\ar[ld, "\seqsym^\Rsym_\Lsym", start anchor=south west, end anchor = north east]
\ar[rd, "\seqsym^\Lsym_\Rsym"', start anchor=south east, end anchor = north west]
&  \\
\begin{array}{c}
\inferrule*[Right={$\Lpull[\sigma]$}]
{\inferrule*[Right={$\Lpush[\pi]$}]
  {\inferrule*[Right={$\Rpush[\tau]$}]
    {\inferrule*[Right={$\Rpull[\rho]$}]
      {N \vdashf{\pi f\rho} P}
      {N \vdashf{\pi f} \pull{\rho}P}}
    {N\vdashf{\pi f \tau} \push{\tau}\pull{\rho}P}}
  {\push{\pi} N\vdashf{f \tau} \push{\tau}\pull{\rho}P}}
{\pull{\sigma}\push{\pi} N \vdashf{\sigma f\tau} \push{\tau}\pull{\rho}P}
\end{array}
& &  
\begin{array}{c}
\inferrule*[Right={$\Rpush[\tau]$}]
{\inferrule*[Right={$\Rpull[\rho]$}]
  {\inferrule*[Right={$\Lpull[\sigma]$}]
    {\inferrule*[Right={$\Lpush[\pi]$}]
      {N \vdashf{\pi f\rho} P}
      {\push{\pi} N\vdashf{f \rho} P}}
    {\pull{\sigma}\push{\pi} N\vdashf{\sigma f \rho} P}}
  {\pull{\sigma}\push{\pi} N\vdashf{\sigma f} \pull{\rho}P}}
{\pull{\sigma}\push{\pi} N \vdashf{\sigma f\tau} \push{\tau}\pull{\rho}P}
\end{array}
\end{tikzcd}
\end{equation*}

\newsavebox\myboxA
\newsavebox\myboxAwstrings
\newsavebox\myboxB
\newsavebox\myboxBwstrings
\newsavebox\myboxC
\newsavebox\myboxCwstrings

\begin{lrbox}{\myboxA}
  \begin{tikzcd}
  |[alias=l0]| \cdot & |[alias=r0]|\cdot \\
  |[alias=l1]| \cdot & |[alias=r1]|\cdot \\
  |[alias=l2]| \cdot & |[alias=r2]|\cdot 
	\arrow["\pi f \rho", from=1-1, to=1-2]
	\arrow["\pi"', from=1-1, to=2-1] \arrow[phantom,from=1-1,to=2-2,"\push \Lsym\pull \Rsym"]
	\arrow["f" description, from=2-1, to=2-2]
	\arrow["\rho"', from=2-2, to=1-2]
	\arrow["\tau", from=2-2, to=3-2]
	\arrow["\sigma", from=3-1, to=2-1] \arrow[phantom,from=3-1,to=2-2,"\pull \Lsym\push \Rsym"]
	\arrow["\sigma f \tau" description, from=3-1, to=3-2]
  \end{tikzcd}
\end{lrbox}

\begin{lrbox}{\myboxAwstrings}
  \begin{tikzcd}[execute at end picture = {
    \draw[arrg, oedge, draw opacity=0.5, out=90,in=0] ($(l2)!0.33!(r2)$) to ($(l1)!0.5!(l2)$);
    \draw[arrf, oedge, draw opacity=0.5, out=90,in=-90] ($(l2)!0.50!(r2)$) to ($(l0)!0.5!(r0)$);
    \draw[arrh, oedge, draw opacity=0.5, out=90,in=180] ($(l2)!0.67!(r2)$) to ($(r1)!0.5!(r2)$);
    \draw[arrgp, oedge, draw opacity=0.5, out=0,in=-90] ($(l0)!0.5!(l1)$) to ($(l0)!0.33!(r0)$);
    \draw[arrhp, oedge, draw opacity=0.5, out=180,in=-90] ($(r0)!0.5!(r1)$) to ($(l0)!0.67!(r0)$);
  }  ]
  |[alias=l0]| \cdot & |[alias=r0]|\cdot \\
  |[alias=l1]| \cdot & |[alias=r1]|\cdot \\
  |[alias=l2]| \cdot & |[alias=r2]|\cdot 
	\arrow["\pi f \rho", from=1-1, to=1-2]
	\arrow["\pi"', from=1-1, to=2-1] \arrow[phantom,from=1-1,to=2-2,"\push \Lsym\pull \Rsym"]
	\arrow["f" description, from=2-1, to=2-2]
	\arrow["\rho"', from=2-2, to=1-2]
	\arrow["\tau", from=2-2, to=3-2]
	\arrow["\sigma", from=3-1, to=2-1] \arrow[phantom,from=3-1,to=2-2,"\pull \Lsym\push \Rsym"]
	\arrow["\sigma f \tau" description, from=3-1, to=3-2]
  \end{tikzcd}
\end{lrbox}

\begin{lrbox}{\myboxB}
  \begin{tikzcd}
    |[alias=l0]| \cdot & |[alias=r0]| \cdot \\
    |[alias=l1]| \cdot & |[alias=r1]| \cdot \\
    |[alias=l2]| \cdot & |[alias=r2]| \cdot \\
    |[alias=l3]| \cdot & |[alias=r3]| \cdot \\
    |[alias=l4]| \cdot & |[alias=r4]| \cdot
        \arrow["\pi f \rho", from=1-1, to=1-2]
	\arrow[equals, from=1-1, to=2-1] \arrow[phantom,from=1-1,to=2-2,"{\pull \Rsym}"]
	\arrow["\pi f" description, from=2-1, to=2-2]
	\arrow["\rho"', from=2-2, to=1-2]
	\arrow["\tau", from=2-2, to=3-2]
	\arrow[equals, from=3-1, to=2-1] \arrow[phantom,from=2-1,to=3-2,"{\push\Rsym}"]
	\arrow["\pi f \tau" description, from=3-1, to=3-2]
	\arrow["\pi"', from=3-1, to=4-1] \arrow[phantom,from=3-1,to=4-2,"\push\Lsym"]
	\arrow[equals, from=3-2, to=4-2]
	\arrow["f\tau" description, from=4-1, to=4-2]
	\arrow[equals, from=4-2, to=5-2]
	\arrow["\sigma", from=5-1, to=4-1] \arrow[phantom,from=4-1,to=5-2,"\pull\Lsym"]
	\arrow["\sigma f\tau"', from=5-1, to=5-2]
  \end{tikzcd}
\end{lrbox}

\begin{lrbox}{\myboxBwstrings}
  \begin{tikzcd}[execute at end picture = {
    \draw[arrg, oedge, draw opacity=0.5, out=90,in=0] ($(l4)!0.33!(r4)$) to ($(l3)!0.5!(l4)$);
    \draw[arrf, oedge, draw opacity=0.5, out=90,in=-90] ($(l4)!0.50!(r4)$) to ($(l0)!0.5!(r0)$);
    \draw[arrh, oedge, draw opacity=0.5, out=90,in=240] ($(l4)!0.67!(r4)$) to ($(r1)!0.5!(r2)$);
    \draw[arrgp, oedge, draw opacity=0.5, out=60,in=-90] ($(l2)!0.5!(l3)$) to ($(l0)!0.33!(r0)$);
    \draw[arrhp, oedge, draw opacity=0.5, out=180,in=-90] ($(r0)!0.5!(r1)$) to ($(l0)!0.67!(r0)$);
  }  ]
    |[alias=l0]| \cdot & |[alias=r0]| \cdot \\
    |[alias=l1]| \cdot & |[alias=r1]| \cdot \\
    |[alias=l2]| \cdot & |[alias=r2]| \cdot \\
    |[alias=l3]| \cdot & |[alias=r3]| \cdot \\
    |[alias=l4]| \cdot & |[alias=r4]| \cdot
        \arrow["\pi f \rho", from=1-1, to=1-2]
	\arrow[equals, from=1-1, to=2-1] \arrow[phantom,from=1-1,to=2-2,"{\pull \Rsym}"]
	\arrow["\pi f" description, from=2-1, to=2-2]
	\arrow["\rho"', from=2-2, to=1-2]
	\arrow["\tau", from=2-2, to=3-2]
	\arrow[equals, from=3-1, to=2-1] \arrow[phantom,from=2-1,to=3-2,"{\push\Rsym}"]
	\arrow["\pi f \tau" description, from=3-1, to=3-2]
	\arrow["\pi"', from=3-1, to=4-1] \arrow[phantom,from=3-1,to=4-2,"\push\Lsym"]
	\arrow[equals, from=3-2, to=4-2]
	\arrow["f\tau" description, from=4-1, to=4-2]
	\arrow[equals, from=4-2, to=5-2]
	\arrow["\sigma", from=5-1, to=4-1] \arrow[phantom,from=4-1,to=5-2,"\pull\Lsym"]
	\arrow["\sigma f\tau"', from=5-1, to=5-2]
  \end{tikzcd}
\end{lrbox}

\begin{lrbox}{\myboxC}
  \begin{tikzcd}
    |[alias=l0]| \cdot & |[alias=r0]| \cdot \\
    |[alias=l1]| \cdot & |[alias=r1]| \cdot \\
    |[alias=l2]| \cdot & |[alias=r2]| \cdot \\
    |[alias=l3]| \cdot & |[alias=r3]| \cdot \\
    |[alias=l4]| \cdot & |[alias=r4]| \cdot
	\arrow["\pi f \rho", from=1-1, to=1-2]
	\arrow["\pi"', from=1-1, to=2-1] \arrow[phantom,from=1-1,to=2-2,"{\push \Lsym}"]
	\arrow["f\rho" description, from=2-1, to=2-2]
	\arrow[equals, from=2-2, to=1-2]
	\arrow[equals, from=2-2, to=3-2]
	\arrow["\sigma", from=3-1, to=2-1] \arrow[phantom,from=2-1,to=3-2,"{\pull\Lsym}"]
	\arrow["\sigma f\rho" description, from=3-1, to=3-2]
	\arrow[equals, from=3-1, to=4-1] \arrow[phantom,from=3-1,to=4-2,"\pull\Rsym"]
	\arrow["\rho"', from=4-2, to=3-2]
	\arrow["\sigma f" description, from=4-1, to=4-2]
	\arrow["\tau", from=4-2, to=5-2]
	\arrow[equals, from=5-1, to=4-1] \arrow[phantom,from=4-1,to=5-2,"\push\Rsym"]
	\arrow["\sigma f\tau"', from=5-1, to=5-2]
  \end{tikzcd}
\end{lrbox}
\begin{lrbox}{\myboxCwstrings}
  \begin{tikzcd}[execute at end picture = {
    \draw[arrg, oedge, draw opacity=0.5, out=90,in=-60] ($(l4)!0.33!(r4)$) to ($(l1)!0.5!(l2)$);
    \draw[arrf, oedge, draw opacity=0.5, out=90,in=-90] ($(l4)!0.50!(r4)$) to ($(l0)!0.5!(r0)$);
    \draw[arrh, oedge, draw opacity=0.5, out=90,in=180] ($(l4)!0.67!(r4)$) to ($(r3)!0.5!(r4)$);
    \draw[arrgp, oedge, draw opacity=0.5, out=0,in=-90] ($(l0)!0.5!(l1)$) to ($(l0)!0.33!(r0)$);
    \draw[arrhp, oedge, draw opacity=0.5, out=120,in=-90] ($(r2)!0.5!(r3)$) to ($(l0)!0.67!(r0)$);
  }  ]
    |[alias=l0]| \cdot & |[alias=r0]| \cdot \\
    |[alias=l1]| \cdot & |[alias=r1]| \cdot \\
    |[alias=l2]| \cdot & |[alias=r2]| \cdot \\
    |[alias=l3]| \cdot & |[alias=r3]| \cdot \\
    |[alias=l4]| \cdot & |[alias=r4]| \cdot
	\arrow["\pi f \rho", from=1-1, to=1-2]
	\arrow["\pi"', from=1-1, to=2-1] \arrow[phantom,from=1-1,to=2-2,"{\push \Lsym}"]
	\arrow["f\rho" description, from=2-1, to=2-2]
	\arrow[equals, from=2-2, to=1-2]
	\arrow[equals, from=2-2, to=3-2]
	\arrow["\sigma", from=3-1, to=2-1] \arrow[phantom,from=2-1,to=3-2,"{\pull\Lsym}"]
	\arrow["\sigma f\rho" description, from=3-1, to=3-2]
	\arrow[equals, from=3-1, to=4-1] \arrow[phantom,from=3-1,to=4-2,"\pull\Rsym"]
	\arrow["\rho"', from=4-2, to=3-2]
	\arrow["\sigma f" description, from=4-1, to=4-2]
	\arrow["\tau", from=4-2, to=5-2]
	\arrow[equals, from=5-1, to=4-1] \arrow[phantom,from=4-1,to=5-2,"\push\Rsym"]
	\arrow["\sigma f\tau"', from=5-1, to=5-2]
  \end{tikzcd}
\end{lrbox}
\noindent
Or equivalently in stack-and-string diagrammatic form:
\begin{equation*}
\begin{tikzpicture}[baseline={([yshift=-.5ex]current bounding box.center)}]
  \node (B) {\usebox\myboxAwstrings};
  \node[left=4em of B] (A) {  \usebox\myboxBwstrings };
  \node[right=4em of B] (C) { \usebox\myboxCwstrings };
  \draw[->] (B.west) -- (A.east) node[midway,below]{$\seqsym^\Rsym_\Lsym$};
  \draw[->] (B.east) -- (C.west) node[midway,below]{$\seqsym^\Lsym_\Rsym$};
\end{tikzpicture}
\end{equation*}
\noindent
Or equivalently again in proof term syntax:
\begin{equation}\label{eq:seqeq-terms}
\begin{tikzcd}
\tLpull \sigma (\uLpush {\pi} {f \tau } ((\beta \uRpull {\pi f} {\rho}) \tRpush \tau))
    & \tLRpullpush \sigma {(\uLRpushpull \pi f \beta \rho)} \tau
    \ar[l, "\seqsym^\Rsym_\Lsym"] \ar[r, "\seqsym^\Lsym_\Rsym"']
    &  
((\tLpull \sigma (\uLpush \pi {f \rho} \beta)) \uRpull {\sigma f} {\rho}) \tRpush \tau
\end{tikzcd}
\end{equation}
Recall that we allow $\pi$ and $\rho$ to be empty, $\pi = \epsilon$ or $\rho = \epsilon$ (in which case, e.g., $\uLpush \epsilon {} \alpha = \alpha$).

Let us write $(\rewtoseq)$ for this sequentialization order, and $(\permeqseq)$ for the corresponding equivalence relation.





\subsubsection{Relating the weak-monofocused and strong-multifocused systems.}~

The $(\rewtoseq)$-normal forms of this relation are the strongly \emph{monofocused} proofs---in other words proofs that do not use the bi-focusing rule, although they can use bi-inversion. Such proofs can then be turned into derivations in the weakly focused calculus by selecting an ordering for each bi-inversion rule. We define $\Seq{\alpha_m}$ as the set of weakly monofocused sequentializations of a strongly multifocused proof $\alpha_m$.

\begin{prop}
  \label{prop:sequentialization-preserves-equiv}
If $\alpha_m \permeqseq \alpha'_m$ then for all $\alpha_w \in \Seq{\alpha_m}$ and $\alpha'_w \in \Seq{\alpha'_m}$ we have $\alpha_w \permeq \alpha'_w$.
\end{prop}
\begin{proof}
\label{prf:sequentialization-preserves-equiv}
The idea of the proof is that for any non-deterministic choice in the sequentialization process, the two choices will sequentialize into permutation-equivalent derivations.

The first choice comes from the definition of the sequentialization order between multifocused proofs. For each $\alpha$ with two direct sequentializations $\alpha_1$, $\alpha_2$ as depicted in \eqref{eq:seqeq-terms} above, we have that $\Seq{\alpha} = \Seq{\alpha_1} \uplus \Seq{\alpha_2}$, and we prove that for any ${\alpha_1}_w \in \Seq{\alpha_1}$ and ${\alpha_2}_w \in \Seq{\alpha_2}$ we have ${\alpha_1}_w \permeq {\alpha_2}_w$. Indeed, the two sequentializations sets are
\begin{align*}
\Seq{\alpha_1} &=
  \left\{\;
    \tLpull \sigma (\uLpush {\pi} {} ((\beta_w \uRpull f {\rho}) \tRpush \tau))
   \;\mid\; \beta_w \in \Seq{\beta} \;\right\} \\
\Seq{\alpha_2} &=
  \left\{\;
    ((\tLpull \sigma (\uLpush {\pi} f \beta_w)) \uRpull {} {\rho}) \tRpush \tau
   \;\mid\; \beta_w \in \Seq{\beta} \;\right\}
\end{align*}
and for any $\beta_w \in \Seq{\beta}$, we do have as expected:
\begin{mathpar}
  \begin{array}{rl}
    & \tLpull \sigma (\uLpush {\pi} {} ((\beta_w \uRpull f {\rho}) \tRpush \tau))
    \\ \permeq
    & \tLpull \sigma ((\uLpush {\pi} {} (\beta_w \uRpull f {\rho})) \tRpush \tau)
    \\ \permeq
    & (\tLpull \sigma (\uLpush {\pi} {} (\beta_w \uRpull f {\rho}))) \tRpush \tau
    \\ \permeq
    & (\tLpull \sigma ((\uLpush {\pi} {f} \beta_w) \uRpull {} {\rho})) \tRpush \tau
    \\ \permeq
    & ((\tLpull \sigma (\uLpush {\pi} {f} \beta_w)) \uRpull {} {\rho}) \tRpush \tau
  \end{array}
\end{mathpar}

The second choice is in the ordering of inversion rules when sequentializing bi-inversion rules, and the two choices are immediately permutation-equivalent.
\end{proof}

\begin{cor}\label{cor:seq-permeq}
  All elements of $\Seq{\alpha_m}$ are permutation-equivalent: for any $\alpha_m$ we have $\alpha_m \permeqseq \alpha_m$, so $\alpha_w \in \Seq{\alpha_m}$ and $\alpha'_w \in \Seq{\alpha_m}$ imply $\alpha_w \permeq \alpha'_w$.
\end{cor}

Conversely, given a weakly focused derivation $\alpha_w : S \vdashf{h} T$, let us write $\Strengthen{\alpha_w}$ for the \emph{strengthening} of $\alpha_w$, a uniquely-determined strongly (mono-)focused proof defined as follows, by induction on judgments:
\begin{mathpar}
  \begin{array}{lll}
    \Strengthen{\alpha : \push {\pi} N \vdashf{f} \pull {\rho} P}
    & :=
    & \uLpush \pi {} {\Strengthen{\opcart \pi N \hcomp \alpha \hcomp \cart \rho P : N \vdashf{\pi f \rho} P}} \uRpull {} \rho
    \\
    \Strengthen{(\tLpull \sigma \alpha) : \pull \sigma P \vdashf{\sigma f} Q}
    & :=
    & \tLpull \sigma {\Strengthen{\alpha : P \vdashf{f} Q}}
    \\
    \Strengthen{(\alpha \tRpush \tau) : N \vdashf{f\tau} \push \tau M}
    & :=
    & \Strengthen{\alpha : N \vdashf{f} M} \tRpush \tau
    \\
    \Strengthen{\atom\delta : \atom X \vdashf{f} \atom Y} &:= & \atom\delta
  \end{array}
\end{mathpar}
In other words, to strengthen $\alpha_w$, whenever we encounter a judgment that admits invertible steps, we do them (both at the same time if applicable), and otherwise we imitate the non-invertible steps.

This strengthening transformation defines a mapping from weakly focused derivations to strongly focused derivations of the same judgment.
We now establish that it sends permutation equivalent weakly focused derivations to sequentially equivalent strongly focused derivations.
\begin{prop}
\label{prop:strengthening-preserves-equiv}
If $\alpha_w \permeq \alpha_w'$ then $\Strengthen{\alpha_w} \permeqseq \Strengthen{\alpha_w'}$.
\end{prop}
\begin{proof}
\label{prf:strengthening-preserves-equiv}
By induction on the underlying judgment derived by both derivations $\alpha_w \permeq \alpha_w'$.  There are two cases to consider, depending on whether or not the judgment is invertible.
\begin{itemize}
\item Case $\push{\pi} N \vdashf{f} \pull{\rho} P$. (Recall we allow $\pi = \epsilon$ or $\rho = \epsilon$.) We have
\[
\Strengthen{\alpha_w} = \uLpush \pi {} {\Strengthen{\opcart \pi N \hcomp \alpha_w \hcomp \cart \rho P}} \uRpull {} \rho
 \permeqseq \uLpush \pi {} {\Strengthen{\opcart \pi N \hcomp \alpha_w' \hcomp \cart \rho P}} \uRpull {} \rho
 = \Strengthen{\alpha_w'}
\]
where the first and last steps are by definition and the middle step is by induction on
$\opcart \pi N \hcomp \alpha_w \hcomp \cart \rho P \permeq \opcart \pi N \hcomp \alpha_w' \hcomp \cart \rho P$ at the smaller judgment $N \vdashf{\pi f\rho} P$.
\item Case $N \vdashf{h} P$. We consider the derivation of the permutation equivalence $\alpha_w \permeq \alpha_w'$.
\begin{itemize}
\item Subcase $\alpha_w\permeq \alpha_w'$ is derived by the generator \eqref{eq:permgen1}, with $\alpha_w =(\tLpull {f} \beta) \tRpush {h}$ and $\alpha_w' = \tLpull {f} (\beta \tRpush {h})$.  In this case the two strengthenings are sequentializations of the same bi-focused derivation $\tLRpullpush f {\Strengthen \beta} h$.
\item Subcase $\alpha_w\permeq \alpha_w'$ is derived by a generator \eqref{eq:permgen2}--\eqref{eq:permgen4}: impossible at the judgment $N \vdashf{h} P$.
\item Subcase $\alpha_w = \tLpull{f} \beta$ and $\alpha'_w = \tLpull{f} \beta'$ for some $\beta \permeq \beta' : Q \vdashf{fh} P$ where $N = \pull{f} Q$.
Then $\Strengthen{\beta} \permeqseq \Strengthen{\beta'}$ by induction and so $\Strengthen{\alpha_w} = \tLpull{f} \Strengthen{\beta} \permeqseq \tLpull{f} \Strengthen{\beta'} = \Strengthen{\alpha_w'}$.
\item Subcase $\alpha_w = \beta\tRpush{g} $ and $\alpha'_w = \beta'\tRpush{g} $ for some $\beta \permeq \beta' : N \vdashf{hg} M$ where $P = \push{g} M$: same as above.
\end{itemize}


\end{itemize}
\end{proof}

The relation between permutation equivalence and sequentialization equivalence is captured by the following diagram:
\[
\begin{tikzcd}
(\mathsf{WeakMono}(S,f,T), \permeq)
\ar[r,bend left,"\Strengthensym"{name=L}]
\ar[phantom,r,"\sim"]
& (\mathsf{StrongMono}(S,f,T), \permeqseq)
\ar[r,bend left,hook]
\ar[l,spanmap,bend left,"\InvSeqsym"{name=R}]
\ar[phantom,r,"\sim"]
& (\mathsf{StrongMulti}(S,f,T), \permeqseq)
\ar[l,spanmap,bend left,"\FocSeqsym"{name=R}]
\end{tikzcd}
\]
where we have factored sequentialization $\Seq{-}$ in two different nondeterministic / multivalued functions:
$\FocSeq{\alpha_m}$ is the sequentialization of the bi-focused rules of a multifocused proof, and 
$\InvSeq{\alpha_m}$ is the sequentialization of the bi-invertible rules of a strongly monofocused proof.
This diagram expresses two equivalences of setoids (sets equipped with an equivalence relation), in the following sense:
\begin{enumerate}
\item $\Strengthensym$ is a morphism of setoids,
  per Proposition~\ref{prop:strengthening-preserves-equiv}.
  (This gives us completeness and relational completeness of
   strongly focused proof with respect to weakly focused proofs).
\item $\InvSeqsym$ and $\FocSeqsym$ are nondeterministic morphisms of setoids, per Proposition~\ref{prop:sequentialization-preserves-equiv}.
\item The equivalence on the right corresponds to the following statement:
  \begin{mathpar}
    \begin{array}{c}
      \alpha_m \permeq \FocSeq{\alpha_m}
      \\
      \forall \alpha'_m \in \FocSeq{\alpha_m},\ \alpha'_m \permeqseq \alpha_m
    \end{array}
  \end{mathpar}
  It is immediate from the fact that $\FocSeq{\alpha_m}$ is defined as the set of $(\rewtoseq)$-normal forms of $\alpha_m$, and $(\permeqseq)$ as the least congruence containing $(\rewtoseq)$.
\item The equivalence on the left corresponds to the following statements:
  \begin{restatable}{prop}{invseqS}
    \label{prop:invseq-S}
    \begin{mathpar}
      \begin{array}{c}
        \forall \alpha_w, \quad \alpha_w \permeq \InvSeq{\Strengthen{\alpha_w}}
        \\
        \forall \alpha_w,\ \forall \alpha'_w \in \InvSeq{\Strengthen{\alpha_w}},\quad \alpha_w \permeq \alpha'_w
      \end{array}
    \end{mathpar}
  \end{restatable}

  \begin{restatable}{prop}{Sinvseq}
    \label{prop:S-invseq}
  \begin{mathpar}
    \begin{array}{c}
      \forall \alpha_m, \quad \alpha_m \permeqseq \Strengthen{\InvSeq{\alpha_m}}
      \\
      \forall \alpha_m,\,\forall \alpha_w \in \InvSeq{\alpha_m},\quad \alpha_m \permeqseq \Strengthen{\alpha_w}
    \end{array}
  \end{mathpar}
  \end{restatable}
  \noindent For each property, on the first line we wrote the statement with $\InvSeqsym$ seen as a multi-valued function nondeterministically returning a weakly focused proof, and on the second line we wrote it with $\InvSeqsym$ seen directly as a function into sets of weakly focused proofs.
Both statements can be established by induction on derivations; we detail the proofs in Appendix~\ref{prf:invseqS}.
\end{enumerate}

The intermediate system $\mathsf{StrongMono}$ is useful to describe the fine-grained properties of strengthening and sequentialization, but note that the equivalence relation $(\permeqseq)$ has no natural definition that is internal to $\mathsf{StrongMono}$, since the source derivation $\alpha_m$ of a sequentialization step $\alpha_m \rewtoseq \alpha'_m$ uses a bi-focused rule.

\subsubsection{Horizontal composition (cut) of strongly multi-focused proofs}
\label{subsubsec:multi-focused-cut}
 Now that we have defined strengthening, sequentialization and their properties, we can easily define horizontal composition (cut) of strongly multifocused proofs modulo permutation equivalence: $\alpha_m \hcomp \beta_m$ can be defined as $\Strengthen{\alpha_w \hcomp \beta_w}$ for some $\alpha_w \in \Seq{\alpha_m}$, $\beta_w \in \Seq{\beta_m}$. The choice of $\alpha_w$ and $\beta_w$ is irrelevant: they are all permutation equivalent, so the resulting strengthening is unique modulo sequentialization equivalence. This definition is easily seen to be associative (lifted from the associativity of composition of non-focused proofs) and unital---for identity derivations defined as the $\Strengthen{\Id[A]}$. Finally, we also directly get that $\Strengthen{-}$ respects composition and preserves identity derivations.

We can thus define a category of strongly multifocused derivations quotiented by $(\permeqseq)$, with $\Strengthen{-}$ a functor from $\Bifib{p}$. But this category is just $\Bifib{p}$ itself, given that our equivalence of setoids gives an equality of quotiented sets.

It is also possible to define cut directly on strongly multifocused derivations. We detail this construction in Appendix~\ref{app:multi-focused-cut}.

\subsection{Normal forms via rewriting}
\label{sec:focusing:normal-forms}

In this section we introduce a strongly normalizing rewriting relation between multifocused proofs that, under a certain condition on the base category, is locally confluent and thus admits unique normal forms: all proofs in each sequentialization equivalence class eventually reduce to the same normal form.

An intuition is that these normal forms will be ``maximally parallel'', in the sense that $\biL$ and $\biR$ bipoles should be combined into an $\biLR$ bipole whenever possible.
We therefore begin by considering the \emph{parallelization} order $(\rewtopar)$, which is just the opposite of the sequentialization order ($\alpha \rewtopar \beta$ iff $\beta \rewtoseq \alpha$).
Diagrammatically it is defined by the pair of rules $\parsym_\Lsym^\Rsym$ and $\parsym_\Rsym^\Lsym$ shown below:

\begin{equation*}
\begin{tikzpicture}
  \node (B) {\usebox\myboxA};
  \node[left=4em of B] (A) {  \usebox\myboxB };
  \node[right=4em of B] (C) { \usebox\myboxC };
  \begin{scope}[transform canvas={yshift=-1em}]
    \draw[->] (B.west) -- (A.east) node[midway,below]{$\seqsym^\Rsym_\Lsym$};
  \end{scope}
  \begin{scope}[transform canvas={yshift=1em}]
  \draw[->] (A.east) -- (B.west) node[midway,above]{$\parsym^\Rsym_\Lsym$};
  \end{scope}
  \begin{scope}[transform canvas={yshift=-1em}]
  \draw[->] (B.east) -- (C.west) node[midway,below]{$\seqsym^\Lsym_\Rsym$};
  \end{scope}  
  \begin{scope}[transform canvas={yshift=1em}]
  \draw[->] (C.west) -- (B.east) node[midway,above]{$\parsym^\Lsym_\Rsym$};
  \end{scope}  
\end{tikzpicture}
\end{equation*}
Let us now make a few observations about the above rules.
While the sequentialization order is non-deterministic in the sense that a given multifocused proof can have many different sequentializations, each rule $\seqsym^\Rsym_\Lsym$ and $\seqsym^\Lsym_\Rsym$ is deterministic: for a given $\biLR$ bipole $\alpha$, there are unique $\beta$ and $\gamma$ such that $\alpha \rewtoseqRL \beta$ and $\alpha\rewtoseqLR \gamma$.
Dually, the rules $\parsym^\Rsym_\Lsym$ and $\parsym^\Lsym_\Rsym$ are \emph{co-deterministic} in the sense that for any given $\biLR$ bipole $\alpha$ there are unique $\beta$ and $\gamma$ such that $\beta \rewtoparRL \alpha$ and $\gamma \rewtoparLR \alpha$.
On the other hand, in general the parallelization rules are not deterministic.
Indeed, consider a rewriting diagram of the following form:
\newsavebox\myboxAg
\newsavebox\myboxAh
\begin{lrbox}{\myboxAg}
  \begin{tikzcd}
	\cdot & \cdot \\
	\cdot & \cdot \\
	\cdot & \cdot 
	\arrow["\pi f \rho", from=1-1, to=1-2]
	\arrow["\pi"', from=1-1, to=2-1] \arrow[phantom,from=1-1,to=2-2,"\push \Lsym\pull \Rsym"]
	\arrow["g" description, from=2-1, to=2-2]
	\arrow["\rho"', from=2-2, to=1-2]
	\arrow["\tau", from=2-2, to=3-2]
	\arrow["\sigma", from=3-1, to=2-1] \arrow[phantom,from=3-1,to=2-2,"\pull \Lsym\push \Rsym"]
	\arrow["\sigma f \tau" description, from=3-1, to=3-2]
  \end{tikzcd}
\end{lrbox}
\begin{lrbox}{\myboxAh}
  \begin{tikzcd}
	\cdot & \cdot \\
	\cdot & \cdot \\
	\cdot & \cdot 
	\arrow["\pi f \rho", from=1-1, to=1-2]
	\arrow["\pi"', from=1-1, to=2-1] \arrow[phantom,from=1-1,to=2-2,"\push \Lsym\pull \Rsym"]
	\arrow["h" description, from=2-1, to=2-2]
	\arrow["\rho"', from=2-2, to=1-2]
	\arrow["\tau", from=2-2, to=3-2]
	\arrow["\sigma", from=3-1, to=2-1] \arrow[phantom,from=3-1,to=2-2,"\pull \Lsym\push \Rsym"]
	\arrow["\sigma f \tau" description, from=3-1, to=3-2]
  \end{tikzcd}
\end{lrbox}
\begin{equation}
\label{eq:criticalLR}
\begin{tikzpicture}
  \node (B) {\usebox\myboxB};
  \node[left=4em of B] (A) {  \usebox\myboxAg };
  \node[right=4em of B] (C) { \usebox\myboxAh };
  \draw[->] (B.west) -- (A.east) node[midway,above]{$\parsym^\Rsym_\Lsym$};
  \draw[->] (B.east) -- (C.west) node[midway,above]{$\parsym^\Rsym_\Lsym$};
\end{tikzpicture}
\end{equation}
Under the assumption that
$f\tau = g\tau = h\tau$
and 
$\pi f = \pi g = \pi h$, 
both rewrites in \eqref{eq:criticalLR} are valid instances of the $\parsym^\Rsym_\Lsym$ rule.
However, the two resulting bipoles will be distinct unless $g = h$.
The same argument shows that the $\parsym^\Lsym_\Rsym$ rule is not deterministic in general.

A sufficient condition for the parallelization rules to be deterministic is that the underlying base category is \emph{factorization preordered} in the sense of Johnstone~\cite{Johnstone1999}.

\begin{defi}\label{def:FP-condition}
  A category $\C$ is \defin{factorization preordered} if for any diagram of composable arrows of the form
\[
\begin{tikzcd}
  \cdot
  \ar[r, "f"]
  &
  \cdot
  \ar[r, "g", yshift=1ex]
  \ar[r, "h"', yshift=-1ex]
  &
  \cdot
  \ar[r, "k"]
  &
  \cdot
\end{tikzcd}
\]
if both $fg = fh$ and $gk = hk$ then necessarily $g = h$.
Equivalently, $\C$ is factorization preordered just in case every commuting square has at most one diagonal filler:
\[
\begin{tikzcd}[column sep=3em, row sep=3em,baseline=1.8em]
  \cdot\ar[r,"x"]\ar[d,"f"']  & \cdot\ar[d,"k"] \\
  \cdot\ar[r, "y"']\ar[ur,dashed,"g=h" description]  & \cdot
\end{tikzcd}
\]
Equivalently again, $\C$ is factorization preordered just in case the \emph{category of factorizations} of every arrow is a preorder, where the category of factorizations of an arrow $w : A \to B$ has objects given by composable pairs of arrows $A \overset u\to X \overset {v}\to B$ such that $w = uv$, and morphisms $(A \overset u\to X \overset v\to B) \to (A \overset s\to Y \overset t\to B)$ given by arrows $h : X \to Y$ such that $uh = s$ and $ht = v$.
\end{defi}
\begin{exa}\label{exa:all-epi-or-mono-FP}
  Any category that has only monic arrows or only epic arrows is factorization preordered.
  In particular, every preorder category is factorization preordered, as is every free category.
\end{exa}
\noindent
Let us call ``\FP condition'' the hypothesis that the base category $\C$ of the free bifibration $\BifibFun{p} : \Bifib{p} \to \C$ is factorization preordered.

\begin{prop}
  Under the \FP condition, the parallelization rules $\parsym^\Rsym_\Lsym$ and $\parsym^\Lsym_\Rsym$ are deterministic, and the sequentialization rules $\seqsym^\Rsym_\Lsym$ and $\seqsym^\Lsym_\Rsym$ are co-deterministic.
\end{prop}
\noindent
Determinism is a special case of confluence, ensuring that a single application of a rule does not induce a critical pair with itself.
On the other hand, general confluence fails for the parallelization order, as one can exhibit a divergent critical pair from a stack of three bipoles.
For example, there is a critical pair whose apex is a stack of bipoles
$\biL\ \biR\ \biL$,

\newsavebox\myboxa
\newsavebox\myboxb
\newsavebox\myboxc

\begin{lrbox}{\myboxa}
  \begin{tikzcd}
	\cdot & \cdot \\
	\cdot & \cdot \\
	\cdot & \cdot \\
	\cdot & \cdot \\
	\cdot & \cdot
	\arrow["afc", from=1-1, to=1-2]
	\arrow["a"', from=1-1, to=2-1] \arrow[phantom,from=1-1,to=2-2,"\push \Lsym"]
	\arrow["fc" description, from=2-1, to=2-2]
	\arrow[equals, from=2-2, to=1-2]
	\arrow[equals, from=2-2, to=3-2]
	\arrow["b", from=3-1, to=2-1] \arrow[phantom,from=3-1,to=2-2,"\pull \Lsym"]
	\arrow["bfc = b'gc" description, from=3-1, to=3-2]
	\arrow["b'"', from=3-1, to=4-1] \arrow[phantom,from=3-1,to=4-2,"{\push \Lsym} {\pull \Rsym}"]
	\arrow["g" description, from=4-1, to=4-2]
	\arrow["c"', from=4-2, to=3-2] \arrow[phantom,from=4-1,to=5-2,"{\pull \Lsym} {\push \Rsym}"]
	\arrow["d", from=4-2, to=5-2]
	\arrow["e", from=5-1, to=4-1]
	\arrow["egd"', from=5-1, to=5-2]
  \end{tikzcd}
\end{lrbox}

\begin{lrbox}{\myboxb}
  \begin{tikzcd}
	\cdot & \cdot \\
	\cdot & \cdot \\
	\cdot & \cdot \\
	\cdot & \cdot \\
	\cdot & \cdot \\
	\cdot & \cdot \\
	\cdot & \cdot
	\arrow["afc", from=1-1, to=1-2]
	\arrow["a"', from=1-1, to=2-1] \arrow[phantom,from=1-1,to=2-2,"\push \Lsym"]
	\arrow["fc" description, from=2-1, to=2-2]
	\arrow[equals, from=2-2, to=1-2]
	\arrow[equals, from=2-2, to=3-2]
	\arrow["b", from=3-1, to=2-1] \arrow[phantom,from=2-1,to=3-2,"\pull \Lsym"]
	\arrow["bfc" description, from=3-1, to=3-2]
	\arrow[equals, from=3-1, to=4-1]
	\arrow["bf = b'g" description, from=4-1, to=4-2]
	\arrow["c"', from=4-2, to=3-2] \arrow[phantom,from=3-1,to=4-2,"\pull \Rsym"]
	\arrow["d", from=4-2, to=5-2] \arrow[phantom,from=4-1,to=5-2,"\push \Rsym"]
	\arrow[equals, from=5-1, to=4-1]
	\arrow["b'gd" description, from=5-1, to=5-2]
	\arrow["b'"', from=5-1, to=6-1] \arrow[phantom,from=5-1,to=6-2,"\push \Lsym"]
	\arrow[equals, from=5-2, to=6-2]
	\arrow["gd" description, from=6-1, to=6-2]
	\arrow[equals, from=6-2, to=7-2]
	\arrow["e", from=7-1, to=6-1] \arrow[phantom,from=6-1,to=7-2,"\pull \Lsym"]
	\arrow["egd"', from=7-1, to=7-2]
  \end{tikzcd}
\end{lrbox}
\begin{lrbox}{\myboxc}
  \begin{tikzcd}
	\cdot & \cdot \\
	\cdot & \cdot \\
	\cdot & \cdot \\
	\cdot & \cdot \\
	\cdot & \cdot
	\arrow["afc", from=1-1, to=1-2]
	\arrow["a"', from=1-1, to=2-1] \arrow[phantom,from=1-1,to=2-2,"{\push \Lsym}{\pull \Rsym}"]
	\arrow["f" description, from=2-1, to=2-2]
	\arrow["c"', from=2-2, to=1-2]
	\arrow["d", from=2-2, to=3-2]
	\arrow["b", from=3-1, to=2-1] \arrow[phantom,from=2-1,to=3-2,"{\pull\Lsym}{\push\Rsym}"]
	\arrow["bfd = b'gd" description, from=3-1, to=3-2]
	\arrow["b'"', from=3-1, to=4-1] \arrow[phantom,from=3-1,to=4-2,"\push\Lsym"]
	\arrow[equals, from=3-2, to=4-2]
	\arrow["gd" description, from=4-1, to=4-2]
	\arrow[equals, from=4-2, to=5-2]
	\arrow["e", from=5-1, to=4-1] \arrow[phantom,from=4-1,to=5-2,"\pull\Lsym"]
	\arrow["egd"', from=5-1, to=5-2]
  \end{tikzcd}
\end{lrbox}

\begin{equation}\label{eq:criticalLRL}
\begin{tikzpicture}
  \node (B) {\usebox\myboxb};
  \node[below left=-5cm and 4em of B] (A) {  \usebox\myboxa };
  \node[below right=-5cm and 4em of B] (C) { \usebox\myboxc };
  \pgfsetfillpattern{north east lines}{blue}
  \draw[->] ([yshift=6.8em]B.south west) -- ([yshift=-3em]A.east) node[midway,below right=5pt and -10pt,fill,nearly transparent,text opacity=1]{$\parsym^\Rsym_\Lsym$};
  \pgfsetfillpattern{north west lines}{red}
  \draw[->] ([yshift=-6.8em]B.north east) -- ([yshift=3em]C.west) node[midway,below left=10pt and -13pt,fill,nearly transparent,text opacity=1]{$\parsym^\Lsym_\Rsym$};
  \pgfsetfillpattern{north west lines}{red}
  \fill[nearly transparent] ([yshift=-0.25em]B.north west) rectangle ([yshift=-3.5em]B.east);
  \pgfsetfillpattern{north east lines}{blue}
  \fill[nearly transparent] ([yshift=3.5em]B.west) rectangle ([yshift=.25em]B.south east);
\end{tikzpicture}
\end{equation}
as well as a symmetric critical pair whose apex is a stack of bipoles $\biR\ \biL\ \biR$.
Here we clarify that the condition that $bf = b'g$ is necessary for the well-formedness of the stack in the middle of \eqref{eq:criticalLRL}, and implies the conditions $bfc = b'gc$ and $bfd = b'gd$ ensuring well-formedness of the stacks on the left and right.

\newsavebox\myboxd
\newsavebox\myboxe

\begin{lrbox}{\myboxe}
  \begin{tikzcd}
	\cdot & \cdot \\
	\cdot & \cdot \\
	\cdot & \cdot \\
	\cdot & \cdot \\
	\cdot & \cdot
	\arrow["afc", from=1-1, to=1-2]
	\arrow[equals, from=1-1, to=2-1] \arrow[phantom,from=1-1,to=2-2,"\pull\Rsym"]
	\arrow["af" description, from=2-1, to=2-2]
	\arrow["c"', from=2-2, to=1-2]
	\arrow["d", from=2-2, to=3-2]
	\arrow[equals, from=3-1, to=2-1] \arrow[phantom,from=3-1,to=2-2,"\push\Rsym"]
	\arrow["afd = agd'" description, from=3-1, to=3-2]
	\arrow["a"', from=3-1, to=4-1] \arrow[phantom,from=3-1,to=4-2,"{\push\Lsym}{\pull\Rsym}"]
	\arrow["g" description, from=4-1, to=4-2]
	\arrow["d'"', from=4-2, to=3-2] \arrow[phantom,from=4-1,to=5-2,"{\pull\Lsym}{\push\Rsym}"]
	\arrow["e", from=4-2, to=5-2]
	\arrow["b", from=5-1, to=4-1]
	\arrow["bge"', from=5-1, to=5-2]
  \end{tikzcd}
\end{lrbox}

\begin{lrbox}{\myboxd}
  \begin{tikzcd}
	\cdot & \cdot \\
	\cdot & \cdot \\
	\cdot & \cdot \\
	\cdot & \cdot \\
	\cdot & \cdot
	\arrow["afc", from=1-1, to=1-2]
	\arrow["a"', from=1-1, to=2-1] \arrow[phantom,from=1-1,to=2-2,"{\push\Lsym}{\pull\Rsym}"]
	\arrow["f" description, from=2-1, to=2-2]
	\arrow["c"', from=2-2, to=1-2]
	\arrow["d", from=2-2, to=3-2]
	\arrow["b", from=3-1, to=2-1] \arrow[phantom,from=2-1,to=3-2,"{\pull\Lsym}{\push\Rsym}"]
	\arrow["bfd = bgd'" description, from=3-1, to=3-2]
	\arrow[equals, from=3-1, to=4-1] \arrow[phantom,from=3-1,to=4-2,"\pull\Rsym"]
	\arrow["d'"', from=4-2, to=3-2]
	\arrow["bg" description, from=4-1, to=4-2]
	\arrow["e", from=4-2, to=5-2]
	\arrow[equals, from=5-1, to=4-1] \arrow[phantom,from=4-1,to=5-2,"\push \Rsym"]
	\arrow["bge"', from=5-1, to=5-2]
  \end{tikzcd}
\end{lrbox}

\newsavebox\myboxf
\begin{lrbox}{\myboxf}
  \begin{tikzcd}
	\cdot & \cdot \\
	\cdot & \cdot \\
	\cdot & \cdot \\
	\cdot & \cdot \\
	\cdot & \cdot \\
	\cdot & \cdot \\
	\cdot & \cdot
	\arrow["afc", from=1-1, to=1-2]
	\arrow["a"', from=1-1, to=2-1] \arrow[phantom,from=1-1,to=2-2,"\push \Lsym \pull \Rsym"]
	\arrow["f" description, from=2-1, to=2-2]
	\arrow["c"', from=2-2, to=1-2]
	\arrow["d", from=2-2, to=3-2]
	\arrow["b", from=3-1, to=2-1] \arrow[phantom,from=2-1,to=3-2,"\pull \Lsym \push \Rsym"]
	\arrow["bfd=b'gd" description, from=3-1, to=3-2]
	\arrow["b'"', from=3-1, to=4-1]
	\arrow["gd" description, from=4-1, to=4-2]
	\arrow[equals, from=4-2, to=3-2] \arrow[phantom,from=3-1,to=4-2,"\push \Lsym"]
	\arrow[equals, from=4-2, to=5-2] \arrow[phantom,from=4-1,to=5-2,"\pull \Lsym"]
	\arrow["e", from=5-1, to=4-1]
	\arrow["egd=egd'" description, from=5-1, to=5-2]
	\arrow[equals, from=5-1, to=6-1] \arrow[phantom,from=5-1,to=6-2,"\pull \Rsym"]
	\arrow["d'"', from=6-2, to=5-2]
	\arrow["eg" description, from=6-1, to=6-2]
	\arrow["i", from=6-2, to=7-2]
	\arrow[equals, from=7-1, to=6-1] \arrow[phantom,from=6-1,to=7-2,"\push \Rsym"]
	\arrow["egi"', from=7-1, to=7-2]
  \end{tikzcd}
\end{lrbox}

\newsavebox\myboxg
\begin{lrbox}{\myboxg}
  \begin{tikzcd}
	\cdot & \cdot \\
	\cdot & \cdot \\
	\cdot & \cdot \\
	\cdot & \cdot \\
	\cdot & \cdot \\
	\cdot & \cdot \\
	\cdot & \cdot
	\arrow["afc", from=1-1, to=1-2]
	\arrow["a"', from=1-1, to=2-1] \arrow[phantom,from=1-1,to=2-2,"\push \Lsym"]
	\arrow["fc" description, from=2-1, to=2-2]
	\arrow[equals, from=2-2, to=1-2]
	\arrow[equals, from=2-2, to=3-2]
	\arrow["b", from=3-1, to=2-1] \arrow[phantom,from=2-1,to=3-2,"\pull \Lsym"]
	\arrow["bfc=b'gc" description, from=3-1, to=3-2]
	\arrow["b'"', from=3-1, to=4-1]
	\arrow["g" description, from=4-1, to=4-2]
	\arrow["c"', from=4-2, to=3-2] \arrow[phantom,from=3-1,to=4-2,"\push \Lsym \pull \Rsym"]
	\arrow["d", from=4-2, to=5-2] \arrow[phantom,from=4-1,to=5-2,"\pull \Lsym \push \Rsym"]
	\arrow["e", from=5-1, to=4-1]
	\arrow["egd=egd'" description, from=5-1, to=5-2]
	\arrow[equals, from=5-1, to=6-1] \arrow[phantom,from=5-1,to=6-2,"\pull \Rsym"]
	\arrow["d'"', from=6-2, to=5-2]
	\arrow["eg" description, from=6-1, to=6-2]
	\arrow["i", from=6-2, to=7-2]
	\arrow[equals, from=7-1, to=6-1] \arrow[phantom,from=6-1,to=7-2,"\push \Rsym"]
	\arrow["egi"', from=7-1, to=7-2]
  \end{tikzcd}
\end{lrbox}

\newsavebox\myboxh
\begin{lrbox}{\myboxh}
  \begin{tikzcd}
	\cdot & \cdot \\
	\cdot & \cdot \\
	\cdot & \cdot \\
	\cdot & \cdot \\
	\cdot & \cdot \\
	\cdot & \cdot \\
	\cdot & \cdot
	\arrow["afc", from=1-1, to=1-2]
	\arrow["a"', from=1-1, to=2-1] \arrow[phantom,from=1-1,to=2-2,"\push \Lsym"]
	\arrow["fc" description, from=2-1, to=2-2]
	\arrow[equals, from=2-2, to=1-2]
	\arrow[equals, from=2-2, to=3-2]
	\arrow["b", from=3-1, to=2-1] \arrow[phantom,from=2-1,to=3-2,"\pull \Lsym"]
	\arrow["bfc=b'gc" description, from=3-1, to=3-2]
	\arrow[equals, from=3-1, to=4-1]
	\arrow["b'g" description, from=4-1, to=4-2]
	\arrow["c"', from=4-2, to=3-2] \arrow[phantom,from=3-1,to=4-2,"\pull \Rsym"]
	\arrow["d", from=4-2, to=5-2] \arrow[phantom,from=4-1,to=5-2,"\push \Rsym"]
	\arrow[equals, from=5-1, to=4-1]
	\arrow["b'gd=b'gd'" description, from=5-1, to=5-2]
	\arrow["b'"', from=5-1, to=6-1] \arrow[phantom,from=5-1,to=6-2,"\push \Lsym \pull \Rsym"]
	\arrow["d'"', from=6-2, to=5-2]
	\arrow["g" description, from=6-1, to=6-2]
	\arrow["i", from=6-2, to=7-2]
	\arrow["e", from=7-1, to=6-1] \arrow[phantom,from=6-1,to=7-2,"\pull \Lsym \push \Rsym"]
	\arrow["egi"', from=7-1, to=7-2]
  \end{tikzcd}
\end{lrbox}

\newsavebox\myboxi
\begin{lrbox}{\myboxi}
  \begin{tikzcd}
	\cdot & \cdot \\
	\cdot & \cdot \\
	\cdot & \cdot \\
	\cdot & \cdot \\
	\cdot & \cdot \\
	\arrow["afc", from=1-1, to=1-2]
	\arrow["a"', from=1-1, to=2-1] \arrow[phantom,from=1-1,to=2-2,"\push \Lsym \pull \Rsym"]
	\arrow["f" description, from=2-1, to=2-2]
	\arrow["c"', from=2-2, to=1-2]
	\arrow["d", from=2-2, to=3-2]
	\arrow["b", from=3-1, to=2-1] \arrow[phantom,from=2-1,to=3-2,"\pull \Lsym \push \Rsym"]
	\arrow["bfd=b'gd'" description, from=3-1, to=3-2]
	\arrow["b'"', from=3-1, to=4-1]
	\arrow["g" description, from=4-1, to=4-2]
	\arrow["d'"', from=4-2, to=3-2] \arrow[phantom,from=3-1,to=4-2,"\push \Lsym \pull \Rsym"]
	\arrow["i", from=4-2, to=5-2] \arrow[phantom,from=4-1,to=5-2,"\pull \Lsym \push \Rsym"]
	\arrow["e", from=5-1, to=4-1]
	\arrow["egi" description, from=5-1, to=5-2]
  \end{tikzcd}
\end{lrbox}

Following the principles of Knuth--Bendix completion \cite{KnuthBendix1983simple}, we can obtain a confluent system by adding two new rewriting rules to orient each of these critical pairs in one direction:
\begin{equation}
  \label{eq:rewrite-gravity}
\begin{tikzpicture}
  \node (A) {  \usebox\myboxa };
  \node[left=4em of A] (C) { \usebox\myboxc };
  \draw[->] (C.east) -- (A.west) node[midway,below right=5pt and -10pt]{$\grasym_\Lsym$} node[midway,above]{\scriptsize$bf=b'g$};;
  \node[right=3em of A] (D) {  \usebox\myboxd };
  \node[right=4em of D] (E) { \usebox\myboxe };
  \draw[->] (D.east) -- (E.west) node[midway,below right=5pt and -10pt]{$\grasym_\Rsym$} node[midway,above]{\scriptsize$fd=gd'$};
\end{tikzpicture}
\end{equation}

\newcommand\criticalLRLR{
\begin{tikzpicture}
  \node (A) {\usebox\myboxf};
  \node[right=3cm of A] (B) {\usebox\myboxg};
  \node[below=3cm of A] (D) {\usebox\myboxi};
  \node[right=3cm of D] (C) {\usebox\myboxh};

  \pgfsetfillpattern{crosshatch dots}{green}
  \draw[->] ([yshift=1.5cm]A.east) -- ([yshift=1.5cm]B.west) node[midway,above,fill,nearly transparent,text opacity=1]{$\grasym_\Lsym$} node[midway,below]{\scriptsize$bf=b'g$};
  \pgfsetfillpattern{crosshatch dots}{violet}
  \draw[->] (B.south) -- (C.north) node[midway,right,fill,nearly transparent,text opacity=1]{$\grasym_\Rsym$} node[midway,left]{\scriptsize$gd=gd'$};
  \pgfsetfillpattern{north west lines}{red}
  \draw[->] (A.south) -- (D.north) node[midway,left,fill,nearly transparent,text opacity=1]{$\parsym^\Lsym_\Rsym$} node[midway,right]{\scriptsize$gd=gd'$};
  \pgfsetfillpattern{north west lines}{red}
  \draw[->] (C.west) -- (D.east) node[midway,below,fill,nearly transparent,text opacity=1]{$\parsym^\Lsym_\Rsym$} node[midway,above]{\scriptsize$bf=b'g$};
  \pgfsetfillpattern{crosshatch dots}{green}
  \fill[nearly transparent] ([yshift=-.25em]A.north west) rectangle ([yshift=-3.5em]A.east);
  \pgfsetfillpattern{north west lines}{red}
  \fill[nearly transparent] ([yshift=3.5em]A.west) rectangle ([yshift=.25em]A.south east);
  \pgfsetfillpattern{crosshatch dots}{violet}
  \fill[nearly transparent] ([yshift=3.5em]B.west) rectangle ([yshift=.25em]B.south east);
  \pgfsetfillpattern{north west lines}{red}
  \fill[nearly transparent] ([yshift=-.25em]C.north west) rectangle ([yshift=-3.5em]C.east);
\end{tikzpicture}
}

\noindent
Here we annotate the rules with the side conditions that are necessary for applying them.
Observe that these conditions, which correspond to the equation appearing in the middle of \eqref{eq:criticalLRL} as well as in the symmetric critical pair, imply that both sides of each rewrite rule are well-formed.

\begin{wrapfigure}[25]{r}{0.41\textwidth}
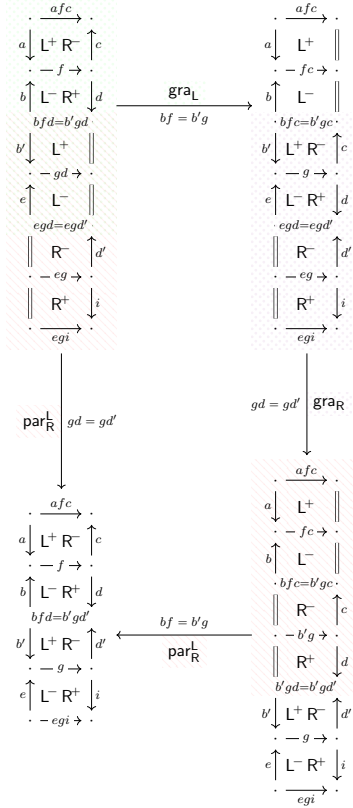
\centering
  \vspace{-1\baselineskip}
  \scalebox{0.60}{\criticalLRLR}\caption{Resolution of a critical pair.}\label{fig:critical-pair}
\end{wrapfigure}
We call these ``gravitation''  rules because they have the effect of pulling bi-focused bipoles down towards the root of the derivation. The choice of moving downards towards the root is arbitrary, and we could have taken the opposite convention to move them up. (In the multifocusing literature the derivations generally have the shape of a tree rather than a stack, and then the convention to move multifocusing steps down towards the root is natural.)

It is clear that the gravitation rules are both deterministic and co-deterministic, without making any assumptions about the base category, since the arrows $a, b, c, d, e, f$ appear in the source and target of each rule.

\begin{restatable}[Local confluence]{thm}{localconfluence}
  \label{thm:local-confluence}
  The rewrite system $(\parsym \cup \grasym)$ is locally confluent: if $\beta_m \leftarrow \alpha_m \rightarrow \beta'_m$ then there exists $\gamma_m$ such that $\beta_m \rightarrow^* \gamma_m \leftarrow^* \beta'_m$.
\end{restatable}

\begin{proof}[Proof sketch]
  The source of each rewrite rule must be a stack of two bipoles. Thus
  all critical pairs (where the source of two rewrite rules overlap)
  can be found by considering stacks of two or three bipoles.

  The proof (detailed in Appendix~\ref{prf:local-confluence}) enumerates all
  possible critical pairs, and shows that each of them can be
  resolved.
There are two non-trivial pairs (modulo symmetries), one
  with three bipoles $\biL ~ \biR ~ \biL$ which is exactly the
  critical pair \eqref{eq:criticalLRL} which is resolved by the rule
  $\grasym_\Lsym$, and one with three bipoles $\biLR ~ \biL ~ \biR$
  involving the rules $\grasym_\Lsym$ on one side and
  $\parsym^\Lsym_\Rsym$ on the other, which is resolved by applying
  $\parsym^\Lsym_\Rsym$ on one side and $\grasym_\Rsym$ on the other (see Figure~\ref{fig:critical-pair}).
\end{proof}

\begin{lem}
  $\parsym \cup \grasym$ is terminating: any valid rewriting sequence is finite.
\end{lem}

\begin{proof}
  \label{prf:par-gra-terminating}
  We provide a non-negative measure for derivations that $\parsym \cup \grasym$ decreases strictly.

  We measure each derivation by summing the weight of all bipoles. If a bipole is at position $i$ starting from the bottom, the root of the derivation, we define its weight as $i$ if it is a $\biL$ or $\biR$ bipole, and $2 \times i$ if it is a $\biLR$ bipole: the weight of $\biLR$ is the weight of $\biL$ plus the weight of $\biR$.

  Consider a $\grasym$ rewrite ($\grasym_\Lsym$ for example), which is of the following shape annotated with vertical positions:
  \begin{mathpar}
    \begin{array}{ll}
        \beta & \\ \biLR & (i + 1) \\ \biL & (i) \\ \alpha &
    \end{array}
    \quad
    \rightarrow_{\grasym_\Lsym}
    \quad
    \begin{array}{ll}
      \beta & \\ \biL & (i + 1) \\ \biLR & (i) \\ \alpha
    \end{array}
  \end{mathpar}

  The starting positions of the sub-stacks $\alpha$ and $\beta$ do not matter in this rewrite as their weight is unchanged. The weight of the middle of the derivation decreases strictly from $2 \times (i + 1) + i$ to $(i + 1) + 2 \times i$.

  Consider a $\parsym$ rewrite ($\parsym^\Lsym_\Rsym$ for example), it is of the shape
  \begin{mathpar}
    \begin{array}{ll}
        \beta & (i + 2) \\ \biL & (i + 1) \\ \biR & (i) \\ \alpha &
    \end{array}
    \quad
    \rightarrow_{\parsym^\Lsym_\Rsym}
    \quad
    \begin{array}{ll}
      ~ \\ \beta & (i + 1) \\ \biLR & (i) \\ \alpha
    \end{array}
  \end{mathpar}

  The weight of $\alpha$ is unchanged by the rewrite. The combined weight of the $\biL$ and $\biR$ bipoles decreases from $i + (i + 1)$ to $2 \times i$. The weight of $\beta$ also decreases strictly as all its bipoles are shifted one position down.
\end{proof}

\newcommand{\NF}[1]{\mathsf{NF}(#1)}

\noindent
From local confluence and termination, we  derive our main canonicity theorem.
\begin{thm}\label{thm:fp-unique-normal-forms}
  Under the \FP condition, any multifocused derivation $\alpha_m$ admits a unique normal form for $\parsym \cup \grasym$, which we write $\NF{\alpha_m}$.
Moreover, for any pair of derivations $\alpha_m$ and $\beta_m$ we have
  \begin{mathpar}
    \alpha_m \permeqseq \beta_m

    \iff

    \NF{\alpha_m} = \NF{\beta_m}.
  \end{mathpar}
\end{thm}

\subsection{Maximal multifocusing}
\label{subsec:maximal-multi-focusing}

We finally introduce a restriction of the multifocused sequent calculus that precisely captures the normal forms for $\parsym \cup \grasym$. Under the \FP condition, proofs in this restricted calculus give canonical representatives for the arrows of $\Bifib{p}$.


An intuition is that, when building the proof from the root to the axioms, we only need information about the last invertible rule that was applied to conclude a bipole. In the immediately following bipole, we can use this information to check that the non-invertible rule does not introduce a $\parsym$ or $\grasym$ redex.

Suppose for example that the last invertible rule was $\Lpush[\pi]$ as part of a $\biL$ bipole with intermediate base arrow $f$, and suppose that---assuming $\pi f$ is of the form $g \tau$ for some $g$---we want to continue the proof by applying $\Rpush[\tau]$ as part of a $\biR$ bipole, as depicted on the left:
\begin{center}
\begin{tabular}{c:c}
\begin{tikzcd}[ampersand replacement=\&]
    \cdot \ar[r, "g \dots"] \& \cdot \\
    \cdot \ar[r, "g"] \ar[u, equal] \& \cdot  \ar[d, "\tau"] \ar[u, "\dots"'] \\
    \cdot \ar[r,"g\tau=\pi f"'] \ar[d, "\pi"'] \ar[u, equal] \& \cdot \\
    \cdot \ar[r, "f"'] \& \cdot \ar[u, equal] \\
    \cdot \ar[r,"\dots f"'] \ar[u, "\dots"] \& \cdot \ar[u, equal]
\end{tikzcd} \hspace{1cm} & \hspace{1cm}
\begin{tikzcd}[ampersand replacement=\&]
  \cdot\ar[d,"\pi"']\ar[r,"g"] \& \cdot\ar[d,"\tau"] \\
  \cdot\ar[r,"f"']\ar[ur,dashed,"h" description] \& \cdot
\end{tikzcd}
\end{tabular}
\end{center}
This creates a $\parsym^\Rsym_\Lsym$ redex exactly when there is an arrow $h$ such that $g = \pi h$ and $f = h \tau$, that is when the square to the right has a diagonal filler, corresponding to a morphism $h : (\pi,f) \to (g,\tau)$ in the category of factorizations of the composite arrow $\pi f = g\tau$.

If instead we try to extend the $\biL$ bipole with a $\biLR$ bipole, we should check that we don't create a $\grasym$ redex. Such an extension takes the shape on the left:
\begin{center}
\begin{tabular}{c:c}
\begin{tikzcd}[ampersand replacement=\&]
    \cdot \ar[r, "\dots g \dots"] \ar[d, "\dots"'] \& \cdot \\
    \cdot \ar[r, "g"] \& \cdot  \ar[d, "\tau"] \ar[u, "\dots"'] \\
    \cdot \ar[r,"\sigma g \tau=\pi f"'] \ar[d, "\pi"'] \ar[u, "\sigma"] \& \cdot \\
    \cdot \ar[r, "f"'] \& \cdot \ar[u, equal] \\
    \cdot \ar[r,"\dots f"] \ar[u, "\dots"] \& \cdot \ar[u, equal] 
\end{tikzcd} \hspace{1cm}\hspace{1cm}&\hspace{1cm}
\begin{tikzcd}[ampersand replacement=\&]
  \cdot\ar[d,"\pi"']\ar[r,"\sigma g"] \& \cdot\ar[d,"\tau"] \\
  \cdot\ar[r,"f"']\ar[ur,dashed,"h" description] \& \cdot
\end{tikzcd}
\end{tabular}
\end{center}
This forms a $\grasym$ redex exactly when $f$ is of the form $h\tau$ for some $h$ such that $\pi h = \sigma g$, that is when the square to the right has a diagonal filler corresponding to a morphism $h : (\pi,f) \to (\sigma g,\tau)$ in the category of factorizations of the composite arrow $\pi f = \sigma g\tau$.

Notice that in both cases, in order to decide whether or not there is a redex we need to know the intermediate arrow $f$ of the $\biL$ bipole.
By keeping track of the last rule applied as well as the intermediate arrow, we can thus rule out the creation of any $\parsym$ or $\grasym$ redexes.

In Figure~\ref{fig:maxmultfocseq} we present a refinement of the multifocused sequent calculus that we call \emph{maximally multifocused,} in which judgments $S \vdashl{f}{q} T$ are annotated with a state $q$ that can be either a \emph{pre-inversion state} $q = \istate$ or a \emph{pre-focus state} $q = \fstate$.
The rules are best read operationally as goal transformers from conclusion to premise, following the conventions of proof search.
The idea is that a pre-inversion state forces an inversion rule to be applied in the next step (recall that we allow $\pi$ and $\rho$ to be empty), while a pre-focus state forces one of the focusing rules or an atomic rule to be applied.
After applying a focusing rule, the state transitions to a pre-inversion state of the form $\lockL$, $\lockR$, or $\ilockNone$, indicating the start of a $\biL$-, $\biR$-, or $\biLR$-bipole respectively.
The subsequent inversion rule supplements the state with additional information about the arrows involved, which is then used at the start of the next bipole to rule out any $\parsym$ or $\grasym$ redices: this is done with the side conditions for the focusing rules, which are indicated to the side in a smaller font in Figure~\ref{fig:maxmultfocseq}.
\begin{figure}
\begin{mathpar}
  \begin{array}{rcl}
    \text{pre-inversion states}\quad\istate & ::= & \ilockNone \mid \lockL \mid \lockR \\
    \text{pre-focus states}\quad\fstate & ::= & \flockNone \mid \lockL[\pi,f] \mid \lockR[f,\rho] \\
  \end{array}

  \ilockNone{[\pi,f]} = \flockNone = \ilockNone{[f,\rho]}

  \inferrule*[Right={$\atom \delta$}]
  {\delta : X \to Y \in \D \\ p(\delta) = f}
  {\atom X \vdashl{f}{\fstate} \atom Y}
\\\\

  \inferrule*[Right={$\Lpush[\pi]$}]
  {N \vdashl{\pi f}{\istate{[\pi,f]}} P}
  {\push{\pi}N \vdashl{f}{\istate} P}

  \inferrule*[Right={$\Lpush[\pi]\Rpull[\rho]$}]
  {N \vdashl{\pi f \rho}{\flockNone} P}
  {\push{\pi}N \vdashl{f}{\ilockNone} \pull{\rho}P}

  \inferrule*[Right={$\Rpull[\rho]$}]
  {N \vdashl{f \rho}{\istate{[f,\rho]}} P}
  {N \vdashl{f}{\istate} \pull{\rho}P}
  \\
  
\inferrule*[Left={\tiny$\fstate = \lockR[f,\rho]\Rightarrow (\sigma,g) \not\le (f,\rho)$},Right={$\Lpull[\sigma]$}]
  {P \vdashl{g}{\lockL} Q}
  {\pull{\sigma}P \vdashl{\sigma g}{\fstate} Q}

  \inferrule*[Left={\tiny$\fstate = \lockL[\pi,f] \Rightarrow (\pi,f) \not\le (g,\tau)$},Right={$\Rpush[\tau]$}]
  {N \vdashl{g}{\lockR} M}
  {N \vdashl{g \tau}{\fstate} \push{\tau}M}  
\\

  \inferrule*[Left={\tiny$\begin{array}{c}\fstate = \lockL[\pi,f]\Rightarrow (\pi,f) \not\le (\sigma g,\tau)\\\fstate = \lockR[f,\rho] \Rightarrow  (\sigma,g\tau) \not\le (f,\rho)\end{array}$},Right={$\Lpull[\sigma]\Rpush[\tau]$}]
  {\deduce{P \vdashl{g}{\ilockNone} N}{\deduce{\vspace{1pt}}{\color{persianred}}}}
  {\pull{\sigma}P \vdashl{\sigma g\tau}{\fstate} \push{\tau}N}
  \\
  
\end{mathpar}
\hfill where $(a,b) \le (c,d) \defeq \exists e .\, c = ae \wedge b = ed$.
\caption{Maximally multifocused derivations (factorization preordered base category)}
\label{fig:maxmultfocseq}
\end{figure}%

For example, suppose we are in a ``left-locked'' state $\fstate = \lockL[\pi,f]$.
That means that in the current stage of proof search, we just performed a left-focus rule to initiate a $\biL$-bipole followed by a $\Lpush[\pi]$ rule to terminate it.
We now have the following choices are available:
\begin{itemize}
\item any left-focus rule can be applied freely; 
\item a right-focus rule $\Rpush[\tau]$ beginning a $\biR$ bipole with intermediate arrow $g$ can only be applied if $(\pi,f) \not\le (g,\tau)$ in the factorization order;
\item a bi-focus rule ${\Lpull[\sigma]\Rpush[\tau]}$ beginning a $\biLR$ bipole with intermediate arrow $g$ can only be applied if $(\pi,f) \not\le (\sigma g,\tau)$ in the factorization order.
\end{itemize}
Dual restrictions apply in a ``right-locked'' state $\fstate = \lockR[f,\rho]$.

Our use of lock states is inspired by formulations of focusing and maximal multifocusing using tag annotations by Uustalu, Veltri, and Wan~\cite{UVW2022,Veltri2023}.

By the preceding analysis we immediately have the following result.
\begin{prop}\label{prop:maxmultfoc-nopargra}
  Maximally multifocused proofs of unlocked judgments $S \vdashl{f}{\ilockNone} T$ are in one-to-one correspondence with $(\parsym \cup \grasym)$-normal multifocused proofs of $S \vdashf{f} T$.
\end{prop}
\begin{proof}
  The side conditions on the rules $\Lpull[\sigma]$ and $\Rpush[\tau]$ ensure that they never create a $\parsym$ redex, while the side condition on the $\Lpull[\sigma]\Rpush[\tau]$ rule ensures that it never creates a $\grasym$ redex.
  Conversely, given a multifocused proof with no $\parsym$ or $\grasym$ redex, we can always annotate it with states following the locking discpline of Figure~\ref{fig:maxmultfocseq} in a uniquely determined way, starting from an unlocked state $\istate = \ilockNone$.
\end{proof}
\noindent
Combining Proposition~\ref{prop:maxmultfoc-nopargra} with Theorem~\ref{thm:fp-unique-normal-forms} then implies the following.
\begin{cor}\label{cor:FP=>maxmultfoc-canonical}
  Under the \FP condition, there is a one-to-one correspondence between arrows $\alpha : S \to T$ in $\Bifib{p}$ such that $\BifibFun{p}(\alpha) = f$ and maximally multifocused proofs of $S \vdashl{f}{\ilockNone} T$.
\end{cor}
\noindent
This correspondence means that if the base category $\C$ is factorization preordered then maximally multifocused proofs give an equation-free inductive definition of arrows of $\Bifib{p}$.
An important special case is when $\C = \Free\G$ is the free category over a graph $\G$.
In that case, the factorization order reduces to a linear order (a path of length $n$ has $n+1$ linearly ordered factorizations) and the calculus can be simplified somewhat: states only need to keep track of the type of the last bipole and the intermediate arrow $f$ (but not $\pi$ or $\rho$), and the factorization order test can be replaced by a single prefix or suffix test (since for any two factorizations of a word $w = uv = u'v'$, $u$ is a prefix of $u'$ just in case $v'$ is a suffix of $v$).

On the other hand, in both the free case and general \FP case, our definition of cut-elimination or horizontal composition of two multifocused proofs in Section~\ref{subsubsec:multi-focused-cut} does not preserve maximality---to obtain a maximal multifocused proof one needs to renormalize the result of the composition. We leave open the question of whether there is a more direct definition of horizontal composition on maximal derivations.

\subsection{Decidability and enumeration}
\label{subsec:maxmultfoc-decidability}

In studying their free adjoint construction, Dawson, Paré, and Pronk~\cite{DPP03b} established the surprising result that, in general, equality of 2-cells in $\Pi_2(\C)$ is undecidable---or more precisely that there is no decision procedure for determining whether two representatives of 2-cells (what they call ``fences'') denote the same 2-cell.
They gave two proofs of this result, first by exhibiting a very concrete reduction from the halting problem for 2-register machines, and second by giving a somewhat simpler and more abstract reduction from the undecidable problem of determining whether two vertices of an infinite (but locally finite) bipartite graph are connected by a path.
Their latter proof translates cleanly to an analogous undecidability result for sequentialization equivalence.
\begin{defi}
  We say that a category $\C$ has \defin{locally finite factorizations} just in case every commuting square in $\C$ has at most finitely many diagonal fillers, or equivalently if the category of factorizations of every arrow is locally finite (i.e., has finite homsets).
\end{defi}
By a \emph{locally finite bipartite graph,} we mean a span of sets $V \overset{p_1}\leftarrow E \overset{p_2}\rightarrow W$ such that the fibers of $p_1$ and $p_2$ are finite.
(Every edge $e\in E$ connects a vertex $v \in V$ to a vertex $w \in W$, where $v = p_1(e)$ and $w = p_2(e)$. So local finiteness says that every vertex of either class is involved in only finitely many edges, i.e., has finite degree.)
To any locally finite bipartite graph $G$, we can associate a category $\C_G$ with locally finite factorizations, defined as the subcategory of $\Set$ generated by the five sets 
\[ 0, \quad 1,\quad V,\quad W,\quad E\]
(where 0 and 1 are respectively the empty set and singleton set), the projection functions
\[
  \begin{tikzcd}[ampersand replacement=\&]
    V \& E\ar[l,"p_1"']\ar[r,"p_2"] \& W
  \end{tikzcd}
\]
describing the incidence relation of $G$,
the constant functions
\[
  \begin{tikzcd}[ampersand replacement=\&]
    1\ar[r,"v"] \& V \& 1\ar[r,"w"] \& W \& 1\ar[r,"e"] \& E
  \end{tikzcd}
\]
associated to every vertex $v\in V$ or $w\in W$ or edge $e\in E$, and finally the unique functions from the empty set
\[
  \begin{tikzcd}[ampersand replacement=\&]
    0 \ar[r,"a_X"] \& X
  \end{tikzcd}
\]
for $X \in \set{1,V,W,E}$.
To see that local finiteness of $G$ implies that $\C_G$ has locally finite factorizations, first note that the only non-degenerate commuting squares are of the form 
\[
  \begin{tikzcd}[ampersand replacement=\&]
    0 \ar[r,"a_E"]\ar[d,"a_1"'] \& E \ar[d,"p_1"] \\
    1 \ar[r,"v"'] \& V
  \end{tikzcd}
  \qquad
  \begin{tikzcd}[ampersand replacement=\&]
    0 \ar[r,"a_E"]\ar[d,"a_1"'] \& E \ar[d,"p_2"] \\
    1 \ar[r,"w"'] \& W
  \end{tikzcd}
\]
and then observe that diagonal fillers for these squares correspond to edges incident to the respective vertices:
\[
  \begin{tikzcd}[ampersand replacement=\&]
    0 \ar[r,"a_E"]\ar[d,"a_1"'] \& E \ar[d,"p_1"] \\
    1 \ar[r,"v"']\ar[ur,dashed,"e" description] \& V
  \end{tikzcd}
  \qquad
  \begin{tikzcd}[ampersand replacement=\&]
    0 \ar[r,"a_E"]\ar[d,"a_1"'] \& E \ar[d,"p_2"] \\
    1 \ar[r,"w"']\ar[ur,dashed,"e" description] \& W
  \end{tikzcd}
\]
Now, to any vertex or edge $x$ of $G$ let us associate a stack $\alpha_x$ of the zigzag double category $\ZZ(\C_G)$, as follows:

\[
  \alpha_v \quad\defeq \quad
  \begin{tikzcd}
	0 & V \\
	1 & V \\
	0 & V \\
	0 & E \\
	0 & W \\
	\arrow["a_V", from=1-1, to=1-2]
	\arrow["a_1"', from=1-1, to=2-1] \arrow[phantom,from=1-1,to=2-2,"\push \Lsym"]
	\arrow["v" description, from=2-1, to=2-2]
	\arrow[equals, from=2-2, to=1-2]
	\arrow[equals, from=2-2, to=3-2]
	\arrow["a_1", from=3-1, to=2-1] \arrow[phantom,from=2-1,to=3-2,"\pull \Lsym"]
	\arrow["a_V" description, from=3-1, to=3-2]
	\arrow[equals, from=3-1, to=4-1]
	\arrow["a_E" description, from=4-1, to=4-2]
	\arrow["p_1"', from=4-2, to=3-2] \arrow[phantom,from=3-1,to=4-2,"\pull \Rsym"]
	\arrow["p_2", from=4-2, to=5-2] \arrow[phantom,from=4-1,to=5-2,"\push \Rsym"]
	\arrow[equals, from=5-1, to=4-1]
	\arrow["a_W" description, from=5-1, to=5-2]
  \end{tikzcd}
  \qquad
  \alpha_e \quad\defeq \quad
  \begin{tikzcd}
	0 & V \\
	1 & E \\
	0 & W \\
	\arrow["a_V", from=1-1, to=1-2]
	\arrow["a_1"', from=1-1, to=2-1] \arrow[phantom,from=1-1,to=2-2,"\push \Lsym\pull \Rsym"]
	\arrow["e" description, from=2-1, to=2-2]
	\arrow["p_1"', from=2-2, to=1-2]
	\arrow["p_2", from=2-2, to=3-2]
	\arrow["a_1", from=3-1, to=2-1] \arrow[phantom,from=2-1,to=3-2,"\pull \Lsym\push \Rsym"]
	\arrow["a_W" description, from=3-1, to=3-2]
  \end{tikzcd}
  \qquad
  \alpha_w \quad\defeq \quad
  \begin{tikzcd}
	0 & V \\
	0 & E \\
	0 & W \\
	1 & W \\
	0 & W \\
	\arrow["a_V", from=1-1, to=1-2]
	\arrow[equals, from=1-1, to=2-1] \arrow[phantom,from=1-1,to=2-2,"\pull \Rsym"]
	\arrow["a_E" description, from=2-1, to=2-2]
	\arrow["p_1"', from=2-2, to=1-2]
	\arrow["p_2", from=2-2, to=3-2]
	\arrow[equals, from=3-1, to=2-1] \arrow[phantom,from=2-1,to=3-2,"\push \Rsym"]
	\arrow["a_W" description, from=3-1, to=3-2]
	\arrow["a_1"', from=3-1, to=4-1]
	\arrow["w" description, from=4-1, to=4-2]
	\arrow[equals, from=4-2, to=3-2] \arrow[phantom,from=3-1,to=4-2,"\push \Lsym"]
	\arrow[equals, from=4-2, to=5-2] \arrow[phantom,from=4-1,to=5-2,"\pull \Lsym"]
	\arrow["a_1", from=5-1, to=4-1]
	\arrow["a_W" description, from=5-1, to=5-2]
    \end{tikzcd}
\]
Observe that $\alpha_e$ sequentializes to $\alpha_v$ and $\alpha_w$, that is, $\alpha_v \oprewtoseq \alpha_e \rewtoseq \alpha_w$,
just in case $v = p_1(e)$ and $w = p_2(e)$.
\begin{thm}[cf.~Theorems~2 and 3 of Dawson, Paré, and Pronk~\cite{DPP03b}]
  \label{thm:derivation-equivalence-undecidable}
  There is a category $\C$ with locally finite factorizations for which the following equivalent problems are both undecidable:
  \begin{itemize}
  \item given two stacks of double cells in $\ZZ(\C)$, determine whether they vertically compose to the same double cell;
  \item given two proofs in the bifibrational calculus generated by $p = \Id[\C]$, determine whether they are permutation equivalent.
  \end{itemize}
\end{thm}
\begin{proof}
  We take $\C = \C_{G_M}$ where $G_M$ is a bipartite version of the configuration graph of some universal Turing machine $M$.
  Without loss of generality we will assume that $M$ is deterministic and has a unique halting configuration $\tau_\text{halt}$.
  Let us write $\Conf{M}$ for the configuration space of $M$, that is, the set of tuples $(q, u, v)$ of a state $q$ and finite strings $u$ and $v$ over the working alphabet denoting the contents of the tape to the left and right of the head.
  We define $G_M$ by taking $V = \Conf{M} \times \set{0}$ and $W = \Conf{M} \times \set{1}$ to be two disjoint copies of the configuration space, and $E \subseteq V\times W$ as the union
\begin{align*}
  E &= \set{((\sigma,0), (\tau,1)) \mid \sigma \to \tau\text{ is a valid transition of }M} \\
   & \cup \set{((\sigma,0), (\sigma,1)) \mid \sigma \in \Conf{M}}
\end{align*}
By the preceding observation, it is clear that any directed chain of configurations
\[
\sigma_1 \to \sigma_2 \to \sigma_3 \to \dots
\]
induces a corresponding chain of sequentialization equivalences:
\[
\alpha_{(\sigma_1,0)} \permeqseq \alpha_{(\sigma_2,1)}\permeqseq \alpha_{(\sigma_2,0)}\permeqseq \alpha_{(\sigma_3,1)}\permeqseq \alpha_{(\sigma_3,0)} \permeqseq \cdots
\]
Conversely, the assumption that $M$ is deterministic and has a unique halting state ensures that any chain of sequentialization equivalences $\alpha_{(\sigma,0)}\permeqseq\cdots\permeqseq \alpha_{(\tau_\text{halt},1)}$ corresponds to a halting execution $\sigma \to^* \tau_\text{halt}$.
We have thus reduced the halting problem for $M$ to deciding sequentialization equivalence $\alpha_{(\sigma,0)} \permeqseq \alpha_{(\tau_\text{halt},1)}$, or equivalently permutation equivalence $\alpha_{(\sigma,0)} \permeq \alpha_{(\tau_\text{halt},1)}$.
\end{proof}

Dawson, Paré, and Pronk proved decidability of (fence equivalence for) the $\Pi_2(\C)$ construction separately in the cases that $\C$ is locally finite or factorization preordered.
We now prove analogues of these two results for deciding equality of morphisms in $\Bifib{p}$, or more precisely for deciding whether two proofs of the same judgment represent the same morphism.
The first is already a consequence of the results of Section~\ref{sec:sequent-calculus}, while the second is a corollary of our normal form theorem.

\begin{thm}[cf.~Propositions 2 and 3 of Dawson, Paré, and Pronk~\cite{DPP03b}]
  \label{thm:derivation-equivalence-decidable}
  Let $p:\D\to\C$ and suppose that $\C$ satisfies either of the following conditions:
  \begin{enumerate}
  \item $\C$ is locally finite, \underline{or}
  \item $\C$ is factorization preordered;
  \end{enumerate}
  then for any two proofs $\alpha_1,\alpha_2 : S \vdashf{f} T$ of the bifibrational calculus generated by $p$, it is decidable whether $\alpha_1 \permeq \alpha_2$.
\end{thm}
\begin{proof}
  \begin{enumerate}
  \item Two proofs are equivalent just in case there is some sequence of permutations transforming one to the other.
    Any such permutation leaves invariant the length of the proof, the choice of terminating axiom $\delta : X \to Y$ in $\D$, and the selection of left and right rules applied, merely changing the order of the rules and the intermediate base arrows that annotate the derivation.
    Since $\C$ is locally finite, there are therefore only finitely many proofs permutation equivalent to any given proof, and we can reduce $\alpha_1 \permeq \alpha_2$ to a finite test $\alpha_2 \in [\alpha_1]$.
    \item This is a corollary of our normal formal theorem, Theorem~\ref{thm:fp-unique-normal-forms}, which states that each multifocused proof has a unique maximal normal form, under the \FP condition.
We can thus decide whether two proofs denote the same morphism of $\Bifib{p}$ simply by computing their normal forms and checking that these are equal.
  \end{enumerate}
\end{proof}
\noindent
\begin{rem}
  Technically, we need $\C$ to satisfy computable versions of either local finiteness or factorization preorder to get a decision procedure.
  For instance, we need to be able to effectively enumerate the finite homset $\C(A,B)$ between any pair of objects.
  Or, for any given square, we need to be able to decide whether that square is empty or else compute a filler.
  We also need the equality of morphisms in $\C$ and $\D$ to be decidable to check that two proofs are identical. 
  We omit these precisions from the statement of the theorem, since it is clear from context that the conditions should be interpreted constructively.
\end{rem}

Besides entailing decidability of the word problem, the normal form theorem has the equally important application of allowing us to \emph{enumerate} homsets of free bifibrations without duplicates, assuming the \FP condition.
We are most interested in the situation that $\Bifib{p}$ is relatively locally finite in the following sense.

We say that a category $\D$ equipped with a functor $p : \D \to \C$ is \defin{relatively locally finite} (relative to $p$) if the \emph{relative homsets} over $f$
\[
  \D_f(X,Y) \defeq \set{\delta \mid \delta : X \to Y, p(\delta) = f}
\]
are finite, for every arrow $f : A \to B$ of $\C$ and pair of objects $X$ and $Y$ of $\D$ lying above $A$ and $B$ respectively. (Thus a category $\C$ is locally finite iff the unique functor $!_C : \C \to \One$ is relatively locally finite.)
\begin{thm}
  \label{thm:relative-local-finiteness}
  Let $p : \D \to \C$ and suppose $\C$ satisfies the \FP condition.
  If $\D$ is relatively locally finite (relative to $p$) and $\C$ has locally finite slices and coslices, then $\Bifib{p}$ is relatively locally finite (relative to $\BifibFun{p}$).
  In particular, the fiber category $\Bifib{p}_A$ of every object $A \in \C$ is locally finite.
\end{thm}
\begin{proof}
  Suppose given an arrow $f : A \to B$ of $\C$ and a pair of bifibrational formulas $S\refs A$ and $T \refs B$.
  By Corollary~\ref{cor:FP=>maxmultfoc-canonical}, we can enumerate canonical representatives of arrows $\alpha \in \Bifib{p}_f(S,T)$ using the sequent calculus of Figure~\ref{fig:maxmultfocseq} to search for maximally multifocused proofs of $S \vdashl{f}{\ilockNone} T$.
  There are only three potential sources of nondeterministism:
\begin{enumerate}
  \item the choice of whether to initiate a $\biL$, $\biR$, or $\biLR$ bipole when proving a neutral sequent;
  \item the choice of an arrow $g$ in the premise of a left-/right-/bi-focus rule;
  \item the choice of an axiom $\delta : X \to Y \in \D$ to terminate the proof.
\end{enumerate}
But each of these choice points are finite branching:
\begin{enumerate}
\item at most three rules to try (some may be forbidden by the lock state $\fstate$ or if one side of the sequent is atomic);
\item finite by the assumption that $\C$ has locally finite slice and coslice categories (for example, to derive a sequent of the form $\pull{\sigma}P \vdashl{h}{q} Q$ using the $\Lpull[\sigma]$ rule, we must pick an arrow $g$ such that $h=\sigma g$; observe that such an arrow corresponds to a morphism $\sigma \to h$ in the coslice category $A/\C$, where $A = \dom \sigma = \dom h$);
\item finite by the assumption that $\D$ is relatively locally finite.
\end{enumerate}
We conclude that there are only finitely many maximally multifocused proofs of any given judgment, and hence that the relative homsets are finite.
In particular, for every object $A\in \C$, the homset of the fiber category $\Bifib{p}_A$, corresponding to the relative homset over $\Id[A]$, is finite.
\end{proof}

\begin{cor}
  Under the assumptions of Theorem~\ref{thm:relative-local-finiteness}, logical equivalence of bifibrational formulas is decidable.
\end{cor}
\begin{proof}
  Recall that by definition, two formulas $S_1,S_2 \refs A$ are logically equivalent if there exist derivations $\alpha : S_1\vdashf{\Id[A]} S_2$ and $\beta : S_2 \vdashf{\Id[A]} S_1$ whose compositions are permutation equivalent to identities (Defn.~\ref{defi:logical-equivalence}), or equivalently just in case they are isomorphic in the fiber category $\Bifib{p}_A$ (Prop.~\ref{prop:equivalence-is-vertical-iso}).
  It suffices to enumerate canonical representatives of all such derivations $\alpha$ and $\beta$ using Theorem~\ref{thm:relative-local-finiteness}, compute their cuts $\alpha\hcomp\beta$ and $\beta\hcomp\alpha$, and decide whether $\alpha\hcomp\beta \permeq \Id[S_1]$ and $\beta\hcomp\alpha \permeq \Id[S_2]$ using Theorem~\ref{thm:derivation-equivalence-decidable}.
\end{proof}
\noindent
As already remarked, free categories satisfy the \FP condition because their factorization categories are linear orders.
Moreover, although free categories $\Free\G$ are not locally finite unless the underlying graph $\G$ is acyclic, their slices and coslices are posetal and hence trivially locally finite.
Thus the enumeration and decision procedures described above are applicable (and even become slightly simpler since the sequent calculus of Figure~\ref{fig:maxmultfocseq} can be simplified a bit, as remarked at the end of Section~\ref{subsec:maximal-multi-focusing}).

Finally, Theorem~\ref{thm:relative-local-finiteness} can clearly be generalized to derive any infinite cardinality bound on the relative homsets of $\Bifib{p}$, assuming the same bound on the relative homsets of $\D$ and the homsets of the slices and coslices of $\C$.
We can also use the sequent calculus to effectively enumerate the (potentially infinite) relative homsets of $\Bifib{p}$, given procedures for doing the same for $\D$ and for the slices and coslices of $\C$.

\section{Free $(\P,\N)$-fibrations}
\label{sec:PN-fibrations}

As previously mentioned, Dawson, Paré, and Pronk's $\Pi_2(\C)$ construction may be seen as a weaker version of the free groupoid construction, which turns a category into a groupoid by freely adding an inverse to every arrow.
More often, one only wants to add formal inverses to a class of arrows $\W$, a process known as \emph{localization} (see, e.g.,~Borceaux~\cite[Ch.5]{BorceauxHCA1}), which freely builds a category $\C[\W^{-1}]$ equipped with a functor $\C \to \C[\W^{-1}]$ sending every arrow $f \in \W$ to an isomorphism.
DPP~\cite[\S4.1]{DPP03} discuss how to adapt their construction of $\Pi_2(\C)$ to define a 2-category $\C[\W^*]$ in which every arrow $f \in \W$ is freely equipped with a right adjoint.
In this section we quickly explain how our construction of the free bifibration on a functor $p : \D \to \C$ may be similarly adapted when the operations of pushing and pulling are restricted to certain classes of arrows of $\C$.
\begin{defi}
  Let $\P \subset \C$ and $\N \subset \C$ be two subsets of arrows (not necessarily subcategories) of $\C$.
  We say that a functor $\pi : \E \to \C$ is a \defin{$(\P,\N)$-fibration} if it has $+$-cartesian liftings of all arrows in $\P$ and $-$-cartesian liftings of all arrows in $\N$.
\end{defi}
\noindent
As special cases, a $(\emptyset,\C)$-fibration is a Grothendieck fibration in the ordinary sense, a $(\C,\emptyset)$-fibration is an opfibration, and a $(\C,\C)$-fibration is a bifibration.

A case of particular interest is when $(\P,\N)$ constitute a factorization system for the base category.
Such $(\P,\N)$-fibrations relative to a factorization system have been termed \emph{ambifibrations} by Joachim Kock and André Joyal.
One interesting property of ambifibrations, noted by Kock~\cite{Kock2010}, is that they give rise to ternary factorization systems in which every arrow of the total category factors as a $+$-cartesian arrow, followed by an arrow lying over an identity, followed by a $-$-cartesian arrow:
\[
  \begin{tikzcd}
    \E \ar[dd,"\pi"] \\\\
    \C
  \end{tikzcd}
  \qquad
  \begin{tikzcd}[row sep=2ex]
    S \ar[rrr,"\alpha"]\ar[rd,"\opcart e S"'] &&& T \\
  &\push{e}S \ar[r,"\uLpush {e} {} {\alpha} \uRpull {} {m}"'] & \pull{m}T\ar[ru,"\cart m T"']& \\\\
   A \ar[rrr,"f"]\ar[two heads,rd,"e\in\P"'] &&& B \\
   &\cdot\ar[r,equals] & \cdot \ar[tail,ru,"m\in\N"']&
  \end{tikzcd}
\]
We will consider an example of a free ambifibration in Section~\ref{sec:examples:ambi}.

Unless otherwise stated we do not assume anything about $\P$ and $\N$ in this section.
We write $\BifibFun{p,\P,\N} : \Bifib{p,\P,\N} \to \C$ for the free $(\P,\N)$-fibration on a functor $p : \D \to \C$.

\paragraph{Sequent calculus}

The presentation of the free bifibration that we gave via the sequent calculus in Section~\ref{sec:sequent-calculus} adapts directly to a presentation of the free $(\P,\N)$-bifibration.
For this one simply has to restrict the formation rules on bifibrational formulas:
\begin{mathpar}
  \inferrule {S \refs A \\ f : A \to B \in \P}
             {\push{f}S \refs B}

  \inferrule {g : B \to C\in \N \\ T \refs C}
             {\pull{g}T \refs B}
\end{mathpar}
Otherwise, everything remains the same: the definition of the inference rules, permutation equivalence, cut-elimination, et cetera.
This presentation makes clear that $\Bifib{p,\P,\N}$ is a full subcategory of $\Bifib{p}$.

\paragraph{The zigzag double category}
In Section~\ref{sec:double-categories} we defined the double category of zigzags $\ZZ(\C)$ by first constructing the free bifibration on the identity functor $\Id[\C]$.
We can similarly define a \emph{double category of restricted zigzags} $\ZZ(\C,\P,\N)$, whose vertical arrows are zigzags moving positively along arrows in $\P$ and negatively along arrows in $\N$.
In particular, we can recover DPP's $\C[\W^*]$ in this way as the vertical 2-category of $\ZZ(\C,\C,\W)$.

\paragraph{Weak focusing}
The syntax of weakly focused derivations given in Section~\ref{sec:focusing:weak-focusing} works just as well for $(\P,\N)$-fibrations, although the proof of correctness based on the equivalence $\Bifib{p} \simeq \saBifib{p}$ has to be slightly adapted.
The equivalence with strictly alternating formulas relied on the pseudofunctoriality laws $\push{(g\circ f)}S \equiv \push g \push f S$ and $\pull {(f \; g)} T \equiv \pull f \pull g T$, whereas $\P$ and $\N$ are not necessarily closed under composition.
However, Theorem~\ref{thm:strictify-equivalence} can be directly adapted to prove the following.
\begin{prop}\label{prop:PNsaBifib}
  If $\P$ and $\N$ are closed under binary composition, then $\Bifib{p,\P,\N} \simeq \saBifib{p,\P,\N}$.
\end{prop}
\noindent
The restrictions on $\P$ and $\N$ are not problematic, because we can always close them under finite compositions and obtain an equivalent free $(\P,\N)$-fibration.
Given a subset of arrows $\W \subset \C$, let us write $\W^*$ for the subcategory of $\C$ generated by adding identity arows $\Id[A]$ and $\Id[B]$ for every $f : A \to B \in \W$, and by closing under binary composition.
The pseudofunctoriality laws immediately imply the following.
\begin{prop}\label{prop:freePNstar}
  $\Bifib{p,\P,\N} \simeq \Bifib{p,\P^*,\N^*}$.
\end{prop}
\begin{cor}
  If $\C = \Free\G/{\sim}$ is freely generated by some graph $\G$ modulo some equations, then $\Bifib{p} \simeq \Bifib{p,\G,\G}$.
\end{cor}

\paragraph{Maximal multifocusing}

Finally, the results of Sections~\ref{sec:focusing:normal-forms}--\ref{subsec:maxmultfoc-decidability} provided a strong notion of normal form for free bifibrations as well as decidability results in the case that the base category is factorization preordered (Definition~\ref{def:FP-condition}).
One benefit of considering free $(\P,\N)$-fibrations is that we can relax this requirement to the a priori weaker condition that $\C$ is both $\P$-factorization preordered and $\N$-factorization preordered.

\begin{defi}\label{def:WFP-condition}
  Let $\W \subset \C$ be a subset of arrows of a category $\C$.
  We say that $\C$ is \defin{$\W$-factorization preordered} (or $\W$-\FP) if for any diagram of the form
\[
\begin{tikzcd}
  \cdot
  \ar[r, "f\in\W"]
  &
  \cdot
  \ar[r, "g", yshift=1ex]
  \ar[r, "h"', yshift=-1ex]
  &
  \cdot
  \ar[r, "k\in\W"]
  &
  \cdot
\end{tikzcd}
\]
if both $fg = fh$ and $gk = hk$ then necessarily $g = h$.
Equivalently, $\C$ is factorization preordered just in case every commuting square of the following form has at most one diagonal filler:
\[
\begin{tikzcd}[column sep=3em, row sep=3em]
  \cdot\ar[r,"x"]\ar[d,"f\in\W"']  & \cdot\ar[d,"k\in\W"] \\
  \cdot\ar[r, "y"']\ar[ur,dashed,"g=h" description]  & \cdot
\end{tikzcd}
\]
\end{defi}
\begin{prop}Every category $\C$ is both $\C_{\mathsf{epi}}$-\FP and $\C_{\mathsf{mono}}$-\FP, where $\C_{\mathsf{epi}}$ and $\C_{\mathsf{mono}}$ consist respectively of the epic arrows and the monic arrows of $\C$.
\end{prop}
\noindent
The analysis of Sections~\ref{sec:focusing:normal-forms}--\ref{subsec:maxmultfoc-decidability} carries over and we obtain the following.
\begin{thm}
\label{thm:PNFP=>maxmultfoc-canonical+}
If $\C$ is both $\P$-\FP and $\N$-\FP, then arrows of $\Bifib{p,\P,\N}$ are canonically represented by maximally multifocused derivations, and permutation equivalence is decidable.
If moreover $\D$ is relatively locally finite and $\C$ has locally finite slices and coslices, then $\Bifib{p,\P,\N}$ is relatively locally finite and logical equivalence is decidable.
\end{thm}

\begin{rem}
  As mentioned, fibrations and opfibrations arise as particular cases of $(\P,\N)$-fibration, taking $\P = \emptyset, \N = \C$ for fibrations and $\P=\C, \N=\emptyset$ for opfibrations.
  Even if $\C$ is not factorization preordered, strongly focused derivations do provide normal forms for the arrows of $\Bifib{p,\emptyset,\C}$ or $\Bifib{p,\C,\emptyset}$, for the simple reason that there are no $\biLR$-bipoles and so sequentialization equivalence reduces to discrete equality.
  The usual constructions of the free fibration and free opfibration as comma categories may be recovered in this way if one restricts to strictly alternating formulas of depth 1, avoiding formulas of depth 0 so as to obtain a split fibration (cf.~Remark~\ref{rem:not-quite-split-bifibration}).
\end{rem}

\section{The combinatorics of free bifibrations and ambifibrations}
\label{sec:examples}

We conclude with a deeper analysis of the specific examples of free bifibrations and ambifibrations that were mentioned in the Introduction.
In all of these examples, some rich combinatorial structure emerges out of very simple generating data, and the exploration of this structure is facilitated by the bifibrational calculus.

\subsection{Reconstructing the simplex category}
\label{sec:examples:simplices}

Recall that we write $\Delta$ for the category of finite ordinals and order-preserving maps, which is sometimes referred to as the \emph{augmented} simplex category to emphasize that it includes the empty ordinal $\ord{0}$, but which we will simply call the simplex category.
We write $\Delta_\bot$ for the category of non-empty finite ordinals and order-and-least-element preserving maps.
We write $L : \Delta \to \Delta_\bot$ for the functor sending an ordinal $\ord{n}$ to a non-empty ordinal $\ord{n}'$ by adding a $\bot$ element, and an order-preserving map $\theta : \ord{m} \to \ord{n}$ to the order-and-least-element-preserving map $\theta' : \ord{m}' \to \ord{n}'$ defined by $\theta'(\bot) = \bot$ and $\theta'(x) = \theta(x)$ for $x > \bot$.
This functor has a right adjoint given by the forgetful functor $R : \Delta_\bot \to \Delta$, which interprets $\ord{n}'$ as $\ord{1{+}n}$ and forgets that a map is $\bot$-preserving.

\begin{prop}\label{prop:simplex-category-free-bifibration}
  Let $F : \Two \to \Adj$ and $p_\Two : \One \to \Two$ be the functors shown below:
  \[ F0 = \Delta \quad F1 = \Delta_\bot \qquad Ff =
    \begin{tikzcd}[column sep=large]
    \Delta \ar[r,"L"{name=0},bend left] & \Delta_\bot \ar[l,"R"{name=1},bend left]
    \ar[from=0,to=1,phantom,"\bot"]
    \end{tikzcd}.
  \]
 \[
\begin{tikzcd}
  \One \ar[d,"p_\Two"'] & \ast & \\
  \Two  & 0 \ar[r,"f"] & 1
\end{tikzcd}
\]
Then the unique morphism of bifibrations $\Bifib{p_\Two} \to \int_\Two F$ sending the atomic formula $\atom\ast$ to the empty ordinal $\ord{0}$ is an equivalence of categories.
In particular, $\Delta$ is equivalent to $\Bifib{p_\Two}_0$.
\end{prop}
\noindent
Proposition~\ref{prop:simplex-category-free-bifibration} is a corollary of Theorem~\ref{thm:tree-category-free-bifibration} below, so we do not give a detailed proof here, but we sketch the correspondence.
Under the equivalence, each ordinal $\ord{n}$ in $\Delta$ is interpreted as the bifibrational formula
\[
\enc{\ord{n}} \defeq \underbrace{\pull{f}\push{f}\dots\pull{f}\push{f}}_{n\text{ times}}\ast \refs 0
\]
while each non-empty ordinal $\ord{n}'$ in $\Delta_\bot$ is interpreted as the bifibrational formula
\[
\enc{\ord{n}'} \defeq \push{f}\underbrace{\pull{f}\push{f}\dots\pull{f}\push{f}}_{n\text{ times}}\ast \refs 1\ .
\]
Conversely, since $\ast$ is the only atom and $f$ is the only non-identity arrow, every bifibrational formula on $p_\Two$ is equivalent to one of these two forms.
To be precise, we note that pushing or pulling along an identity arrow $\Id[0]$ or $\Id[1]$ is also possible, but logically equivalent to doing nothing (Prop.~\ref{prop:pseudofunctoriality})---hence why we only state an equivalence $\int_\Two F \equiv \Bifib{p_\Two}$ rather than an isomorphism of categories.
Clearly we obtain an isomorphism if we restrict to the full subcategory $\Bifib{p_\Two,f,f} \subset \Bifib{p_\Two}$.

Observe that ordinal formulas $\enc{\ord{n}}$ and $\enc{\ord{n}'}$ are strictly alternating.
The sequent calculus of Figure~\ref{fig:maxmultfocseq} becomes particularly simple when restricted to such ordinal formulas: since there are no non-trivial factorizations in the base category, the bi-focusing and bi-inversion rules can be removed (they are impossible to apply), as can the lock conditions on the left- and right-focusing rules (they are always validated).
Thus any morphism in $\Bifib{p_\Two}$ may be represented canonically as a strongly monofocused derivation, which in turn is essentially just a shuffle of $\biL$ and $\biR$ bipoles.

As an illustration, the three order-preserving maps $\ord{2} \to \ord{2}$ depicted on the left correspond to the following three derivations of $\pull{f} \push{f}\pull{f}\push{f}\ast \vdashf{\Id[0]} \pull{f} \push{f}\pull{f}\push{f}\ast$ depicted on the right (drawn as stacks with overlaid string diagrams):

\[
\begin{array}{c}
\begin{tikzcd}[arrows={-},ampersand replacement=\&]
  \bullet \ar[r] \&  \bullet \\
  \bullet \ar[ru] \& \bullet
\end{tikzcd} \\[2cm]
\begin{tikzcd}[arrows={-},ampersand replacement=\&]
  \bullet \ar[r] \& \bullet \\
  \bullet \ar[r] \& \bullet
\end{tikzcd} \\[2cm]
\begin{tikzcd}[arrows={-},ampersand replacement=\&]
  \bullet \ar[rd] \& \bullet \\
  \bullet \ar[r] \& \bullet
\end{tikzcd}
\end{array} \hspace{3cm}
\scalebox{0.8}{\begin{tikzcd}[ampersand replacement=\&,row sep=0.75cm,execute at end picture = { 
\draw[draw=violet, draw opacity=0.5, line width=1mm,out=90,in=180] ($(l1)!0.50!(r1)$) to ($(r1)!0.50!(r0)$);
\draw[draw=violet, draw opacity=0.5, line width=1mm,out=0,in=-90] ($(l2)!0.50!(l1)$) to ($(l1)!0.50!(r1)$);
\draw[draw=violet, draw opacity=0.5, line width=1mm,out=90,in=0] ($(l3)!0.50!(r3)$) to ($(l3)!0.50!(l2)$);
\draw[draw=violet, draw opacity=0.5, line width=1mm,out=0,in=-90] ($(l4)!0.50!(l3)$) to ($(l3)!0.50!(r3)$);
\draw[draw=violet, draw opacity=0.5, line width=1mm,out=90,in=0] ($(l5)!0.50!(r5)$) to ($(l5)!0.50!(l4)$);
\draw[draw=violet, draw opacity=0.5, line width=1mm,out=180,in=-90] ($(r6)!0.50!(r5)$) to ($(l5)!0.50!(r5)$);
\draw[draw=violet, draw opacity=0.5, line width=1mm,out=90,in=180] ($(l7)!0.50!(r7)$) to ($(r7)!0.50!(r6)$);
\draw[draw=violet, draw opacity=0.5, line width=1mm,out=180,in=-90] ($(r8)!0.50!(r7)$) to ($(l7)!0.50!(r7)$);
}
 ]
|[alias=l0]| 0 \ar[r,equals] \& |[alias=r0]| 0 
\\ |[alias=l1]|0\ar[u,equals]\ar[r] \& |[alias=r1]|1\ar[u,<-]
\\ |[alias=l2]|1\ar[u,<-]\ar[r,equals] \& |[alias=r2]|1\ar[u,equals]
\\ |[alias=l3]|0\ar[u]\ar[r] \& |[alias=r3]|1\ar[u,equals]
\\ |[alias=l4]|1\ar[u,<-]\ar[r,equals] \& |[alias=r4]|1\ar[u,equals]
\\ |[alias=l5]|0\ar[u]\ar[r] \& |[alias=r5]|1\ar[u,equals]
\\ |[alias=l6]|0\ar[u,equals]\ar[r,equals] \& |[alias=r6]|0\ar[u]
\\ |[alias=l7]|0\ar[u,equals]\ar[r] \& |[alias=r7]|1\ar[u,<-]
\\ |[alias=l8]|0\ar[u,equals]\ar[r,equals] \& |[alias=r8]|0\ar[u]
\arrow[gray,phantom,from=1-1,to=2-2,"\push \Rsym"]
\arrow[gray,phantom,from=2-1,to=3-2,"\push \Lsym"]
\arrow[gray,phantom,from=3-1,to=4-2,"\pull \Lsym"]
\arrow[gray,phantom,from=4-1,to=5-2,"\push \Lsym"]
\arrow[gray,phantom,from=5-1,to=6-2,"\pull \Lsym"]
\arrow[gray,phantom,from=6-1,to=7-2,"\pull \Rsym"]
\arrow[gray,phantom,from=7-1,to=8-2,"\push \Rsym"]
\arrow[gray,phantom,from=8-1,to=9-2,"\pull \Rsym"]
\end{tikzcd}
}
\qquad\qquad
\scalebox{0.8}{\begin{tikzcd}[ampersand replacement=\&,row sep=0.75cm,execute at end picture = { 
\draw[draw=violet, draw opacity=0.5, line width=1mm,out=90,in=180] ($(l1)!0.50!(r1)$) to ($(r1)!0.50!(r0)$);
\draw[draw=violet, draw opacity=0.5, line width=1mm,out=0,in=-90] ($(l2)!0.50!(l1)$) to ($(l1)!0.50!(r1)$);
\draw[draw=violet, draw opacity=0.5, line width=1mm,out=90,in=0] ($(l3)!0.50!(r3)$) to ($(l3)!0.50!(l2)$);
\draw[draw=violet, draw opacity=0.5, line width=1mm,out=180,in=-90] ($(r4)!0.50!(r3)$) to ($(l3)!0.50!(r3)$);
\draw[draw=violet, draw opacity=0.5, line width=1mm,out=90,in=180] ($(l5)!0.50!(r5)$) to ($(r5)!0.50!(r4)$);
\draw[draw=violet, draw opacity=0.5, line width=1mm,out=0,in=-90] ($(l6)!0.50!(l5)$) to ($(l5)!0.50!(r5)$);
\draw[draw=violet, draw opacity=0.5, line width=1mm,out=90,in=0] ($(l7)!0.50!(r7)$) to ($(l7)!0.50!(l6)$);
\draw[draw=violet, draw opacity=0.5, line width=1mm,out=180,in=-90] ($(r8)!0.50!(r7)$) to ($(l7)!0.50!(r7)$);
}
 ]
|[alias=l0]| 0 \ar[r,equals] \& |[alias=r0]| 0 
\\ |[alias=l1]|0\ar[u,equals]\ar[r] \& |[alias=r1]|1\ar[u,<-]
\\ |[alias=l2]|1\ar[u,<-]\ar[r,equals] \& |[alias=r2]|1\ar[u,equals]
\\ |[alias=l3]|0\ar[u]\ar[r] \& |[alias=r3]|1\ar[u,equals]
\\ |[alias=l4]|0\ar[u,equals]\ar[r,equals] \& |[alias=r4]|0\ar[u]
\\ |[alias=l5]|0\ar[u,equals]\ar[r] \& |[alias=r5]|1\ar[u,<-]
\\ |[alias=l6]|1\ar[u,<-]\ar[r,equals] \& |[alias=r6]|1\ar[u,equals]
\\ |[alias=l7]|0\ar[u]\ar[r] \& |[alias=r7]|1\ar[u,equals]
\\ |[alias=l8]|0\ar[u,equals]\ar[r,equals] \& |[alias=r8]|0\ar[u]
\arrow[gray,phantom,from=1-1,to=2-2,"\push \Rsym"]
\arrow[gray,phantom,from=2-1,to=3-2,"\push \Lsym"]
\arrow[gray,phantom,from=3-1,to=4-2,"\pull \Lsym"]
\arrow[gray,phantom,from=4-1,to=5-2,"\pull \Rsym"]
\arrow[gray,phantom,from=5-1,to=6-2,"\push \Rsym"]
\arrow[gray,phantom,from=6-1,to=7-2,"\push \Lsym"]
\arrow[gray,phantom,from=7-1,to=8-2,"\pull \Lsym"]
\arrow[gray,phantom,from=8-1,to=9-2,"\pull \Rsym"]
\end{tikzcd}
}
\qquad\qquad
\scalebox{0.8}{\begin{tikzcd}[ampersand replacement=\&,row sep=0.75cm,execute at end picture = { 
\draw[draw=violet, draw opacity=0.5, line width=1mm,out=90,in=180] ($(l1)!0.50!(r1)$) to ($(r1)!0.50!(r0)$);
\draw[draw=violet, draw opacity=0.5, line width=1mm,out=180,in=-90] ($(r2)!0.50!(r1)$) to ($(l1)!0.50!(r1)$);
\draw[draw=violet, draw opacity=0.5, line width=1mm,out=90,in=180] ($(l3)!0.50!(r3)$) to ($(r3)!0.50!(r2)$);
\draw[draw=violet, draw opacity=0.5, line width=1mm,out=0,in=-90] ($(l4)!0.50!(l3)$) to ($(l3)!0.50!(r3)$);
\draw[draw=violet, draw opacity=0.5, line width=1mm,out=90,in=0] ($(l5)!0.50!(r5)$) to ($(l5)!0.50!(l4)$);
\draw[draw=violet, draw opacity=0.5, line width=1mm,out=0,in=-90] ($(l6)!0.50!(l5)$) to ($(l5)!0.50!(r5)$);
\draw[draw=violet, draw opacity=0.5, line width=1mm,out=90,in=0] ($(l7)!0.50!(r7)$) to ($(l7)!0.50!(l6)$);
\draw[draw=violet, draw opacity=0.5, line width=1mm,out=180,in=-90] ($(r8)!0.50!(r7)$) to ($(l7)!0.50!(r7)$);
}
 ]
|[alias=l0]| 0 \ar[r,equals] \& |[alias=r0]| 0 
\\ |[alias=l1]|0\ar[u,equals]\ar[r] \& |[alias=r1]|1\ar[u,<-]
\\ |[alias=l2]|0\ar[u,equals]\ar[r,equals] \& |[alias=r2]|0\ar[u]
\\ |[alias=l3]|0\ar[u,equals]\ar[r] \& |[alias=r3]|1\ar[u,<-]
\\ |[alias=l4]|1\ar[u,<-]\ar[r,equals] \& |[alias=r4]|1\ar[u,equals]
\\ |[alias=l5]|0\ar[u]\ar[r] \& |[alias=r5]|1\ar[u,equals]
\\ |[alias=l6]|1\ar[u,<-]\ar[r,equals] \& |[alias=r6]|1\ar[u,equals]
\\ |[alias=l7]|0\ar[u]\ar[r] \& |[alias=r7]|1\ar[u,equals]
\\ |[alias=l8]|0\ar[u,equals]\ar[r,equals] \& |[alias=r8]|0\ar[u]
\arrow[gray,phantom,from=1-1,to=2-2,"\push \Rsym"]
\arrow[gray,phantom,from=2-1,to=3-2,"\pull \Rsym"]
\arrow[gray,phantom,from=3-1,to=4-2,"\push \Rsym"]
\arrow[gray,phantom,from=4-1,to=5-2,"\push \Lsym"]
\arrow[gray,phantom,from=5-1,to=6-2,"\pull \Lsym"]
\arrow[gray,phantom,from=6-1,to=7-2,"\push \Lsym"]
\arrow[gray,phantom,from=7-1,to=8-2,"\pull \Lsym"]
\arrow[gray,phantom,from=8-1,to=9-2,"\pull \Rsym"]
\end{tikzcd}
}
\]
The graphical correspondence should be suggestive.
In general, any morphism $\enc{\ord{m}} \to \enc{\ord{n}}$ in $\Bifib{p_\Two}_0$ can be canonically represented as a derivation that opens (reading from the bottom) with a $\pull{\Rsym}$ rule, followed by an arbitrary shuffle of $m$ $\biL$ bipoles with $n-1$ $\biR$ bipoles, followed by a closing $\push{\Rsym}$ rule.
Since there are $\binom{m+n-1}{m}$ such shuffles, we thus derive the well-known formula for the number of order-preserving maps.
\begin{cor}
There are $\frac{(m+n-1)!}{m!(n-1)!}$ order-preserving maps $\ord{m} \to \ord{n}$.
\end{cor}

\subsection{Categories of trees and walks}
\label{sec:examples:trees}

A natural generalization of the previous example is obtained by considering the free bifibration generated by the functor $p_\omega$ illustrated below which maps the point to the initial object of $\omega$, the totally ordered set of natural numbers considered as a category.
\[
\begin{tikzcd}
  \One \ar[d,"p_\omega"'] & \ast & \\
  \omega & 0 \ar[r,"f_0"] & 1 \ar[r,"f_1"] & 2 \ar[r,"f_2"] & \cdots
\end{tikzcd}
\]

As we observed in the introduction, up to isomorphism the objects of $\Bifib{p_\omega}_0$ may be read as Dyck words, which are in bijection with rooted plane trees, that is, trees embedded in the plane with one vertex marked as the root, considered up to planar isotopy.
We draw such trees growing upwards from the root, as depicted on the left:
\begin{mathpar}
  \begin{tikzcd}[column sep=1em, arrows={-}]
    \bullet\ar[d]    &    &  & & & \\
   \bullet\ar[dr] && \bullet\ar[dl] & & &\bullet\ar[d] \\
   &\bullet\ar[drr] && \bullet\ar[d] & &\bullet\ar[dll] \\
     &&      & \bullet & &
  \end{tikzcd}

  \begin{tikzcd}[column sep=1em, arrows={-}]
    \bullet\ar[d]    &    &  & & & \\
        \bullet\ar[dr] && \bullet\ar[dl] & & &
                \bullet\ar[d]
                \ar[d,<-,xshift=4pt,shorten=2pt,purple]
                \ar[d,->,xshift=-4pt,shorten=2pt,forestgreen] \\
        &\bullet\ar[drr] && \bullet\ar[d] & &\bullet
          \ar[dll]
            \ar[dll,<-,xshift=4pt,yshift=-4pt,shorten=2pt,blue]
            \ar[dll,->,xshift=-4pt,yshift=4pt,shorten=2pt,orange] \\
     &&      & \bullet & &
  \end{tikzcd}
\end{mathpar}
To encode a plane tree by a Dyck word, a canonical procedure is to traverse the tree (say) from right to left starting from the root, recording a $+1$ every time you move up towards a child, and a $-1$ every time you move down back to a parent.
We represent the beginning of this traversal, visually, on the right. (We will use these same colors below to visualize the traversal order.)

The corresponding bifibrational formula may be read off directly:
\[\pull{f}\pull{f}\pull{f}\push{f}\push{f}\pull{f}\push{f}\push{f}\pull{f}\push{f}
  \pull{\textcolor{orange}{f}}
  \pull{\textcolor{forestgreen}{f}}
  \push{\textcolor{purple}{f}}
  \push{\textcolor{blue}{f}}
  \atom\ast\]
Here to keep the notation simple we have elided the indices from the $f_i$'s, which are uniquely determined.
Such a word may also be interpreted as a walk on the upper plane, starting at the $x$-axis (at the point marked $\ast$ below), constantly moving up or down while moving leftwards and eventually terminating back at the $x$-axis:
\[
\begin{tikzpicture}
  \draw[very thin,color=gray] (-0.25,-0.25) grid (14.25,3.25);
  \draw [->] plot coordinates {(14,0) (13,1) (12,2) (11,1) (10,0) (9,1) (8,0) (7,1) (6,2) (5,1) (4,2) (3,3) (2,2) (1,1) (0,0)};
  \draw [-,blue,thick] (14, 0) edge (13, 1);
  \draw [-,purple,thick] (13, 1) edge (12, 2);
  \draw [-,forestgreen,thick] (12, 2) edge (11, 1);
  \draw [-,orange,thick] (11, 1) edge (10, 0);
  \node at (14,0) {$\ast$};
\end{tikzpicture}
\]
Observe that all of the information about the walk is contained in its peaks and valleys, and a tree may also therefore be represented more concisely as a strictly alternating zigzag:
\[
\begin{tikzcd} 0 \ar[r] & 3 \ar[r,<-] & 1\ar[r] & 2\ar[r,<-] & 0\ar[r] & 1\ar[r,<-]
  & |[alias=A]| 0 \ar[r, from=A, to=B, phantom, ""{name=AB}]
      \ar[from=A,to=AB.center, orange]
      \ar[from=AB.center, to=B, forestgreen]
  & |[alias=B]| 2 \ar[r, from=B, to=C, phantom, ""{name=BC}]
      \ar[from=BC.center,to=B, purple]
      \ar[from=C, to=BC.center, blue]
  & |[alias=C]| 0
\end{tikzcd}
\]
or again as a formula in strictly alternating syntax: $\pull[3]{f}\push[2]{f}\pull[1]{f}\push[2]{f}\pull[1]{f}\push[1]{f}\pull[2]{f}\push[2]{f}\ast$.

Next we recall how plane trees admit an alternative representation as functors.
(This idea goes back at least to Joyal~\cite{Joyal1997}, but may be earlier folklore.)
Define the \emph{height} of a vertex as its distance from the root, and the height of a tree as the maximum height of any of its vertices.
Let $\ord{[k]} = \ord{k{+}1}$ be the linear order $\set{ 0 < \dots < k}$ seen as a posetal category.
It is clear from the definitions that any rooted plane tree of height $k$ may be represented by a functor $T : \ord{[k]}^\op \to \Delta$ such that $T(0) = \ord{1}$: each linearly ordered set $T(i)$ defines the vertices of height $i$, and the parents of the non-root vertices are specified by the functions $T(i+1) \to T(i)$, which must be order-preserving by the constraints of planarity.
For instance, the height-3 tree given in the example above corresponds to the functor $T : \ord{[3]} \to \Delta$ depicted below.
\[
  \begin{tikzcd}[column sep=1em, arrows={-}]
T(3) \ar[d,->] \rlap{$ = \ord 1  = \set{0}$} &\hspace*{3cm}&   0\ar[d]     &    &  & & &  \\
T(2) \ar[d,->] \rlap{$ = \ord 3 = \set{0 < 1 < 2}$} &&   0\ar[dr] && 1\ar[dl] & & &2\ar[d] \\
T(1) \ar[d,->] \rlap{$ = \ord 3 = \set{0 < 1 < 2}$} &&    &0\ar[drr] && 1\ar[d] & &2\ar[dll]\\
T(0) \rlap{$ = \ord 1 = \set{0}$}  &&     &&      & 0 & & 
  \end{tikzcd}
\]
This tree also has \emph{width} 3, defined as the maximum number of vertices at any given height.
More generally, a tree of unbounded height and width may be represented as a functor $T : \omega^\op \to \Delta$ such that $T(0) = \ord{1}$.
In other words, the only requirement is that $T$ sends the initial object of $\omega$ to the terminal object of $\Delta$.
Finiteness is ensured if we moreover ask that there is some natural number $k$ such that $T(k+1) = \ord 0$ (and hence $T(j) = \ord 0$ for all $j > k$).

We define $\PTree$ as the full subcategory of $[\omega^\op,\Delta]$ spanned by the finite rooted plane trees, that is, the category whose objects are functors $T : \omega^\op \to \Delta$ such that $T(0) = \ord{1}$ and $T(k+1) = \ord 0$ for some $k$, and all natural transformations between them.
For example, below we depict two different morphisms of plane trees $S \to T$ between the same pair of trees, using different color lines to indicate the mappings $S(k) \to T(k)$ of the natural transformations at each height $k$. 
\begin{equation}\label{eq:tree-morphisms}
\begin{tikzcd}[column sep=1em, arrows={-}]
S(3) = \ord 0\ar[d,->]\ar[r,->] & T(3) = \ord 1\ar[d,->] &\qquad
 &&&\qquad&
 0\ar[d] & & &          \\
S(2) = \ord 3\ar[d,->]\ar[r,->] & T(2) = \ord 3\ar[d,->] &
  0\ar[d]\ar[rrrr,bend left=15,line width=3pt,opacity=0.5,frenchblue]\ar[rrrrrr,bend right=15,line width=2pt,opacity=0.5,persianred] & 1\ar[rd]\ar[rrrrrrrr,bend left=15,line width=3pt,opacity=0.5,frenchblue]\ar[rrrrr,bend right=15,line width=2pt,opacity=0.5,persianred] & 2\ar[d] \ar[rrrrrrr,bend left=15,line width=3pt,opacity=0.5,frenchblue]\ar[rrrr,bend right=15,line width=2pt,opacity=0.5,persianred] &&
  0\ar[dr] && 1\ar[dl] & & &2\ar[d] \\
S(1) = \ord 3\ar[d,->]\ar[r,->] & T(1) = \ord 3\ar[d,->] &
 0\ar[dr]\ar[rrrrr,bend left=15,line width=3pt,opacity=0.5,frenchblue]\ar[rrrrr,bend right=15,line width=2pt,opacity=0.5,persianred] & 1\ar[d]\ar[rrrr,bend left=15,line width=3pt,opacity=0.5,frenchblue]\ar[rrrr,bend right=15,line width=2pt,opacity=0.5,persianred]  & 2\ar[dl]\ar[rrrrrrr,bend left=15,line width=3pt,opacity=0.5,frenchblue]\ar[rrr,bend right=15,line width=2pt,opacity=0.5,persianred] &&
 & 0\ar[drr] && 1\ar[d] & &2\ar[dll]\\
S(0) = \ord 1\ar[r,->] & T(0) = \ord 1 &
  &0\ar[rrrrrr,bend left=15,line width=3pt,opacity=0.5,frenchblue]\ar[rrrrrr,bend right=8,line width=2pt,opacity=0.5,persianred] &&& 
  &&      & 0 & & &               
\end{tikzcd}
\end{equation}
Readers familiar with the ``topos of trees'' $[\omega^\op,\Set]$ may recognize that $\PTree$ is defined in essentially the same way, but replacing $\Set$ by $\Delta$ to reflect planarity, and taking a full subcategory of $[\omega^\op,\Delta]$ in order to restrict to finite trees rather than potentially infinite forests.
The combination of order-preservation and naturality imposes strong constraints on the shape of plane tree morphisms.
The reader may check that there are exactly 11 morphisms $S \to T$ between the trees $S$ and $T$ shown above, including the two in \eqref{eq:tree-morphisms}.

We should note that $\PTree$ is strongly related to but different from a category of plane trees defined by Joyal~\cite{Joyal1997}.
Joyal considered a more liberal notion of morphism $S \to T$ between plane trees encoded in the same way as functors $S,T : \omega^\op \to \Delta$.
Instead of considering natural transformations in $[\omega^\op,\Delta]$, he took more general natural transformations between the underlying functors in $[\omega^\op,\Set]$ so long as they respect the linear orders on the \emph{fibers} of the maps $S(n+1) \to S(n)$. (In other words, respect the ordering of siblings, but not necessarily of cousins.)
Every morphism in $\PTree$ is a morphism in Joyal's category of plane trees (i.e., it is a wide subcategory), but the converse is not true as the following example illustrates:
\begin{equation}\label{eq:joyal-morphism}
\begin{tikzcd}[column sep=1em, arrows={-}]
S(2) = \ord 2\ar[d,->]\ar[r,->] & T(2) = \ord 2\ar[d,->] &
  0\ar[d]\ar[rrrrrrr,bend left=15,line width=3pt,opacity=0.5,gray] && 1\ar[d]\ar[rrr,bend left=15,line width=3pt,opacity=0.5,gray] &&
  & 0\ar[dr] && 1\ar[dl] & & & \\
S(1) = \ord 2\ar[d,->]\ar[r,->] & T(1) = \ord 1\ar[d,->] &
 0\ar[dr]\ar[rrrrrr,bend left=15,line width=3pt,opacity=0.5,gray] & & 1\ar[dl]\ar[rrrr,bend left=15,line width=3pt,opacity=0.5,gray] &&
 & &0\ar[d] &&  & &\\
S(0) = \ord 1\ar[r,->] & T(0) = \ord 1 &
  &0\ar[rrrrr,bend left=15,line width=3pt,opacity=0.5,gray] &&& 
  &      & 0 & & &               
\end{tikzcd}
\end{equation}
Joyal used this more relaxed definition of tree morphism in order to axiomatize a notion of $\omega$-category, and his category of trees, which is equivalent to the category $\Omega$ studied by Batanin and Street~\cite{BataninStreet2000universal}, satisfies some important universal properties (see Theorems~1~and~2 of~\cite{BataninStreet2000universal}).
Despite being a non-full subcategory of the Joyal-Batanin-Street category of trees, the category of plane trees $\PTree$ defined above as a full subcategory of $[\omega^\op,\Delta]$ also seems to us very natural.

Let us now prove that $\PTree$ corresponds to the fiber over $0$ of the free bifibration generated by $p_\omega$.
In order to prove this we give a more general characterization of the fiber categories $\Bifib{p_\omega}_k$ for arbitrary $k \ge 0$.
We write $\PTree_{\bot(k)}$ for the category whose objects are rooted plane trees with a marked vertex of height $k$ \emph{on their leftmost branch,} and whose arrows are morphisms of plane trees preserving the mark.
It is clear that $\PTree \cong \PTree_{\bot(0)}$, since every tree morphism $S \to T$ sends the root of $S$ to the root of $T$, and also that every tree in $\PTree_{\bot(k)}$ must have height at least $k$.
But not every morphism between trees of height $\ge k$ is a morphism in $\PTree_{\bot(k)}$.
For example, one of the two tree morphisms $S \to T$ depicted in \eqref{eq:tree-morphisms} lifts to a morphism in $\PTree_{\bot(2)}$ (the one in blue), but the other (in red) only lifts to $\PTree_{\bot(1)}$.

Objects of $\Bifib{p_\omega}_k$ may be interpreted as left-marked trees: the idea is that a bifibrational formula in $S \refs k$, corresponding to a walk that starts at height 0 and ends at height $k$, may be read as describing a walk from right to left along the tree that stops at the marked node.
Under this correspondence, extending a walk by an up-step $k \to k+1$ corresponds to growing a fresh marked leaf as a left child of the marked node, while extending a walk by a down-step $k+1 \to k$ corresponds to moving the marked node down towards the root.
We call these respective operations $L_k$ and $R_k$, and illustrate them below on an example for the case $k=2$:
\[
\begin{tikzcd}[column sep=.5em, row sep=.75em,arrows={-}]
\circ\ar[dr] & \bullet\ar[d] & & &          \\
&  \bullet\ar[dr] && \bullet\ar[dl] & & &\bullet\ar[d] \\
& & \bullet \ar[drr] && \bullet\ar[d] & &\bullet\ar[dll]\\
&  &&      & \bullet & & &               
\end{tikzcd}
\quad\overset{L_2}\longmapsfrom\quad
\begin{tikzcd}[column sep=.5em, row sep=.75em,arrows={-}]
 \bullet\ar[d] & & &          \\
  \circ\ar[dr] && \bullet\ar[dl] & & &\bullet\ar[d] \\
 & \bullet \ar[drr] && \bullet\ar[d] & &\bullet\ar[dll]\\
  &&      & \bullet & & &               
\end{tikzcd}
\quad\overset{R_2}\longmapsto\quad
\begin{tikzcd}[column sep=.5em, row sep=.75em,arrows={-}]
 \bullet\ar[d] & & &          \\
  \bullet\ar[dr] && \bullet\ar[dl] & & &\bullet\ar[d] \\
 & \circ \ar[drr] && \bullet\ar[d] & &\bullet\ar[dll]\\
  &&      & \bullet & & &               
\end{tikzcd}
\]
It is easy to see that $L_k$ and $R_k$ extend to functors defining a family of adjunctions
\[
    \begin{tikzcd}[column sep=large]
    \PTree_{\bot(k)} \ar[r,"L_k"{name=0},bend left] & \PTree_{\bot(k+1)} \ar[l,"R_k"{name=1},bend left]
    \ar[from=0,to=1,phantom,"\bot"]
    \end{tikzcd}
\]
between the categories of left-marked plane trees.
\begin{thm}\label{thm:tree-category-free-bifibration}
  Let $F : \omega \to \Adj$ be the functor sending $k$ to $\PTree_{\bot(k)}$ and each arrow $f_k : k \to k+1$ to the adjunction $L_k \dashv R_k$ defined above.
  Let $p_\omega : \One \to \omega$ be the functor mapping $\atom\ast$ to 0.
  Then the unique morphism of bifibrations $\Bifib{p_\omega} \to \int_\omega F$ sending the atomic formula $\atom\ast$ to the one-vertex tree $\bullet$ is an equivalence of categories.
  In particular, $\PTree$ is equivalent to $\Bifib{p_\omega}_0$.
\end{thm}
\begin{proof}
  Since $\pi_F(\bullet) = 0 = p_\omega(\ast)$, the universal property of the free bifibration determines a unique morphism of bifibrations $\sem{-}{\bullet} : \Bifib{p_\omega} \to \int_\omega F$ mapping $\atom \ast$ to $\bullet$.

  On objects, $\sem{-}\bullet$ interprets a bifibrational formula $T$ as the left-marked tree $\sem{T}\bullet$ encoded by the corresponding walk. This mapping is surjective since every such tree is described by a unique walk.  (It is many-to-one, since many formulas isomorphic by the pseudofunctoriality laws represent the same walk.)
  
  As a morphism of bifibrations, the mapping on arrows is determined by its action on the homsets of the fiber categories.
We now give an explicit description of a family of bijections
\[ (\psi_k)_{S,T}: \Bifib{p_\omega}_k(S,T) \overset\sim\longrightarrow \PTree_{\bot(k)}(\sem{S}\bullet,\sem{T}\bullet)\]
mapping zigzag morphisms to plane tree morphisms.
Since these bijections are natural in $S$ and $T$, they uniquely determine $\sem{-}\bullet : \Bifib{p_\omega} \to \int_\omega F$ as a morphism of bifibrations.
We conclude that $\sem{-}\bullet$ is both surjective on objects and fully faithful, hence an equivalence.

To define $\psi_k$, we rely on an inductive characterization of plane tree morphisms that we show can be reflected in both $\PTree_{\bot(k)}$ and $\Bifib{p_\omega}_k$.
Let us begin by illustrating this bijection in the case $k=0$, before explaining how to generalize it to arbitrary $k\ge 0$.

\paragraph{Bijection for the case $k=0$}
A plane tree can be described inductively as a root node with a list of plane trees as children. In other words, a tree is built by taking a (possibly empty) forest of trees, and attaching the root of each tree in the forest to a new root vertex.
The following grammar summarizes this inductive structure, where $T$ ranges over trees and $\Phi$ over forests:
\begin{equation}\label{eq:tree-formulas}
  T ::= \bullet \Phi
  \qquad
  \Phi ::= \cdot \mid T,\Phi'
\end{equation}
Order-preserving morphisms of plane trees can similarly be defined inductively.
We find it convenient to introduce two auxiliary judgments $\Phi \to T$ and $\Phi' \to \Phi$, describing forest-to-tree morphisms and forest-to-forest morphisms, respectively; we then consider tree morphisms as a special case of forest-to-tree morphisms with a singleton forest as input.
The two kinds of morphisms can be built inductively using the following rules:
\begin{equation}\label{eq:tree-morphism-rules}
    \infer{\bullet\Phi'_1,\dots,\bullet\Phi'_i \to \bullet\Phi}{\Phi'_1,\dots,\Phi'_i \to \Phi}
    \qquad
    \infer{\Phi'_1,\dots,\Phi'_n \to T_1,\dots,T_n}{\Phi'_1 \to T_1 & \cdots & \Phi'_n \to T_n}
\end{equation}
The forest-to-tree rule (on the left) expresses that the root nodes of the input forest are necessarily mapped to the unique root of the output tree, so it suffices to provide a forest-to-forest morphism at height 1.
The forest-to-forest rule (on the right) expresses that a morphism into a forest with $n$ trees is determined by an ordered partition of the input forest into $n$ subforests together with a list of forest-to-tree morphisms.

It is clear that morphisms in $\PTree$ may be constructed inductively in this way---to see this, it is helpful to begin by observing that there is an isomorphism 
\[
  \begin{tikzcd}
    \PTree \ar[r,bend left,"T(-) \mapsto T(1+-)"]\ar[r,phantom,"\sim"] & \PForest \ar[l,bend left,"\Phi \mapsto\bullet\Phi"]
  \end{tikzcd}
\]
where $\PForest$ is the category of (finite rooted) plane forests, that is, the full subcategory of $[\omega^\op,\Delta]$ spanned by the functors $\Phi : \omega^\op \to \Delta$ that eventually vanish $\Phi(k) = \ord{0}$ (with no constraint on $\Phi(0)$).
Indeed, any natural transformation $\theta : S \to T$ in $\PTree$ is uniquely determined at height 0 (since $S(0) = T(0) = \ord{1}$) and induces a natural transformation $\theta': \Phi_S \to \Phi_T$ in $\PForest$ between the underlying forests, where $S = \bullet\Phi_S$ and $T = \bullet\Phi_T$.
The height 0 component of such a natural transformation $\theta'_0 : \Phi_S(0) \to \Phi_T(0)$ determines an ordered partition of the roots of $\Phi_S$ into $n$ components, where $\Phi_T = T_1,\dots,T_n$.
By naturality of $\theta'$, this extends to an ordered partition of the forest itself $\Phi_S = \Phi'_1,\dots,\Phi'_n$, together with a list of natural transformations $\theta_1 : \Phi'_1 \to T_1, \dots, \theta_n : \Phi'_n \to T_n$.

Let us now explain how to similarly decompose morphisms in $\Bifib{p_\omega}_0$, illustrating with an example.
Consider the maximally multifocused proof below, which we will see corresponds to the blue morphism in \eqref{eq:tree-morphisms}:
\[
\rotatebox{-90}{\resizebox{!}{7cm}{\begin{tikzcd}[ampersand replacement=\&,row sep=1cm,execute at end picture = { 
\draw[draw=frenchblue, draw opacity=0.5, line width=1mm,out=90,in=180] ($(l1)!0.67!(r1)$) to ($(r1)!0.33!(r0)$);
\draw[draw=frenchblue, draw opacity=0.5, line width=1mm,out=90,in=180] ($(l1)!0.33!(r1)$) to ($(r1)!0.67!(r0)$);
\draw[draw=frenchblue, draw opacity=0.5, line width=1mm,out=0,in=-90] ($(l2)!0.67!(l1)$) to ($(l1)!0.33!(r1)$);
\draw[draw=frenchblue, draw opacity=0.5, line width=1mm,out=0,in=-90] ($(l2)!0.33!(l1)$) to ($(l1)!0.67!(r1)$);
\draw[draw=frenchblue, draw opacity=0.5, line width=1mm,out=90,in=0] ($(l3)!0.50!(r3)$) to ($(l3)!0.50!(l2)$);
\draw[draw=frenchblue, draw opacity=0.5, line width=1mm,out=0,in=-90] ($(l4)!0.50!(l3)$) to ($(l3)!0.50!(r3)$);
\draw[draw=frenchblue, draw opacity=0.5, line width=1mm,out=90,in=0] ($(l5)!0.33!(r5)$) to ($(l5)!0.33!(l4)$);
\draw[draw=frenchblue, draw opacity=0.5, line width=1mm,out=90,in=0] ($(l5)!0.67!(r5)$) to ($(l5)!0.67!(l4)$);
\draw[draw=frenchblue, draw opacity=0.5, line width=1mm,out=180,in=-90] ($(r6)!0.67!(r5)$) to ($(l5)!0.67!(r5)$);
\draw[draw=frenchblue, draw opacity=0.5, line width=1mm,out=180,in=-90] ($(r6)!0.33!(r5)$) to ($(l5)!0.33!(r5)$);
\draw[draw=frenchblue, draw opacity=0.5, line width=1mm,out=90,in=180] ($(l7)!0.50!(r7)$) to ($(r7)!0.50!(r6)$);
\draw[draw=frenchblue, draw opacity=0.5, line width=1mm,out=180,in=-90] ($(r8)!0.50!(r7)$) to ($(l7)!0.50!(r7)$);
\draw[draw=frenchblue, draw opacity=0.5, line width=1mm,out=90,in=180] ($(l9)!0.67!(r9)$) to ($(r9)!0.33!(r8)$);
\draw[draw=frenchblue, draw opacity=0.5, line width=1mm,out=90,in=180] ($(l9)!0.33!(r9)$) to ($(r9)!0.67!(r8)$);
\draw[draw=frenchblue, draw opacity=0.5, line width=1mm,out=0,in=-90] ($(l10)!0.50!(l9)$) to ($(l9)!0.33!(r9)$);
\draw[draw=frenchblue, draw opacity=0.5, line width=1mm,out=180,in=-90] ($(r10)!0.50!(r9)$) to ($(l9)!0.67!(r9)$);
\draw[draw=frenchblue, draw opacity=0.5, line width=1mm,out=90,in=0] ($(l11)!0.25!(r11)$) to ($(l11)!0.50!(l10)$);
\draw[draw=frenchblue, draw opacity=0.5, line width=1mm,out=90,in=180] ($(l11)!0.75!(r11)$) to ($(r11)!0.33!(r10)$);
\draw[draw=frenchblue, draw opacity=0.5, line width=1mm,out=90,in=180] ($(l11)!0.50!(r11)$) to ($(r11)!0.67!(r10)$);
\draw[draw=frenchblue, draw opacity=0.5, line width=1mm,out=0,in=-90] ($(l12)!0.67!(l11)$) to ($(l11)!0.25!(r11)$);
\draw[draw=frenchblue, draw opacity=0.5, line width=1mm,out=0,in=-90] ($(l12)!0.33!(l11)$) to ($(l11)!0.50!(r11)$);
\draw[draw=frenchblue, draw opacity=0.5, line width=1mm,out=90,in=-90] ($(l12)!0.50!(r12)$) to ($(l11)!0.75!(r11)$);
\draw[draw=frenchblue, draw opacity=0.5, line width=1mm,out=90,in=0] ($(l13)!0.25!(r13)$) to ($(l13)!0.33!(l12)$);
\draw[draw=frenchblue, draw opacity=0.5, line width=1mm,out=90,in=0] ($(l13)!0.50!(r13)$) to ($(l13)!0.67!(l12)$);
\draw[draw=frenchblue, draw opacity=0.5, line width=1mm,out=90,in=-90] ($(l13)!0.75!(r13)$) to ($(l12)!0.50!(r12)$);
\draw[draw=frenchblue, draw opacity=0.5, line width=1mm,out=180,in=-90] ($(r14)!0.75!(r13)$) to ($(l13)!0.75!(r13)$);
\draw[draw=frenchblue, draw opacity=0.5, line width=1mm,out=180,in=-90] ($(r14)!0.50!(r13)$) to ($(l13)!0.50!(r13)$);
\draw[draw=frenchblue, draw opacity=0.5, line width=1mm,out=180,in=-90] ($(r14)!0.25!(r13)$) to ($(l13)!0.25!(r13)$);
}
 ]
|[alias=l0]| 0 \ar[r,equals] \& |[alias=r0]| 0 
\\ |[alias=l1]|0\ar[u,equals]\ar[r] \& |[alias=r1]|2\ar[u,<-]
\\ |[alias=l2]|2\ar[u,<-]\ar[r,equals] \& |[alias=r2]|2\ar[u,equals]
\\ |[alias=l3]|1\ar[u]\ar[r] \& |[alias=r3]|2\ar[u,equals]
\\ |[alias=l4]|2\ar[u,<-]\ar[r,equals] \& |[alias=r4]|2\ar[u,equals]
\\ |[alias=l5]|0\ar[u]\ar[r] \& |[alias=r5]|2\ar[u,equals]
\\ |[alias=l6]|0\ar[u,equals]\ar[r,equals] \& |[alias=r6]|0\ar[u]
\\ |[alias=l7]|0\ar[u,equals]\ar[r] \& |[alias=r7]|1\ar[u,<-]
\\ |[alias=l8]|0\ar[u,equals]\ar[r,equals] \& |[alias=r8]|0\ar[u]
\\ |[alias=l9]|0\ar[u,equals]\ar[r] \& |[alias=r9]|2\ar[u,<-]
\\ |[alias=l10]|1\ar[u,<-]\ar[r,equals] \& |[alias=r10]|1\ar[u]
\\ |[alias=l11]|0\ar[u]\ar[r] \& |[alias=r11]|3\ar[u,<-]
\\ |[alias=l12]|2\ar[u,<-]\ar[r] \& |[alias=r12]|3\ar[u,equals]
\\ |[alias=l13]|0\ar[u]\ar[r] \& |[alias=r13]|3\ar[u,equals]
\\ |[alias=l14]|0\ar[u,equals]\ar[r,equals] \& |[alias=r14]|0\ar[u]
\arrow[gray,phantom,from=1-1,to=2-2,"\push \Rsym"]
\arrow[gray,phantom,from=2-1,to=3-2,"\push \Lsym"]
\arrow[gray,phantom,from=3-1,to=4-2,"\pull \Lsym"]
\arrow[gray,phantom,from=4-1,to=5-2,"\push \Lsym"]
\arrow[gray,phantom,from=5-1,to=6-2,"\pull \Lsym"]
\arrow[gray,phantom,from=6-1,to=7-2,"\pull \Rsym"]
\arrow[gray,phantom,from=7-1,to=8-2,"\push \Rsym"]
\arrow[gray,phantom,from=8-1,to=9-2,"\pull \Rsym"]
\arrow[gray,phantom,from=9-1,to=10-2,"\push \Rsym"]
\arrow[gray,phantom,from=10-1,to=11-2,"\push \Lsym\pull \Rsym"]
\arrow[gray,phantom,from=11-1,to=12-2,"\pull \Lsym \push \Rsym"]
\arrow[gray,phantom,from=12-1,to=13-2,"\push \Lsym"]
\arrow[gray,phantom,from=13-1,to=14-2,"\pull \Lsym"]
\arrow[gray,phantom,from=14-1,to=15-2,"\pull \Rsym"]
\end{tikzcd}}}
\]
We have rotated the proof 90 degrees clockwise to better fit on the page. It should be read as a morphism $S \to T$ from the zigzag $S$ at the top to the zigzag $T$ at the bottom.
Remember that each zigzag describes a walk along the contour of a tree.

The idea is to begin by splitting the proof along \begin{tikzcd}[column sep=1em]0\ar[r,equals] & 0\end{tikzcd} boundaries:
\[
\rotatebox{-90}{\resizebox{!}{7cm}{\begin{tikzcd}[ampersand replacement=\&,row sep=1cm,execute at end picture = { 
\draw[draw=frenchblue, draw opacity=0.5, line width=1mm,out=90,in=180] ($(l1)!0.67!(r1)$) to ($(r1)!0.33!(r0)$);
\draw[draw=frenchblue, draw opacity=0.5, line width=1mm,out=90,in=180] ($(l1)!0.33!(r1)$) to ($(r1)!0.67!(r0)$);
\draw[draw=frenchblue, draw opacity=0.5, line width=1mm,out=0,in=-90] ($(l2)!0.67!(l1)$) to ($(l1)!0.33!(r1)$);
\draw[draw=frenchblue, draw opacity=0.5, line width=1mm,out=0,in=-90] ($(l2)!0.33!(l1)$) to ($(l1)!0.67!(r1)$);
\draw[draw=frenchblue, draw opacity=0.5, line width=1mm,out=90,in=0] ($(l3)!0.50!(r3)$) to ($(l3)!0.50!(l2)$);
\draw[draw=frenchblue, draw opacity=0.5, line width=1mm,out=0,in=-90] ($(l4)!0.50!(l3)$) to ($(l3)!0.50!(r3)$);
\draw[draw=frenchblue, draw opacity=0.5, line width=1mm,out=90,in=0] ($(l5)!0.33!(r5)$) to ($(l5)!0.33!(l4)$);
\draw[draw=frenchblue, draw opacity=0.5, line width=1mm,out=90,in=0] ($(l5)!0.67!(r5)$) to ($(l5)!0.67!(l4)$);
\draw[draw=frenchblue, draw opacity=0.5, line width=1mm,out=180,in=-90] ($(r61)!0.67!(r5)$) to ($(l5)!0.67!(r5)$);
\draw[draw=frenchblue, draw opacity=0.5, line width=1mm,out=180,in=-90] ($(r61)!0.33!(r5)$) to ($(l5)!0.33!(r5)$);
\draw[draw=frenchblue, draw opacity=0.5, line width=1mm,out=90,in=180] ($(l7)!0.50!(r7)$) to ($(r7)!0.50!(r62)$);
\draw[draw=frenchblue, draw opacity=0.5, line width=1mm,out=180,in=-90] ($(r81)!0.50!(r7)$) to ($(l7)!0.50!(r7)$);
\draw[draw=frenchblue, draw opacity=0.5, line width=1mm,out=90,in=180] ($(l9)!0.67!(r9)$) to ($(r9)!0.33!(r82)$);
\draw[draw=frenchblue, draw opacity=0.5, line width=1mm,out=90,in=180] ($(l9)!0.33!(r9)$) to ($(r9)!0.67!(r82)$);
\draw[draw=frenchblue, draw opacity=0.5, line width=1mm,out=0,in=-90] ($(l10)!0.50!(l9)$) to ($(l9)!0.33!(r9)$);
\draw[draw=frenchblue, draw opacity=0.5, line width=1mm,out=180,in=-90] ($(r10)!0.50!(r9)$) to ($(l9)!0.67!(r9)$);
\draw[draw=frenchblue, draw opacity=0.5, line width=1mm,out=90,in=0] ($(l11)!0.25!(r11)$) to ($(l11)!0.50!(l10)$);
\draw[draw=frenchblue, draw opacity=0.5, line width=1mm,out=90,in=180] ($(l11)!0.75!(r11)$) to ($(r11)!0.33!(r10)$);
\draw[draw=frenchblue, draw opacity=0.5, line width=1mm,out=90,in=180] ($(l11)!0.50!(r11)$) to ($(r11)!0.67!(r10)$);
\draw[draw=frenchblue, draw opacity=0.5, line width=1mm,out=0,in=-90] ($(l12)!0.67!(l11)$) to ($(l11)!0.25!(r11)$);
\draw[draw=frenchblue, draw opacity=0.5, line width=1mm,out=0,in=-90] ($(l12)!0.33!(l11)$) to ($(l11)!0.50!(r11)$);
\draw[draw=frenchblue, draw opacity=0.5, line width=1mm,out=90,in=-90] ($(l12)!0.50!(r12)$) to ($(l11)!0.75!(r11)$);
\draw[draw=frenchblue, draw opacity=0.5, line width=1mm,out=90,in=0] ($(l13)!0.25!(r13)$) to ($(l13)!0.33!(l12)$);
\draw[draw=frenchblue, draw opacity=0.5, line width=1mm,out=90,in=0] ($(l13)!0.50!(r13)$) to ($(l13)!0.67!(l12)$);
\draw[draw=frenchblue, draw opacity=0.5, line width=1mm,out=90,in=-90] ($(l13)!0.75!(r13)$) to ($(l12)!0.50!(r12)$);
\draw[draw=frenchblue, draw opacity=0.5, line width=1mm,out=180,in=-90] ($(r14)!0.75!(r13)$) to ($(l13)!0.75!(r13)$);
\draw[draw=frenchblue, draw opacity=0.5, line width=1mm,out=180,in=-90] ($(r14)!0.50!(r13)$) to ($(l13)!0.50!(r13)$);
\draw[draw=frenchblue, draw opacity=0.5, line width=1mm,out=180,in=-90] ($(r14)!0.25!(r13)$) to ($(l13)!0.25!(r13)$);
}
 ]
|[alias=l0]| 0 \ar[r,equals] \& |[alias=r0]| 0 
\\ |[alias=l1]|0\ar[u,equals]\ar[r] \& |[alias=r1]|2\ar[u,<-]
\\ |[alias=l2]|2\ar[u,<-]\ar[r,equals] \& |[alias=r2]|2\ar[u,equals]
\\ |[alias=l3]|1\ar[u]\ar[r] \& |[alias=r3]|2\ar[u,equals]
\\ |[alias=l4]|2\ar[u,<-]\ar[r,equals] \& |[alias=r4]|2\ar[u,equals]
\\ |[alias=l5]|0\ar[u]\ar[r] \& |[alias=r5]|2\ar[u,equals]
\\ |[alias=l61]|0\ar[u,equals]\ar[r,equals] \& |[alias=r61]|0\ar[u]
\\ |[alias=l62]|0\ar[r,equals] \& |[alias=r62]|0
\\ |[alias=l7]|0\ar[u,equals]\ar[r] \& |[alias=r7]|1\ar[u,<-]
\\ |[alias=l81]|0\ar[u,equals]\ar[r,equals] \& |[alias=r81]|0\ar[u]
\\ |[alias=l82]|0\ar[r,equals] \& |[alias=r82]|0
\\ |[alias=l9]|0\ar[u,equals]\ar[r] \& |[alias=r9]|2\ar[u,<-]
\\ |[alias=l10]|1\ar[u,<-]\ar[r,equals] \& |[alias=r10]|1\ar[u]
\\ |[alias=l11]|0\ar[u]\ar[r] \& |[alias=r11]|3\ar[u,<-]
\\ |[alias=l12]|2\ar[u,<-]\ar[r] \& |[alias=r12]|3\ar[u,equals]
\\ |[alias=l13]|0\ar[u]\ar[r] \& |[alias=r13]|3\ar[u,equals]
\\ |[alias=l14]|0\ar[u,equals]\ar[r,equals] \& |[alias=r14]|0\ar[u]
\arrow[gray,phantom,from=1-1,to=2-2,"\push \Rsym"]
\arrow[gray,phantom,from=2-1,to=3-2,"\push \Lsym"]
\arrow[gray,phantom,from=3-1,to=4-2,"\pull \Lsym"]
\arrow[gray,phantom,from=4-1,to=5-2,"\push \Lsym"]
\arrow[gray,phantom,from=5-1,to=6-2,"\pull \Lsym"]
\arrow[gray,phantom,from=6-1,to=7-2,"\pull \Rsym"]
\arrow[gray,phantom,from=8-1,to=9-2,"\push \Rsym"]
\arrow[gray,phantom,from=9-1,to=10-2,"\pull \Rsym"]
\arrow[gray,phantom,from=11-1,to=12-2,"\push \Rsym"]
\arrow[gray,phantom,from=12-1,to=13-2,"\push \Lsym\pull \Rsym"]
\arrow[gray,phantom,from=13-1,to=14-2,"\pull \Lsym \push \Rsym"]
\arrow[gray,phantom,from=14-1,to=15-2,"\push \Lsym"]
\arrow[gray,phantom,from=15-1,to=16-2,"\pull \Lsym"]
\arrow[gray,phantom,from=16-1,to=17-2,"\pull \Rsym"]
\end{tikzcd}}}
\]
This splits the proof into exactly $n$ components, where $T = \bullet(T_1,\dots,T_n)$.
Next, in each of these components, we ``gray out'' the chords that are incident to 0:
\[
\rotatebox{-90}{\resizebox{!}{7cm}{\begin{tikzcd}[ampersand replacement=\&,row sep=1cm,execute at end picture = { 
\draw[draw=frenchblue, draw opacity=0.5, line width=1mm,out=90,in=180] ($(l1)!0.67!(r1)$) to ($(r1)!0.33!(r0)$);
\draw[draw=lightgray, draw opacity=0.5, line width=1mm,out=90,in=180] ($(l1)!0.33!(r1)$) to ($(r1)!0.67!(r0)$);
\draw[draw=lightgray, draw opacity=0.5, line width=1mm,out=0,in=-90] ($(l2)!0.67!(l1)$) to ($(l1)!0.33!(r1)$);
\draw[draw=frenchblue, draw opacity=0.5, line width=1mm,out=0,in=-90] ($(l2)!0.33!(l1)$) to ($(l1)!0.67!(r1)$);
\draw[draw=frenchblue, draw opacity=0.5, line width=1mm,out=90,in=0] ($(l3)!0.50!(r3)$) to ($(l3)!0.50!(l2)$);
\draw[draw=frenchblue, draw opacity=0.5, line width=1mm,out=0,in=-90] ($(l4)!0.50!(l3)$) to ($(l3)!0.50!(r3)$);
\draw[draw=lightgray, draw opacity=0.5, line width=1mm,out=90,in=0] ($(l5)!0.33!(r5)$) to ($(l5)!0.33!(l4)$);
\draw[draw=frenchblue, draw opacity=0.5, line width=1mm,out=90,in=0] ($(l5)!0.67!(r5)$) to ($(l5)!0.67!(l4)$);
\draw[draw=frenchblue, draw opacity=0.5, line width=1mm,out=180,in=-90] ($(r61)!0.67!(r5)$) to ($(l5)!0.67!(r5)$);
\draw[draw=lightgray, draw opacity=0.5, line width=1mm,out=180,in=-90] ($(r61)!0.33!(r5)$) to ($(l5)!0.33!(r5)$);
\draw[draw=lightgray, draw opacity=0.5, line width=1mm,out=90,in=180] ($(l7)!0.50!(r7)$) to ($(r7)!0.50!(r62)$);
\draw[draw=lightgray, draw opacity=0.5, line width=1mm,out=180,in=-90] ($(r81)!0.50!(r7)$) to ($(l7)!0.50!(r7)$);
\draw[draw=frenchblue, draw opacity=0.5, line width=1mm,out=90,in=180] ($(l9)!0.67!(r9)$) to ($(r9)!0.33!(r82)$);
\draw[draw=lightgray, draw opacity=0.5, line width=1mm,out=90,in=180] ($(l9)!0.33!(r9)$) to ($(r9)!0.67!(r82)$);
\draw[draw=lightgray, draw opacity=0.5, line width=1mm,out=0,in=-90] ($(l10)!0.50!(l9)$) to ($(l9)!0.33!(r9)$);
\draw[draw=frenchblue, draw opacity=0.5, line width=1mm,out=180,in=-90] ($(r10)!0.50!(r9)$) to ($(l9)!0.67!(r9)$);
\draw[draw=lightgray, draw opacity=0.5, line width=1mm,out=90,in=0] ($(l11)!0.25!(r11)$) to ($(l11)!0.50!(l10)$);
\draw[draw=frenchblue, draw opacity=0.5, line width=1mm,out=90,in=180] ($(l11)!0.75!(r11)$) to ($(r11)!0.33!(r10)$);
\draw[draw=frenchblue, draw opacity=0.5, line width=1mm,out=90,in=180] ($(l11)!0.50!(r11)$) to ($(r11)!0.67!(r10)$);
\draw[draw=lightgray, draw opacity=0.5, line width=1mm,out=0,in=-90] ($(l12)!0.67!(l11)$) to ($(l11)!0.25!(r11)$);
\draw[draw=frenchblue, draw opacity=0.5, line width=1mm,out=0,in=-90] ($(l12)!0.33!(l11)$) to ($(l11)!0.50!(r11)$);
\draw[draw=frenchblue, draw opacity=0.5, line width=1mm,out=90,in=-90] ($(l12)!0.50!(r12)$) to ($(l11)!0.75!(r11)$);
\draw[draw=lightgray, draw opacity=0.5, line width=1mm,out=90,in=0] ($(l13)!0.25!(r13)$) to ($(l13)!0.33!(l12)$);
\draw[draw=frenchblue, draw opacity=0.5, line width=1mm,out=90,in=0] ($(l13)!0.50!(r13)$) to ($(l13)!0.67!(l12)$);
\draw[draw=frenchblue, draw opacity=0.5, line width=1mm,out=90,in=-90] ($(l13)!0.75!(r13)$) to ($(l12)!0.50!(r12)$);
\draw[draw=frenchblue, draw opacity=0.5, line width=1mm,out=180,in=-90] ($(r14)!0.75!(r13)$) to ($(l13)!0.75!(r13)$);
\draw[draw=frenchblue, draw opacity=0.5, line width=1mm,out=180,in=-90] ($(r14)!0.50!(r13)$) to ($(l13)!0.50!(r13)$);
\draw[draw=lightgray, draw opacity=0.5, line width=1mm,out=180,in=-90] ($(r14)!0.25!(r13)$) to ($(l13)!0.25!(r13)$);
}
 ]
|[alias=l0]| 0 \ar[r,equals] \& |[alias=r0]| 0 
\\ |[alias=l1]|0\ar[u,equals]\ar[r] \& |[alias=r1]|2\ar[u,<-]
\\ |[alias=l2]|2\ar[u,<-]\ar[r,equals] \& |[alias=r2]|2\ar[u,equals]
\\ |[alias=l3]|1\ar[u]\ar[r] \& |[alias=r3]|2\ar[u,equals]
\\ |[alias=l4]|2\ar[u,<-]\ar[r,equals] \& |[alias=r4]|2\ar[u,equals]
\\ |[alias=l5]|0\ar[u]\ar[r] \& |[alias=r5]|2\ar[u,equals]
\\ |[alias=l61]|0\ar[u,equals]\ar[r,equals] \& |[alias=r61]|0\ar[u]
\\ |[alias=l62]|0\ar[r,equals] \& |[alias=r62]|0
\\ |[alias=l7]|0\ar[u,equals]\ar[r] \& |[alias=r7]|1\ar[u,<-]
\\ |[alias=l81]|0\ar[u,equals]\ar[r,equals] \& |[alias=r81]|0\ar[u]
\\ |[alias=l82]|0\ar[r,equals] \& |[alias=r82]|0
\\ |[alias=l9]|0\ar[u,equals]\ar[r] \& |[alias=r9]|2\ar[u,<-]
\\ |[alias=l10]|1\ar[u,<-]\ar[r,equals] \& |[alias=r10]|1\ar[u]
\\ |[alias=l11]|0\ar[u]\ar[r] \& |[alias=r11]|3\ar[u,<-]
\\ |[alias=l12]|2\ar[u,<-]\ar[r] \& |[alias=r12]|3\ar[u,equals]
\\ |[alias=l13]|0\ar[u]\ar[r] \& |[alias=r13]|3\ar[u,equals]
\\ |[alias=l14]|0\ar[u,equals]\ar[r,equals] \& |[alias=r14]|0\ar[u]
\arrow[gray,phantom,from=1-1,to=2-2,"\push \Rsym"]
\arrow[gray,phantom,from=2-1,to=3-2,"\push \Lsym"]
\arrow[gray,phantom,from=3-1,to=4-2,"\pull \Lsym"]
\arrow[gray,phantom,from=4-1,to=5-2,"\push \Lsym"]
\arrow[gray,phantom,from=5-1,to=6-2,"\pull \Lsym"]
\arrow[gray,phantom,from=6-1,to=7-2,"\pull \Rsym"]
\arrow[gray,phantom,from=8-1,to=9-2,"\push \Rsym"]
\arrow[gray,phantom,from=9-1,to=10-2,"\pull \Rsym"]
\arrow[gray,phantom,from=11-1,to=12-2,"\push \Rsym"]
\arrow[gray,phantom,from=12-1,to=13-2,"\push \Lsym\pull \Rsym"]
\arrow[gray,phantom,from=13-1,to=14-2,"\pull \Lsym \push \Rsym"]
\arrow[gray,phantom,from=14-1,to=15-2,"\push \Lsym"]
\arrow[gray,phantom,from=15-1,to=16-2,"\pull \Lsym"]
\arrow[gray,phantom,from=16-1,to=17-2,"\pull \Rsym"]
\end{tikzcd}}}
\]
Each component now implicitly describes a forest-to-tree morphism $\Phi_i \to T_i$, where $S = \bullet\Phi_S$ and $\Phi_1,\dots,\Phi_n$ is an ordered partition of $\Phi_S$.
If we remove the grayed-out chords and shift levels---in terms of the sequent calculus, this corresponds to removing the pair of $\pull\Rsym$ and $\push\Rsym$ steps bracketing each component, as well as a list of pairs of $\pull\Lsym$ and $\push\Lsym$ steps---then we can iterate the process on smaller proofs.
Eventually we obtain a full recursive decomposition of the original zigzag morphism $S \to T$ using the inductive rules in \eqref{eq:tree-morphism-rules}, which we can play back as a morphism $\sem{S}\bullet \to \sem{T}\bullet$ in $\PTree$.
This defines the bijection $\psi_k$ for the case $k=0$.

The bijection may be better visualized by imagining the string diagram for the zigzag morphism as a kind of topographic relief map:
\[
\rotatebox{-90}{\resizebox{!}{7cm}{\begin{tikzcd}[ampersand replacement=\&,row sep=1cm,execute at end picture = { 
\draw[draw=forestgreen, draw opacity=0.5, line width=1mm,out=90,in=180] ($(l1)!0.67!(r1)$) to ($(r1)!0.33!(r0)$);
\draw[draw=goldenpoppy, draw opacity=0.5, line width=1mm,out=90,in=180] ($(l1)!0.33!(r1)$) to ($(r1)!0.67!(r0)$);
\draw[draw=goldenpoppy, draw opacity=0.5, line width=1mm,out=0,in=-90] ($(l2)!0.67!(l1)$) to ($(l1)!0.33!(r1)$);
\draw[draw=forestgreen, draw opacity=0.5, line width=1mm,out=0,in=-90] ($(l2)!0.33!(l1)$) to ($(l1)!0.67!(r1)$);
\draw[draw=forestgreen, draw opacity=0.5, line width=1mm,out=90,in=0] ($(l3)!0.50!(r3)$) to ($(l3)!0.50!(l2)$);
\draw[draw=forestgreen, draw opacity=0.5, line width=1mm,out=0,in=-90] ($(l4)!0.50!(l3)$) to ($(l3)!0.50!(r3)$);
\draw[draw=goldenpoppy, draw opacity=0.5, line width=1mm,out=90,in=0] ($(l5)!0.33!(r5)$) to ($(l5)!0.33!(l4)$);
\draw[draw=forestgreen, draw opacity=0.5, line width=1mm,out=90,in=0] ($(l5)!0.67!(r5)$) to ($(l5)!0.67!(l4)$);
\draw[draw=forestgreen, draw opacity=0.5, line width=1mm,out=180,in=-90] ($(r6)!0.67!(r5)$) to ($(l5)!0.67!(r5)$);
\draw[draw=goldenpoppy, draw opacity=0.5, line width=1mm,out=180,in=-90] ($(r6)!0.33!(r5)$) to ($(l5)!0.33!(r5)$);
\draw[draw=goldenpoppy, draw opacity=0.5, line width=1mm,out=90,in=180] ($(l7)!0.50!(r7)$) to ($(r7)!0.50!(r6)$);
\draw[draw=goldenpoppy, draw opacity=0.5, line width=1mm,out=180,in=-90] ($(r8)!0.50!(r7)$) to ($(l7)!0.50!(r7)$);
\draw[draw=forestgreen, draw opacity=0.5, line width=1mm,out=90,in=180] ($(l9)!0.67!(r9)$) to ($(r9)!0.33!(r8)$);
\draw[draw=goldenpoppy, draw opacity=0.5, line width=1mm,out=90,in=180] ($(l9)!0.33!(r9)$) to ($(r9)!0.67!(r8)$);
\draw[draw=goldenpoppy, draw opacity=0.5, line width=1mm,out=0,in=-90] ($(l10)!0.50!(l9)$) to ($(l9)!0.33!(r9)$);
\draw[draw=forestgreen, draw opacity=0.5, line width=1mm,out=180,in=-90] ($(r10)!0.50!(r9)$) to ($(l9)!0.67!(r9)$);
\draw[draw=goldenpoppy, draw opacity=0.5, line width=1mm,out=90,in=0] ($(l11)!0.25!(r11)$) to ($(l11)!0.50!(l10)$);
\draw[draw=goldenbrown, draw opacity=0.5, line width=1mm,out=90,in=180] ($(l11)!0.75!(r11)$) to ($(r11)!0.33!(r10)$);
\draw[draw=forestgreen, draw opacity=0.5, line width=1mm,out=90,in=180] ($(l11)!0.50!(r11)$) to ($(r11)!0.67!(r10)$);
\draw[draw=goldenpoppy, draw opacity=0.5, line width=1mm,out=0,in=-90] ($(l12)!0.67!(l11)$) to ($(l11)!0.25!(r11)$);
\draw[draw=forestgreen, draw opacity=0.5, line width=1mm,out=0,in=-90] ($(l12)!0.33!(l11)$) to ($(l11)!0.50!(r11)$);
\draw[draw=goldenbrown, draw opacity=0.5, line width=1mm,out=90,in=-90] ($(l12)!0.50!(r12)$) to ($(l11)!0.75!(r11)$);
\draw[draw=goldenpoppy, draw opacity=0.5, line width=1mm,out=90,in=0] ($(l13)!0.25!(r13)$) to ($(l13)!0.33!(l12)$);
\draw[draw=forestgreen, draw opacity=0.5, line width=1mm,out=90,in=0] ($(l13)!0.50!(r13)$) to ($(l13)!0.67!(l12)$);
\draw[draw=goldenbrown, draw opacity=0.5, line width=1mm,out=90,in=-90] ($(l13)!0.75!(r13)$) to ($(l12)!0.50!(r12)$);
\draw[draw=goldenbrown, draw opacity=0.5, line width=1mm,out=180,in=-90] ($(r14)!0.75!(r13)$) to ($(l13)!0.75!(r13)$);
\draw[draw=forestgreen, draw opacity=0.5, line width=1mm,out=180,in=-90] ($(r14)!0.50!(r13)$) to ($(l13)!0.50!(r13)$);
\draw[draw=goldenpoppy, draw opacity=0.5, line width=1mm,out=180,in=-90] ($(r14)!0.25!(r13)$) to ($(l13)!0.25!(r13)$);
\begin{scope}[on background layer]
\fill[height0] (r14.center) to (l14.center) to ($(l13)!0.33!(l12)$) to[out=0,in=90] ($(l13)!0.25!(r13)$) to[out=-90,in=180] ($(r14)!0.25!(r13)$) -- cycle;
\fill[height0] ($(l12)!0.67!(l11)$) to[out=0,in=-90] ($(l11)!0.25!(r11)$) to[out=90,in=0] ($(l11)!0.5!(l10)$) -- cycle;
\fill[height0] ($(l10)!0.5!(l9)$) to[out=0,in=-90] ($(l9)!0.33!(r9)$) to[out=90,in=180] ($(r9)!0.67!(r8)$) to ($(r8)!0.5!(r7)$) to[out=180,in=-90] ($(l7)!0.5!(r7)$) to[out=90,in=0] ($(r7)!0.5!(r6)$) to ($(r6)!0.33!(r5)$) to[out=180,in=-90] ($(l5)!0.33!(r5)$) to[out=90,in=0] ($(l5)!0.33!(l4)$) -- cycle;
\fill[height0] (r0.center) to (l0.center) to ($(l2)!0.67!(l1)$) to[out=0,in=-90] ($(l1)!0.33!(r1)$) to[out=90,in=180] ($(r1)!0.67!(r0)$) -- cycle;
\fill[height1] ($(r14)!0.25!(r13)$) to[out=180,in=-90] ($(l13)!0.25!(r13)$) to[out=90,in=0] ($(l13)!0.33!(l12)$) to ($(l13)!0.67!(l12)$) to[out=0,in=90] ($(l13)!0.5!(r13)$) to[out=-90,in=180] ($(r14)!0.5!(r13)$) -- cycle;
\fill[height1] ($(r11)!0.67!(r10)$) to[out=180,in=90] ($(l11)!0.5!(r11)$) to ($(l11)!0.45!(r11)$) to [in=-90,out=0] ($(l12)!0.33!(l11)$) to ($(l12)!0.67!(l11)$) to [out=0,in=-90] ($(l11)!0.25!(r11)$) to [out=90,in=0] ($(l11)!0.5!(l10)$) to ($(l10)!0.5!(l9)$) to [out=0,in=-90] ($(l9)!0.33!(r9)$) to[out=90,in=-180] ($(r9)!0.67!(r8)$) to ($(r9)!0.33!(r8)$) to[out=180,in=90] ($(l9)!0.67!(r9)$) to [out=-90,in=180] ($(r10)!0.5!(r9)$) -- cycle;
\fill[height1] ($(r8)!0.5!(r7)$) to [out=180,in=-90] ($(l7)!0.5!(r7)$) to[out=90,in=180] ($(r7)!0.5!(r6)$) -- cycle;
\fill[height1] ($(r6)!0.33!(r5)$) to [out=180,in=-90] ($(l5)!0.33!(r5)$) to[out=90,in=0] ($(l5)!0.33!(l4)$) to ($(l5)!0.67!(l4)$) to [out=0,in=90] ($(l5)!0.67!(r5)$) to [out=-90,in=180] ($(r6)!0.67!(r5)$) -- cycle;
\fill[height1] ($(l4)!0.5!(l3)$) to [out=0,in=-90] ($(l3)!0.5!(r3)$) to[out=90,in=0] ($(l3)!0.5!(l2)$) -- cycle;
\fill[height1] ($(r1)!0.33!(r0)$) to [out=180,in=90] ($(l1)!0.67!(r1)$) to[out=-90,in=0] ($(l2)!0.33!(l1)$) to ($(l2)!0.67!(l1)$) to [out=0,in=-90] ($(l1)!0.33!(r1)$) to [out=90,in=180] ($(r1)!0.67!(r0)$) -- cycle;
\fill[height2] ($(r14)!0.5!(r13)$) to[out=180,in=-90] ($(l13)!0.5!(r13)$) to[out=90,in=0] ($(l13)!0.67!(l12)$) to ($(l12)!0.33!(l11)$) to[out=0,in=-90] ($(l11)!0.5!(r11)$) to[out=90,in=-180] ($(r11)!0.67!(r10)$) to ($(r11)!0.33!(r10)$) to[out=180,in=90] ($(l11)!0.75!(r11)$) to[out=-90,in=90] ($(l12)!0.5!(r12)$) to[out=-90,in=90] ($(l13)!0.75!(r13)$) to [out=180,in=-90] ($(r14)!0.75!(r13)$) -- cycle;
\fill[height2] ($(r10)!0.5!(r9)$) to [out=180,in=-90] ($(l9)!0.67!(r9)$) to[out=90,in=180] ($(r9)!0.33!(r8)$) -- cycle;
\fill[height2] ($(r6)!0.67!(r5)$) to [out=180,in=-90] ($(l5)!0.67!(r5)$) to[out=90,in=0] ($(l5)!0.67!(l4)$) to ($(l4)!0.5!(l3)$) to [out=0,in=-90] ($(l3)!0.5!(r3)$) to [out=90,in=0] ($(l3)!0.5!(l2)$) to ($(l2)!0.33!(l1)$) to [out=0,in=-90] ($(l1)!0.67!(r1)$) to [out=90,in=180] ($(r1)!0.33!(r0)$) -- cycle;
\fill[height3] ($(r14)!0.75!(r13)$) to[out=180,in=-90] ($(l13)!0.75!(r13)$) to[out=90,in=-90] ($(l12)!0.5!(r12)$) to[out=90,in=-90] ($(l11)!0.75!(r11)$) to [out=90,in=-180] ($(r10)!0.67!(r11)$) -- cycle;
\end{scope}
}
 ]
|[alias=l0]| 0 \ar[r,equals]
\& |[alias=r0]| 0 
\\ |[alias=l1]|0\ar[u,equals]
\& |[alias=r1]|2\ar[u,<-]
\\ |[alias=l2]|2\ar[u,<-]
\& |[alias=r2]|2\ar[u,equals]
\\ |[alias=l3]|1\ar[u]
\& |[alias=r3]|2\ar[u,equals]
\\ |[alias=l4]|2\ar[u,<-]
\& |[alias=r4]|2\ar[u,equals]
\\ |[alias=l5]|0\ar[u]
\& |[alias=r5]|2\ar[u,equals]
\\ |[alias=l6]|0\ar[u,equals]
\& |[alias=r6]|0\ar[u]
\\ |[alias=l7]|0\ar[u,equals]
\& |[alias=r7]|1\ar[u,<-]
\\ |[alias=l8]|0\ar[u,equals]
\& |[alias=r8]|0\ar[u]
\\ |[alias=l9]|0\ar[u,equals]
\& |[alias=r9]|2\ar[u,<-]
\\ |[alias=l10]|1\ar[u,<-]
\& |[alias=r10]|1\ar[u]
\\ |[alias=l11]|0\ar[u]
\& |[alias=r11]|3\ar[u,<-]
\\ |[alias=l12]|2\ar[u,<-]
\& |[alias=r12]|3\ar[u,equals]
\\ |[alias=l13]|0\ar[u]
\& |[alias=r13]|3\ar[u,equals]
\\ |[alias=l14]|0\ar[u,equals]\ar[r,equals] \& |[alias=r14]|0\ar[u]
\end{tikzcd}}}
\]
The source and target trees are rooted at the North and South ends of this map, with the morphism exhibiting a continuous deformation between them.
\[\raisebox{-.5\height}{\includegraphics[width=9cm]{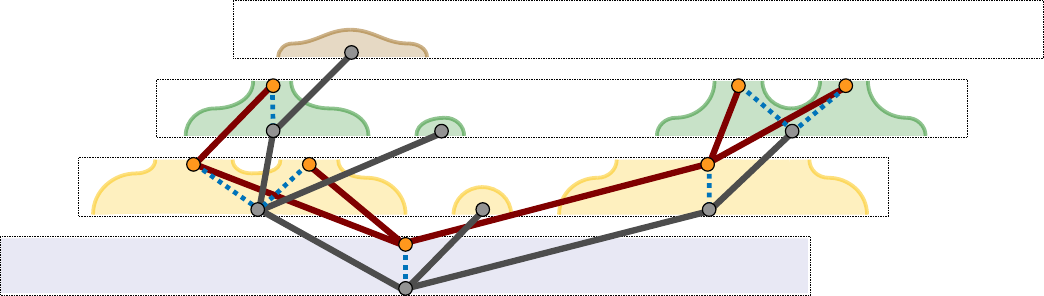}}\quad
\scalebox{0.6}{\begin{tikzcd}[column sep=1em, arrows={-}, ampersand replacement=\&]
 \& \& \& \& \color{nodeT}\bullet\ar[d] \& \& \&          \\
 \color{nodeS}\bullet \ar[d]\ar[rrrr,bend left=15,line width=3pt,opacity=0.5,frenchblue] \& \color{nodeS}\bullet\ar[rd]\ar[rrrrrrrr,bend left=15,line width=3pt,opacity=0.5,frenchblue] \& \color{nodeS}\bullet\ar[d] \ar[rrrrrrr,bend left=15,line width=3pt,opacity=0.5,frenchblue] \&\&
  \color{nodeT}\bullet\ar[dr] \&\& \color{nodeT}\bullet\ar[dl] \& \& \&\color{nodeT}\bullet\ar[d] \\
 \color{nodeS}\bullet\ar[dr]\ar[rrrrr,bend left=15,line width=3pt,opacity=0.5,frenchblue] \& \color{nodeS}\bullet\ar[d]\ar[rrrr,bend left=15,line width=3pt,opacity=0.5,frenchblue]  \& \color{nodeS}\bullet\ar[dl]\ar[rrrrrrr,bend left=15,line width=3pt,opacity=0.5,frenchblue] \&\&
 \& \color{nodeT}\bullet\ar[drr] \&\& \color{nodeT}\bullet\ar[d] \& \&\color{nodeT}\bullet\ar[dll]\\
  \&\color{nodeS}\bullet\ar[rrrrrr,bend left=15,line width=3pt,opacity=0.5,frenchblue] \&\&\& 
  \&\&      \& \color{nodeT}\bullet \& \& \& 
\end{tikzcd}}
\]

\paragraph{Bijection in the general case}
The inductive procedure we have sketched may be adapted to the general case $k \ge 0$ by extending the rules \eqref{eq:tree-formulas} and \eqref{eq:tree-morphism-rules} to represent trees with a marked node on their leftmost branch.
We begin by introducing a grammar of formulas for marked trees ${}^\circ T$ and marked forests ${}^\circ\Phi$:
\begin{equation}\label{eq:marked-tree-formulas}
  {}^\circ T ::= \circ \Phi \mid \bullet {}^\circ\Phi
  \qquad
  {}^\circ \Phi ::= {}^\circ T,\Phi'  
\end{equation}
A marked forest ${}^\circ\Phi$ starts with a marked tree ${}^\circ T$ followed by an unmarked forest $\Phi'$.
A marked tree ${}^\circ T$ can be built either by taking an unmarked forest $\Phi$ and attaching the root of each tree to a marked root vertex, $\circ \Phi$, or by taking a marked forest ${}^\circ\Phi$ and attaching each of its roots to a new unmarked root vertex, $\bullet {}^\circ \Phi$.
We introduce two additional judgments ${}^\circ\Phi \to_k {}^\circ T$ and ${}^\circ \Phi' \to_k {}^\circ \Phi$ defining marked-forest-to-marked-tree and marked-forest-to-marked-forest morphisms, where the index $k$ specifies the height of the mark:
\begin{equation}\label{eq:marked-tree-morphism-rules}
    \infer{\circ \Phi'_0,\bullet\Phi'_1,\dots,\bullet\Phi'_i \to_0 \circ\Phi}{\Phi'_0,\Phi'_1,\dots,\Phi'_i \to \Phi}
    \qquad
    \infer{\bullet ({}^\circ \Phi'_0),\bullet\Phi'_1,\dots,\bullet\Phi'_i \to_{k+1} \bullet{}^\circ\Phi}{{}^\circ \Phi'_0,\Phi'_1,\dots,\Phi'_i \to_k \Phi}
    \qquad
    \infer{{}^\circ\Phi_0,\Phi \to_k {}^\circ T,\Phi'}{{}^\circ\Phi_0 \to_k {}^\circ T & \Phi \to \Phi'}  
\end{equation}
The key invariant enforced by these rules is that the marked node of the forest on the left is sent to the marked node of the tree or forest on the right.
Just as we argued for $\PTree$, it is clear that any morphism in $\PTree_{\bot(k)}$ may be built inductively using the above rules.

Similarly, for $k > 0$, the string diagrams of morphisms in $\Bifib{p_\omega}_k$ may be decomposed recursively by first splitting along \begin{tikzcd}[column sep=1em]0\ar[r,equals] & 0\end{tikzcd} boundaries essentially as before, with the difference that the bottom boundary (corresponding to the left side in our rotated diagrams) is of the form \begin{tikzcd}[column sep=1em]k\ar[r,equals] & k\end{tikzcd}.
For example, the following diagram represents a morphism in $\Bifib{p_\omega}_2$:
\[
\rotatebox{-90}{\resizebox{!}{7cm}{\begin{tikzcd}[ampersand replacement=\&,row sep=1cm,execute at end picture = { 
\draw[draw=frenchblue, draw opacity=0.5, line width=1mm,out=90,in=180] ($(l1)!0.67!(r1)$) to ($(r1)!0.33!(r0)$);
\draw[draw=frenchblue, draw opacity=0.5, line width=1mm,out=90,in=180] ($(l1)!0.33!(r1)$) to ($(r1)!0.67!(r0)$);
\draw[draw=frenchblue, draw opacity=0.5, line width=1mm,out=0,in=-90] ($(l2)!0.67!(l1)$) to ($(l1)!0.33!(r1)$);
\draw[draw=frenchblue, draw opacity=0.5, line width=1mm,out=0,in=-90] ($(l2)!0.33!(l1)$) to ($(l1)!0.67!(r1)$);
\draw[draw=frenchblue, draw opacity=0.5, line width=1mm,out=90,in=0] ($(l3)!0.50!(r3)$) to ($(l3)!0.50!(l2)$);
\draw[draw=frenchblue, draw opacity=0.5, line width=1mm,out=0,in=-90] ($(l4)!0.50!(l3)$) to ($(l3)!0.50!(r3)$);
\draw[draw=frenchblue, draw opacity=0.5, line width=1mm,out=90,in=0] ($(l5)!0.33!(r5)$) to ($(l5)!0.33!(l4)$);
\draw[draw=frenchblue, draw opacity=0.5, line width=1mm,out=90,in=0] ($(l5)!0.67!(r5)$) to ($(l5)!0.67!(l4)$);
\draw[draw=frenchblue, draw opacity=0.5, line width=1mm,out=180,in=-90] ($(r6)!0.67!(r5)$) to ($(l5)!0.67!(r5)$);
\draw[draw=frenchblue, draw opacity=0.5, line width=1mm,out=180,in=-90] ($(r6)!0.33!(r5)$) to ($(l5)!0.33!(r5)$);
\draw[draw=frenchblue, draw opacity=0.5, line width=1mm,out=90,in=180] ($(l7)!0.50!(r7)$) to ($(r7)!0.50!(r6)$);
\draw[draw=frenchblue, draw opacity=0.5, line width=1mm,out=180,in=-90] ($(r8)!0.50!(r7)$) to ($(l7)!0.50!(r7)$);
\draw[draw=frenchblue, draw opacity=0.5, line width=1mm,out=90,in=180] ($(l9)!0.67!(r9)$) to ($(r9)!0.33!(r8)$);
\draw[draw=frenchblue, draw opacity=0.5, line width=1mm,out=90,in=180] ($(l9)!0.33!(r9)$) to ($(r9)!0.67!(r8)$);
\draw[draw=frenchblue, draw opacity=0.5, line width=1mm,out=0,in=-90] ($(l10)!0.50!(l9)$) to ($(l9)!0.33!(r9)$);
\draw[draw=frenchblue, draw opacity=0.5, line width=1mm,out=180,in=-90] ($(r10)!0.50!(r9)$) to ($(l9)!0.67!(r9)$);
\draw[draw=frenchblue, draw opacity=0.5, line width=1mm,out=90,in=0] ($(l11)!0.25!(r11)$) to ($(l11)!0.50!(l10)$);
\draw[draw=frenchblue, draw opacity=0.5, line width=1mm,out=90,in=180] ($(l11)!0.75!(r11)$) to ($(r11)!0.33!(r10)$);
\draw[draw=frenchblue, draw opacity=0.5, line width=1mm,out=90,in=180] ($(l11)!0.50!(r11)$) to ($(r11)!0.67!(r10)$);
\draw[draw=frenchblue, draw opacity=0.5, line width=1mm,out=0,in=-90] ($(l14)!0.67!(l11)$) to ($(l11)!0.25!(r11)$);
\draw[draw=frenchblue, draw opacity=0.5, line width=1mm,out=0,in=-90] ($(l14)!0.33!(l11)$) to ($(l11)!0.50!(r11)$);
\draw[draw=frenchblue, draw opacity=0.5, line width=1mm,out=180,in=-90] ($(r14)!0.50!(r11)$) to ($(l11)!0.75!(r11)$);
}
 ]
|[alias=l0]| 0 \ar[r,equals] \& |[alias=r0]| 0 
\\ |[alias=l1]|0\ar[u,equals]\ar[r] \& |[alias=r1]|2\ar[u,<-]
\\ |[alias=l2]|2\ar[u,<-]\ar[r,equals] \& |[alias=r2]|2\ar[u,equals]
\\ |[alias=l3]|1\ar[u]\ar[r] \& |[alias=r3]|2\ar[u,equals]
\\ |[alias=l4]|2\ar[u,<-]\ar[r,equals] \& |[alias=r4]|2\ar[u,equals]
\\ |[alias=l5]|0\ar[u]\ar[r] \& |[alias=r5]|2\ar[u,equals]
\\ |[alias=l6]|0\ar[u,equals]\ar[r,equals] \& |[alias=r6]|0\ar[u]
\\ |[alias=l7]|0\ar[u,equals]\ar[r] \& |[alias=r7]|1\ar[u,<-]
\\ |[alias=l8]|0\ar[u,equals]\ar[r,equals] \& |[alias=r8]|0\ar[u]
\\ |[alias=l9]|0\ar[u,equals]\ar[r] \& |[alias=r9]|2\ar[u,<-]
\\ |[alias=l10]|1\ar[u,<-]\ar[r,equals] \& |[alias=r10]|1\ar[u]
\\ |[alias=l11]|0\ar[u]\ar[r] \& |[alias=r11]|3\ar[u,<-]
\\ |[alias=l14]|2\ar[u,equals]\ar[r,equals] \& |[alias=r14]|2\ar[u]
\arrow[gray,phantom,from=1-1,to=2-2,"\push \Rsym"]
\arrow[gray,phantom,from=2-1,to=3-2,"\push \Lsym"]
\arrow[gray,phantom,from=3-1,to=4-2,"\pull \Lsym"]
\arrow[gray,phantom,from=4-1,to=5-2,"\push \Lsym"]
\arrow[gray,phantom,from=5-1,to=6-2,"\pull \Lsym"]
\arrow[gray,phantom,from=6-1,to=7-2,"\pull \Rsym"]
\arrow[gray,phantom,from=7-1,to=8-2,"\push \Rsym"]
\arrow[gray,phantom,from=8-1,to=9-2,"\pull \Rsym"]
\arrow[gray,phantom,from=9-1,to=10-2,"\push \Rsym"]
\arrow[gray,phantom,from=10-1,to=11-2,"\push \Lsym\pull \Rsym"]
\arrow[gray,phantom,from=11-1,to=12-2,"\pull \Lsym \push \Rsym"]
\arrow[gray,phantom,from=12-1,to=13-2,"\push\Lsym\pull \Rsym"]
\end{tikzcd}
}}
\]
After splitting along \begin{tikzcd}[column sep=1em]0\ar[r,equals] & 0\end{tikzcd} boundaries, we obtain as before a list of components that may be interpreted as morphisms in $\Bifib{p_\omega}_0$, except for the leftmost component which may be reduced to a morphism in $\Bifib{p_\omega}_1$ after removing the chords incident to 0 and shifting levels.
We recognize the inductive structure of marked morphisms of \eqref{eq:marked-tree-morphism-rules}.

\noindent
\begin{minipage}{0.84\textwidth}
\paragraph{Naturality}
Finally we need to show that the bijections $(\psi_k)_{S,T}$ are natural in $S$ and $T$, which is equivalent to functoriality: 
$\psi_k(\alpha\hcomp\beta) = \psi_k(\alpha)\hcomp \psi_k(\beta)$ for every pair of composable zigzag morphisms $\alpha$ and $\beta$.

For example, consider our above example restricted to a prefix for clarity, along with a second morphism which is composable along the middle zigzag formula, as represented on the right.

These morphisms can be seen as relief maps: the tree morphism in $T \to T'$ is determined by the order-preserving functions $T(n) \to T'(n)$ between nodes of height $n$, which can be read from the $n$-th layer of the relief map.
\end{minipage}
\begin{minipage}{0.16\textwidth}
  $\quad$
  \resizebox{!}{2.5cm}{\begin{tikzcd}[ampersand replacement=\&,row sep=1cm,execute at end picture = { 
\draw[draw=forestgreen, draw opacity=0.5, line width=1mm,out=90,in=180] ($(l9)!0.67!(r9)$) to ($(r9)!0.50!(r8)$);
\draw[draw=goldenpoppy, draw opacity=0.5, line width=1mm,out=90,in=180] ($(l9)!0.33!(r9)$) to ($(r9)!0.67!(r8)$);
\draw[draw=goldenpoppy, draw opacity=0.5, line width=1mm,out=0,in=-90] ($(l10)!0.50!(l9)$) to ($(l9)!0.33!(r9)$);
\draw[draw=forestgreen, draw opacity=0.5, line width=1mm,out=180,in=-90] ($(r10)!0.50!(r9)$) to ($(l9)!0.67!(r9)$);
\draw[draw=goldenpoppy, draw opacity=0.5, line width=1mm,out=90,in=0] ($(l11)!0.25!(r11)$) to ($(l11)!0.50!(l10)$);
\draw[draw=goldenbrown, draw opacity=0.5, line width=1mm,out=90,in=180] ($(l11)!0.75!(r11)$) to ($(r11)!0.33!(r10)$);
\draw[draw=forestgreen, draw opacity=0.5, line width=1mm,out=90,in=180] ($(l11)!0.50!(r11)$) to ($(r11)!0.67!(r10)$);
\draw[draw=goldenpoppy, draw opacity=0.5, line width=1mm,out=0,in=-90] ($(l12)!0.67!(l11)$) to ($(l11)!0.25!(r11)$);
\draw[draw=forestgreen, draw opacity=0.5, line width=1mm,out=0,in=-90] ($(l12)!0.33!(l11)$) to ($(l11)!0.50!(r11)$);
\draw[draw=goldenbrown, draw opacity=0.5, line width=1mm,out=90,in=-90] ($(l12)!0.50!(r12)$) to ($(l11)!0.75!(r11)$);
\draw[draw=goldenpoppy, draw opacity=0.5, line width=1mm,out=90,in=0] ($(l13)!0.25!(r13)$) to ($(l13)!0.33!(l12)$);
\draw[draw=forestgreen, draw opacity=0.5, line width=1mm,out=90,in=0] ($(l13)!0.50!(r13)$) to ($(l13)!0.67!(l12)$);
\draw[draw=goldenbrown, draw opacity=0.5, line width=1mm,out=90,in=-90] ($(l13)!0.75!(r13)$) to ($(l12)!0.50!(r12)$);
\draw[draw=goldenbrown, draw opacity=0.5, line width=1mm,out=180,in=-90] ($(r14)!0.75!(r13)$) to ($(l13)!0.75!(r13)$);
\draw[draw=forestgreen, draw opacity=0.5, line width=1mm,out=180,in=-90] ($(r14)!0.50!(r13)$) to ($(l13)!0.50!(r13)$);
\draw[draw=goldenpoppy, draw opacity=0.5, line width=1mm,out=180,in=-90] ($(r14)!0.25!(r13)$) to ($(l13)!0.25!(r13)$);
\begin{scope}[on background layer]
\fill[height0] (r14.center) to (l14.center) to (l8.center) to (r8.center) -- cycle;
\fill[height1] ($(r14)!0.25!(r13)$) to[out=180,in=-90] ($(l13)!0.25!(r13)$) to[out=90,in=0] ($(l13)!0.33!(l12)$) to ($(l12)!0.67!(l11)$) to[out=0,in=0] ($(l11)!0.5!(l10)$) to ($(l10)!0.5!(l9)$) to[in=-90,out=0] ($(l9)!0.33!(r9)$) to[out=90,in=180] ($(r9)!0.67!(r8)$) -- cycle;
\fill[height2] ($(r14)!0.5!(r13)$) to[out=180,in=-90] ($(l13)!0.5!(r13)$) to[out=90,in=0] ($(l13)!0.67!(l12)$) to ($(l12)!0.33!(l11)$) to[out=0,in=-90] ($(l11)!0.5!(r11)$) to[out=90,in=-180] ($(r11)!0.67!(r10)$) -- cycle;
\fill[height2] ($(r10)!0.5!(r9)$) to [out=180,in=-90] ($(l9)!0.67!(r9)$) to[out=90,in=180] ($(r9)!0.50!(r8)$) -- cycle;
\fill[height3] ($(r14)!0.75!(r13)$) to[out=180,in=-90] ($(l13)!0.75!(r13)$) to[out=90,in=-90] ($(l12)!0.5!(r12)$) to[out=90,in=-90] ($(l11)!0.75!(r11)$) to [out=90,in=-180] ($(r10)!0.67!(r11)$) -- cycle;
\end{scope}
}
 ]
   |[alias=l7]|\phantom{0}
\& |[alias=r7]|\phantom{0}
\\ |[alias=l8]|0\ar[r,equals]
\& |[alias=r8]|0
\\ |[alias=l9]|0\ar[u,equals]
\& |[alias=r9]|2\ar[u,<-]
\\ |[alias=l10]|1\ar[u,<-]
\& |[alias=r10]|1\ar[u]
\\ |[alias=l11]|0\ar[u]
\& |[alias=r11]|3\ar[u,<-]
\\ |[alias=l12]|2\ar[u,<-]
\& |[alias=r12]|\phantom{3}
\\ |[alias=l13]|0\ar[u]
\& |[alias=r13]|3\ar[uu,equals]
\\ |[alias=l14]|0\ar[u,equals]\ar[r,equals] \& |[alias=r14]|0\ar[u]
\end{tikzcd}}
  \resizebox{!}{2.5cm}{\begin{tikzcd}[ampersand replacement=\&,row sep=1cm,execute at end picture = { 
\draw[draw=forestgreen, draw opacity=0.5, line width=1mm,out=180,in=0] ($(R8)!0.50!(R7)$) to[in=90,out=180] ($(R8)!0.5!(L8)$) to[out=-90] ($(L9)!0.50!(L8)$);
\draw[draw=goldenpoppy, draw opacity=0.5, line width=1mm,out=180] ($(R8)!0.67!(R7)$) to[in=90] ($(R8)!0.67!(L8)$) to[in=0,out=-90] ($(L9)!0.67!(L8)$);
\draw[draw=goldenbrown, draw opacity=0.5, line width=1mm] ($(R9)!0.50!(R8)$) to [out=180,in=180] ($(R8)!0.33!(R7)$);
\draw[draw=forestgreen, draw opacity=0.5, line width=1mm,out=0,in=0] ($(L10)!0.50!(L9)$) to ($(L10)!0.33!(L11)$);
\draw[draw=goldenbrown, draw opacity=0.5, line width=1mm,out=0,in=180] ($(L11)!0.33!(L10)$) to ($(R11)!0.50!(R10)$);
\draw[draw=goldenbrown, draw opacity=0.5, line width=1mm,out=0,in=180] ($(L14)!0.75!(L13)$) to ($(R14)!0.75!(R13)$);
\draw[draw=forestgreen, draw opacity=0.5, line width=1mm,out=0,in=180] ($(L14)!0.50!(L13)$) to ($(R14)!0.50!(R13)$);
\draw[draw=goldenpoppy, draw opacity=0.5, line width=1mm,out=0,in=180] ($(L14)!0.25!(L13)$) to ($(R14)!0.25!(R13)$);
\begin{scope}[on background layer]
\fill[height0] (R14.center) to (L14.center) to (L7.center) to (R7.center) -- cycle;
\fill[height1] ($(R14)!0.25!(R13)$) to ($(L14)!0.25!(L13)$) to ($(L9)!0.67!(L8)$) to[in=0,out=-90] ($(L8)!0.33!(R8)$) to[out=90,in=180] ($(R8)!0.67!(R7)$) -- cycle;
\fill[height2] ($(R14)!0.50!(R13)$) to ($(L14)!0.50!(L13)$) to ($(L11)!0.67!(L10)$) to [out=0,in=0] ($(L10)!0.50!(L9)$) to ($(L9)!0.50!(L8)$) to[out=0,in=-90] ($(L8)!0.5!(R8)$) to[out=90,in=180] ($(R8)!0.5!(R7)$) -- cycle;
\fill[height3] ($(R14)!0.75!(R13)$) to ($(R11)!0.50!(R10)$) to [in=0,out=180] ($(L11)!0.33!(L10)$) to ($(L14)!0.75!(L13)$) -- cycle;
\fill[height3] ($(R9)!0.50!(R8)$) to [out=180,in=180] ($(R8)!0.33!(R7)$) -- cycle;
\end{scope}
}
 ]
   |[alias=L7]|0\ar[r,equals]
\& |[alias=R7]|0
\\ |[alias=L8]|0\ar[u,equals]
\& |[alias=R8]|3\ar[u,<-]
\\ |[alias=L9]|2\ar[u,<-]
\& |[alias=R9]|2\ar[u,<-]
\\ |[alias=L10]|1\ar[u]
\& |[alias=R10]|2\ar[u,equals]
\\ |[alias=L11]|3\ar[u,<-]
\& |[alias=R11]|3\ar[u]
\\ |[alias=L12]|\phantom{3}
\& |[alias=R12]|\phantom{3}
\\ |[alias=L13]|3\ar[uu,equals]
\& |[alias=R13]|3\ar[uu,equals]
\\ |[alias=L14]|0\ar[u]\ar[r,equals]
\& |[alias=R14]|0\ar[u]
\end{tikzcd}}
\end{minipage}

If the two diagrams are composable, then each layer of the relief map is composable:
\newcommand{\hcomparray}[2]{#1 \hcomp #2}
\begin{mathpar}
\hcomparray{
\rotatebox{0}{\resizebox{!}{2.5cm}{\begin{tikzcd}[ampersand replacement=\&,row sep=1cm,execute at end picture = { 
\draw[draw=goldenpoppy, draw opacity=0.5, line width=1mm,out=90,in=180] ($(l9)!0.33!(r9)$) to ($(r9)!0.67!(r8)$);
\draw[draw=goldenpoppy, draw opacity=0.5, line width=1mm,out=0,in=-90] ($(l10)!0.50!(l9)$) to ($(l9)!0.33!(r9)$);
\draw[draw=goldenpoppy, draw opacity=0.5, line width=1mm,out=90,in=0] ($(l11)!0.25!(r11)$) to ($(l11)!0.50!(l10)$);
\draw[draw=goldenpoppy, draw opacity=0.5, line width=1mm,out=0,in=-90] ($(l12)!0.67!(l11)$) to ($(l11)!0.25!(r11)$);
\draw[draw=goldenpoppy, draw opacity=0.5, line width=1mm,out=90,in=0] ($(l13)!0.25!(r13)$) to ($(l13)!0.33!(l12)$);
\draw[draw=goldenpoppy, draw opacity=0.5, line width=1mm,out=180,in=-90] ($(r14)!0.25!(r13)$) to ($(l13)!0.25!(r13)$);
\begin{scope}[on background layer]
\fill[height0] (r14.center) to (l14.center) to (l8.center) to (r8.center) -- cycle;
\fill[height1] ($(r14)!0.25!(r13)$) to[out=180,in=-90] ($(l13)!0.25!(r13)$) to[out=90,in=0] ($(l13)!0.33!(l12)$) to ($(l12)!0.67!(l11)$) to[out=0,in=0] ($(l11)!0.5!(l10)$) to ($(l10)!0.5!(l9)$) to[in=-90,out=0] ($(l9)!0.33!(r9)$) to[out=90,in=180] ($(r9)!0.67!(r8)$) -- cycle;
\end{scope}
}
 ]
   |[alias=l7]|\phantom{0}
\& |[alias=r7]|\phantom{0}
\\ |[alias=l8]|0\ar[r,equals]
\& |[alias=r8]|0
\\ |[alias=l9]|0\ar[u,equals]
\& |[alias=r9]|1\ar[u,<-]
\\ |[alias=l10]|1\ar[u,<-]
\& |[alias=r10]|
\\ |[alias=l11]|0\ar[u]
\& |[alias=r11]|
\\ |[alias=l12]|1\ar[u,<-]
\& |[alias=r12]|
\\ |[alias=l13]|0\ar[u]
\& |[alias=r13]|1\ar[uuuu,equals]
\\ |[alias=l14]|0\ar[u,equals]\ar[r,equals] \& |[alias=r14]|0\ar[u]
\end{tikzcd}}}
}{
\rotatebox{0}{\resizebox{!}{2.5cm}{\begin{tikzcd}[ampersand replacement=\&,row sep=1cm,execute at end picture = { 
\draw[draw=goldenpoppy, draw opacity=0.5, line width=1mm,out=180] ($(R8)!0.67!(R7)$) to[in=90] ($(R8)!0.67!(L8)$) to[in=0,out=-90] ($(L9)!0.67!(L8)$);
\draw[draw=goldenpoppy, draw opacity=0.5, line width=1mm,out=0,in=180] ($(L14)!0.25!(L13)$) to ($(R14)!0.25!(R13)$);
\begin{scope}[on background layer]
\fill[height0] (R14.center) to (L14.center) to (L7.center) to (R7.center) -- cycle;
\fill[height1] ($(R14)!0.25!(R13)$) to ($(L14)!0.25!(L13)$) to ($(L9)!0.67!(L8)$) to[in=0,out=-90] ($(L8)!0.33!(R8)$) to[out=90,in=180] ($(R8)!0.67!(R7)$) -- cycle;
\end{scope}
}
 ]
   |[alias=L7]|0\ar[r,equals]
\& |[alias=R7]|0
\\ |[alias=L8]|0\ar[u,equals]
\& |[alias=R8]|1\ar[u,<-]
\\ |[alias=L9]|1\ar[u,<-]
\& |[alias=R9]|\phantom{1}
\\ |[alias=L10]|\phantom{1}
\& |[alias=R10]|\phantom{1}
\\ |[alias=L11]|\phantom{1}
\& |[alias=R11]|\phantom{1}
\\ |[alias=L12]|\phantom{1}
\& |[alias=R12]|\phantom{1}
\\ |[alias=L13]|1\ar[uuuu,equals]
\& |[alias=R13]|1\ar[uuuuu,equals]
\\ |[alias=L14]|0\ar[u]\ar[r,equals]
\& |[alias=R14]|0\ar[u]
\end{tikzcd}
}}
}

\hcomparray{
\rotatebox{0}{\resizebox{!}{2.5cm}{\begin{tikzcd}[ampersand replacement=\&,row sep=1cm,execute at end picture = { 
\draw[draw=forestgreen, draw opacity=0.5, line width=1mm,out=90,in=180] ($(l9)!0.67!(r9)$) to ($(r9)!0.50!(r8)$);
\draw[draw=forestgreen, draw opacity=0.5, line width=1mm,out=180,in=-90] ($(r10)!0.50!(r9)$) to ($(l9)!0.67!(r9)$);
\draw[draw=forestgreen, draw opacity=0.5, line width=1mm,out=90,in=180] ($(l11)!0.50!(r11)$) to ($(r11)!0.67!(r10)$);
\draw[draw=forestgreen, draw opacity=0.5, line width=1mm,out=0,in=-90] ($(l12)!0.33!(l11)$) to ($(l11)!0.50!(r11)$);
\draw[draw=forestgreen, draw opacity=0.5, line width=1mm,out=90,in=0] ($(l13)!0.50!(r13)$) to ($(l13)!0.67!(l12)$);
\draw[draw=forestgreen, draw opacity=0.5, line width=1mm,out=180,in=-90] ($(r14)!0.50!(r13)$) to ($(l13)!0.50!(r13)$);
\begin{scope}[on background layer]
\fill[height0] (r14.center) to (l14.center) to (l8.center) to (r8.center) -- cycle;
\fill[height2] ($(r14)!0.5!(r13)$) to[out=180,in=-90] ($(l13)!0.5!(r13)$) to[out=90,in=0] ($(l13)!0.67!(l12)$) to ($(l12)!0.33!(l11)$) to[out=0,in=-90] ($(l11)!0.5!(r11)$) to[out=90,in=-180] ($(r11)!0.67!(r10)$) -- cycle;
\fill[height2] ($(r10)!0.5!(r9)$) to [out=180,in=-90] ($(l9)!0.67!(r9)$) to[out=90,in=180] ($(r9)!0.50!(r8)$) -- cycle;
\end{scope}
}
 ]
   |[alias=l7]|\phantom{0}
\& |[alias=r7]|\phantom{0}
\\ |[alias=l8]|1\ar[r,equals]
\& |[alias=r8]|1
\\ |[alias=l9]|\phantom{1}
\& |[alias=r9]|2\ar[u,<-]
\\ |[alias=l10]|\phantom{1}
\& |[alias=r10]|1\ar[u]
\\ |[alias=l11]|1\ar[uuu,equals]
\& |[alias=r11]|2\ar[u,<-]
\\ |[alias=l12]|2\ar[u,<-]
\& |[alias=r12]|\phantom{2}
\\ |[alias=l13]|1\ar[u]
\& |[alias=r13]|2\ar[uu,equals]
\\ |[alias=l14]|1\ar[u,equals]\ar[r,equals]
\& |[alias=r14]|1\ar[u]
\end{tikzcd}}}
}{
\rotatebox{0}{\resizebox{!}{2.5cm}{\begin{tikzcd}[ampersand replacement=\&,row sep=1cm,execute at end picture = { 
\draw[draw=forestgreen, draw opacity=0.5, line width=1mm,out=180,in=0] ($(R8)!0.50!(R7)$) to[in=90,out=180] ($(R8)!0.5!(L8)$) to[out=-90] ($(L9)!0.50!(L8)$);
\draw[draw=forestgreen, draw opacity=0.5, line width=1mm,out=0,in=0] ($(L10)!0.50!(L9)$) to ($(L10)!0.33!(L11)$);
\draw[draw=forestgreen, draw opacity=0.5, line width=1mm,out=0,in=180] ($(L14)!0.50!(L13)$) to ($(R14)!0.50!(R13)$);
\begin{scope}[on background layer]
\fill[height0] (R14.center) to (L14.center) to (L7.center) to (R7.center) -- cycle;
\fill[height2] ($(R14)!0.50!(R13)$) to ($(L14)!0.50!(L13)$) to ($(L11)!0.67!(L10)$) to [out=0,in=0] ($(L10)!0.50!(L9)$) to ($(L9)!0.50!(L8)$) to[out=0,in=-90] ($(L8)!0.5!(R8)$) to[out=90,in=180] ($(R8)!0.5!(R7)$) -- cycle;
\end{scope}
}
 ]
   |[alias=L7]|1\ar[r,equals]
\& |[alias=R7]|1
\\ |[alias=L8]|1\ar[u,equals]
\& |[alias=R8]|2\ar[u,<-]
\\ |[alias=L9]|2\ar[u,<-]
\& |[alias=R9]|\phantom{2}
\\ |[alias=L10]|1\ar[u]
\& |[alias=R10]|\phantom{2}
\\ |[alias=L11]|2\ar[u,<-]
\& |[alias=R11]|\phantom{2}
\\ |[alias=L12]|\phantom{2}
\& |[alias=R12]|\phantom{2}
\\ |[alias=L13]|2\ar[uu,equals]
\& |[alias=R13]|2\ar[uuuuu,equals]
\\ |[alias=L14]|1\ar[u]\ar[r,equals]
\& |[alias=R14]|1\ar[u]
\end{tikzcd}
}}
}

\hcomparray{
\rotatebox{0}{\resizebox{!}{2.5cm}{\begin{tikzcd}[ampersand replacement=\&,row sep=1cm,execute at end picture = { 
\draw[draw=goldenbrown, draw opacity=0.5, line width=1mm,out=90,in=180] ($(l11)!0.75!(r11)$) to ($(r11)!0.33!(r10)$);
\draw[draw=goldenbrown, draw opacity=0.5, line width=1mm,out=90,in=-90] ($(l12)!0.50!(r12)$) to ($(l11)!0.75!(r11)$);
\draw[draw=goldenbrown, draw opacity=0.5, line width=1mm,out=90,in=-90] ($(l13)!0.75!(r13)$) to ($(l12)!0.50!(r12)$);
\draw[draw=goldenbrown, draw opacity=0.5, line width=1mm,out=180,in=-90] ($(r14)!0.75!(r13)$) to ($(l13)!0.75!(r13)$);
\begin{scope}[on background layer]
\fill[height0] (r14.center) to (l14.center) to (l8.center) to (r8.center) -- cycle;
\fill[height3] ($(r14)!0.75!(r13)$) to[out=180,in=-90] ($(l13)!0.75!(r13)$) to[out=90,in=-90] ($(l12)!0.5!(r12)$) to[out=90,in=-90] ($(l11)!0.75!(r11)$) to [out=90,in=-180] ($(r10)!0.67!(r11)$) -- cycle;
\end{scope}
}
 ]
   |[alias=l7]|\phantom{2}
\& |[alias=r7]|\phantom{2}
\\ |[alias=l8]|2\ar[r,equals]
\& |[alias=r8]|2
\\ |[alias=l9]|\phantom{2}
\& |[alias=r9]|\phantom{2}
\\ |[alias=l10]|\phantom{2}
\& |[alias=r10]|2\ar[uu,equals]
\\ |[alias=l11]|\phantom{2}
\& |[alias=r11]|3\ar[u,<-]
\\ |[alias=l12]|\phantom{2}
\& |[alias=r12]|\phantom{3}
\\ |[alias=l13]|\phantom{2}
\& |[alias=r13]|3\ar[uu,equals]
\\ |[alias=l14]|2\ar[uuuuuu,equals]\ar[r,equals]
\& |[alias=r14]|2\ar[u]
\end{tikzcd}
}}
}{
\rotatebox{0}{\resizebox{!}{2.5cm}{\begin{tikzcd}[ampersand replacement=\&,row sep=1cm,execute at end picture = { 
\draw[draw=goldenbrown, draw opacity=0.5, line width=1mm] ($(R9)!0.50!(R8)$) to [out=180,in=180] ($(R8)!0.33!(R7)$);
\draw[draw=goldenbrown, draw opacity=0.5, line width=1mm,out=0,in=180] ($(L11)!0.33!(L10)$) to ($(R11)!0.50!(R10)$);
\draw[draw=goldenbrown, draw opacity=0.5, line width=1mm,out=0,in=180] ($(L14)!0.75!(L13)$) to ($(R14)!0.75!(R13)$);
\begin{scope}[on background layer]
\fill[height0] (R14.center) to (L14.center) to (L7.center) to (R7.center) -- cycle;
\fill[height3] ($(R14)!0.75!(R13)$) to ($(R11)!0.50!(R10)$) to [in=0,out=180] ($(L11)!0.33!(L10)$) to ($(L14)!0.75!(L13)$) -- cycle;
\fill[height3] ($(R9)!0.50!(R8)$) to [out=180,in=180] ($(R8)!0.33!(R7)$) -- cycle;
\end{scope}
}
 ]
   |[alias=L7]|2\ar[r,equals]
\& |[alias=R7]|2
\\ |[alias=L8]|\phantom{2}
\& |[alias=R8]|3\ar[u,<-]
\\ |[alias=L9]|\phantom{2}
\& |[alias=R9]|2\ar[u,<-]
\\ |[alias=L10]|2\ar[uuu,equals]
\& |[alias=R10]|2\ar[u,equals]
\\ |[alias=L11]|3\ar[u,<-]
\& |[alias=R11]|3\ar[u]
\\ |[alias=L12]|\phantom{2}
\& |[alias=R12]|\phantom{2}
\\ |[alias=L13]|3\ar[uu,equals]
\& |[alias=R13]|3\ar[uu,equals]
\\ |[alias=L14]|2\ar[u]\ar[r,equals]
\& |[alias=R14]|2\ar[u]
\end{tikzcd}
}}
}
\end{mathpar}
 
For each layer of the relief map, the order-preserving function is determined by connectedness in the diagram: each node in the input tree is related to the (unique) node in the output tree that belongs to the same connected component. This property of connectedness is preserved by diagrammatic concatenation: the connectedness relation of the composite diagram (corresponding to a layer of $\alpha \hcomp \beta$) is the composition of the connectedness relations of the two diagrams.
\end{proof}

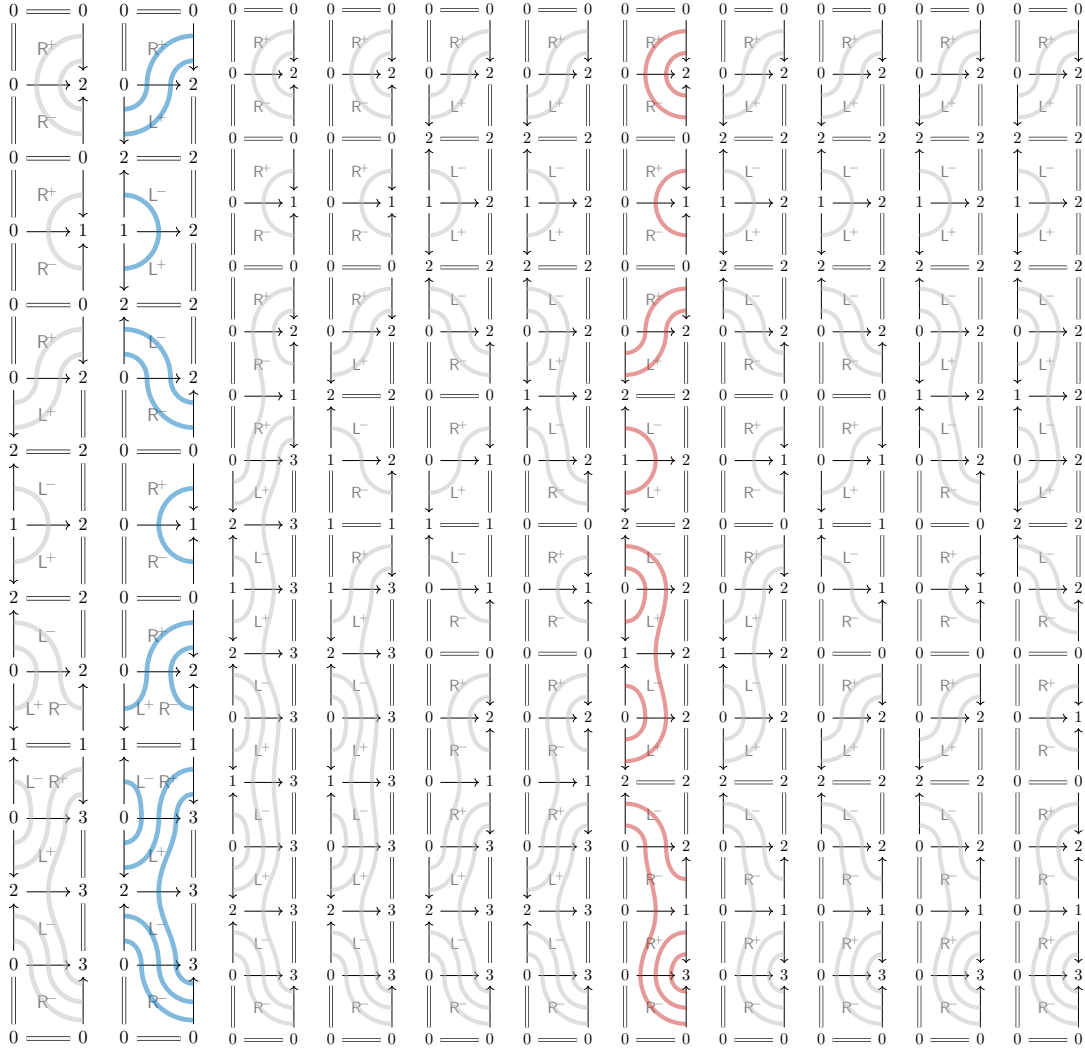
\begin{figure}[t]
  \centering

\resizebox{!}{7cm}{\begin{tikzcd}[ampersand replacement=\&,row sep=1cm,execute at end picture = { 
\draw[draw=lightgray, draw opacity=0.5, line width=1mm,out=90,in=180] ($(l1)!0.67!(r1)$) to ($(r1)!0.33!(r0)$);
\draw[draw=lightgray, draw opacity=0.5, line width=1mm,out=90,in=180] ($(l1)!0.33!(r1)$) to ($(r1)!0.67!(r0)$);
\draw[draw=lightgray, draw opacity=0.5, line width=1mm,out=180,in=-90] ($(r2)!0.67!(r1)$) to ($(l1)!0.67!(r1)$);
\draw[draw=lightgray, draw opacity=0.5, line width=1mm,out=180,in=-90] ($(r2)!0.33!(r1)$) to ($(l1)!0.33!(r1)$);
\draw[draw=lightgray, draw opacity=0.5, line width=1mm,out=90,in=180] ($(l3)!0.50!(r3)$) to ($(r3)!0.50!(r2)$);
\draw[draw=lightgray, draw opacity=0.5, line width=1mm,out=180,in=-90] ($(r4)!0.50!(r3)$) to ($(l3)!0.50!(r3)$);
\draw[draw=lightgray, draw opacity=0.5, line width=1mm,out=90,in=180] ($(l5)!0.67!(r5)$) to ($(r5)!0.33!(r4)$);
\draw[draw=lightgray, draw opacity=0.5, line width=1mm,out=90,in=180] ($(l5)!0.33!(r5)$) to ($(r5)!0.67!(r4)$);
\draw[draw=lightgray, draw opacity=0.5, line width=1mm,out=0,in=-90] ($(l6)!0.67!(l5)$) to ($(l5)!0.33!(r5)$);
\draw[draw=lightgray, draw opacity=0.5, line width=1mm,out=0,in=-90] ($(l6)!0.33!(l5)$) to ($(l5)!0.67!(r5)$);
\draw[draw=lightgray, draw opacity=0.5, line width=1mm,out=90,in=0] ($(l7)!0.50!(r7)$) to ($(l7)!0.50!(l6)$);
\draw[draw=lightgray, draw opacity=0.5, line width=1mm,out=0,in=-90] ($(l8)!0.50!(l7)$) to ($(l7)!0.50!(r7)$);
\draw[draw=lightgray, draw opacity=0.5, line width=1mm,out=90,in=0] ($(l9)!0.33!(r9)$) to ($(l9)!0.33!(l8)$);
\draw[draw=lightgray, draw opacity=0.5, line width=1mm,out=90,in=0] ($(l9)!0.67!(r9)$) to ($(l9)!0.67!(l8)$);
\draw[draw=lightgray, draw opacity=0.5, line width=1mm,out=0,in=-90] ($(l10)!0.50!(l9)$) to ($(l9)!0.33!(r9)$);
\draw[draw=lightgray, draw opacity=0.5, line width=1mm,out=180,in=-90] ($(r10)!0.50!(r9)$) to ($(l9)!0.67!(r9)$);
\draw[draw=lightgray, draw opacity=0.5, line width=1mm,out=90,in=0] ($(l11)!0.25!(r11)$) to ($(l11)!0.50!(l10)$);
\draw[draw=lightgray, draw opacity=0.5, line width=1mm,out=90,in=180] ($(l11)!0.75!(r11)$) to ($(r11)!0.33!(r10)$);
\draw[draw=lightgray, draw opacity=0.5, line width=1mm,out=90,in=180] ($(l11)!0.50!(r11)$) to ($(r11)!0.67!(r10)$);
\draw[draw=lightgray, draw opacity=0.5, line width=1mm,out=0,in=-90] ($(l12)!0.67!(l11)$) to ($(l11)!0.25!(r11)$);
\draw[draw=lightgray, draw opacity=0.5, line width=1mm,out=0,in=-90] ($(l12)!0.33!(l11)$) to ($(l11)!0.50!(r11)$);
\draw[draw=lightgray, draw opacity=0.5, line width=1mm,out=90,in=-90] ($(l12)!0.50!(r12)$) to ($(l11)!0.75!(r11)$);
\draw[draw=lightgray, draw opacity=0.5, line width=1mm,out=90,in=0] ($(l13)!0.25!(r13)$) to ($(l13)!0.33!(l12)$);
\draw[draw=lightgray, draw opacity=0.5, line width=1mm,out=90,in=0] ($(l13)!0.50!(r13)$) to ($(l13)!0.67!(l12)$);
\draw[draw=lightgray, draw opacity=0.5, line width=1mm,out=90,in=-90] ($(l13)!0.75!(r13)$) to ($(l12)!0.50!(r12)$);
\draw[draw=lightgray, draw opacity=0.5, line width=1mm,out=180,in=-90] ($(r14)!0.75!(r13)$) to ($(l13)!0.75!(r13)$);
\draw[draw=lightgray, draw opacity=0.5, line width=1mm,out=180,in=-90] ($(r14)!0.50!(r13)$) to ($(l13)!0.50!(r13)$);
\draw[draw=lightgray, draw opacity=0.5, line width=1mm,out=180,in=-90] ($(r14)!0.25!(r13)$) to ($(l13)!0.25!(r13)$);
}
 ]
|[alias=l0]| 0 \ar[r,equals] \& |[alias=r0]| 0 
\\ |[alias=l1]|0\ar[u,equals]\ar[r] \& |[alias=r1]|2\ar[u,<-]
\\ |[alias=l2]|0\ar[u,equals]\ar[r,equals] \& |[alias=r2]|0\ar[u]
\\ |[alias=l3]|0\ar[u,equals]\ar[r] \& |[alias=r3]|1\ar[u,<-]
\\ |[alias=l4]|0\ar[u,equals]\ar[r,equals] \& |[alias=r4]|0\ar[u]
\\ |[alias=l5]|0\ar[u,equals]\ar[r] \& |[alias=r5]|2\ar[u,<-]
\\ |[alias=l6]|2\ar[u,<-]\ar[r,equals] \& |[alias=r6]|2\ar[u,equals]
\\ |[alias=l7]|1\ar[u]\ar[r] \& |[alias=r7]|2\ar[u,equals]
\\ |[alias=l8]|2\ar[u,<-]\ar[r,equals] \& |[alias=r8]|2\ar[u,equals]
\\ |[alias=l9]|0\ar[u]\ar[r] \& |[alias=r9]|2\ar[u,equals]
\\ |[alias=l10]|1\ar[u,<-]\ar[r,equals] \& |[alias=r10]|1\ar[u]
\\ |[alias=l11]|0\ar[u]\ar[r] \& |[alias=r11]|3\ar[u,<-]
\\ |[alias=l12]|2\ar[u,<-]\ar[r] \& |[alias=r12]|3\ar[u,equals]
\\ |[alias=l13]|0\ar[u]\ar[r] \& |[alias=r13]|3\ar[u,equals]
\\ |[alias=l14]|0\ar[u,equals]\ar[r,equals] \& |[alias=r14]|0\ar[u]
\arrow[gray,phantom,from=1-1,to=2-2,"\push \Rsym"]
\arrow[gray,phantom,from=2-1,to=3-2,"\pull \Rsym"]
\arrow[gray,phantom,from=3-1,to=4-2,"\push \Rsym"]
\arrow[gray,phantom,from=4-1,to=5-2,"\pull \Rsym"]
\arrow[gray,phantom,from=5-1,to=6-2,"\push \Rsym"]
\arrow[gray,phantom,from=6-1,to=7-2,"\push \Lsym"]
\arrow[gray,phantom,from=7-1,to=8-2,"\pull \Lsym"]
\arrow[gray,phantom,from=8-1,to=9-2,"\push \Lsym"]
\arrow[gray,phantom,from=9-1,to=10-2,"\pull \Lsym"]
\arrow[gray,phantom,from=10-1,to=11-2,"\push \Lsym\pull \Rsym"]
\arrow[gray,phantom,from=11-1,to=12-2,"\pull \Lsym \push \Rsym"]
\arrow[gray,phantom,from=12-1,to=13-2,"\push \Lsym"]
\arrow[gray,phantom,from=13-1,to=14-2,"\pull \Lsym"]
\arrow[gray,phantom,from=14-1,to=15-2,"\pull \Rsym"]
\end{tikzcd}}
\resizebox{!}{7cm}{\begin{tikzcd}[ampersand replacement=\&,row sep=1cm,execute at end picture = { 
\draw[draw=frenchblue, draw opacity=0.5, line width=1mm,out=90,in=180] ($(l1)!0.67!(r1)$) to ($(r1)!0.33!(r0)$);
\draw[draw=frenchblue, draw opacity=0.5, line width=1mm,out=90,in=180] ($(l1)!0.33!(r1)$) to ($(r1)!0.67!(r0)$);
\draw[draw=frenchblue, draw opacity=0.5, line width=1mm,out=0,in=-90] ($(l2)!0.67!(l1)$) to ($(l1)!0.33!(r1)$);
\draw[draw=frenchblue, draw opacity=0.5, line width=1mm,out=0,in=-90] ($(l2)!0.33!(l1)$) to ($(l1)!0.67!(r1)$);
\draw[draw=frenchblue, draw opacity=0.5, line width=1mm,out=90,in=0] ($(l3)!0.50!(r3)$) to ($(l3)!0.50!(l2)$);
\draw[draw=frenchblue, draw opacity=0.5, line width=1mm,out=0,in=-90] ($(l4)!0.50!(l3)$) to ($(l3)!0.50!(r3)$);
\draw[draw=frenchblue, draw opacity=0.5, line width=1mm,out=90,in=0] ($(l5)!0.33!(r5)$) to ($(l5)!0.33!(l4)$);
\draw[draw=frenchblue, draw opacity=0.5, line width=1mm,out=90,in=0] ($(l5)!0.67!(r5)$) to ($(l5)!0.67!(l4)$);
\draw[draw=frenchblue, draw opacity=0.5, line width=1mm,out=180,in=-90] ($(r6)!0.67!(r5)$) to ($(l5)!0.67!(r5)$);
\draw[draw=frenchblue, draw opacity=0.5, line width=1mm,out=180,in=-90] ($(r6)!0.33!(r5)$) to ($(l5)!0.33!(r5)$);
\draw[draw=frenchblue, draw opacity=0.5, line width=1mm,out=90,in=180] ($(l7)!0.50!(r7)$) to ($(r7)!0.50!(r6)$);
\draw[draw=frenchblue, draw opacity=0.5, line width=1mm,out=180,in=-90] ($(r8)!0.50!(r7)$) to ($(l7)!0.50!(r7)$);
\draw[draw=frenchblue, draw opacity=0.5, line width=1mm,out=90,in=180] ($(l9)!0.67!(r9)$) to ($(r9)!0.33!(r8)$);
\draw[draw=frenchblue, draw opacity=0.5, line width=1mm,out=90,in=180] ($(l9)!0.33!(r9)$) to ($(r9)!0.67!(r8)$);
\draw[draw=frenchblue, draw opacity=0.5, line width=1mm,out=0,in=-90] ($(l10)!0.50!(l9)$) to ($(l9)!0.33!(r9)$);
\draw[draw=frenchblue, draw opacity=0.5, line width=1mm,out=180,in=-90] ($(r10)!0.50!(r9)$) to ($(l9)!0.67!(r9)$);
\draw[draw=frenchblue, draw opacity=0.5, line width=1mm,out=90,in=0] ($(l11)!0.25!(r11)$) to ($(l11)!0.50!(l10)$);
\draw[draw=frenchblue, draw opacity=0.5, line width=1mm,out=90,in=180] ($(l11)!0.75!(r11)$) to ($(r11)!0.33!(r10)$);
\draw[draw=frenchblue, draw opacity=0.5, line width=1mm,out=90,in=180] ($(l11)!0.50!(r11)$) to ($(r11)!0.67!(r10)$);
\draw[draw=frenchblue, draw opacity=0.5, line width=1mm,out=0,in=-90] ($(l12)!0.67!(l11)$) to ($(l11)!0.25!(r11)$);
\draw[draw=frenchblue, draw opacity=0.5, line width=1mm,out=0,in=-90] ($(l12)!0.33!(l11)$) to ($(l11)!0.50!(r11)$);
\draw[draw=frenchblue, draw opacity=0.5, line width=1mm,out=90,in=-90] ($(l12)!0.50!(r12)$) to ($(l11)!0.75!(r11)$);
\draw[draw=frenchblue, draw opacity=0.5, line width=1mm,out=90,in=0] ($(l13)!0.25!(r13)$) to ($(l13)!0.33!(l12)$);
\draw[draw=frenchblue, draw opacity=0.5, line width=1mm,out=90,in=0] ($(l13)!0.50!(r13)$) to ($(l13)!0.67!(l12)$);
\draw[draw=frenchblue, draw opacity=0.5, line width=1mm,out=90,in=-90] ($(l13)!0.75!(r13)$) to ($(l12)!0.50!(r12)$);
\draw[draw=frenchblue, draw opacity=0.5, line width=1mm,out=180,in=-90] ($(r14)!0.75!(r13)$) to ($(l13)!0.75!(r13)$);
\draw[draw=frenchblue, draw opacity=0.5, line width=1mm,out=180,in=-90] ($(r14)!0.50!(r13)$) to ($(l13)!0.50!(r13)$);
\draw[draw=frenchblue, draw opacity=0.5, line width=1mm,out=180,in=-90] ($(r14)!0.25!(r13)$) to ($(l13)!0.25!(r13)$);
}
 ]
|[alias=l0]| 0 \ar[r,equals] \& |[alias=r0]| 0 
\\ |[alias=l1]|0\ar[u,equals]\ar[r] \& |[alias=r1]|2\ar[u,<-]
\\ |[alias=l2]|2\ar[u,<-]\ar[r,equals] \& |[alias=r2]|2\ar[u,equals]
\\ |[alias=l3]|1\ar[u]\ar[r] \& |[alias=r3]|2\ar[u,equals]
\\ |[alias=l4]|2\ar[u,<-]\ar[r,equals] \& |[alias=r4]|2\ar[u,equals]
\\ |[alias=l5]|0\ar[u]\ar[r] \& |[alias=r5]|2\ar[u,equals]
\\ |[alias=l6]|0\ar[u,equals]\ar[r,equals] \& |[alias=r6]|0\ar[u]
\\ |[alias=l7]|0\ar[u,equals]\ar[r] \& |[alias=r7]|1\ar[u,<-]
\\ |[alias=l8]|0\ar[u,equals]\ar[r,equals] \& |[alias=r8]|0\ar[u]
\\ |[alias=l9]|0\ar[u,equals]\ar[r] \& |[alias=r9]|2\ar[u,<-]
\\ |[alias=l10]|1\ar[u,<-]\ar[r,equals] \& |[alias=r10]|1\ar[u]
\\ |[alias=l11]|0\ar[u]\ar[r] \& |[alias=r11]|3\ar[u,<-]
\\ |[alias=l12]|2\ar[u,<-]\ar[r] \& |[alias=r12]|3\ar[u,equals]
\\ |[alias=l13]|0\ar[u]\ar[r] \& |[alias=r13]|3\ar[u,equals]
\\ |[alias=l14]|0\ar[u,equals]\ar[r,equals] \& |[alias=r14]|0\ar[u]
\arrow[gray,phantom,from=1-1,to=2-2,"\push \Rsym"]
\arrow[gray,phantom,from=2-1,to=3-2,"\push \Lsym"]
\arrow[gray,phantom,from=3-1,to=4-2,"\pull \Lsym"]
\arrow[gray,phantom,from=4-1,to=5-2,"\push \Lsym"]
\arrow[gray,phantom,from=5-1,to=6-2,"\pull \Lsym"]
\arrow[gray,phantom,from=6-1,to=7-2,"\pull \Rsym"]
\arrow[gray,phantom,from=7-1,to=8-2,"\push \Rsym"]
\arrow[gray,phantom,from=8-1,to=9-2,"\pull \Rsym"]
\arrow[gray,phantom,from=9-1,to=10-2,"\push \Rsym"]
\arrow[gray,phantom,from=10-1,to=11-2,"\push \Lsym\pull \Rsym"]
\arrow[gray,phantom,from=11-1,to=12-2,"\pull \Lsym \push \Rsym"]
\arrow[gray,phantom,from=12-1,to=13-2,"\push \Lsym"]
\arrow[gray,phantom,from=13-1,to=14-2,"\pull \Lsym"]
\arrow[gray,phantom,from=14-1,to=15-2,"\pull \Rsym"]
\end{tikzcd}}
\resizebox{!}{7cm}{\begin{tikzcd}[ampersand replacement=\&,row sep=1cm,execute at end picture = { 
\draw[draw=lightgray, draw opacity=0.5, line width=1mm,out=90,in=180] ($(l1)!0.67!(r1)$) to ($(r1)!0.33!(r0)$);
\draw[draw=lightgray, draw opacity=0.5, line width=1mm,out=90,in=180] ($(l1)!0.33!(r1)$) to ($(r1)!0.67!(r0)$);
\draw[draw=lightgray, draw opacity=0.5, line width=1mm,out=180,in=-90] ($(r2)!0.67!(r1)$) to ($(l1)!0.67!(r1)$);
\draw[draw=lightgray, draw opacity=0.5, line width=1mm,out=180,in=-90] ($(r2)!0.33!(r1)$) to ($(l1)!0.33!(r1)$);
\draw[draw=lightgray, draw opacity=0.5, line width=1mm,out=90,in=180] ($(l3)!0.50!(r3)$) to ($(r3)!0.50!(r2)$);
\draw[draw=lightgray, draw opacity=0.5, line width=1mm,out=180,in=-90] ($(r4)!0.50!(r3)$) to ($(l3)!0.50!(r3)$);
\draw[draw=lightgray, draw opacity=0.5, line width=1mm,out=90,in=180] ($(l5)!0.67!(r5)$) to ($(r5)!0.33!(r4)$);
\draw[draw=lightgray, draw opacity=0.5, line width=1mm,out=90,in=180] ($(l5)!0.33!(r5)$) to ($(r5)!0.67!(r4)$);
\draw[draw=lightgray, draw opacity=0.5, line width=1mm,out=90,in=-90] ($(l6)!0.50!(r6)$) to ($(l5)!0.33!(r5)$);
\draw[draw=lightgray, draw opacity=0.5, line width=1mm,out=180,in=-90] ($(r6)!0.50!(r5)$) to ($(l5)!0.67!(r5)$);
\draw[draw=lightgray, draw opacity=0.5, line width=1mm,out=90,in=-90] ($(l7)!0.25!(r7)$) to ($(l6)!0.50!(r6)$);
\draw[draw=lightgray, draw opacity=0.5, line width=1mm,out=90,in=180] ($(l7)!0.75!(r7)$) to ($(r7)!0.33!(r6)$);
\draw[draw=lightgray, draw opacity=0.5, line width=1mm,out=90,in=180] ($(l7)!0.50!(r7)$) to ($(r7)!0.67!(r6)$);
\draw[draw=lightgray, draw opacity=0.5, line width=1mm,out=0,in=-90] ($(l8)!0.67!(l7)$) to ($(l7)!0.25!(r7)$);
\draw[draw=lightgray, draw opacity=0.5, line width=1mm,out=0,in=-90] ($(l8)!0.33!(l7)$) to ($(l7)!0.50!(r7)$);
\draw[draw=lightgray, draw opacity=0.5, line width=1mm,out=90,in=-90] ($(l8)!0.50!(r8)$) to ($(l7)!0.75!(r7)$);
\draw[draw=lightgray, draw opacity=0.5, line width=1mm,out=90,in=0] ($(l9)!0.33!(r9)$) to ($(l9)!0.50!(l8)$);
\draw[draw=lightgray, draw opacity=0.5, line width=1mm,out=90,in=-90] ($(l9)!0.67!(r9)$) to ($(l8)!0.50!(r8)$);
\draw[draw=lightgray, draw opacity=0.5, line width=1mm,out=0,in=-90] ($(l10)!0.50!(l9)$) to ($(l9)!0.33!(r9)$);
\draw[draw=lightgray, draw opacity=0.5, line width=1mm,out=90,in=-90] ($(l10)!0.50!(r10)$) to ($(l9)!0.67!(r9)$);
\draw[draw=lightgray, draw opacity=0.5, line width=1mm,out=90,in=0] ($(l11)!0.25!(r11)$) to ($(l11)!0.33!(l10)$);
\draw[draw=lightgray, draw opacity=0.5, line width=1mm,out=90,in=0] ($(l11)!0.50!(r11)$) to ($(l11)!0.67!(l10)$);
\draw[draw=lightgray, draw opacity=0.5, line width=1mm,out=90,in=-90] ($(l11)!0.75!(r11)$) to ($(l10)!0.50!(r10)$);
\draw[draw=lightgray, draw opacity=0.5, line width=1mm,out=0,in=-90] ($(l12)!0.50!(l11)$) to ($(l11)!0.25!(r11)$);
\draw[draw=lightgray, draw opacity=0.5, line width=1mm,out=90,in=-90] ($(l12)!0.33!(r12)$) to ($(l11)!0.50!(r11)$);
\draw[draw=lightgray, draw opacity=0.5, line width=1mm,out=90,in=-90] ($(l12)!0.67!(r12)$) to ($(l11)!0.75!(r11)$);
\draw[draw=lightgray, draw opacity=0.5, line width=1mm,out=90,in=0] ($(l13)!0.25!(r13)$) to ($(l13)!0.50!(l12)$);
\draw[draw=lightgray, draw opacity=0.5, line width=1mm,out=90,in=-90] ($(l13)!0.50!(r13)$) to ($(l12)!0.33!(r12)$);
\draw[draw=lightgray, draw opacity=0.5, line width=1mm,out=90,in=-90] ($(l13)!0.75!(r13)$) to ($(l12)!0.67!(r12)$);
\draw[draw=lightgray, draw opacity=0.5, line width=1mm,out=0,in=-90] ($(l14)!0.67!(l13)$) to ($(l13)!0.25!(r13)$);
\draw[draw=lightgray, draw opacity=0.5, line width=1mm,out=0,in=-90] ($(l14)!0.33!(l13)$) to ($(l13)!0.50!(r13)$);
\draw[draw=lightgray, draw opacity=0.5, line width=1mm,out=90,in=-90] ($(l14)!0.50!(r14)$) to ($(l13)!0.75!(r13)$);
\draw[draw=lightgray, draw opacity=0.5, line width=1mm,out=90,in=0] ($(l15)!0.25!(r15)$) to ($(l15)!0.33!(l14)$);
\draw[draw=lightgray, draw opacity=0.5, line width=1mm,out=90,in=0] ($(l15)!0.50!(r15)$) to ($(l15)!0.67!(l14)$);
\draw[draw=lightgray, draw opacity=0.5, line width=1mm,out=90,in=-90] ($(l15)!0.75!(r15)$) to ($(l14)!0.50!(r14)$);
\draw[draw=lightgray, draw opacity=0.5, line width=1mm,out=180,in=-90] ($(r16)!0.75!(r15)$) to ($(l15)!0.75!(r15)$);
\draw[draw=lightgray, draw opacity=0.5, line width=1mm,out=180,in=-90] ($(r16)!0.50!(r15)$) to ($(l15)!0.50!(r15)$);
\draw[draw=lightgray, draw opacity=0.5, line width=1mm,out=180,in=-90] ($(r16)!0.25!(r15)$) to ($(l15)!0.25!(r15)$);
}
 ]
|[alias=l0]| 0 \ar[r,equals] \& |[alias=r0]| 0 
\\ |[alias=l1]|0\ar[u,equals]\ar[r] \& |[alias=r1]|2\ar[u,<-]
\\ |[alias=l2]|0\ar[u,equals]\ar[r,equals] \& |[alias=r2]|0\ar[u]
\\ |[alias=l3]|0\ar[u,equals]\ar[r] \& |[alias=r3]|1\ar[u,<-]
\\ |[alias=l4]|0\ar[u,equals]\ar[r,equals] \& |[alias=r4]|0\ar[u]
\\ |[alias=l5]|0\ar[u,equals]\ar[r] \& |[alias=r5]|2\ar[u,<-]
\\ |[alias=l6]|0\ar[u,equals]\ar[r] \& |[alias=r6]|1\ar[u]
\\ |[alias=l7]|0\ar[u,equals]\ar[r] \& |[alias=r7]|3\ar[u,<-]
\\ |[alias=l8]|2\ar[u,<-]\ar[r] \& |[alias=r8]|3\ar[u,equals]
\\ |[alias=l9]|1\ar[u]\ar[r] \& |[alias=r9]|3\ar[u,equals]
\\ |[alias=l10]|2\ar[u,<-]\ar[r] \& |[alias=r10]|3\ar[u,equals]
\\ |[alias=l11]|0\ar[u]\ar[r] \& |[alias=r11]|3\ar[u,equals]
\\ |[alias=l12]|1\ar[u,<-]\ar[r] \& |[alias=r12]|3\ar[u,equals]
\\ |[alias=l13]|0\ar[u]\ar[r] \& |[alias=r13]|3\ar[u,equals]
\\ |[alias=l14]|2\ar[u,<-]\ar[r] \& |[alias=r14]|3\ar[u,equals]
\\ |[alias=l15]|0\ar[u]\ar[r] \& |[alias=r15]|3\ar[u,equals]
\\ |[alias=l16]|0\ar[u,equals]\ar[r,equals] \& |[alias=r16]|0\ar[u]
\arrow[gray,phantom,from=1-1,to=2-2,"\push \Rsym"]
\arrow[gray,phantom,from=2-1,to=3-2,"\pull \Rsym"]
\arrow[gray,phantom,from=3-1,to=4-2,"\push \Rsym"]
\arrow[gray,phantom,from=4-1,to=5-2,"\pull \Rsym"]
\arrow[gray,phantom,from=5-1,to=6-2,"\push \Rsym"]
\arrow[gray,phantom,from=6-1,to=7-2,"\pull \Rsym"]
\arrow[gray,phantom,from=7-1,to=8-2,"\push \Rsym"]
\arrow[gray,phantom,from=8-1,to=9-2,"\push \Lsym"]
\arrow[gray,phantom,from=9-1,to=10-2,"\pull \Lsym"]
\arrow[gray,phantom,from=10-1,to=11-2,"\push \Lsym"]
\arrow[gray,phantom,from=11-1,to=12-2,"\pull \Lsym"]
\arrow[gray,phantom,from=12-1,to=13-2,"\push \Lsym"]
\arrow[gray,phantom,from=13-1,to=14-2,"\pull \Lsym"]
\arrow[gray,phantom,from=14-1,to=15-2,"\push \Lsym"]
\arrow[gray,phantom,from=15-1,to=16-2,"\pull \Lsym"]
\arrow[gray,phantom,from=16-1,to=17-2,"\pull \Rsym"]
\end{tikzcd}}
\resizebox{!}{7cm}{\begin{tikzcd}[ampersand replacement=\&,row sep=1cm,execute at end picture = { 
\draw[draw=lightgray, draw opacity=0.5, line width=1mm,out=90,in=180] ($(l1)!0.67!(r1)$) to ($(r1)!0.33!(r0)$);
\draw[draw=lightgray, draw opacity=0.5, line width=1mm,out=90,in=180] ($(l1)!0.33!(r1)$) to ($(r1)!0.67!(r0)$);
\draw[draw=lightgray, draw opacity=0.5, line width=1mm,out=180,in=-90] ($(r2)!0.67!(r1)$) to ($(l1)!0.67!(r1)$);
\draw[draw=lightgray, draw opacity=0.5, line width=1mm,out=180,in=-90] ($(r2)!0.33!(r1)$) to ($(l1)!0.33!(r1)$);
\draw[draw=lightgray, draw opacity=0.5, line width=1mm,out=90,in=180] ($(l3)!0.50!(r3)$) to ($(r3)!0.50!(r2)$);
\draw[draw=lightgray, draw opacity=0.5, line width=1mm,out=180,in=-90] ($(r4)!0.50!(r3)$) to ($(l3)!0.50!(r3)$);
\draw[draw=lightgray, draw opacity=0.5, line width=1mm,out=90,in=180] ($(l5)!0.67!(r5)$) to ($(r5)!0.33!(r4)$);
\draw[draw=lightgray, draw opacity=0.5, line width=1mm,out=90,in=180] ($(l5)!0.33!(r5)$) to ($(r5)!0.67!(r4)$);
\draw[draw=lightgray, draw opacity=0.5, line width=1mm,out=0,in=-90] ($(l6)!0.67!(l5)$) to ($(l5)!0.33!(r5)$);
\draw[draw=lightgray, draw opacity=0.5, line width=1mm,out=0,in=-90] ($(l6)!0.33!(l5)$) to ($(l5)!0.67!(r5)$);
\draw[draw=lightgray, draw opacity=0.5, line width=1mm,out=90,in=0] ($(l7)!0.50!(r7)$) to ($(l7)!0.50!(l6)$);
\draw[draw=lightgray, draw opacity=0.5, line width=1mm,out=180,in=-90] ($(r8)!0.50!(r7)$) to ($(l7)!0.50!(r7)$);
\draw[draw=lightgray, draw opacity=0.5, line width=1mm,out=90,in=180] ($(l9)!0.67!(r9)$) to ($(r9)!0.33!(r8)$);
\draw[draw=lightgray, draw opacity=0.5, line width=1mm,out=90,in=180] ($(l9)!0.33!(r9)$) to ($(r9)!0.67!(r8)$);
\draw[draw=lightgray, draw opacity=0.5, line width=1mm,out=0,in=-90] ($(l10)!0.50!(l9)$) to ($(l9)!0.33!(r9)$);
\draw[draw=lightgray, draw opacity=0.5, line width=1mm,out=90,in=-90] ($(l10)!0.50!(r10)$) to ($(l9)!0.67!(r9)$);
\draw[draw=lightgray, draw opacity=0.5, line width=1mm,out=90,in=0] ($(l11)!0.25!(r11)$) to ($(l11)!0.33!(l10)$);
\draw[draw=lightgray, draw opacity=0.5, line width=1mm,out=90,in=0] ($(l11)!0.50!(r11)$) to ($(l11)!0.67!(l10)$);
\draw[draw=lightgray, draw opacity=0.5, line width=1mm,out=90,in=-90] ($(l11)!0.75!(r11)$) to ($(l10)!0.50!(r10)$);
\draw[draw=lightgray, draw opacity=0.5, line width=1mm,out=0,in=-90] ($(l12)!0.50!(l11)$) to ($(l11)!0.25!(r11)$);
\draw[draw=lightgray, draw opacity=0.5, line width=1mm,out=90,in=-90] ($(l12)!0.33!(r12)$) to ($(l11)!0.50!(r11)$);
\draw[draw=lightgray, draw opacity=0.5, line width=1mm,out=90,in=-90] ($(l12)!0.67!(r12)$) to ($(l11)!0.75!(r11)$);
\draw[draw=lightgray, draw opacity=0.5, line width=1mm,out=90,in=0] ($(l13)!0.25!(r13)$) to ($(l13)!0.50!(l12)$);
\draw[draw=lightgray, draw opacity=0.5, line width=1mm,out=90,in=-90] ($(l13)!0.50!(r13)$) to ($(l12)!0.33!(r12)$);
\draw[draw=lightgray, draw opacity=0.5, line width=1mm,out=90,in=-90] ($(l13)!0.75!(r13)$) to ($(l12)!0.67!(r12)$);
\draw[draw=lightgray, draw opacity=0.5, line width=1mm,out=0,in=-90] ($(l14)!0.67!(l13)$) to ($(l13)!0.25!(r13)$);
\draw[draw=lightgray, draw opacity=0.5, line width=1mm,out=0,in=-90] ($(l14)!0.33!(l13)$) to ($(l13)!0.50!(r13)$);
\draw[draw=lightgray, draw opacity=0.5, line width=1mm,out=90,in=-90] ($(l14)!0.50!(r14)$) to ($(l13)!0.75!(r13)$);
\draw[draw=lightgray, draw opacity=0.5, line width=1mm,out=90,in=0] ($(l15)!0.25!(r15)$) to ($(l15)!0.33!(l14)$);
\draw[draw=lightgray, draw opacity=0.5, line width=1mm,out=90,in=0] ($(l15)!0.50!(r15)$) to ($(l15)!0.67!(l14)$);
\draw[draw=lightgray, draw opacity=0.5, line width=1mm,out=90,in=-90] ($(l15)!0.75!(r15)$) to ($(l14)!0.50!(r14)$);
\draw[draw=lightgray, draw opacity=0.5, line width=1mm,out=180,in=-90] ($(r16)!0.75!(r15)$) to ($(l15)!0.75!(r15)$);
\draw[draw=lightgray, draw opacity=0.5, line width=1mm,out=180,in=-90] ($(r16)!0.50!(r15)$) to ($(l15)!0.50!(r15)$);
\draw[draw=lightgray, draw opacity=0.5, line width=1mm,out=180,in=-90] ($(r16)!0.25!(r15)$) to ($(l15)!0.25!(r15)$);
}
 ]
|[alias=l0]| 0 \ar[r,equals] \& |[alias=r0]| 0 
\\ |[alias=l1]|0\ar[u,equals]\ar[r] \& |[alias=r1]|2\ar[u,<-]
\\ |[alias=l2]|0\ar[u,equals]\ar[r,equals] \& |[alias=r2]|0\ar[u]
\\ |[alias=l3]|0\ar[u,equals]\ar[r] \& |[alias=r3]|1\ar[u,<-]
\\ |[alias=l4]|0\ar[u,equals]\ar[r,equals] \& |[alias=r4]|0\ar[u]
\\ |[alias=l5]|0\ar[u,equals]\ar[r] \& |[alias=r5]|2\ar[u,<-]
\\ |[alias=l6]|2\ar[u,<-]\ar[r,equals] \& |[alias=r6]|2\ar[u,equals]
\\ |[alias=l7]|1\ar[u]\ar[r] \& |[alias=r7]|2\ar[u,equals]
\\ |[alias=l8]|1\ar[u,equals]\ar[r,equals] \& |[alias=r8]|1\ar[u]
\\ |[alias=l9]|1\ar[u,equals]\ar[r] \& |[alias=r9]|3\ar[u,<-]
\\ |[alias=l10]|2\ar[u,<-]\ar[r] \& |[alias=r10]|3\ar[u,equals]
\\ |[alias=l11]|0\ar[u]\ar[r] \& |[alias=r11]|3\ar[u,equals]
\\ |[alias=l12]|1\ar[u,<-]\ar[r] \& |[alias=r12]|3\ar[u,equals]
\\ |[alias=l13]|0\ar[u]\ar[r] \& |[alias=r13]|3\ar[u,equals]
\\ |[alias=l14]|2\ar[u,<-]\ar[r] \& |[alias=r14]|3\ar[u,equals]
\\ |[alias=l15]|0\ar[u]\ar[r] \& |[alias=r15]|3\ar[u,equals]
\\ |[alias=l16]|0\ar[u,equals]\ar[r,equals] \& |[alias=r16]|0\ar[u]
\arrow[gray,phantom,from=1-1,to=2-2,"\push \Rsym"]
\arrow[gray,phantom,from=2-1,to=3-2,"\pull \Rsym"]
\arrow[gray,phantom,from=3-1,to=4-2,"\push \Rsym"]
\arrow[gray,phantom,from=4-1,to=5-2,"\pull \Rsym"]
\arrow[gray,phantom,from=5-1,to=6-2,"\push \Rsym"]
\arrow[gray,phantom,from=6-1,to=7-2,"\push \Lsym"]
\arrow[gray,phantom,from=7-1,to=8-2,"\pull \Lsym"]
\arrow[gray,phantom,from=8-1,to=9-2,"\pull \Rsym"]
\arrow[gray,phantom,from=9-1,to=10-2,"\push \Rsym"]
\arrow[gray,phantom,from=10-1,to=11-2,"\push \Lsym"]
\arrow[gray,phantom,from=11-1,to=12-2,"\pull \Lsym"]
\arrow[gray,phantom,from=12-1,to=13-2,"\push \Lsym"]
\arrow[gray,phantom,from=13-1,to=14-2,"\pull \Lsym"]
\arrow[gray,phantom,from=14-1,to=15-2,"\push \Lsym"]
\arrow[gray,phantom,from=15-1,to=16-2,"\pull \Lsym"]
\arrow[gray,phantom,from=16-1,to=17-2,"\pull \Rsym"]
\end{tikzcd}}
\resizebox{!}{7cm}{\begin{tikzcd}[ampersand replacement=\&,row sep=1cm,execute at end picture = { 
\draw[draw=lightgray, draw opacity=0.5, line width=1mm,out=90,in=180] ($(l1)!0.67!(r1)$) to ($(r1)!0.33!(r0)$);
\draw[draw=lightgray, draw opacity=0.5, line width=1mm,out=90,in=180] ($(l1)!0.33!(r1)$) to ($(r1)!0.67!(r0)$);
\draw[draw=lightgray, draw opacity=0.5, line width=1mm,out=0,in=-90] ($(l2)!0.67!(l1)$) to ($(l1)!0.33!(r1)$);
\draw[draw=lightgray, draw opacity=0.5, line width=1mm,out=0,in=-90] ($(l2)!0.33!(l1)$) to ($(l1)!0.67!(r1)$);
\draw[draw=lightgray, draw opacity=0.5, line width=1mm,out=90,in=0] ($(l3)!0.50!(r3)$) to ($(l3)!0.50!(l2)$);
\draw[draw=lightgray, draw opacity=0.5, line width=1mm,out=0,in=-90] ($(l4)!0.50!(l3)$) to ($(l3)!0.50!(r3)$);
\draw[draw=lightgray, draw opacity=0.5, line width=1mm,out=90,in=0] ($(l5)!0.33!(r5)$) to ($(l5)!0.33!(l4)$);
\draw[draw=lightgray, draw opacity=0.5, line width=1mm,out=90,in=0] ($(l5)!0.67!(r5)$) to ($(l5)!0.67!(l4)$);
\draw[draw=lightgray, draw opacity=0.5, line width=1mm,out=180,in=-90] ($(r6)!0.67!(r5)$) to ($(l5)!0.67!(r5)$);
\draw[draw=lightgray, draw opacity=0.5, line width=1mm,out=180,in=-90] ($(r6)!0.33!(r5)$) to ($(l5)!0.33!(r5)$);
\draw[draw=lightgray, draw opacity=0.5, line width=1mm,out=90,in=180] ($(l7)!0.50!(r7)$) to ($(r7)!0.50!(r6)$);
\draw[draw=lightgray, draw opacity=0.5, line width=1mm,out=0,in=-90] ($(l8)!0.50!(l7)$) to ($(l7)!0.50!(r7)$);
\draw[draw=lightgray, draw opacity=0.5, line width=1mm,out=90,in=0] ($(l9)!0.50!(r9)$) to ($(l9)!0.50!(l8)$);
\draw[draw=lightgray, draw opacity=0.5, line width=1mm,out=180,in=-90] ($(r10)!0.50!(r9)$) to ($(l9)!0.50!(r9)$);
\draw[draw=lightgray, draw opacity=0.5, line width=1mm,out=90,in=180] ($(l11)!0.67!(r11)$) to ($(r11)!0.33!(r10)$);
\draw[draw=lightgray, draw opacity=0.5, line width=1mm,out=90,in=180] ($(l11)!0.33!(r11)$) to ($(r11)!0.67!(r10)$);
\draw[draw=lightgray, draw opacity=0.5, line width=1mm,out=90,in=-90] ($(l12)!0.50!(r12)$) to ($(l11)!0.33!(r11)$);
\draw[draw=lightgray, draw opacity=0.5, line width=1mm,out=180,in=-90] ($(r12)!0.50!(r11)$) to ($(l11)!0.67!(r11)$);
\draw[draw=lightgray, draw opacity=0.5, line width=1mm,out=90,in=-90] ($(l13)!0.25!(r13)$) to ($(l12)!0.50!(r12)$);
\draw[draw=lightgray, draw opacity=0.5, line width=1mm,out=90,in=180] ($(l13)!0.75!(r13)$) to ($(r13)!0.33!(r12)$);
\draw[draw=lightgray, draw opacity=0.5, line width=1mm,out=90,in=180] ($(l13)!0.50!(r13)$) to ($(r13)!0.67!(r12)$);
\draw[draw=lightgray, draw opacity=0.5, line width=1mm,out=0,in=-90] ($(l14)!0.67!(l13)$) to ($(l13)!0.25!(r13)$);
\draw[draw=lightgray, draw opacity=0.5, line width=1mm,out=0,in=-90] ($(l14)!0.33!(l13)$) to ($(l13)!0.50!(r13)$);
\draw[draw=lightgray, draw opacity=0.5, line width=1mm,out=90,in=-90] ($(l14)!0.50!(r14)$) to ($(l13)!0.75!(r13)$);
\draw[draw=lightgray, draw opacity=0.5, line width=1mm,out=90,in=0] ($(l15)!0.25!(r15)$) to ($(l15)!0.33!(l14)$);
\draw[draw=lightgray, draw opacity=0.5, line width=1mm,out=90,in=0] ($(l15)!0.50!(r15)$) to ($(l15)!0.67!(l14)$);
\draw[draw=lightgray, draw opacity=0.5, line width=1mm,out=90,in=-90] ($(l15)!0.75!(r15)$) to ($(l14)!0.50!(r14)$);
\draw[draw=lightgray, draw opacity=0.5, line width=1mm,out=180,in=-90] ($(r16)!0.75!(r15)$) to ($(l15)!0.75!(r15)$);
\draw[draw=lightgray, draw opacity=0.5, line width=1mm,out=180,in=-90] ($(r16)!0.50!(r15)$) to ($(l15)!0.50!(r15)$);
\draw[draw=lightgray, draw opacity=0.5, line width=1mm,out=180,in=-90] ($(r16)!0.25!(r15)$) to ($(l15)!0.25!(r15)$);
}
 ]
|[alias=l0]| 0 \ar[r,equals] \& |[alias=r0]| 0 
\\ |[alias=l1]|0\ar[u,equals]\ar[r] \& |[alias=r1]|2\ar[u,<-]
\\ |[alias=l2]|2\ar[u,<-]\ar[r,equals] \& |[alias=r2]|2\ar[u,equals]
\\ |[alias=l3]|1\ar[u]\ar[r] \& |[alias=r3]|2\ar[u,equals]
\\ |[alias=l4]|2\ar[u,<-]\ar[r,equals] \& |[alias=r4]|2\ar[u,equals]
\\ |[alias=l5]|0\ar[u]\ar[r] \& |[alias=r5]|2\ar[u,equals]
\\ |[alias=l6]|0\ar[u,equals]\ar[r,equals] \& |[alias=r6]|0\ar[u]
\\ |[alias=l7]|0\ar[u,equals]\ar[r] \& |[alias=r7]|1\ar[u,<-]
\\ |[alias=l8]|1\ar[u,<-]\ar[r,equals] \& |[alias=r8]|1\ar[u,equals]
\\ |[alias=l9]|0\ar[u]\ar[r] \& |[alias=r9]|1\ar[u,equals]
\\ |[alias=l10]|0\ar[u,equals]\ar[r,equals] \& |[alias=r10]|0\ar[u]
\\ |[alias=l11]|0\ar[u,equals]\ar[r] \& |[alias=r11]|2\ar[u,<-]
\\ |[alias=l12]|0\ar[u,equals]\ar[r] \& |[alias=r12]|1\ar[u]
\\ |[alias=l13]|0\ar[u,equals]\ar[r] \& |[alias=r13]|3\ar[u,<-]
\\ |[alias=l14]|2\ar[u,<-]\ar[r] \& |[alias=r14]|3\ar[u,equals]
\\ |[alias=l15]|0\ar[u]\ar[r] \& |[alias=r15]|3\ar[u,equals]
\\ |[alias=l16]|0\ar[u,equals]\ar[r,equals] \& |[alias=r16]|0\ar[u]
\arrow[gray,phantom,from=1-1,to=2-2,"\push \Rsym"]
\arrow[gray,phantom,from=2-1,to=3-2,"\push \Lsym"]
\arrow[gray,phantom,from=3-1,to=4-2,"\pull \Lsym"]
\arrow[gray,phantom,from=4-1,to=5-2,"\push \Lsym"]
\arrow[gray,phantom,from=5-1,to=6-2,"\pull \Lsym"]
\arrow[gray,phantom,from=6-1,to=7-2,"\pull \Rsym"]
\arrow[gray,phantom,from=7-1,to=8-2,"\push \Rsym"]
\arrow[gray,phantom,from=8-1,to=9-2,"\push \Lsym"]
\arrow[gray,phantom,from=9-1,to=10-2,"\pull \Lsym"]
\arrow[gray,phantom,from=10-1,to=11-2,"\pull \Rsym"]
\arrow[gray,phantom,from=11-1,to=12-2,"\push \Rsym"]
\arrow[gray,phantom,from=12-1,to=13-2,"\pull \Rsym"]
\arrow[gray,phantom,from=13-1,to=14-2,"\push \Rsym"]
\arrow[gray,phantom,from=14-1,to=15-2,"\push \Lsym"]
\arrow[gray,phantom,from=15-1,to=16-2,"\pull \Lsym"]
\arrow[gray,phantom,from=16-1,to=17-2,"\pull \Rsym"]
\end{tikzcd}}
\resizebox{!}{7cm}{\begin{tikzcd}[ampersand replacement=\&,row sep=1cm,execute at end picture = { 
\draw[draw=lightgray, draw opacity=0.5, line width=1mm,out=90,in=180] ($(l1)!0.67!(r1)$) to ($(r1)!0.33!(r0)$);
\draw[draw=lightgray, draw opacity=0.5, line width=1mm,out=90,in=180] ($(l1)!0.33!(r1)$) to ($(r1)!0.67!(r0)$);
\draw[draw=lightgray, draw opacity=0.5, line width=1mm,out=0,in=-90] ($(l2)!0.67!(l1)$) to ($(l1)!0.33!(r1)$);
\draw[draw=lightgray, draw opacity=0.5, line width=1mm,out=0,in=-90] ($(l2)!0.33!(l1)$) to ($(l1)!0.67!(r1)$);
\draw[draw=lightgray, draw opacity=0.5, line width=1mm,out=90,in=0] ($(l3)!0.50!(r3)$) to ($(l3)!0.50!(l2)$);
\draw[draw=lightgray, draw opacity=0.5, line width=1mm,out=0,in=-90] ($(l4)!0.50!(l3)$) to ($(l3)!0.50!(r3)$);
\draw[draw=lightgray, draw opacity=0.5, line width=1mm,out=90,in=0] ($(l5)!0.33!(r5)$) to ($(l5)!0.33!(l4)$);
\draw[draw=lightgray, draw opacity=0.5, line width=1mm,out=90,in=0] ($(l5)!0.67!(r5)$) to ($(l5)!0.67!(l4)$);
\draw[draw=lightgray, draw opacity=0.5, line width=1mm,out=0,in=-90] ($(l6)!0.50!(l5)$) to ($(l5)!0.33!(r5)$);
\draw[draw=lightgray, draw opacity=0.5, line width=1mm,out=90,in=-90] ($(l6)!0.50!(r6)$) to ($(l5)!0.67!(r5)$);
\draw[draw=lightgray, draw opacity=0.5, line width=1mm,out=90,in=0] ($(l7)!0.33!(r7)$) to ($(l7)!0.50!(l6)$);
\draw[draw=lightgray, draw opacity=0.5, line width=1mm,out=90,in=-90] ($(l7)!0.67!(r7)$) to ($(l6)!0.50!(r6)$);
\draw[draw=lightgray, draw opacity=0.5, line width=1mm,out=180,in=-90] ($(r8)!0.67!(r7)$) to ($(l7)!0.67!(r7)$);
\draw[draw=lightgray, draw opacity=0.5, line width=1mm,out=180,in=-90] ($(r8)!0.33!(r7)$) to ($(l7)!0.33!(r7)$);
\draw[draw=lightgray, draw opacity=0.5, line width=1mm,out=90,in=180] ($(l9)!0.50!(r9)$) to ($(r9)!0.50!(r8)$);
\draw[draw=lightgray, draw opacity=0.5, line width=1mm,out=180,in=-90] ($(r10)!0.50!(r9)$) to ($(l9)!0.50!(r9)$);
\draw[draw=lightgray, draw opacity=0.5, line width=1mm,out=90,in=180] ($(l11)!0.67!(r11)$) to ($(r11)!0.33!(r10)$);
\draw[draw=lightgray, draw opacity=0.5, line width=1mm,out=90,in=180] ($(l11)!0.33!(r11)$) to ($(r11)!0.67!(r10)$);
\draw[draw=lightgray, draw opacity=0.5, line width=1mm,out=90,in=-90] ($(l12)!0.50!(r12)$) to ($(l11)!0.33!(r11)$);
\draw[draw=lightgray, draw opacity=0.5, line width=1mm,out=180,in=-90] ($(r12)!0.50!(r11)$) to ($(l11)!0.67!(r11)$);
\draw[draw=lightgray, draw opacity=0.5, line width=1mm,out=90,in=-90] ($(l13)!0.25!(r13)$) to ($(l12)!0.50!(r12)$);
\draw[draw=lightgray, draw opacity=0.5, line width=1mm,out=90,in=180] ($(l13)!0.75!(r13)$) to ($(r13)!0.33!(r12)$);
\draw[draw=lightgray, draw opacity=0.5, line width=1mm,out=90,in=180] ($(l13)!0.50!(r13)$) to ($(r13)!0.67!(r12)$);
\draw[draw=lightgray, draw opacity=0.5, line width=1mm,out=0,in=-90] ($(l14)!0.67!(l13)$) to ($(l13)!0.25!(r13)$);
\draw[draw=lightgray, draw opacity=0.5, line width=1mm,out=0,in=-90] ($(l14)!0.33!(l13)$) to ($(l13)!0.50!(r13)$);
\draw[draw=lightgray, draw opacity=0.5, line width=1mm,out=90,in=-90] ($(l14)!0.50!(r14)$) to ($(l13)!0.75!(r13)$);
\draw[draw=lightgray, draw opacity=0.5, line width=1mm,out=90,in=0] ($(l15)!0.25!(r15)$) to ($(l15)!0.33!(l14)$);
\draw[draw=lightgray, draw opacity=0.5, line width=1mm,out=90,in=0] ($(l15)!0.50!(r15)$) to ($(l15)!0.67!(l14)$);
\draw[draw=lightgray, draw opacity=0.5, line width=1mm,out=90,in=-90] ($(l15)!0.75!(r15)$) to ($(l14)!0.50!(r14)$);
\draw[draw=lightgray, draw opacity=0.5, line width=1mm,out=180,in=-90] ($(r16)!0.75!(r15)$) to ($(l15)!0.75!(r15)$);
\draw[draw=lightgray, draw opacity=0.5, line width=1mm,out=180,in=-90] ($(r16)!0.50!(r15)$) to ($(l15)!0.50!(r15)$);
\draw[draw=lightgray, draw opacity=0.5, line width=1mm,out=180,in=-90] ($(r16)!0.25!(r15)$) to ($(l15)!0.25!(r15)$);
}
 ]
|[alias=l0]| 0 \ar[r,equals] \& |[alias=r0]| 0 
\\ |[alias=l1]|0\ar[u,equals]\ar[r] \& |[alias=r1]|2\ar[u,<-]
\\ |[alias=l2]|2\ar[u,<-]\ar[r,equals] \& |[alias=r2]|2\ar[u,equals]
\\ |[alias=l3]|1\ar[u]\ar[r] \& |[alias=r3]|2\ar[u,equals]
\\ |[alias=l4]|2\ar[u,<-]\ar[r,equals] \& |[alias=r4]|2\ar[u,equals]
\\ |[alias=l5]|0\ar[u]\ar[r] \& |[alias=r5]|2\ar[u,equals]
\\ |[alias=l6]|1\ar[u,<-]\ar[r] \& |[alias=r6]|2\ar[u,equals]
\\ |[alias=l7]|0\ar[u]\ar[r] \& |[alias=r7]|2\ar[u,equals]
\\ |[alias=l8]|0\ar[u,equals]\ar[r,equals] \& |[alias=r8]|0\ar[u]
\\ |[alias=l9]|0\ar[u,equals]\ar[r] \& |[alias=r9]|1\ar[u,<-]
\\ |[alias=l10]|0\ar[u,equals]\ar[r,equals] \& |[alias=r10]|0\ar[u]
\\ |[alias=l11]|0\ar[u,equals]\ar[r] \& |[alias=r11]|2\ar[u,<-]
\\ |[alias=l12]|0\ar[u,equals]\ar[r] \& |[alias=r12]|1\ar[u]
\\ |[alias=l13]|0\ar[u,equals]\ar[r] \& |[alias=r13]|3\ar[u,<-]
\\ |[alias=l14]|2\ar[u,<-]\ar[r] \& |[alias=r14]|3\ar[u,equals]
\\ |[alias=l15]|0\ar[u]\ar[r] \& |[alias=r15]|3\ar[u,equals]
\\ |[alias=l16]|0\ar[u,equals]\ar[r,equals] \& |[alias=r16]|0\ar[u]
\arrow[gray,phantom,from=1-1,to=2-2,"\push \Rsym"]
\arrow[gray,phantom,from=2-1,to=3-2,"\push \Lsym"]
\arrow[gray,phantom,from=3-1,to=4-2,"\pull \Lsym"]
\arrow[gray,phantom,from=4-1,to=5-2,"\push \Lsym"]
\arrow[gray,phantom,from=5-1,to=6-2,"\pull \Lsym"]
\arrow[gray,phantom,from=6-1,to=7-2,"\push \Lsym"]
\arrow[gray,phantom,from=7-1,to=8-2,"\pull \Lsym"]
\arrow[gray,phantom,from=8-1,to=9-2,"\pull \Rsym"]
\arrow[gray,phantom,from=9-1,to=10-2,"\push \Rsym"]
\arrow[gray,phantom,from=10-1,to=11-2,"\pull \Rsym"]
\arrow[gray,phantom,from=11-1,to=12-2,"\push \Rsym"]
\arrow[gray,phantom,from=12-1,to=13-2,"\pull \Rsym"]
\arrow[gray,phantom,from=13-1,to=14-2,"\push \Rsym"]
\arrow[gray,phantom,from=14-1,to=15-2,"\push \Lsym"]
\arrow[gray,phantom,from=15-1,to=16-2,"\pull \Lsym"]
\arrow[gray,phantom,from=16-1,to=17-2,"\pull \Rsym"]
\end{tikzcd}}
\resizebox{!}{7cm}{\begin{tikzcd}[ampersand replacement=\&,row sep=1cm,execute at end picture = { 
\draw[draw=persianred, draw opacity=0.5, line width=1mm,out=90,in=180] ($(l1)!0.67!(r1)$) to ($(r1)!0.33!(r0)$);
\draw[draw=persianred, draw opacity=0.5, line width=1mm,out=90,in=180] ($(l1)!0.33!(r1)$) to ($(r1)!0.67!(r0)$);
\draw[draw=persianred, draw opacity=0.5, line width=1mm,out=180,in=-90] ($(r2)!0.67!(r1)$) to ($(l1)!0.67!(r1)$);
\draw[draw=persianred, draw opacity=0.5, line width=1mm,out=180,in=-90] ($(r2)!0.33!(r1)$) to ($(l1)!0.33!(r1)$);
\draw[draw=persianred, draw opacity=0.5, line width=1mm,out=90,in=180] ($(l3)!0.50!(r3)$) to ($(r3)!0.50!(r2)$);
\draw[draw=persianred, draw opacity=0.5, line width=1mm,out=180,in=-90] ($(r4)!0.50!(r3)$) to ($(l3)!0.50!(r3)$);
\draw[draw=persianred, draw opacity=0.5, line width=1mm,out=90,in=180] ($(l5)!0.67!(r5)$) to ($(r5)!0.33!(r4)$);
\draw[draw=persianred, draw opacity=0.5, line width=1mm,out=90,in=180] ($(l5)!0.33!(r5)$) to ($(r5)!0.67!(r4)$);
\draw[draw=persianred, draw opacity=0.5, line width=1mm,out=0,in=-90] ($(l6)!0.67!(l5)$) to ($(l5)!0.33!(r5)$);
\draw[draw=persianred, draw opacity=0.5, line width=1mm,out=0,in=-90] ($(l6)!0.33!(l5)$) to ($(l5)!0.67!(r5)$);
\draw[draw=persianred, draw opacity=0.5, line width=1mm,out=90,in=0] ($(l7)!0.50!(r7)$) to ($(l7)!0.50!(l6)$);
\draw[draw=persianred, draw opacity=0.5, line width=1mm,out=0,in=-90] ($(l8)!0.50!(l7)$) to ($(l7)!0.50!(r7)$);
\draw[draw=persianred, draw opacity=0.5, line width=1mm,out=90,in=0] ($(l9)!0.33!(r9)$) to ($(l9)!0.33!(l8)$);
\draw[draw=persianred, draw opacity=0.5, line width=1mm,out=90,in=0] ($(l9)!0.67!(r9)$) to ($(l9)!0.67!(l8)$);
\draw[draw=persianred, draw opacity=0.5, line width=1mm,out=0,in=-90] ($(l10)!0.50!(l9)$) to ($(l9)!0.33!(r9)$);
\draw[draw=persianred, draw opacity=0.5, line width=1mm,out=90,in=-90] ($(l10)!0.50!(r10)$) to ($(l9)!0.67!(r9)$);
\draw[draw=persianred, draw opacity=0.5, line width=1mm,out=90,in=0] ($(l11)!0.33!(r11)$) to ($(l11)!0.50!(l10)$);
\draw[draw=persianred, draw opacity=0.5, line width=1mm,out=90,in=-90] ($(l11)!0.67!(r11)$) to ($(l10)!0.50!(r10)$);
\draw[draw=persianred, draw opacity=0.5, line width=1mm,out=0,in=-90] ($(l12)!0.67!(l11)$) to ($(l11)!0.33!(r11)$);
\draw[draw=persianred, draw opacity=0.5, line width=1mm,out=0,in=-90] ($(l12)!0.33!(l11)$) to ($(l11)!0.67!(r11)$);
\draw[draw=persianred, draw opacity=0.5, line width=1mm,out=90,in=0] ($(l13)!0.33!(r13)$) to ($(l13)!0.33!(l12)$);
\draw[draw=persianred, draw opacity=0.5, line width=1mm,out=90,in=0] ($(l13)!0.67!(r13)$) to ($(l13)!0.67!(l12)$);
\draw[draw=persianred, draw opacity=0.5, line width=1mm,out=90,in=-90] ($(l14)!0.50!(r14)$) to ($(l13)!0.33!(r13)$);
\draw[draw=persianred, draw opacity=0.5, line width=1mm,out=180,in=-90] ($(r14)!0.50!(r13)$) to ($(l13)!0.67!(r13)$);
\draw[draw=persianred, draw opacity=0.5, line width=1mm,out=90,in=-90] ($(l15)!0.25!(r15)$) to ($(l14)!0.50!(r14)$);
\draw[draw=persianred, draw opacity=0.5, line width=1mm,out=90,in=180] ($(l15)!0.75!(r15)$) to ($(r15)!0.33!(r14)$);
\draw[draw=persianred, draw opacity=0.5, line width=1mm,out=90,in=180] ($(l15)!0.50!(r15)$) to ($(r15)!0.67!(r14)$);
\draw[draw=persianred, draw opacity=0.5, line width=1mm,out=180,in=-90] ($(r16)!0.75!(r15)$) to ($(l15)!0.75!(r15)$);
\draw[draw=persianred, draw opacity=0.5, line width=1mm,out=180,in=-90] ($(r16)!0.50!(r15)$) to ($(l15)!0.50!(r15)$);
\draw[draw=persianred, draw opacity=0.5, line width=1mm,out=180,in=-90] ($(r16)!0.25!(r15)$) to ($(l15)!0.25!(r15)$);
}
 ]
|[alias=l0]| 0 \ar[r,equals] \& |[alias=r0]| 0 
\\ |[alias=l1]|0\ar[u,equals]\ar[r] \& |[alias=r1]|2\ar[u,<-]
\\ |[alias=l2]|0\ar[u,equals]\ar[r,equals] \& |[alias=r2]|0\ar[u]
\\ |[alias=l3]|0\ar[u,equals]\ar[r] \& |[alias=r3]|1\ar[u,<-]
\\ |[alias=l4]|0\ar[u,equals]\ar[r,equals] \& |[alias=r4]|0\ar[u]
\\ |[alias=l5]|0\ar[u,equals]\ar[r] \& |[alias=r5]|2\ar[u,<-]
\\ |[alias=l6]|2\ar[u,<-]\ar[r,equals] \& |[alias=r6]|2\ar[u,equals]
\\ |[alias=l7]|1\ar[u]\ar[r] \& |[alias=r7]|2\ar[u,equals]
\\ |[alias=l8]|2\ar[u,<-]\ar[r,equals] \& |[alias=r8]|2\ar[u,equals]
\\ |[alias=l9]|0\ar[u]\ar[r] \& |[alias=r9]|2\ar[u,equals]
\\ |[alias=l10]|1\ar[u,<-]\ar[r] \& |[alias=r10]|2\ar[u,equals]
\\ |[alias=l11]|0\ar[u]\ar[r] \& |[alias=r11]|2\ar[u,equals]
\\ |[alias=l12]|2\ar[u,<-]\ar[r,equals] \& |[alias=r12]|2\ar[u,equals]
\\ |[alias=l13]|0\ar[u]\ar[r] \& |[alias=r13]|2\ar[u,equals]
\\ |[alias=l14]|0\ar[u,equals]\ar[r] \& |[alias=r14]|1\ar[u]
\\ |[alias=l15]|0\ar[u,equals]\ar[r] \& |[alias=r15]|3\ar[u,<-]
\\ |[alias=l16]|0\ar[u,equals]\ar[r,equals] \& |[alias=r16]|0\ar[u]
\arrow[gray,phantom,from=1-1,to=2-2,"\push \Rsym"]
\arrow[gray,phantom,from=2-1,to=3-2,"\pull \Rsym"]
\arrow[gray,phantom,from=3-1,to=4-2,"\push \Rsym"]
\arrow[gray,phantom,from=4-1,to=5-2,"\pull \Rsym"]
\arrow[gray,phantom,from=5-1,to=6-2,"\push \Rsym"]
\arrow[gray,phantom,from=6-1,to=7-2,"\push \Lsym"]
\arrow[gray,phantom,from=7-1,to=8-2,"\pull \Lsym"]
\arrow[gray,phantom,from=8-1,to=9-2,"\push \Lsym"]
\arrow[gray,phantom,from=9-1,to=10-2,"\pull \Lsym"]
\arrow[gray,phantom,from=10-1,to=11-2,"\push \Lsym"]
\arrow[gray,phantom,from=11-1,to=12-2,"\pull \Lsym"]
\arrow[gray,phantom,from=12-1,to=13-2,"\push \Lsym"]
\arrow[gray,phantom,from=13-1,to=14-2,"\pull \Lsym"]
\arrow[gray,phantom,from=14-1,to=15-2,"\pull \Rsym"]
\arrow[gray,phantom,from=15-1,to=16-2,"\push \Rsym"]
\arrow[gray,phantom,from=16-1,to=17-2,"\pull \Rsym"]
\end{tikzcd}}
\resizebox{!}{7cm}{\begin{tikzcd}[ampersand replacement=\&,row sep=1cm,execute at end picture = { 
\draw[draw=lightgray, draw opacity=0.5, line width=1mm,out=90,in=180] ($(l1)!0.67!(r1)$) to ($(r1)!0.33!(r0)$);
\draw[draw=lightgray, draw opacity=0.5, line width=1mm,out=90,in=180] ($(l1)!0.33!(r1)$) to ($(r1)!0.67!(r0)$);
\draw[draw=lightgray, draw opacity=0.5, line width=1mm,out=0,in=-90] ($(l2)!0.67!(l1)$) to ($(l1)!0.33!(r1)$);
\draw[draw=lightgray, draw opacity=0.5, line width=1mm,out=0,in=-90] ($(l2)!0.33!(l1)$) to ($(l1)!0.67!(r1)$);
\draw[draw=lightgray, draw opacity=0.5, line width=1mm,out=90,in=0] ($(l3)!0.50!(r3)$) to ($(l3)!0.50!(l2)$);
\draw[draw=lightgray, draw opacity=0.5, line width=1mm,out=0,in=-90] ($(l4)!0.50!(l3)$) to ($(l3)!0.50!(r3)$);
\draw[draw=lightgray, draw opacity=0.5, line width=1mm,out=90,in=0] ($(l5)!0.33!(r5)$) to ($(l5)!0.33!(l4)$);
\draw[draw=lightgray, draw opacity=0.5, line width=1mm,out=90,in=0] ($(l5)!0.67!(r5)$) to ($(l5)!0.67!(l4)$);
\draw[draw=lightgray, draw opacity=0.5, line width=1mm,out=180,in=-90] ($(r6)!0.67!(r5)$) to ($(l5)!0.67!(r5)$);
\draw[draw=lightgray, draw opacity=0.5, line width=1mm,out=180,in=-90] ($(r6)!0.33!(r5)$) to ($(l5)!0.33!(r5)$);
\draw[draw=lightgray, draw opacity=0.5, line width=1mm,out=90,in=180] ($(l7)!0.50!(r7)$) to ($(r7)!0.50!(r6)$);
\draw[draw=lightgray, draw opacity=0.5, line width=1mm,out=180,in=-90] ($(r8)!0.50!(r7)$) to ($(l7)!0.50!(r7)$);
\draw[draw=lightgray, draw opacity=0.5, line width=1mm,out=90,in=180] ($(l9)!0.67!(r9)$) to ($(r9)!0.33!(r8)$);
\draw[draw=lightgray, draw opacity=0.5, line width=1mm,out=90,in=180] ($(l9)!0.33!(r9)$) to ($(r9)!0.67!(r8)$);
\draw[draw=lightgray, draw opacity=0.5, line width=1mm,out=0,in=-90] ($(l10)!0.50!(l9)$) to ($(l9)!0.33!(r9)$);
\draw[draw=lightgray, draw opacity=0.5, line width=1mm,out=90,in=-90] ($(l10)!0.50!(r10)$) to ($(l9)!0.67!(r9)$);
\draw[draw=lightgray, draw opacity=0.5, line width=1mm,out=90,in=0] ($(l11)!0.33!(r11)$) to ($(l11)!0.50!(l10)$);
\draw[draw=lightgray, draw opacity=0.5, line width=1mm,out=90,in=-90] ($(l11)!0.67!(r11)$) to ($(l10)!0.50!(r10)$);
\draw[draw=lightgray, draw opacity=0.5, line width=1mm,out=0,in=-90] ($(l12)!0.67!(l11)$) to ($(l11)!0.33!(r11)$);
\draw[draw=lightgray, draw opacity=0.5, line width=1mm,out=0,in=-90] ($(l12)!0.33!(l11)$) to ($(l11)!0.67!(r11)$);
\draw[draw=lightgray, draw opacity=0.5, line width=1mm,out=90,in=0] ($(l13)!0.33!(r13)$) to ($(l13)!0.33!(l12)$);
\draw[draw=lightgray, draw opacity=0.5, line width=1mm,out=90,in=0] ($(l13)!0.67!(r13)$) to ($(l13)!0.67!(l12)$);
\draw[draw=lightgray, draw opacity=0.5, line width=1mm,out=90,in=-90] ($(l14)!0.50!(r14)$) to ($(l13)!0.33!(r13)$);
\draw[draw=lightgray, draw opacity=0.5, line width=1mm,out=180,in=-90] ($(r14)!0.50!(r13)$) to ($(l13)!0.67!(r13)$);
\draw[draw=lightgray, draw opacity=0.5, line width=1mm,out=90,in=-90] ($(l15)!0.25!(r15)$) to ($(l14)!0.50!(r14)$);
\draw[draw=lightgray, draw opacity=0.5, line width=1mm,out=90,in=180] ($(l15)!0.75!(r15)$) to ($(r15)!0.33!(r14)$);
\draw[draw=lightgray, draw opacity=0.5, line width=1mm,out=90,in=180] ($(l15)!0.50!(r15)$) to ($(r15)!0.67!(r14)$);
\draw[draw=lightgray, draw opacity=0.5, line width=1mm,out=180,in=-90] ($(r16)!0.75!(r15)$) to ($(l15)!0.75!(r15)$);
\draw[draw=lightgray, draw opacity=0.5, line width=1mm,out=180,in=-90] ($(r16)!0.50!(r15)$) to ($(l15)!0.50!(r15)$);
\draw[draw=lightgray, draw opacity=0.5, line width=1mm,out=180,in=-90] ($(r16)!0.25!(r15)$) to ($(l15)!0.25!(r15)$);
}
 ]
|[alias=l0]| 0 \ar[r,equals] \& |[alias=r0]| 0 
\\ |[alias=l1]|0\ar[u,equals]\ar[r] \& |[alias=r1]|2\ar[u,<-]
\\ |[alias=l2]|2\ar[u,<-]\ar[r,equals] \& |[alias=r2]|2\ar[u,equals]
\\ |[alias=l3]|1\ar[u]\ar[r] \& |[alias=r3]|2\ar[u,equals]
\\ |[alias=l4]|2\ar[u,<-]\ar[r,equals] \& |[alias=r4]|2\ar[u,equals]
\\ |[alias=l5]|0\ar[u]\ar[r] \& |[alias=r5]|2\ar[u,equals]
\\ |[alias=l6]|0\ar[u,equals]\ar[r,equals] \& |[alias=r6]|0\ar[u]
\\ |[alias=l7]|0\ar[u,equals]\ar[r] \& |[alias=r7]|1\ar[u,<-]
\\ |[alias=l8]|0\ar[u,equals]\ar[r,equals] \& |[alias=r8]|0\ar[u]
\\ |[alias=l9]|0\ar[u,equals]\ar[r] \& |[alias=r9]|2\ar[u,<-]
\\ |[alias=l10]|1\ar[u,<-]\ar[r] \& |[alias=r10]|2\ar[u,equals]
\\ |[alias=l11]|0\ar[u]\ar[r] \& |[alias=r11]|2\ar[u,equals]
\\ |[alias=l12]|2\ar[u,<-]\ar[r,equals] \& |[alias=r12]|2\ar[u,equals]
\\ |[alias=l13]|0\ar[u]\ar[r] \& |[alias=r13]|2\ar[u,equals]
\\ |[alias=l14]|0\ar[u,equals]\ar[r] \& |[alias=r14]|1\ar[u]
\\ |[alias=l15]|0\ar[u,equals]\ar[r] \& |[alias=r15]|3\ar[u,<-]
\\ |[alias=l16]|0\ar[u,equals]\ar[r,equals] \& |[alias=r16]|0\ar[u]
\arrow[gray,phantom,from=1-1,to=2-2,"\push \Rsym"]
\arrow[gray,phantom,from=2-1,to=3-2,"\push \Lsym"]
\arrow[gray,phantom,from=3-1,to=4-2,"\pull \Lsym"]
\arrow[gray,phantom,from=4-1,to=5-2,"\push \Lsym"]
\arrow[gray,phantom,from=5-1,to=6-2,"\pull \Lsym"]
\arrow[gray,phantom,from=6-1,to=7-2,"\pull \Rsym"]
\arrow[gray,phantom,from=7-1,to=8-2,"\push \Rsym"]
\arrow[gray,phantom,from=8-1,to=9-2,"\pull \Rsym"]
\arrow[gray,phantom,from=9-1,to=10-2,"\push \Rsym"]
\arrow[gray,phantom,from=10-1,to=11-2,"\push \Lsym"]
\arrow[gray,phantom,from=11-1,to=12-2,"\pull \Lsym"]
\arrow[gray,phantom,from=12-1,to=13-2,"\push \Lsym"]
\arrow[gray,phantom,from=13-1,to=14-2,"\pull \Lsym"]
\arrow[gray,phantom,from=14-1,to=15-2,"\pull \Rsym"]
\arrow[gray,phantom,from=15-1,to=16-2,"\push \Rsym"]
\arrow[gray,phantom,from=16-1,to=17-2,"\pull \Rsym"]
\end{tikzcd}}
\resizebox{!}{7cm}{\begin{tikzcd}[ampersand replacement=\&,row sep=1cm,execute at end picture = { 
\draw[draw=lightgray, draw opacity=0.5, line width=1mm,out=90,in=180] ($(l1)!0.67!(r1)$) to ($(r1)!0.33!(r0)$);
\draw[draw=lightgray, draw opacity=0.5, line width=1mm,out=90,in=180] ($(l1)!0.33!(r1)$) to ($(r1)!0.67!(r0)$);
\draw[draw=lightgray, draw opacity=0.5, line width=1mm,out=0,in=-90] ($(l2)!0.67!(l1)$) to ($(l1)!0.33!(r1)$);
\draw[draw=lightgray, draw opacity=0.5, line width=1mm,out=0,in=-90] ($(l2)!0.33!(l1)$) to ($(l1)!0.67!(r1)$);
\draw[draw=lightgray, draw opacity=0.5, line width=1mm,out=90,in=0] ($(l3)!0.50!(r3)$) to ($(l3)!0.50!(l2)$);
\draw[draw=lightgray, draw opacity=0.5, line width=1mm,out=0,in=-90] ($(l4)!0.50!(l3)$) to ($(l3)!0.50!(r3)$);
\draw[draw=lightgray, draw opacity=0.5, line width=1mm,out=90,in=0] ($(l5)!0.33!(r5)$) to ($(l5)!0.33!(l4)$);
\draw[draw=lightgray, draw opacity=0.5, line width=1mm,out=90,in=0] ($(l5)!0.67!(r5)$) to ($(l5)!0.67!(l4)$);
\draw[draw=lightgray, draw opacity=0.5, line width=1mm,out=180,in=-90] ($(r6)!0.67!(r5)$) to ($(l5)!0.67!(r5)$);
\draw[draw=lightgray, draw opacity=0.5, line width=1mm,out=180,in=-90] ($(r6)!0.33!(r5)$) to ($(l5)!0.33!(r5)$);
\draw[draw=lightgray, draw opacity=0.5, line width=1mm,out=90,in=180] ($(l7)!0.50!(r7)$) to ($(r7)!0.50!(r6)$);
\draw[draw=lightgray, draw opacity=0.5, line width=1mm,out=0,in=-90] ($(l8)!0.50!(l7)$) to ($(l7)!0.50!(r7)$);
\draw[draw=lightgray, draw opacity=0.5, line width=1mm,out=90,in=0] ($(l9)!0.50!(r9)$) to ($(l9)!0.50!(l8)$);
\draw[draw=lightgray, draw opacity=0.5, line width=1mm,out=180,in=-90] ($(r10)!0.50!(r9)$) to ($(l9)!0.50!(r9)$);
\draw[draw=lightgray, draw opacity=0.5, line width=1mm,out=90,in=180] ($(l11)!0.67!(r11)$) to ($(r11)!0.33!(r10)$);
\draw[draw=lightgray, draw opacity=0.5, line width=1mm,out=90,in=180] ($(l11)!0.33!(r11)$) to ($(r11)!0.67!(r10)$);
\draw[draw=lightgray, draw opacity=0.5, line width=1mm,out=0,in=-90] ($(l12)!0.67!(l11)$) to ($(l11)!0.33!(r11)$);
\draw[draw=lightgray, draw opacity=0.5, line width=1mm,out=0,in=-90] ($(l12)!0.33!(l11)$) to ($(l11)!0.67!(r11)$);
\draw[draw=lightgray, draw opacity=0.5, line width=1mm,out=90,in=0] ($(l13)!0.33!(r13)$) to ($(l13)!0.33!(l12)$);
\draw[draw=lightgray, draw opacity=0.5, line width=1mm,out=90,in=0] ($(l13)!0.67!(r13)$) to ($(l13)!0.67!(l12)$);
\draw[draw=lightgray, draw opacity=0.5, line width=1mm,out=90,in=-90] ($(l14)!0.50!(r14)$) to ($(l13)!0.33!(r13)$);
\draw[draw=lightgray, draw opacity=0.5, line width=1mm,out=180,in=-90] ($(r14)!0.50!(r13)$) to ($(l13)!0.67!(r13)$);
\draw[draw=lightgray, draw opacity=0.5, line width=1mm,out=90,in=-90] ($(l15)!0.25!(r15)$) to ($(l14)!0.50!(r14)$);
\draw[draw=lightgray, draw opacity=0.5, line width=1mm,out=90,in=180] ($(l15)!0.75!(r15)$) to ($(r15)!0.33!(r14)$);
\draw[draw=lightgray, draw opacity=0.5, line width=1mm,out=90,in=180] ($(l15)!0.50!(r15)$) to ($(r15)!0.67!(r14)$);
\draw[draw=lightgray, draw opacity=0.5, line width=1mm,out=180,in=-90] ($(r16)!0.75!(r15)$) to ($(l15)!0.75!(r15)$);
\draw[draw=lightgray, draw opacity=0.5, line width=1mm,out=180,in=-90] ($(r16)!0.50!(r15)$) to ($(l15)!0.50!(r15)$);
\draw[draw=lightgray, draw opacity=0.5, line width=1mm,out=180,in=-90] ($(r16)!0.25!(r15)$) to ($(l15)!0.25!(r15)$);
}
 ]
|[alias=l0]| 0 \ar[r,equals] \& |[alias=r0]| 0 
\\ |[alias=l1]|0\ar[u,equals]\ar[r] \& |[alias=r1]|2\ar[u,<-]
\\ |[alias=l2]|2\ar[u,<-]\ar[r,equals] \& |[alias=r2]|2\ar[u,equals]
\\ |[alias=l3]|1\ar[u]\ar[r] \& |[alias=r3]|2\ar[u,equals]
\\ |[alias=l4]|2\ar[u,<-]\ar[r,equals] \& |[alias=r4]|2\ar[u,equals]
\\ |[alias=l5]|0\ar[u]\ar[r] \& |[alias=r5]|2\ar[u,equals]
\\ |[alias=l6]|0\ar[u,equals]\ar[r,equals] \& |[alias=r6]|0\ar[u]
\\ |[alias=l7]|0\ar[u,equals]\ar[r] \& |[alias=r7]|1\ar[u,<-]
\\ |[alias=l8]|1\ar[u,<-]\ar[r,equals] \& |[alias=r8]|1\ar[u,equals]
\\ |[alias=l9]|0\ar[u]\ar[r] \& |[alias=r9]|1\ar[u,equals]
\\ |[alias=l10]|0\ar[u,equals]\ar[r,equals] \& |[alias=r10]|0\ar[u]
\\ |[alias=l11]|0\ar[u,equals]\ar[r] \& |[alias=r11]|2\ar[u,<-]
\\ |[alias=l12]|2\ar[u,<-]\ar[r,equals] \& |[alias=r12]|2\ar[u,equals]
\\ |[alias=l13]|0\ar[u]\ar[r] \& |[alias=r13]|2\ar[u,equals]
\\ |[alias=l14]|0\ar[u,equals]\ar[r] \& |[alias=r14]|1\ar[u]
\\ |[alias=l15]|0\ar[u,equals]\ar[r] \& |[alias=r15]|3\ar[u,<-]
\\ |[alias=l16]|0\ar[u,equals]\ar[r,equals] \& |[alias=r16]|0\ar[u]
\arrow[gray,phantom,from=1-1,to=2-2,"\push \Rsym"]
\arrow[gray,phantom,from=2-1,to=3-2,"\push \Lsym"]
\arrow[gray,phantom,from=3-1,to=4-2,"\pull \Lsym"]
\arrow[gray,phantom,from=4-1,to=5-2,"\push \Lsym"]
\arrow[gray,phantom,from=5-1,to=6-2,"\pull \Lsym"]
\arrow[gray,phantom,from=6-1,to=7-2,"\pull \Rsym"]
\arrow[gray,phantom,from=7-1,to=8-2,"\push \Rsym"]
\arrow[gray,phantom,from=8-1,to=9-2,"\push \Lsym"]
\arrow[gray,phantom,from=9-1,to=10-2,"\pull \Lsym"]
\arrow[gray,phantom,from=10-1,to=11-2,"\pull \Rsym"]
\arrow[gray,phantom,from=11-1,to=12-2,"\push \Rsym"]
\arrow[gray,phantom,from=12-1,to=13-2,"\push \Lsym"]
\arrow[gray,phantom,from=13-1,to=14-2,"\pull \Lsym"]
\arrow[gray,phantom,from=14-1,to=15-2,"\pull \Rsym"]
\arrow[gray,phantom,from=15-1,to=16-2,"\push \Rsym"]
\arrow[gray,phantom,from=16-1,to=17-2,"\pull \Rsym"]
\end{tikzcd}}
\resizebox{!}{7cm}{\begin{tikzcd}[ampersand replacement=\&,row sep=1cm,execute at end picture = { 
\draw[draw=lightgray, draw opacity=0.5, line width=1mm,out=90,in=180] ($(l1)!0.67!(r1)$) to ($(r1)!0.33!(r0)$);
\draw[draw=lightgray, draw opacity=0.5, line width=1mm,out=90,in=180] ($(l1)!0.33!(r1)$) to ($(r1)!0.67!(r0)$);
\draw[draw=lightgray, draw opacity=0.5, line width=1mm,out=0,in=-90] ($(l2)!0.67!(l1)$) to ($(l1)!0.33!(r1)$);
\draw[draw=lightgray, draw opacity=0.5, line width=1mm,out=0,in=-90] ($(l2)!0.33!(l1)$) to ($(l1)!0.67!(r1)$);
\draw[draw=lightgray, draw opacity=0.5, line width=1mm,out=90,in=0] ($(l3)!0.50!(r3)$) to ($(l3)!0.50!(l2)$);
\draw[draw=lightgray, draw opacity=0.5, line width=1mm,out=0,in=-90] ($(l4)!0.50!(l3)$) to ($(l3)!0.50!(r3)$);
\draw[draw=lightgray, draw opacity=0.5, line width=1mm,out=90,in=0] ($(l5)!0.33!(r5)$) to ($(l5)!0.33!(l4)$);
\draw[draw=lightgray, draw opacity=0.5, line width=1mm,out=90,in=0] ($(l5)!0.67!(r5)$) to ($(l5)!0.67!(l4)$);
\draw[draw=lightgray, draw opacity=0.5, line width=1mm,out=0,in=-90] ($(l6)!0.50!(l5)$) to ($(l5)!0.33!(r5)$);
\draw[draw=lightgray, draw opacity=0.5, line width=1mm,out=90,in=-90] ($(l6)!0.50!(r6)$) to ($(l5)!0.67!(r5)$);
\draw[draw=lightgray, draw opacity=0.5, line width=1mm,out=90,in=0] ($(l7)!0.33!(r7)$) to ($(l7)!0.50!(l6)$);
\draw[draw=lightgray, draw opacity=0.5, line width=1mm,out=90,in=-90] ($(l7)!0.67!(r7)$) to ($(l6)!0.50!(r6)$);
\draw[draw=lightgray, draw opacity=0.5, line width=1mm,out=180,in=-90] ($(r8)!0.67!(r7)$) to ($(l7)!0.67!(r7)$);
\draw[draw=lightgray, draw opacity=0.5, line width=1mm,out=180,in=-90] ($(r8)!0.33!(r7)$) to ($(l7)!0.33!(r7)$);
\draw[draw=lightgray, draw opacity=0.5, line width=1mm,out=90,in=180] ($(l9)!0.50!(r9)$) to ($(r9)!0.50!(r8)$);
\draw[draw=lightgray, draw opacity=0.5, line width=1mm,out=180,in=-90] ($(r10)!0.50!(r9)$) to ($(l9)!0.50!(r9)$);
\draw[draw=lightgray, draw opacity=0.5, line width=1mm,out=90,in=180] ($(l11)!0.67!(r11)$) to ($(r11)!0.33!(r10)$);
\draw[draw=lightgray, draw opacity=0.5, line width=1mm,out=90,in=180] ($(l11)!0.33!(r11)$) to ($(r11)!0.67!(r10)$);
\draw[draw=lightgray, draw opacity=0.5, line width=1mm,out=0,in=-90] ($(l12)!0.67!(l11)$) to ($(l11)!0.33!(r11)$);
\draw[draw=lightgray, draw opacity=0.5, line width=1mm,out=0,in=-90] ($(l12)!0.33!(l11)$) to ($(l11)!0.67!(r11)$);
\draw[draw=lightgray, draw opacity=0.5, line width=1mm,out=90,in=0] ($(l13)!0.33!(r13)$) to ($(l13)!0.33!(l12)$);
\draw[draw=lightgray, draw opacity=0.5, line width=1mm,out=90,in=0] ($(l13)!0.67!(r13)$) to ($(l13)!0.67!(l12)$);
\draw[draw=lightgray, draw opacity=0.5, line width=1mm,out=90,in=-90] ($(l14)!0.50!(r14)$) to ($(l13)!0.33!(r13)$);
\draw[draw=lightgray, draw opacity=0.5, line width=1mm,out=180,in=-90] ($(r14)!0.50!(r13)$) to ($(l13)!0.67!(r13)$);
\draw[draw=lightgray, draw opacity=0.5, line width=1mm,out=90,in=-90] ($(l15)!0.25!(r15)$) to ($(l14)!0.50!(r14)$);
\draw[draw=lightgray, draw opacity=0.5, line width=1mm,out=90,in=180] ($(l15)!0.75!(r15)$) to ($(r15)!0.33!(r14)$);
\draw[draw=lightgray, draw opacity=0.5, line width=1mm,out=90,in=180] ($(l15)!0.50!(r15)$) to ($(r15)!0.67!(r14)$);
\draw[draw=lightgray, draw opacity=0.5, line width=1mm,out=180,in=-90] ($(r16)!0.75!(r15)$) to ($(l15)!0.75!(r15)$);
\draw[draw=lightgray, draw opacity=0.5, line width=1mm,out=180,in=-90] ($(r16)!0.50!(r15)$) to ($(l15)!0.50!(r15)$);
\draw[draw=lightgray, draw opacity=0.5, line width=1mm,out=180,in=-90] ($(r16)!0.25!(r15)$) to ($(l15)!0.25!(r15)$);
}
 ]
|[alias=l0]| 0 \ar[r,equals] \& |[alias=r0]| 0 
\\ |[alias=l1]|0\ar[u,equals]\ar[r] \& |[alias=r1]|2\ar[u,<-]
\\ |[alias=l2]|2\ar[u,<-]\ar[r,equals] \& |[alias=r2]|2\ar[u,equals]
\\ |[alias=l3]|1\ar[u]\ar[r] \& |[alias=r3]|2\ar[u,equals]
\\ |[alias=l4]|2\ar[u,<-]\ar[r,equals] \& |[alias=r4]|2\ar[u,equals]
\\ |[alias=l5]|0\ar[u]\ar[r] \& |[alias=r5]|2\ar[u,equals]
\\ |[alias=l6]|1\ar[u,<-]\ar[r] \& |[alias=r6]|2\ar[u,equals]
\\ |[alias=l7]|0\ar[u]\ar[r] \& |[alias=r7]|2\ar[u,equals]
\\ |[alias=l8]|0\ar[u,equals]\ar[r,equals] \& |[alias=r8]|0\ar[u]
\\ |[alias=l9]|0\ar[u,equals]\ar[r] \& |[alias=r9]|1\ar[u,<-]
\\ |[alias=l10]|0\ar[u,equals]\ar[r,equals] \& |[alias=r10]|0\ar[u]
\\ |[alias=l11]|0\ar[u,equals]\ar[r] \& |[alias=r11]|2\ar[u,<-]
\\ |[alias=l12]|2\ar[u,<-]\ar[r,equals] \& |[alias=r12]|2\ar[u,equals]
\\ |[alias=l13]|0\ar[u]\ar[r] \& |[alias=r13]|2\ar[u,equals]
\\ |[alias=l14]|0\ar[u,equals]\ar[r] \& |[alias=r14]|1\ar[u]
\\ |[alias=l15]|0\ar[u,equals]\ar[r] \& |[alias=r15]|3\ar[u,<-]
\\ |[alias=l16]|0\ar[u,equals]\ar[r,equals] \& |[alias=r16]|0\ar[u]
\arrow[gray,phantom,from=1-1,to=2-2,"\push \Rsym"]
\arrow[gray,phantom,from=2-1,to=3-2,"\push \Lsym"]
\arrow[gray,phantom,from=3-1,to=4-2,"\pull \Lsym"]
\arrow[gray,phantom,from=4-1,to=5-2,"\push \Lsym"]
\arrow[gray,phantom,from=5-1,to=6-2,"\pull \Lsym"]
\arrow[gray,phantom,from=6-1,to=7-2,"\push \Lsym"]
\arrow[gray,phantom,from=7-1,to=8-2,"\pull \Lsym"]
\arrow[gray,phantom,from=8-1,to=9-2,"\pull \Rsym"]
\arrow[gray,phantom,from=9-1,to=10-2,"\push \Rsym"]
\arrow[gray,phantom,from=10-1,to=11-2,"\pull \Rsym"]
\arrow[gray,phantom,from=11-1,to=12-2,"\push \Rsym"]
\arrow[gray,phantom,from=12-1,to=13-2,"\push \Lsym"]
\arrow[gray,phantom,from=13-1,to=14-2,"\pull \Lsym"]
\arrow[gray,phantom,from=14-1,to=15-2,"\pull \Rsym"]
\arrow[gray,phantom,from=15-1,to=16-2,"\push \Rsym"]
\arrow[gray,phantom,from=16-1,to=17-2,"\pull \Rsym"]
\end{tikzcd}} 
\resizebox{!}{7cm}{\begin{tikzcd}[ampersand replacement=\&,row sep=1cm,execute at end picture = { 
\draw[draw=lightgray, draw opacity=0.5, line width=1mm,out=90,in=180] ($(l1)!0.67!(r1)$) to ($(r1)!0.33!(r0)$);
\draw[draw=lightgray, draw opacity=0.5, line width=1mm,out=90,in=180] ($(l1)!0.33!(r1)$) to ($(r1)!0.67!(r0)$);
\draw[draw=lightgray, draw opacity=0.5, line width=1mm,out=0,in=-90] ($(l2)!0.67!(l1)$) to ($(l1)!0.33!(r1)$);
\draw[draw=lightgray, draw opacity=0.5, line width=1mm,out=0,in=-90] ($(l2)!0.33!(l1)$) to ($(l1)!0.67!(r1)$);
\draw[draw=lightgray, draw opacity=0.5, line width=1mm,out=90,in=0] ($(l3)!0.50!(r3)$) to ($(l3)!0.50!(l2)$);
\draw[draw=lightgray, draw opacity=0.5, line width=1mm,out=0,in=-90] ($(l4)!0.50!(l3)$) to ($(l3)!0.50!(r3)$);
\draw[draw=lightgray, draw opacity=0.5, line width=1mm,out=90,in=0] ($(l5)!0.33!(r5)$) to ($(l5)!0.33!(l4)$);
\draw[draw=lightgray, draw opacity=0.5, line width=1mm,out=90,in=0] ($(l5)!0.67!(r5)$) to ($(l5)!0.67!(l4)$);
\draw[draw=lightgray, draw opacity=0.5, line width=1mm,out=0,in=-90] ($(l6)!0.50!(l5)$) to ($(l5)!0.33!(r5)$);
\draw[draw=lightgray, draw opacity=0.5, line width=1mm,out=90,in=-90] ($(l6)!0.50!(r6)$) to ($(l5)!0.67!(r5)$);
\draw[draw=lightgray, draw opacity=0.5, line width=1mm,out=90,in=0] ($(l7)!0.33!(r7)$) to ($(l7)!0.50!(l6)$);
\draw[draw=lightgray, draw opacity=0.5, line width=1mm,out=90,in=-90] ($(l7)!0.67!(r7)$) to ($(l6)!0.50!(r6)$);
\draw[draw=lightgray, draw opacity=0.5, line width=1mm,out=0,in=-90] ($(l8)!0.67!(l7)$) to ($(l7)!0.33!(r7)$);
\draw[draw=lightgray, draw opacity=0.5, line width=1mm,out=0,in=-90] ($(l8)!0.33!(l7)$) to ($(l7)!0.67!(r7)$);
\draw[draw=lightgray, draw opacity=0.5, line width=1mm,out=90,in=0] ($(l9)!0.33!(r9)$) to ($(l9)!0.33!(l8)$);
\draw[draw=lightgray, draw opacity=0.5, line width=1mm,out=90,in=0] ($(l9)!0.67!(r9)$) to ($(l9)!0.67!(l8)$);
\draw[draw=lightgray, draw opacity=0.5, line width=1mm,out=180,in=-90] ($(r10)!0.67!(r9)$) to ($(l9)!0.67!(r9)$);
\draw[draw=lightgray, draw opacity=0.5, line width=1mm,out=180,in=-90] ($(r10)!0.33!(r9)$) to ($(l9)!0.33!(r9)$);
\draw[draw=lightgray, draw opacity=0.5, line width=1mm,out=90,in=180] ($(l11)!0.50!(r11)$) to ($(r11)!0.50!(r10)$);
\draw[draw=lightgray, draw opacity=0.5, line width=1mm,out=180,in=-90] ($(r12)!0.50!(r11)$) to ($(l11)!0.50!(r11)$);
\draw[draw=lightgray, draw opacity=0.5, line width=1mm,out=90,in=180] ($(l13)!0.67!(r13)$) to ($(r13)!0.33!(r12)$);
\draw[draw=lightgray, draw opacity=0.5, line width=1mm,out=90,in=180] ($(l13)!0.33!(r13)$) to ($(r13)!0.67!(r12)$);
\draw[draw=lightgray, draw opacity=0.5, line width=1mm,out=90,in=-90] ($(l14)!0.50!(r14)$) to ($(l13)!0.33!(r13)$);
\draw[draw=lightgray, draw opacity=0.5, line width=1mm,out=180,in=-90] ($(r14)!0.50!(r13)$) to ($(l13)!0.67!(r13)$);
\draw[draw=lightgray, draw opacity=0.5, line width=1mm,out=90,in=-90] ($(l15)!0.25!(r15)$) to ($(l14)!0.50!(r14)$);
\draw[draw=lightgray, draw opacity=0.5, line width=1mm,out=90,in=180] ($(l15)!0.75!(r15)$) to ($(r15)!0.33!(r14)$);
\draw[draw=lightgray, draw opacity=0.5, line width=1mm,out=90,in=180] ($(l15)!0.50!(r15)$) to ($(r15)!0.67!(r14)$);
\draw[draw=lightgray, draw opacity=0.5, line width=1mm,out=180,in=-90] ($(r16)!0.75!(r15)$) to ($(l15)!0.75!(r15)$);
\draw[draw=lightgray, draw opacity=0.5, line width=1mm,out=180,in=-90] ($(r16)!0.50!(r15)$) to ($(l15)!0.50!(r15)$);
\draw[draw=lightgray, draw opacity=0.5, line width=1mm,out=180,in=-90] ($(r16)!0.25!(r15)$) to ($(l15)!0.25!(r15)$);
}
 ]
|[alias=l0]| 0 \ar[r,equals] \& |[alias=r0]| 0 
\\ |[alias=l1]|0\ar[u,equals]\ar[r] \& |[alias=r1]|2\ar[u,<-]
\\ |[alias=l2]|2\ar[u,<-]\ar[r,equals] \& |[alias=r2]|2\ar[u,equals]
\\ |[alias=l3]|1\ar[u]\ar[r] \& |[alias=r3]|2\ar[u,equals]
\\ |[alias=l4]|2\ar[u,<-]\ar[r,equals] \& |[alias=r4]|2\ar[u,equals]
\\ |[alias=l5]|0\ar[u]\ar[r] \& |[alias=r5]|2\ar[u,equals]
\\ |[alias=l6]|1\ar[u,<-]\ar[r] \& |[alias=r6]|2\ar[u,equals]
\\ |[alias=l7]|0\ar[u]\ar[r] \& |[alias=r7]|2\ar[u,equals]
\\ |[alias=l8]|2\ar[u,<-]\ar[r,equals] \& |[alias=r8]|2\ar[u,equals]
\\ |[alias=l9]|0\ar[u]\ar[r] \& |[alias=r9]|2\ar[u,equals]
\\ |[alias=l10]|0\ar[u,equals]\ar[r,equals] \& |[alias=r10]|0\ar[u]
\\ |[alias=l11]|0\ar[u,equals]\ar[r] \& |[alias=r11]|1\ar[u,<-]
\\ |[alias=l12]|0\ar[u,equals]\ar[r,equals] \& |[alias=r12]|0\ar[u]
\\ |[alias=l13]|0\ar[u,equals]\ar[r] \& |[alias=r13]|2\ar[u,<-]
\\ |[alias=l14]|0\ar[u,equals]\ar[r] \& |[alias=r14]|1\ar[u]
\\ |[alias=l15]|0\ar[u,equals]\ar[r] \& |[alias=r15]|3\ar[u,<-]
\\ |[alias=l16]|0\ar[u,equals]\ar[r,equals] \& |[alias=r16]|0\ar[u]
\arrow[gray,phantom,from=1-1,to=2-2,"\push \Rsym"]
\arrow[gray,phantom,from=2-1,to=3-2,"\push \Lsym"]
\arrow[gray,phantom,from=3-1,to=4-2,"\pull \Lsym"]
\arrow[gray,phantom,from=4-1,to=5-2,"\push \Lsym"]
\arrow[gray,phantom,from=5-1,to=6-2,"\pull \Lsym"]
\arrow[gray,phantom,from=6-1,to=7-2,"\push \Lsym"]
\arrow[gray,phantom,from=7-1,to=8-2,"\pull \Lsym"]
\arrow[gray,phantom,from=8-1,to=9-2,"\push \Lsym"]
\arrow[gray,phantom,from=9-1,to=10-2,"\pull \Lsym"]
\arrow[gray,phantom,from=10-1,to=11-2,"\pull \Rsym"]
\arrow[gray,phantom,from=11-1,to=12-2,"\push \Rsym"]
\arrow[gray,phantom,from=12-1,to=13-2,"\pull \Rsym"]
\arrow[gray,phantom,from=13-1,to=14-2,"\push \Rsym"]
\arrow[gray,phantom,from=14-1,to=15-2,"\pull \Rsym"]
\arrow[gray,phantom,from=15-1,to=16-2,"\push \Rsym"]
\arrow[gray,phantom,from=16-1,to=17-2,"\pull \Rsym"]
\end{tikzcd}}

\caption{An exhaustive listing of maximally multifocused derivations representing all morphisms $\alpha : S \to T$ between a pair of plane trees.
  The derivations corresponding to the two natural transformations in \eqref{eq:tree-morphisms} are highlighted in blue and red.}
  \label{fig:all-tree-morphisms}
\end{figure}

Proposition~\ref{prop:simplex-category-free-bifibration} is a corollary of Theorem~\ref{thm:tree-category-free-bifibration}, since the bijection between zigzags and plane trees sends zigzags that strictly alternate between 0 and 1 to trees of height 1, which may be interpreted as ordinals.

The correspondence established in Theorem~\ref{thm:tree-category-free-bifibration} was unexpected to us, and we think its interest is that it relates two very different representations of the same categories.
The calculus of maximally multifocused proofs gives an effective procedure for enumerating tree morphisms by building them up step-by-step in a linear fashion, very differently from how one might first attempt to define a natural transformation through some ad hoc procedure, or even using the inductive rules \eqref{eq:tree-morphism-rules} and \eqref{eq:marked-tree-morphism-rules} described in the proof of the theorem (which are easy enough to formulate but non-linear in the sense that some of the rules have more than one premise).
In Figure~\ref{fig:all-tree-morphisms}, we show the maximally multifocused proofs (and associated string diagrams) representing all morphisms $\alpha : S \to T$ between the pair of trees \eqref{eq:tree-morphisms} considered earlier.
These were generated by applying the proof search procedure described in Section~\ref{subsec:maximal-multi-focusing}, which we implemented as a short Haskell program.\footnote{\label{footnote:source-code} Available at \href{https://github.com/noamz/free-bifibrations}{github.com/noamz/free-bifibrations}.
This implementation was useful to us for running experiments, formulating conjectures, and generating diagrams. Counting proofs using this procedure even helped to reveal a bug in a previous incorrect formulation of the lock conditions for maximally multifocused proofs, before we arrived at the correct version in Section~\ref{subsec:maximal-multi-focusing}.} 

Before moving on to the last example, let us make one more easy observation about $\Bifib{p_\omega}$.
The ordinal $\omega$ may be seen as the category of elements of the regular action of the monoid of natural numbers $(\NN,+,0)$ on itself.
In particular, there is an evident forgetful functor $\delta : \omega \to \BN$ into the monoid of natural numbers seen as the one-object category freely generated on a loop $\begin{tikzcd}\ast \ar[loop left,"f"]\end{tikzcd}$.
This functor sends every pair of natural numbers $m \le n$ to the unique arrow $f^k : \ast \to \ast$ such that $k = n-m$; since $\omega$ is a poset, $\delta$ is trivially a faithful functor.
\begin{prop}\label{prop:BifibFun-faithful-is-fully-faithful} If $\delta : \C \to \B$  is a faithful functor then the induced morphism of bifibrations
  \[\begin{tikzcd}
      \Bifib{p} \ar[d,"\BifibFun{p}"'] \ar[r,"\BifibFun{\delta}"] & \Bifib{\delta \circ p} \ar[d,"\BifibFun{\delta\circ p}"]\\
      \C \ar[r,"\delta"'] & \B
    \end{tikzcd}\]
  is full and faithful relative to $\delta$, in the sense that the induced mapping on relative homsets $\Bifib{p}_f(S,T) \to \Bifib{\delta\circ p}_{\delta f}(\BifibFun{\delta}S, \BifibFun{\delta}T)$ is a bijection for all compatible $S,f,T$. (Equivalently, the induced functor from $\Bifib{p}$ into the fiber product $\C \pullback{\delta}{\BifibFun{\delta\circ p}} \Bifib{\delta\circ p}$ is fully faithful in the ordinary sense.)
\end{prop}
\begin{proof}
The existence of the morphism of bifibrations $\BifibFun{\delta} : \BifibFun{p} \to \BifibFun{\delta \circ p}$ comes from the extended universal property \eqref{eq:free-bifibration-pyramid} of the free bifibration $\BifibFun{p}$ and corresponds to functoriality of the $\BifibFun{}$ construction.
It has a simple inductive definition:
\begin{mathpar}
  \BifibFun{\delta}X = X

  \BifibFun{\delta}(\push f S) = \push {(\delta f)} \BifibFun{\delta}S

  \BifibFun{\delta}(\pull g T) = \pull {(\delta g)} \BifibFun{\delta}T\\

  \BifibFun{\delta}(\uLpush f g \alpha) = \uLpush {\delta f} {\delta g} \BifibFun{\delta}\alpha
\quad
  \BifibFun{\delta}(\alpha \tRpush f) = \BifibFun{\delta}\alpha \tRpush \delta f
\quad
  \BifibFun{\delta}(\tLpull g \alpha) = \tLpull {\delta g} \BifibFun{\delta}\alpha
\quad
  \BifibFun{\delta}(\alpha \uRpull f g) = \BifibFun{\delta}\alpha \uRpull {\delta f} {\delta g}
\end{mathpar}
The action of $\BifibFun{\delta}$ may be even easier to see in terms of the double category of zigzags: every arrow of every generator is simply hit with the functor $\delta$.
For example, every $\Lpush$-cell of $\ZZ(\C)$ is mapped to a $\Lpush[(\delta f)]$-cell of $\ZZ(\B)$, and every $\Lpull[g]$-cell to a $\Lpull[(\delta g)]$-cell:
\begin{equation*}
\begin{tikzcd}
    A
    \arrow[d, " f"']
    \arrow[r, "fg"]
    &
    C
    \arrow[d, equal]
    \\
    B
    \arrow[r, "g"']
    &
    C
    \end{tikzcd}
    \quad\mapsto\quad
\begin{tikzcd}
    \delta A
    \arrow[d, "\delta f"']
    \arrow[r, "\delta (fg)"]
    &
    \delta C
    \arrow[d, equal]
    \\
    \delta B
    \arrow[r, "\delta g"']
    &
    \delta C
    \end{tikzcd}\qquad\qquad
\begin{tikzcd}
    C
    \arrow[from=d, "g"]
    \arrow[r, "g'"]
    &
    C'
    \arrow[d, equal]
    \\
    B
    \arrow[r, "gg'"']
    &
    C'
  \end{tikzcd}
  \quad\mapsto\quad
\begin{tikzcd}
    C
    \arrow[from=d, "\delta g"]
    \arrow[r, "\delta g'"]
    &
    C'
    \arrow[d, equal]
    \\
    B
    \arrow[r, "\delta (gg')"']
    &
    C'
    \end{tikzcd}
\end{equation*}
In general, we obtain a mapping of relative homsets of the form:
\[\Bifib{p}_f(S,T) \longrightarrow \Bifib{\delta\circ p}_{\delta f}(\BifibFun{\delta}S, \BifibFun{\delta}T) \]
By inspection of the action on the generators, we can see that surjectivity of this mapping reduces to functoriality of $\delta$, and that injectivity reduces to faithfulness.
\end{proof}
\begin{cor}\label{cor:trees-in-walks}
  Let $!_\NN : \One \to \BN$ be the unique functor of one-object categories corresponding to the unique monoid homomorphism from the trivial monoid into $(\NN,+,0)$.
  Then the induced morphism of bifibrations $\BifibFun{\delta} : \BifibFun{p_\omega} \to \BifibFun{!_\NN}$ is fully faithful relative to $\delta : \omega \to \BN$.
  In particular, $\PTree$ is equivalent to the full subcategory of the fiber category $\Bifib{!_\NN}_\ast$ spanned by the objects in the image of $\BifibFun{\delta}$.
\end{cor}
\begin{proof}
  We apply Proposition~\ref{prop:BifibFun-faithful-is-fully-faithful} substituting $!_\NN = \delta\circ p_\omega$.
  Then we instantiate the bijection of relative homsets $\Bifib{p_\omega}_f(S,T) \cong \Bifib{!_\NN}_{\delta f}(\BifibFun{\delta}S,\BifibFun{\delta}T)$ at $f = \Id[0]$, $\delta f = \Id[\ast]$ and conclude using Theorem~\ref{thm:tree-category-free-bifibration}.
  \end{proof}
\noindent
The category $\Bifib{!_\NN}$ is in some sense more fundamental than $\Bifib{p_\omega}$.
Its objects may be interpreted as arbitrary one-dimensional walks, starting at the $x$-axis and moving up and down with no restrictions.
Below we show an example formula and the associated walk, which crosses below the $x$-axis:
\[\pull{f}\push{f}\push{f}\push{f}\pull{f}\pull{f}\push{f}
  \atom\ast
\qquad 
\begin{tikzpicture}[baseline=(mid)]
  \draw[very thin,color=gray] (-0.25,-0.25) grid (7.25,3.25);
  \draw [->] plot coordinates {(7,0) (6,1) (5,0) (4,-1) (3,0) (2,1) (1,2) (0,1)};
  \node at (7,0) {$\ast$};
  \node (mid) at (0,1) {};
\end{tikzpicture}
\]
Every such formula can be assigned an integer corresponding to its final \emph{displacement} relative to the $x$-axis (in the above case: +1).
Indeed, displacement corresponds to the object component of a functor $\sem{-}0 : \Bifib{!_\NN} \to (\ZZ,\le)$ obtained from the universal property of $\Bifib{!_\NN}$ by considering the poset $(\ZZ,\le)$ as a discrete bifibration over $\BN$ (where $\push{f}n = n+1$ and $\pull{f}n = n-1$), equipped with the origin point 0.

$\Bifib{!_\NN}$ contains $\PTree$ as a full subcategory, per Corollary~\ref{cor:trees-in-walks}, but it also contains morphisms between objects that are not Dyck walks, as the below example illustrates:
\begin{equation}\label{diagram:non-dyck-morphism}
  \rotatebox{-90}{\scalebox{0.8}{\begin{tikzcd}[ampersand replacement=\&,row sep=1cm,execute at end picture = { 
\draw[draw=violet, draw opacity=0.5, line width=1mm,out=90,in=180] ($(l1)!0.67!(r1)$) to ($(r1)!0.33!(r0)$);
\draw[draw=violet, draw opacity=0.5, line width=1mm,out=90,in=180] ($(l1)!0.33!(r1)$) to ($(r1)!0.67!(r0)$);
\draw[draw=violet, draw opacity=0.5, line width=1mm,out=0,in=-90] ($(l2)!0.50!(l1)$) to ($(l1)!0.33!(r1)$);
\draw[draw=violet, draw opacity=0.5, line width=1mm,out=180,in=-90] ($(r2)!0.50!(r1)$) to ($(l1)!0.67!(r1)$);
\draw[draw=violet, draw opacity=0.5, line width=1mm,out=90,in=0] ($(l3)!0.25!(r3)$) to ($(l3)!0.33!(l2)$);
\draw[draw=violet, draw opacity=0.5, line width=1mm,out=90,in=0] ($(l3)!0.50!(r3)$) to ($(l3)!0.67!(l2)$);
\draw[draw=violet, draw opacity=0.5, line width=1mm,out=90,in=180] ($(l3)!0.75!(r3)$) to ($(r3)!0.50!(r2)$);
\draw[draw=violet, draw opacity=0.5, line width=1mm,out=180,in=-90] ($(r4)!0.75!(r3)$) to ($(l3)!0.75!(r3)$);
\draw[draw=violet, draw opacity=0.5, line width=1mm,out=180,in=-90] ($(r4)!0.50!(r3)$) to ($(l3)!0.50!(r3)$);
\draw[draw=violet, draw opacity=0.5, line width=1mm,out=180,in=-90] ($(r4)!0.25!(r3)$) to ($(l3)!0.25!(r3)$);
\draw[draw=violet, draw opacity=0.5, line width=1mm,out=90,in=180] ($(l5)!0.75!(r5)$) to ($(r5)!0.25!(r4)$);
\draw[draw=violet, draw opacity=0.5, line width=1mm,out=90,in=180] ($(l5)!0.50!(r5)$) to ($(r5)!0.50!(r4)$);
\draw[draw=violet, draw opacity=0.5, line width=1mm,out=90,in=180] ($(l5)!0.25!(r5)$) to ($(r5)!0.75!(r4)$);
\draw[draw=violet, draw opacity=0.5, line width=1mm,out=90,in=-90] ($(l6)!0.33!(r6)$) to ($(l5)!0.25!(r5)$);
\draw[draw=violet, draw opacity=0.5, line width=1mm,out=90,in=-90] ($(l6)!0.67!(r6)$) to ($(l5)!0.50!(r5)$);
\draw[draw=violet, draw opacity=0.5, line width=1mm,out=180,in=-90] ($(r6)!0.50!(r5)$) to ($(l5)!0.75!(r5)$);
\draw[draw=violet, draw opacity=0.5, line width=1mm,out=90,in=-90] ($(l7)!0.25!(r7)$) to ($(l6)!0.33!(r6)$);
\draw[draw=violet, draw opacity=0.5, line width=1mm,out=90,in=-90] ($(l7)!0.50!(r7)$) to ($(l6)!0.67!(r6)$);
\draw[draw=violet, draw opacity=0.5, line width=1mm,out=90,in=180] ($(l7)!0.75!(r7)$) to ($(r7)!0.50!(r6)$);
\draw[draw=violet, draw opacity=0.5, line width=1mm,out=0,in=-90] ($(l8)!0.75!(l7)$) to ($(l7)!0.25!(r7)$);
\draw[draw=violet, draw opacity=0.5, line width=1mm,out=0,in=-90] ($(l8)!0.50!(l7)$) to ($(l7)!0.50!(r7)$);
\draw[draw=violet, draw opacity=0.5, line width=1mm,out=0,in=-90] ($(l8)!0.25!(l7)$) to ($(l7)!0.75!(r7)$);
\draw[draw=violet, draw opacity=0.5, line width=1mm,out=90,in=0] ($(l9)!0.50!(r9)$) to ($(l9)!0.50!(l8)$);
\draw[draw=violet, draw opacity=0.5, line width=1mm,out=180,in=-90] ($(r10)!0.50!(r9)$) to ($(l9)!0.50!(r9)$);
}
 ]
|[alias=l0]| 0 \ar[r,equals] \& |[alias=r0]| 0 
\\ |[alias=l1]|0\ar[u,equals]\ar[r] \& |[alias=r1]|2\ar[u,<-]
\\ |[alias=l2]|1\ar[u,<-]\ar[r,equals] \& |[alias=r2]|1\ar[u]
\\ |[alias=l3]|-1\ar[u]\ar[r] \& |[alias=r3]|2\ar[u,<-]
\\ |[alias=l4]|-1\ar[u,equals]\ar[r,equals] \& |[alias=r4]|-1\ar[u]
\\ |[alias=l5]|-1\ar[u,equals]\ar[r] \& |[alias=r5]|2\ar[u,<-]
\\ |[alias=l6]|-1\ar[u,equals]\ar[r] \& |[alias=r6]|1\ar[u]
\\ |[alias=l7]|-1\ar[u,equals]\ar[r] \& |[alias=r7]|2\ar[u,<-]
\\ |[alias=l8]|2\ar[u,<-]\ar[r,equals] \& |[alias=r8]|2\ar[u,equals]
\\ |[alias=l9]|1\ar[u]\ar[r] \& |[alias=r9]|2\ar[u,equals]
\\ |[alias=l10]|1\ar[u,equals]\ar[r,equals] \& |[alias=r10]|1\ar[u]
\arrow[gray,phantom,from=1-1,to=2-2,"\push \Rsym"]
\arrow[gray,phantom,from=2-1,to=3-2,"\push \Lsym\pull \Rsym"]
\arrow[gray,phantom,from=3-1,to=4-2,"\pull \Lsym \push \Rsym"]
\arrow[gray,phantom,from=4-1,to=5-2,"\pull \Rsym"]
\arrow[gray,phantom,from=5-1,to=6-2,"\push \Rsym"]
\arrow[gray,phantom,from=6-1,to=7-2,"\pull \Rsym"]
\arrow[gray,phantom,from=7-1,to=8-2,"\push \Rsym"]
\arrow[gray,phantom,from=8-1,to=9-2,"\push \Lsym"]
\arrow[gray,phantom,from=9-1,to=10-2,"\pull \Lsym"]
\arrow[gray,phantom,from=10-1,to=11-2,"\pull \Rsym"]
\end{tikzcd}}}
\end{equation}
Here we have indicated the running totals of the displacements of the walks as a guide even though $\BN$ really only has one object $\ast$.
Speaking literally, the diagram above can be interpreted as a double cell of $\ZZ(\ZZ,\le)$ mapping canonically to an arrow of $\Bifib{!_\NN}$.

\subsection{A free ambifibration over $\Delta$}
\label{sec:examples:ambi}

The last example we consider is inspired by the so-called \emph{fat Delta} category $\fDelta$, introduced by Joachim Kock \cite{Kock2006weak} to axiomatize a notion of higher category with weak identity arrows.
One way of defining $\fDelta$ is as a subcategory of the arrow category of $\Delta$, with objects given by epis $n \repi k$ and with arrows given by commutative squares whose upper leg is mono:
\begin{equation}\label{square:fatDelta}
\begin{tikzcd}
  n \ar[d,two heads]\ar[r,hook] & p\ar[d,two heads] \\
  k \ar[r] & j
\end{tikzcd}
\end{equation}
Observe that there is a canonical embedding $j : \NN \to \fDelta$ of the set of natural numbers viewed as a discrete category into fat Delta, defined by $j(n) = \Id[n]$.
Moreover, there is an evident projection functor $\cod : \fDelta \to \Delta$, which is known to be an ambifibration relative to the (epi,~mono) factorization system on $\Delta$ (see Sattler~\cite{Sattler2017} and Pradal~\cite{Pradal2025}).

As well-known (see Mac Lane~\cite[VII.5]{MacLaneCWM}), the simplex category admits an (epi,~mono) factorization system, with the epis being generated by the maps $\sigma_i^n : \ord{n+1} \to \ord{n}$ ($0 \le i < n$) and the monos by the maps $\delta_i^n : \ord{n} \to \ord{n+1}$ ($0 \le i \le n$), where $\sigma_i^n$ is the unique order-preserving surjection sending $i$ and $i+1$ to $i$, and $\delta_i^n$ is the unique order-preserving injection whose image avoids element $i$.
Moreover, these generators satisfy the following standard equations (please recall that we write composition in diagrammatic order):
\begin{align}
  \delta_i^n \delta_{j+1}^{n+1} & = \delta_j^n\delta_i^{n+1} & \text{if }i \le j \label{eq:simplex1}\\
  \sigma_{j+1}^{n+1}\sigma_i^n & = \sigma_i^{n+1}\sigma_j^n & \text{if }i\le j \label{eq:simplex2}\\
  \delta_i^n\sigma_j^n &= \sigma_{j-1}^{n-1}\delta_i^{n-1} & \text{if }i<j \label{eq:simplex3}\\
      & = \Id[n] & \text{if }i=j\text{ or }i=j+1 \label{eq:simplex4} \\
      & = \sigma_j^{n-1}\delta_{i-1}^{n-1} & \text{if }i > j+1 \label{eq:simplex5}
\end{align}
Now let $\P = \set{\sigma_i^n \mid 0 \le i < n}$ and $\N = \set{\delta_i^n \mid 0 \le i \le n}$ and consider the free $(\P,\N)$-fibration generated by the inclusion $i : \NN \to \Delta$ of the set of natural numbers viewed as a discrete category into the simplex category.
Since $\P$ and $\N$ generate the epis and monos respectively, $\P^* = \Delta_\epi$, $\N^* = \Delta_\mono$, we have that $\Bifib{i,\P,\N} \simeq \Bifib{i,\Delta_\epi,\Delta_\mono}$ by Prop.~\ref{prop:freePNstar}, and so $\BifibFun{i,\P,\N} : \Bifib{i,\P,\N} \to \Delta$ is the free (epi, mono)-ambifibration on $i$.

The universal property of $\Bifib{i,\P,\N}$ and the equation $i = \cod \circ j$ imply that there is a canonical morphism of (epi, mono)-ambifibrations $\Bifib{i,\P,\N} \to \fDelta$.

Let us now forget about this connection with fat Delta, and just make some combinatorial observations about $\Bifib{i,\P,\N}$.
We refer to objects of $\Bifib{i,\P,\N}$ as \emph{ambisimplicial formulas}, and will pay particular attention to formulas $S \refs \ord{0}$, which we call \emph{closed}.
Observe that a closed ambisimplicial formula corresponds to a walk $\ord{n} \zigzag \ord{0}$ along the graph below:
\[
\begin{tikzcd}
    \ord{0} \ar[r] & \ord{1} \ar[r,shift left=1ex]\ar[r,<-]\ar[r,shift right=1ex] & \ord{2} \ar[r,shift left=2ex]\ar[r,<-,shift left=1ex]\ar[r]\ar[r,<-,shift right=1ex]\ar[r,shift right=2ex] & \ord{3} \quad\cdots
\end{tikzcd}
\]
At each positive ordinal $\ord{k{+}1}$, we choose to either pull along some $\delta_i^k$ or push along some $\sigma_i^k$.
Hence there are double factorial $(2n-1)!! = (2n-1)\cdot (2n-3)\cdots 3\cdot 1$ many such formulas with source $\chi(S) = n$.
The double factorial numbers are known to have many different combinatorial interpretations \cite{Callan2009survey}, and we state here another: they count $n$-leaf \emph{increasing binary forests,} by which we mean a planar forest of binary trees where every binary node and every root node is assigned a distinct label in such a way that the labels increase when moving from parent to child, and from a binary node bounding a region to a root contained in that region.
See Figure~\ref{fig:15formulas}.

To interpret a closed ambisimplicial formula $S$ as an increasing binary forest, the idea is to start by drawing the $\chi(S) = n$ edges, and then to progressively add either a binary node (joining two edges into one) or a root node (terminating an edge), labelled $m$ and attached to the $i$th free edge, depending on whether we pushed along an epi $\sigma_i^m : m+1 \to m$ or pulled along a mono $\delta_i^m : m \to m+1$ in constructing the formula.
For example, here is how we build the increasing binary forest corresponding to the formula $\pull{\delta_0^0}\push{\sigma_0^1}\pull{\delta_1^2}\ord{3}$:
\[
\begin{tikzpicture}[baseline={([yshift=-.5ex]current bounding box.center)}]
  \node(x) {};
  \node[right=1em of x] (y) {};
  \node[right=1em of y] (z) {};
  \node[delta,below=1em of y] (yend) {2};
  \node[sigma,below=2.5em of $(x)!0.5!(z)$] (xz) {1};
  \node[delta,below=1em of xz] (xzend) {0};
  \draw[bend right] (x) to (xz);
  \draw[bend left] (z) to (xz);
  \draw (xz) to (xzend);
  \draw (y) to (yend);
\end{tikzpicture}
\quad\overset{\pull{\delta_0^0}}\longleftarrow\quad
  \begin{tikzpicture}[baseline={([yshift=-.5ex]current bounding box.center)}]
  \node(x) {};
  \node[right=1em of x] (y) {};
  \node[right=1em of y] (z) {};
  \node[delta,below=1em of y] (yend) {2};
  \node[sigma,below=2.5em of $(x)!0.5!(z)$] (xz) {1};
  \node[below=1em of xz] (xzend) {};
  \draw[bend right] (x) to (xz);
  \draw[bend left] (z) to (xz);
  \draw (xz) to (xzend);
  \draw (y) to (yend);
\end{tikzpicture}
\quad\overset{\push{\sigma_0^1}}\longleftarrow\quad
  \begin{tikzpicture}[baseline={([yshift=-.5ex]current bounding box.center)}]
  \node(x) {};
  \node[right=1em of x] (y) {};
  \node[right=1em of y] (z) {};
  \node[delta,below=1em of y] (yend) {2};
  \node[below=2.5em of x] (xend) {};
  \node[below=2.5em of z] (zend) {};
  \draw (y) to (yend);
  \draw (x) to (xend);
  \draw (z) to (zend);
\end{tikzpicture}
\quad\overset{\pull{\delta_1^2}}\longleftarrow\quad
  \begin{tikzpicture}[baseline={([yshift=-.5ex]current bounding box.center)}]
  \node(x) {};
  \node[right=1em of x] (y) {};
  \node[right=1em of y] (z) {};
  \node[below=2em of y] (yend) {};
  \node[below=2em of x] (xend) {};
  \node[below=2em of z] (zend) {};
  \draw (y) to (yend);
  \draw (x) to (xend);
  \draw (z) to (zend);
\end{tikzpicture}
\]
Now, let us quickly point out that different forests/formulas can be \emph{isomorphic} in the category $\Bifib{i,\P,\N}$.
In particular, equations \eqref{eq:simplex1} and \eqref{eq:simplex2} induce natural isomorphisms
\begin{align}
  \pull{\delta_i^n}\pull{\delta_{j+1}^{n+1}} S &\equiv \pull{\delta_j^n}\pull{\delta_i^{n+1}} S \label{simp:equiv1} \\
\push{\sigma_i^n}\push{\sigma_{j+1}^{n+1}} S & \equiv \push{\sigma_j^n}\push{\sigma_i^{n+1}} S \label{simp:equiv2}
\end{align}
for all $i \le j$.
In Figure~\ref{fig:15formulas} we indicate the equivalence classes for the case $\chi(S) = 3$ by the letters A--G.
Each of these equivalence classes may be canonically represented by a strictly alternating zigzag in $\ZZ(\Delta,\Delta_\epi,\Delta_\mono)$, as per the following correspondence for the seven equivalence classes of the figure (we only label the arrows that are not uniquely determined):
\begin{mathpar}
  \mathrm{A} = \ord{0} \rmono \ord{3}

  \mathrm{B} = \ord{0} \rmono \ord{2} \overset{\sigma_0}\lepi {3}

  \mathrm{C} = \ord{0} \rmono \ord{2} \overset{\sigma_1}\lepi {3}

  \mathrm{D} = \ord{0} \rmono \ord{1} \lepi {2} \overset{\delta_0}\rmono{3}

  \mathrm{E} = \ord{0} \rmono \ord{1} \lepi{3}

  \mathrm{F} = \ord{0} \rmono \ord{1} \lepi{2} \overset{\delta_1}\rmono{3}

  \mathrm{G} = \ord{0} \rmono \ord{1} \lepi{2} \overset{\delta_2}\rmono{3}
\end{mathpar}
The sequence counting the number of equivalence classes of closed ambisimplicial formulas,
\[1, 1, 2, 7, 35, 226, 1787, 16717, \dots, \]
is apparently OEIS \href{https://oeis.org/A014307}{A014307}, known to have various other interpretations \cite{Ren2015ordered}.

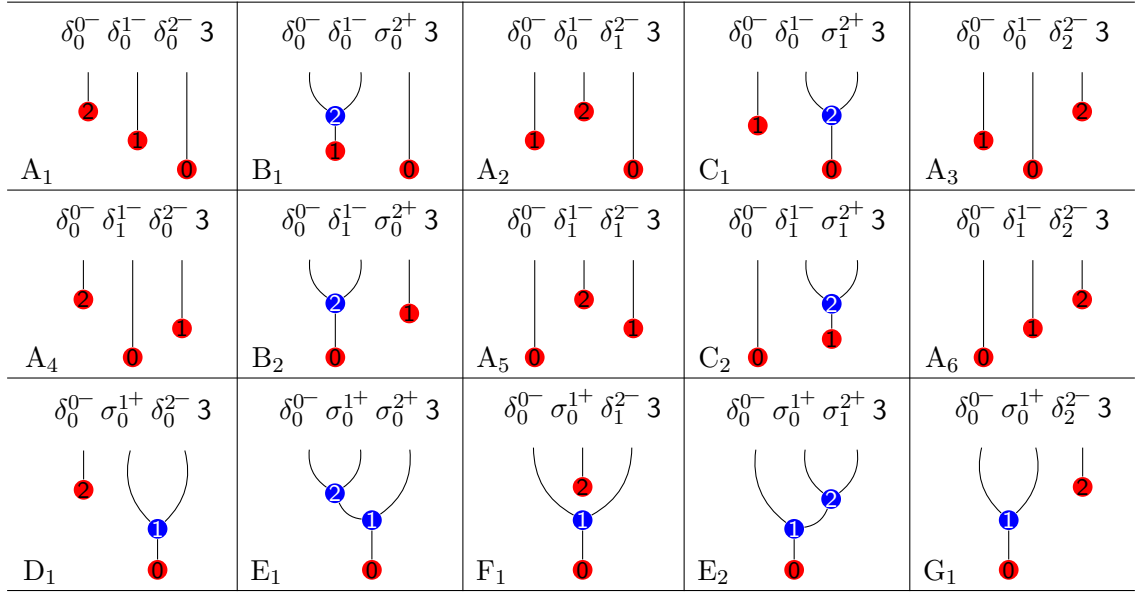
\begin{figure}
\begin{center}
\begin{tabular}{c|c|c|c|c}
\hline
A${}_1$\hspace{-5pt}
\begin{tikzpicture}
  \node(x) {};
  \node[right=1em of x] (y) {};
  \node[right=1em of y] (z) {};
  \node[delta,below=1em of x] (xend) {2};
  \node[delta,below=2em of y] (yend) {1};
  \node[delta,below=3em of z] (zend) {0};
  \draw (x) to (xend);
  \draw (y) to (yend);
  \draw (z) to (zend);
  \node[above=-.25em of y] (frm) {$\pull{\delta_0^0}\pull{\delta_0^1}\pull{\delta_0^2}\ord{3}$};
\end{tikzpicture}
  &
B${}_1$\hspace{-5pt}\begin{tikzpicture}
  \node(x) {};
  \node[right=1em of x] (y) {};
  \node[right=1em of y] (z) {};
  \node[sigma,below=1.5em of $(x)!0.5!(y)$] (xy) {2};
  \node[delta,below=.5em of xy] (xyend) {1};
  \node[delta,below=3em of z] (zend) {0};
  \draw[bend right] (x) to (xy);
  \draw[bend left] (y) to (xy);
  \draw (xy) to (xyend);
  \draw (z) to (zend);
  \node[above=-.25em of y] (frm) {$\pull{\delta_0^0}\pull{\delta_0^1}\push{\sigma_0^2}\ord{3}$};
\end{tikzpicture}
  &
A${}_2$\hspace{-5pt}\begin{tikzpicture}
  \node(x) {};
  \node[right=1em of x] (y) {};
  \node[right=1em of y] (z) {};
  \node[delta,below=2em of x] (xend) {1};
  \node[delta,below=1em of y] (yend) {2};
  \node[delta,below=3em of z] (zend) {0};
  \draw (x) to (xend);
  \draw (y) to (yend);
  \draw (z) to (zend);
  \node[above=-.25em of y] (frm) {$\pull{\delta_0^0}\pull{\delta_0^1}\pull{\delta_1^2}\ord{3}$};
\end{tikzpicture}
  &
C${}_1$\hspace{-5pt}\begin{tikzpicture}
  \node(x) {};
  \node[right=1em of x] (y) {};
  \node[right=1em of y] (z) {};
  \node[sigma,below=1.5em of $(y)!0.5!(z)$] (yz) {2};
  \node[delta,below=1.5em of x] (xend) {1};
  \node[delta,below=1.15em of yz] (yzend) {0};
  \draw[bend right] (y) to (yz);
  \draw[bend left] (z) to (yz);
  \draw (yz) to (yzend);
  \draw (x) to (xend);
  \node[above=-.25em of y] (frm) {$\pull{\delta_0^0}\pull{\delta_0^1}\push{\sigma_1^2}\ord{3}$};
\end{tikzpicture}
  &
A${}_3$\hspace{-5pt}\begin{tikzpicture}
  \node(x) {};
  \node[right=1em of x] (y) {};
  \node[right=1em of y] (z) {};
  \node[delta,below=2em of x] (xend) {1};
  \node[delta,below=3em of y] (yend) {0};
  \node[delta,below=1em of z] (zend) {2};
  \draw (x) to (xend);
  \draw (y) to (yend);
  \draw (z) to (zend);
  \node[above=-.25em of y] (frm) {$\pull{\delta_0^0}\pull{\delta_0^1}\pull{\delta_2^2}\ord{3}$};
\end{tikzpicture} \\
\hline
A${}_4$\hspace{-5pt}\begin{tikzpicture}
  \node(x) {};
  \node[right=1em of x] (y) {};
  \node[right=1em of y] (z) {};
  \node[delta,below=1em of x] (xend) {2};
  \node[delta,below=3em of y] (yend) {0};
  \node[delta,below=2em of z] (zend) {1};
  \draw (x) to (xend);
  \draw (y) to (yend);
  \draw (z) to (zend);
  \node[above=-.25em of y] (frm) {$\pull{\delta_0^0}\pull{\delta_1^1}\pull{\delta_0^2}\ord{3}$};
\end{tikzpicture}
  &
B${}_2$\hspace{-5pt}\begin{tikzpicture}
  \node(x) {};
  \node[right=1em of x] (y) {};
  \node[right=1em of y] (z) {};
  \node[sigma,below=1.5em of $(x)!0.5!(y)$] (xy) {2};
  \node[delta,below=1.15em of xy] (xyend) {0};
  \node[delta,below=1.5em of z] (zend) {1};
  \draw[bend right] (x) to (xy);
  \draw[bend left] (y) to (xy);
  \draw (xy) to (xyend);
  \draw (z) to (zend);
  \node[above=-.25em of y] (frm) {$\pull{\delta_0^0}\pull{\delta_1^1}\push{\sigma_0^2}\ord{3}$};
\end{tikzpicture}
  &
A${}_5$\hspace{-5pt}\begin{tikzpicture}
  \node(x) {};
  \node[right=1em of x] (y) {};
  \node[right=1em of y] (z) {};
  \node[delta,below=3em of x] (xend) {0};
  \node[delta,below=1em of y] (yend) {2};
  \node[delta,below=2em of z] (zend) {1};
  \draw (x) to (xend);
  \draw (y) to (yend);
  \draw (z) to (zend);
  \node[above=-.25em of y] (frm) {$\pull{\delta_0^0}\pull{\delta_1^1}\pull{\delta_1^2}\ord{3}$};
\end{tikzpicture}
  &
C${}_2$\hspace{-5pt}\begin{tikzpicture}
  \node(x) {};
  \node[right=1em of x] (y) {};
  \node[right=1em of y] (z) {};
  \node[sigma,below=1.5em of $(y)!0.5!(z)$] (yz) {2};
  \node[delta,below=3em of x] (xend) {0};
  \node[delta,below=.5em of yz] (yzend) {1};
  \draw[bend right] (y) to (yz);
  \draw[bend left] (z) to (yz);
  \draw (yz) to (yzend);
  \draw (x) to (xend);
  \node[above=-.25em of y] (frm) {$\pull{\delta_0^0}\pull{\delta_1^1}\push{\sigma_1^2}\ord{3}$};
\end{tikzpicture}
  &
A${}_6$\hspace{-5pt}\begin{tikzpicture}
  \node(x) {};
  \node[right=1em of x] (y) {};
  \node[right=1em of y] (z) {};
  \node[delta,below=3em of x] (xend) {0};
  \node[delta,below=2em of y] (yend) {1};
  \node[delta,below=1em of z] (zend) {2};
  \draw (x) to (xend);
  \draw (y) to (yend);
  \draw (z) to (zend);
  \node[above=-.25em of y] (frm) {$\pull{\delta_0^0}\pull{\delta_1^1}\pull{\delta_2^2}\ord{3}$};
\end{tikzpicture} \\
\hline
D${}_1$\hspace{-5pt}\begin{tikzpicture}
  \node(x) {};
  \node[right=1em of x] (y) {};
  \node[right=1em of y] (z) {};
  \node[sigma,below=2.7em of $(y)!0.5!(z)$] (yz) {1};
  \node[delta,below=1em of x] (xend) {2};
  \node[delta,below=.7em of yz] (yzend) {0};
  \draw[bend right] (y) to (yz);
  \draw[bend left] (z) to (yz);
  \draw (yz) to (yzend);
  \draw (x) to (xend);
  \node[above=-.25em of y] (frm) {$\pull{\delta_0^0}\push{\sigma_0^1}\pull{\delta_0^2}\ord{3}$};
\end{tikzpicture}
  &
E${}_1$\hspace{-5pt}\begin{tikzpicture}
  \node(x) {};
  \node[right=1em of x] (y) {};
  \node[right=1em of y] (z) {};
  \node[sigma,below=1.5em of $(x)!0.5!(y)$] (xy) {2};
  \node[sigma,below=1.5em of $(xy)!0.5!(z)$] (xyz) {1};
  \node[delta,below=1em of xyz] (xyzend) {0};
  \draw[bend right] (x) to (xy);
  \draw[bend left] (y) to (xy);
  \draw[bend right] (xy) to (xyz);
  \draw[bend left] (z) to (xyz);
  \draw (xyz) to (xyzend);
  \node[above=-.25em of y] (frm) {$\pull{\delta_0^0}\push{\sigma_0^1}\push{\sigma_0^2}\ord{3}$};
\end{tikzpicture}
  &
F${}_1$\hspace{-5pt}\begin{tikzpicture}
  \node(x) {};
  \node[right=1em of x] (y) {};
  \node[right=1em of y] (z) {};
  \node[delta,below=1em of y] (yend) {2};
  \node[sigma,below=2.5em of $(x)!0.5!(z)$] (xz) {1};
  \node[delta,below=1em of xz] (xzend) {0};
  \draw[bend right] (x) to (xz);
  \draw[bend left] (z) to (xz);
  \draw (xz) to (xzend);
  \draw (y) to (yend);
  \node[above=-.25em of y] (frm) {$\pull{\delta_0^0}\push{\sigma_0^1}\pull{\delta_1^2}\ord{3}$};
\end{tikzpicture}
  &
E${}_2$\hspace{-5pt}\begin{tikzpicture}
  \node(x) {};
  \node[right=1em of x] (y) {};
  \node[right=1em of y] (z) {};
  \node[sigma,below=1.7em of $(y)!0.5!(z)$] (yz) {2};
  \node[sigma,below=1.7em of $(x)!0.5!(yz)$] (xyz) {1};
  \node[delta,below=.7em of xyz] (xyzend) {0};
  \draw[bend right] (y) to (yz);
  \draw[bend left] (z) to (yz);
  \draw[bend right] (x) to (xyz);
  \draw[bend left] (yz) to (xyz);
  \draw (xyz) to (xyzend);
  \node[above=-.25em of y] (frm) {$\pull{\delta_0^0}\push{\sigma_0^1}\push{\sigma_1^2}\ord{3}$};
\end{tikzpicture}
  &
G${}_1$\hspace{-5pt}\begin{tikzpicture}
  \node(x) {};
  \node[right=1em of x] (y) {};
  \node[right=1em of y] (z) {};
  \node[sigma,below=2.5em of $(x)!0.5!(y)$] (xy) {1};
  \node[delta,below=1em of xy] (xyend) {0};
  \node[delta,below=1em of z] (zend) {2};
  \draw[bend right] (x) to (xy);
  \draw[bend left] (y) to (xy);
  \draw (xy) to (xyend);
  \draw (z) to (zend);
  \node[above=-.25em of y] (frm) {$\pull{\delta_0^0}\push{\sigma_0^1}\pull{\delta_2^2}\ord{3}$};
\end{tikzpicture}\\
  \hline
\end{tabular}
\end{center}
\caption{Depiction of all 15 closed ambisimplicial formulas with $\chi(S) = 3$ as increasing binary forests. The labels A--G indicate the split into seven equivalence classes.}
\label{fig:15formulas}
\end{figure}

Perhaps most interestingly, the remaining equations of the simplex category \eqref{eq:simplex3}--\eqref{eq:simplex5}, which mix  $\delta$ and $\sigma$, induce non-invertible \emph{logical entailments} between formulas:
  \begin{align}
    \push{\sigma_{j-1}^{n-1}}\pull{\delta_i^n} S &\vdashf{} \pull{\delta_i^{n-1}}\push{\sigma_j^n} S & \text{if }i<j \label{ordrel1}\\
    \pull{\delta_i^n} S &\vdashf{} \push{\sigma_j^n} S & \text{if }i=j\text{ or }i=j+1 \label{ordrel2}\\
    \push{\sigma_j^{n-1}}\pull{\delta_i^n} S &\vdashf{} \pull{\delta_{i-1}^{n-1}}\push{\sigma_j^n} S & \text{if }i > j+1 \label{ordrel3}
  \end{align}
The middle entailment is obtained as the composite of cartesian arrows $\cart {\delta_i} S \hcomp \opcart {\sigma_j} S$, which lies over the identity by equation \eqref{eq:simplex4}.
For the first and third, recall that in general, any commutative square $ab = cd$ in the base of a bifibration (or in the base of a $(\P,\N)$-fibration with $b,c\in\P$ and $a,d\in \N$) induces a canonical natural transformation $\push{c}\pull{a} S \vdashf{} \pull{d}\push{b} S$:
\begin{equation}\label{der:abcd}\tag{abcd}
\inferrule*[fraction={---},Right={$\Lpush[c]\Rpull[d]$}]
{\inferrule*[fraction={===}]
  {\inferrule*[fraction={---},Right={$\Lpull[a]\Rpush[b]$}]{\inferrule*[Right={$\Id[S]$}]{ }{S \vdashf{} S} }
{\pull{a} S \vdashf{ab} \push{b} S} }
{\pull{a} S \vdashf{cd} \push{b} S}}
{\push{c}\pull{a} S \vdashf{} \pull{d}\push{b} S}
\end{equation}
The Beck-Chevalley condition asks that this natural transformation be invertible whenever $ab = cd$ is a pullback square, but we are in a free setting where the BC condition fails.

All of these non-invertible entailments have simple graphical interpretations, either as moving a binary node upwards relative to a root node somewhere on its right/left, 
or as joining a root node to an edge on its immediate left/right to form a binary node, as drawn below 
(we only show one version of each type of transformation, eliding the mirror image):
\begin{equation}\label{eq:ordrel-diagrams}
\begin{tikzpicture}[baseline={([yshift=-.5ex]current bounding box.center)},scale=0.75]
  \path (0,0) coordinate (x0)
  +(1,0) coordinate (x1)
  +(2,0) coordinate (x2)
  +(0.5,-1.5) coordinate (x01s)
  +(2,-1) coordinate (x2d)
  +(0.5,-2) coordinate (x01end)
  ;
  \draw[bend right] (x0) to (x01s);
  \draw[bend left] (x1) to (x01s);
  \draw (x01s) to (x01end);
  \node[sigma] at (x01s) {};
  \draw (x2) to (x2d);
  \node[delta] at (x2d) {};
  \node[below=1.5em of x1] (ymid) [right] {\dots};
  \end{tikzpicture}
  \quad\Longrightarrow \quad
\begin{tikzpicture}[baseline={([yshift=-.5ex]current bounding box.center)},scale=0.75]
  \path (0,0) coordinate (x0)
  +(1,0) coordinate (x1)
  +(2,0) coordinate (x2)
  +(0.5,-1) coordinate (x01s)
  +(2,-1.5) coordinate (x2d)
  +(0.5,-2) coordinate (x01end)
  ;
  \draw[bend right] (x0) to (x01s);
  \draw[bend left] (x1) to (x01s);
  \draw (x01s) to (x01end);
  \node[sigma] at (x01s) {};
  \draw (x2) to (x2d);
  \node[delta] at (x2d) {};
  \node[below=1.5em of x1] (ymid) [right] {\dots};
  \end{tikzpicture}
  \qquad\qquad\qquad
  \begin{tikzpicture}[baseline={([yshift=-.5ex]current bounding box.center)},scale=0.75]
  \path (0,0) coordinate (x0)
  +(1,0) coordinate (x1)
  +(1,-1) coordinate (x1d)
  +(0,-2) coordinate (x0end)
  ;
  \draw (x0) to (x0end);
  \draw (x1) to (x1d);
  \node[delta] at (x1d) {};
  \end{tikzpicture}
  \quad\Longrightarrow \quad
  \begin{tikzpicture}[baseline={([yshift=-.5ex]current bounding box.center)},scale=0.75]
  \path (0,0) coordinate (x0)
  +(1,0) coordinate (x1)
  +(0.5,-1) coordinate (x01s)
  +(0.5,-2) coordinate (x01end)
  ;
  \draw [bend right] (x0) to (x01s);
  \draw [bend left] (x1) to (x01s);
  \draw (x01s) to (x01end);
  \node[sigma] at (x01s) {};
  \end{tikzpicture}
\end{equation}
Now define a family of posets $(F_{k,n})_{k,n\in \NN}$ as follows.
The carrier of $F_{k,n}$ is the set of ambisimplicial formulas $S \refs \ord{k}$ such that $\chi(S) = n$ (i.e., forests with $k$ open roots and $n$ open leaves), quotiented by \eqref{simp:equiv1} and \eqref{simp:equiv2}.
The order relation is the least congruence generated by \eqref{ordrel1}--\eqref{ordrel3}, where by congruence we mean that $S \le T$ implies $\push{e} S \le \push{e}T$ and $\pull{m} S \le \pull{m} T$.
For example, here is the Hasse diagram of $F_{0,3}$ using the same labels for the equivalence classes from Figure~\ref{fig:15formulas}:
\[
\begin{tikzcd}[row sep=1.5em,column sep=2em,arrows={draw,dash}]
&E& \\
C\ar[ur]&&B\ar[ul]\\
& F\ar[uu] & \\
D\ar[uu]&&G\ar[uu]& \\
& A\ar[ul]\ar[uu]\ar[ur] &
\end{tikzcd}
\]
\begin{thm}
  $\Bifib{i,\P,\N}_{\ord{k}} \simeq \coprod_{n\in\NN} F_{k,n}$.
\end{thm}
\begin{proof}[Proof sketch]
  We have already established that the laws \eqref{simp:equiv1}--\eqref{ordrel3} hold in the free ambifibration, which means that the $F_{k,n}$ embed into $\Bifib{i,\P,\N}_{\ord{k}}$.
  
  Conversely, we will begin by showing more generally that if $\alpha : S \vdashf{f} T$ is an arrow of $\Bifib{i,\P,\N}$ and $f = em$ is the epi-mono factorization of the underlying arrow in $\Delta$, then $\push{e} S \le \pull{m}T$ in the appropriate poset $F_{k,n}$ with $n = \chi(S) = \chi(T)$ and $\ord{k} = \cod(e) = \dom(m)$.
  We prove this by induction on derivations.
  The interesting case is when $\alpha$ ends in a non-invertible rule, for example suppose it is of the form $\alpha = \tLpull g \alpha' : \pull{g}S \vdashf{gh} T$ for some $\alpha' : S \vdashf{h} T$ and $g \in \N\subset \Delta_\mono$.
  By the induction hypothesis, we have that $\push{e}S \le \pull{m}T$ where $h = em$, $e\in\Delta_\epi$, $m \in \Delta_\mono$.
  Let $f = dk$, $d\in \Delta_\epi$, $k\in \Delta_\mono$ be the epi-mono factorization of $f=gh$.
  Thus we have $dk = gem$.
  By the orthogonal lifting property there is a unique mono $\ell$ making the diagram below commute:
  \[
    \begin{tikzcd}
     .\ar[d,two heads,"d"']\ar[r,"g",tail] &. \ar[r,two heads,"e"] & .\ar[d,tail,"m"] & . \\
      . \ar[rr,tail,"k"']\ar[rru,tail,dashed,"\ell" description] && .
    \end{tikzcd}
  \]
  Hence $\push{d}\pull{g}S \le \pull{\ell}\push{e} S \le \pull{\ell} \pull{m} T \equiv \pull{k}T$, where the first entailment is an instance of the \eqref{der:abcd} scheme on the square $ge = d\ell$.
  (Here we need a lemma that every instance of the scheme is valid, $\push{c}\pull{a}S \le \pull{d}\push{b} S$, for $a,b \in \Delta_\epi$ and $c,d \in \Delta_\mono$, which we can prove by induction on the length of the $\P$-factorization of $c$.)

  Finally, to establish the equivalence $\Bifib{i,\P,\N}_{\ord{k}} \simeq \coprod_{n\in\NN} F_{k,n}$ we need to verify that the fiber categories are preorders.
  In fact, we claim that the total category $\Bifib{i,\P,\N}$ is itself a preorder: for any pair of ambisimplicial formulas $S\refs \ord{n},T \refs \ord{m}$ there exists at most one arrow $\alpha : S \vdashf{f} T$ lying over a unique order-preserving function $f : \ord{n} \to \ord{m}$.
  Intuitively, the reason why is that left division by an epi~$e$ and right division by a mono~$m$ are \emph{deterministic} and \emph{independent} operations:
  \begin{itemize}
  \item determinism:   given an arrow $f$, there is at most one $g$ such that $f = eg$ and at most one $h$ such that $f = hm$, so given any derivation $\alpha$ with underlying base arrow $f$, we can write
  $\uLpush e {} \alpha$ and
  $\alpha \uRpull {} m$ rather than  
  $\uLpush e {g} \alpha$ and
  $\alpha \uRpull {h} m$ without risk of ambiguity;
  \item independence: if $\alpha$ is divisible on the left by $e$ and divisible on the right by $m$ (i.e., both $\uLpush e {} \alpha$ and $\alpha \uRpull {} m$ are defined), then it is simultaneously divisible on the left and right (i.e., $\uLpush e {} \alpha \uRpull {} m$ is defined), by the orthogonal lifting property:
  \[
    \begin{tikzcd}
     .\ar[d,two heads,"e"']\ar[r,"h"] & .\ar[d,tail,"m"] \\
      . \ar[r,"g"']\ar[ru,dashed,"i" description] & .
    \end{tikzcd}
  \]
  \end{itemize}
  Determinism and independence imply that any two derivations of arrows $\alpha_1,\alpha_2 : S \to T$ are permutation equivalent: they necessarily perform the same series of multiplication and division operations, although not necessarily in the same order, but if we examine the first place that they differ --- say that $\alpha_1$ performs a left division $\uLpush e {} \beta$ and $\alpha_2$ a right division $\beta \uRpull {} m$ --- then we can always replace that by a bi-division $\uLpush e {} \beta \uRpull {} m$ and continue unifying the rest of the proofs.  
\end{proof}
\noindent
We remark that since $\Delta$ is both $\P$-\FP and $\N$-\FP (like any category relative to a class of epis and monos), every entailment $S \le T$ in $F_{k,n}$ is witnessed by a unique maximally multifocused proof $\alpha : S \vdashf{} T$.
For example, below on the right we show the maximally multifocused proof corresponding to the forest entailment on the left (we work with strictly alternating zigzags, using graphical conventions for the corresponding forests that should be clear, and we notate order-preserving functions $f : \ord{n} \to \ord{m}$ by tuples of numbers  $a_0,\dots,a_{m-1}$ that indicate the cardinalities of the fibers $a_i = |f^{-1}(i)|$):
\[
\begin{tikzpicture}[baseline={([yshift=-.5ex]current bounding box.center)},scale=0.4]
\path (0,0) coordinate (x0)
+(1,0) coordinate (x1)
+(2,0) coordinate (x2)
+(3,0) coordinate (x3)
+(4,0) coordinate (x4)
+(5,0) coordinate (x5)
+(6,0) coordinate (x6)
+(7,0) coordinate (x7)
+(3,-1) coordinate (x3d)
+(5,-1) coordinate (x5d)
+(3,-2) coordinate (x24s)
+(3,-3) coordinate (x24d)
+(0.5,-4) coordinate (x01s)
+(0.5,-5) coordinate (x01d)
+(6.5,-4) coordinate (x67s)
+(6.5,-5) coordinate (x67d)
;
\draw (x5) to (x5d);
\draw (x3) to (x3d);
\draw[bend left] (x4) to (x24s);
\draw[bend right] (x2) to (x24s);
\draw (x24s) to (x24d);
\draw[bend right] (x0) to (x01s);
\draw[bend left] (x1) to (x01s);
\draw[bend right] (x6) to (x67s);
\draw[bend left] (x7) to (x67s);
\draw (x01s) to (x01d);
\draw (x67s) to (x67d);
\node[delta] at (x3d) {};
\node[delta] at (x5d) {};
\node[sigma] at (x24s) {};
\node[delta] at (x24d) {};
\node[sigma] at (x01s) {};
\node[sigma] at (x67s) {};
\node[delta] at (x01d) {};
\node[delta] at (x67d) {};
\end{tikzpicture}\quad\Longrightarrow\quad
\begin{tikzpicture}[baseline={([yshift=-.5ex]current bounding box.center)},scale=0.4]
\path (0,0) coordinate (x0)
+(1,0) coordinate (x1)
+(2,0) coordinate (x2)
+(3,0) coordinate (x3)
+(4,0) coordinate (x4)
+(5,0) coordinate (x5)
+(6,0) coordinate (x6)
+(7,0) coordinate (x7)
+(5,-2) coordinate (x5d)
+(2,-3) coordinate (x01234s)
+(2,-4) coordinate (x01234d)
+(6.5,-1) coordinate (x67s)
+(6.5,-4) coordinate (x67d)
;
\draw (x5) to (x5d);
\draw[bend right] (x0) to (x01234s);
\draw[bend right] (x1) to (x01234s);
\draw (x2) to (x01234s);
\draw[bend left] (x3) to (x01234s);
\draw[bend left] (x4) to (x01234s);
\draw (x01234s) to (x01234d);
\draw[bend right] (x6) to (x67s);
\draw[bend left] (x7) to (x67s);
\draw (x67s) to (x67d);
\node[delta] at (x5d) {};
\node[sigma] at (x01234s) {};
\node[delta] at (x01234d) {};
\node[sigma] at (x67s) {};
\node[delta] at (x67d) {};
\end{tikzpicture}
\qquad
\begin{tikzcd}[ampersand replacement=\&,column sep=2.5cm]
|[alias=l0]| 8 \ar[r,equals] \& |[alias=r0]| 8 
\\ |[alias=l1]|6\ar[u,tail,"{1,1,1,0,1,0,1,1}"]\ar[r,"{1,1,1,0,1,0,2}" description] \& |[alias=r1]|7\ar[u,<-,twoheadleftarrow,"{1,1,1,1,1,1,2}"']
\\ |[alias=l2]|6\ar[u,equals]\ar[r,"{1,1,1,0,1,2}" description] \& |[alias=r2]|6\ar[u,tail,"{1,1,1,1,1,0,1}"']
\\ |[alias=l3]|6\ar[u,equals]\ar[r,"{4,2}" description] \& |[alias=r3]|2\ar[u,<-,twoheadleftarrow,"{5,1}"']
\\ |[alias=l4]|5\ar[u,<-,twoheadleftarrow,"{1,1,2,1,1}"]\ar[r,"{3,2}" description] \& |[alias=r4]|2\ar[u,equals]
\\ |[alias=l5]|4\ar[u,tail,"{1,1,0,1,1}"]\ar[r,"{2,2}" description] \& |[alias=r5]|2\ar[u,equals]
\\ |[alias=l6]|2\ar[u,<-,twoheadleftarrow,"{2,2}"]\ar[r,equals] \& |[alias=r6]|2\ar[u,equals]
\\ |[alias=l7]|0\ar[u,tail,"{0,0}"]\ar[r,"{0,0}" description] \& |[alias=r7]|2\ar[u,equals]
\\ |[alias=l8]|0\ar[u,equals]\ar[r,equals] \& |[alias=r8]|0\ar[u,tail,"{0,0}"']
\arrow[gray,phantom,from=1-1,to=2-2,"\pull \Lsym \push \Rsym"]
\arrow[gray,phantom,from=2-1,to=3-2,"\pull \Rsym"]
\arrow[gray,phantom,from=3-1,to=4-2,"\push \Rsym"]
\arrow[gray,phantom,from=4-1,to=5-2,"\push \Lsym"]
\arrow[gray,phantom,from=5-1,to=6-2,"\pull \Lsym"]
\arrow[gray,phantom,from=6-1,to=7-2,"\push \Lsym"]
\arrow[gray,phantom,from=7-1,to=8-2,"\pull \Lsym"]
\arrow[gray,phantom,from=8-1,to=9-2,"\pull \Rsym"]
\end{tikzcd}
  \]

The $(F_{k,n})_{n\in \NN}$ seem to define an interesting family of posets.
For example, we observe that $F_{0,n}$ contains the \emph{lattice of noncrossing partitions} as a quotient.
A noncrossing partition of order $n$ is a partition of the ordinal $\ord{n} = \set{0,\dots,n-1}$ such that the diagram defined by drawing the elements in order and then connecting all of the elements within the same block by a corolla is planar.
For example, of the ${4 \brace 2} = 7$ partitions of $\ord{4}$ into two blocks, all but one are noncrossing:
\begin{mathpar}
\begin{tikzpicture}
\path coordinate[label=above:0] (1) 
+(.5,0) coordinate[label=above:1] (2)
+(1,0) coordinate[label=above:2] (3)
+(1.5,0) coordinate[label=above:3] (4)
+(.25,-.4) coordinate (12)
+(1.25,-.4) coordinate (34)
;
\draw[out=180,in=-90] (12) to (1);
\draw[out=0,in=-90] (12) to (2);
\draw[out=180,in=-90] (34) to (3);
\draw[out=0,in=-90] (34) to (4);
\node[kappa] (x) at (12) {};
\node[kappa] (y) at (34) {};
\end{tikzpicture}

\begin{tikzpicture}
\path coordinate[label=above:0] (1) 
+(.5,0) coordinate[label=above:1] (2)
+(1,0) coordinate[label=above:2] (3)
+(1.5,0) coordinate[label=above:3] (4)
+(.5,-.4) coordinate (13)
+(1,-.4) coordinate (24)
;
\draw[out=180,in=-90] (13) to (1);
\draw (13) to (3);
\draw (24) to (2);
\draw[out=0,in=-90] (24) to (4);
\node[kappa] (x) at (13) {};
\node[kappa] (y) at (24) {};
\end{tikzpicture}

\begin{tikzpicture}
\path coordinate[label=above:0] (1) 
+(.5,0) coordinate[label=above:1] (2)
+(1,0) coordinate[label=above:2] (3)
+(1.5,0) coordinate[label=above:3] (4)
+(.75,-.4) coordinate (14)
+(.75,-.15) coordinate (23)
;
\draw[out=180,in=-90] (14) to (1);
\draw[out=0,in=-90] (14) to (4);
\draw[out=180,in=-90] (23) to (2);
\draw[out=0,in=-90] (23) to (3);
\node[kappa] (x) at (14) {};
\node[kappa] (y) at (23) {};
\end{tikzpicture}

\begin{tikzpicture}
\path coordinate[label=above:0] (1) 
+(.5,0) coordinate[label=above:1] (2)
+(1,0) coordinate[label=above:2] (3)
+(1.5,0) coordinate[label=above:3] (4)
+(.5,-.4) coordinate (123)
+(1.5,-.4) coordinate (44)
;
\draw[out=180,in=-90] (123) to (1);
\draw (123) to (2);
\draw[out=0,in=-90] (123) to (3);
\draw (44) to (4);
\node[kappa] (x) at (123) {};
\node[kappa] (y) at (44) {};
\end{tikzpicture}

\begin{tikzpicture}
\path coordinate[label=above:0] (1) 
+(.5,0) coordinate[label=above:1] (2)
+(1,0) coordinate[label=above:2] (3)
+(1.5,0) coordinate[label=above:3] (4)
+(.5,-.4) coordinate (124)
+(1,-.2) coordinate (33)
;
\draw[out=180,in=-90] (124) to (1);
\draw (123) to (2);
\draw[out=0,in=-90] (124) to (4);
\draw (33) to (3);
\node[kappa] (x) at (124) {};
\node[kappa] (y) at (33) {};
\end{tikzpicture}

\begin{tikzpicture}
\path coordinate[label=above:0] (1) 
+(.5,0) coordinate[label=above:1] (2)
+(1,0) coordinate[label=above:2] (3)
+(1.5,0) coordinate[label=above:3] (4)
+(1,-.4) coordinate (134)
+(.5,-.2) coordinate (22)
;
\draw[out=180,in=-90] (134) to (1);
\draw (134) to (3);
\draw[out=0,in=-90] (134) to (4);
\draw (22) to (2);
\node[kappa] (x) at (134) {};
\node[kappa] (y) at (22) {};
\end{tikzpicture}

\begin{tikzpicture}
\path coordinate[label=above:0] (1) 
+(.5,0) coordinate[label=above:1] (2)
+(1,0) coordinate[label=above:2] (3)
+(1.5,0) coordinate[label=above:3] (4)
+(1,-.4) coordinate (234)
+(0,-.4) coordinate (11)
;
\draw[out=180,in=-90] (234) to (2);
\draw (234) to (3);
\draw[out=0,in=-90] (234) to (4);
\draw (11) to (1);
\node[kappa] (x) at (234) {};
\node[kappa] (y) at (11) {};
\end{tikzpicture}
\end{mathpar}
The restriction of the usual refinement ordering to noncrossing partitions of order $n$ defines a poset $K_n$ known as the \emph{Kreweras lattice}.
\begin{prop}\label{prop:Kn-is-a-quotient-of-Fn}
$K_n$ is the quotient of $F_{0,n}$ obtained by inverting \eqref{ordrel1} and \eqref{ordrel3}, or equivalently by imposing the Beck-Chevalley condition for bicartesian squares.
\end{prop}
\begin{proof}
The noncrossing partition associated to a closed ambisimplicial formula is visually obvious if one looks at the corresponding increasing forest, since imposing \eqref{ordrel1} and \eqref{ordrel3} as equivalences in addition to \eqref{simp:equiv1} and \eqref{simp:equiv2} has the effect of forgetting the order in which nodes are created and reducing the forest down to its connected components.
For instance, formulas A${}_1$--A${}_6$ from Figure~\ref{fig:15formulas} correspond to the partition $\set{\{0\},\{1\},\{2\}}$, while formulas C${}_1$, C${}_2$, and D${}_1$ correspond to the partition $\set{\{0\},\{1,2\}}$.
The only remaining law \eqref{ordrel2} (see right side of \eqref{eq:ordrel-diagrams} for its diagrammatic depiction) exactly describes the covering relation for the refinement order on noncrossing partitions. Thus we recover $K_n$ as a quotient of $F_{0,n}$.
Finally, we establish that inverting \eqref{ordrel1} and \eqref{ordrel3} is equivalent to asking for the Beck-Chevalley condition on \emph{bicartesian} squares, that is, squares
\begin{equation*}
    \begin{tikzcd}
     .\ar[d,tail,"a"']\ar[r,two heads,"c"] & .\ar[d,tail,"d"] \\
      . \ar[r,two heads,"b"']\ar[ru,phantom,"\text{BC}"] & .
    \end{tikzcd}
\end{equation*}
  in $\Delta$ that are both pullback and pushout squares.
(Here we can fortuitously write ``BC'' as an abbreviation for both bicartesian and Beck-Chevalley.)
One can verify that the squares \eqref{eq:simplex3} and \eqref{eq:simplex5} are bicartesian, and hence that imposing the Beck-Chevalley condition for such squares strengthens \eqref{ordrel1} and \eqref{ordrel3} to equivalences:
  \begin{align}
    \push{\sigma_{j-1}^{n-1}} \pull{\delta_i^n} S &\equiv \pull{\delta_i^{n-1}}\push{\sigma_j^n} S & \text{if }i<j \label{eqrel1}\\
    \push{\sigma_j^{n-1}}\pull{\delta_i^n} S &\equiv \pull{\delta_{i-1}^{n-1}}\push{\sigma_j^n} S & \text{if }i > j+1 \label{eqrel3}
  \end{align}
  Moreover, we claim that all bicartesian squares mixing epi and mono legs in $\Delta$ are generated as pastings of \eqref{eq:simplex3} and \eqref{eq:simplex5}.
  Indeed, this follows from a characterization of pushout squares in $\Delta$ by Constantin, Fritz, Perrone, and Shapiro \cite[Theorem 3.10]{constantin2023weak}, since \eqref{eq:simplex3} and \eqref{eq:simplex5} are exactly the ``basic'' pushout squares mixing epis and monos that are also pullback~squares.
\end{proof}
\noindent
For example, $K_3$ is obtained from $F_{0,3}$ by identifying two pairs of equivalence classes $C\sim D$ and $B\sim G$.
Proposition~\ref{prop:Kn-is-a-quotient-of-Fn} and the seminal results of Kreweras~\cite{Kreweras1971} on the lattice of noncrossing partitions as well as related results for other Catalan lattices \cite{BeBo2009} motivate us to ask two open questions.\footnote{Soon after the original version of this paper was posted, we found out that the answer to Question~\ref{qu:lattice} is negative: already $F_{0,4}$ is not a lattice. We thank Wenjie~Fang for determining this using SageMath from data provided by our decision procedure for entailment (see Footnote~\ref{footnote:source-code}).}
\begin{qu}\label{qu:lattice}
$K_n$ is a lattice for all $n$: is $F_{0,n}$ (and more generally $F_{k,n}$) likewise a lattice?
\end{qu}
\begin{qu}\label{qu:intervals}
The number of intervals $I[K_n] = \set{(x,y) \mid x \le y}$ in $K_n$ is given by a simple formula $|I[K_n]| = \frac{1}{2n+1}\binom{3n}{n}$: is there likewise a nice formula for $|I[F_{0,n}]|$ (and more generally for $|I[F_{k,n}]|)$?
\end{qu}
\noindent
It is known that the lattice of noncrossing partitions is self-dual $K_n \cong K_n^\op$, and even \emph{locally self-dual} in the sense that for every pair of noncrossing partitions $x,y \in K_n$, the restriction of $K_n$ to the interval $[x,y]$ is self-dual \cite[V2:465]{StanleyEC1+2}.
It may be interesting to study this duality from a fibrational perspective, in particular relating it to the duality between $\Delta$ and the category of finite strict intervals and order-and-boundary-preserving maps \cite[\S1.1]{Joyal1997}.

\section{Conclusion and future directions}
\label{sec:conclusion}

This work originally began as an exercise in categorical proof theory, with our hope being that a systematic understanding of the construction of the free bifibration on a functor of categories $p : \D \to \C$ could eventually help in analyzing more complex fibrational situations that arise in logic and type theory.
This ``easy'' exercise proved more challenging than we expected (particularly the question of characterizing normal forms to obtain a canonicity result), but was also rewarding as we realized the implicit complexity and richness contained in the problem.
We were encouraged by the close connection between the sequent calculus presentation of the free bifibration and the presentation of the double category of zigzags by generators and relations, showing the unity of purpose between two basic tools in the proof-theoretic and category-theoretic repertoires; and we were excited to discover the examples of free bifibrations and ambifibrations in which the combinatorial complexity emerged from much simpler generating data.
We think there are many natural paths for further study.

Perhaps the most obvious next step would be to generalize the constructions beyond functors of categories.
As mentioned in the Introduction, we were originally motivated by phenomena that arise when considering bifibrations of monoidal closed categories and of generalized multicategories and polycategories.
We expect that the construction of the free bifibration on a functor should generalize to functors of polycategories, but the details are by no means obvious, especially the appropriate generalization of the \FP condition and whether there is a reasonable notion of maximal multifocusing.
Another natural generalization would be to incorporate higher cells in the base category, such as by considering functors of 2-categories or double categories.
Indeed, we already somewhat implicitly treated the base category as a 2-category in our conventions for drawing string diagrams where we represented equalities by vertices, as in Figure~\ref{fig:example-functor-and-derivation}; taking the base to be an actual 2-category would give this convention a more formal status, and may also make it possible to relax the \FP condition while retaining canonicity.
It should be noted that Licata, Shulman, and Riley~\cite{LSR2017} considered bifibrations of cartesian 2-multicategories, and that 2-cells were used to good effect in their notion of ``mode theory''.

We used string diagrams in this paper mainly for their evocative power (particularly in the proof of Theorem~\ref{thm:tree-category-free-bifibration}), but we believe there is value in taking them seriously as geometric objects and developing a purely topological or purely combinatorial description of free bifibrations.
Dawson, Paré, and Pronk did exactly that in their description of $\Pi_2(\C)$ for a free category, defining a bijection between 2-cells in $\Pi_2(\C)$ and isotopy classes of planar diagrams presented by matchings on a set of oriented points.
It could be worthwhile to extend such a correspondence for free bifibrations beyond the case of a free base category, to general categories or 2-categories.

Although we did not explore the direction in this paper, there is clearly a link to be developed between bifibrations or $(\P,\N)$-fibrations and game semantics~\cite{hirschowitz2007theory,Mellies2012game}.
It is noteworthy that the free bifibration on $p_\Two : \One \to \Two$ appeared incognito in a paper on the categorical combinatorics of innocent strategies \cite{HMM2007}, with the total category $\Bifib{p_\Two}$ of the free bifibration being equivalent to the authors' \emph{category of schedules} $\Upsilon$.

The surprising connection between the free bifibration on $p_\omega : \One \to \omega$ and the Joyal-Batanin-Street category of plane trees---which contains $\Bifib{p_\omega}_0$ as a wide subcategory but is not equivalent to it---is still somewhat mysterious to us, and we are not sure whether it is a mere coincidence.
Similarly, we do not know whether the free ambifibration on $i : \NN \to \Delta$ may (like the fat Delta ambifibration) be useful for defining weak higher categories, but it appears to be an interesting object of study in its own right.

\subsection*{Acknowledgments}
We warmly thank François Lamarche for stimulating discussions about his work on bifibrations~\cite{Lam13,Lam14} and for his enthusiasm about the project.
We also thank Nathanael~Arkor, Jonas~Frey, Tom~Hirshowitz, Paul-André~Melliès, Dorette Pronk, and Gilles~Schaeffer for discussions and pointers.
In particular, we are grateful to Jonas for discussions about the Beck-Chevalley condition, which made us realize that our original statement of Proposition~\ref{prop:Kn-is-a-quotient-of-Fn} was incorrect and that the BC condition should be restricted to bicartesian squares.

This work began while the three authors were at LIX and the first author was on a postdoc supported by the ANR project Repro (ANR-20-CE48-0016).
The third author was partially supported by the ANR project LambdaComb (ANR-21-CE48-0017).

\bibliographystyle{plainurl}
\bibliography{references.bib}

\newpage

\appendix

\section{Some more detailed proofs}
\label{app:proofs}

\subsection{Section~\ref{sec:sequent-calculus}}


\cutrespectspermeq*

\begin{proof}
  \label{prf:cut-respects-permeq}
  By symmetry it suffices to prove that
  $\alpha_1 \hcomp \beta \permeq \alpha_2 \hcomp \beta$. The proof is
  by case analysis on the proof of $\alpha_1 \permeq \alpha_2$: the
  equivalence and congruence cases are direct induction, the more delicate cases are the generating equations, which require an induction on $\beta$.

  \paragraph{Case $\tLpull f (\alpha \tRpush h) \permeq (\tLpull f \alpha) \tRpush h$}
  To reason about
  $((\tLpull f \alpha) \tRpush h) \hcomp \beta$,
  we perform a case analysis on $\beta$, depending on whether
  it starts with a left rule or a right rule. (It cannot be an atom as its domain is of the form $\push h U$.)

  If $\beta$ starts with a left rule and its domain is of the form $\push h U$, we have $\beta = \uLpush h i {\beta'}$ and
  \begin{mathpar}
    \begin{array}{cl@{\qquad}l}
      & (\tLpull f (\alpha \tRpush h)) \hcomp \beta
      & \\ \permeq
      & \tLpull f ((\alpha \tRpush h) \hcomp \beta)
      & \\ =
      & \tLpull f ((\alpha \tRpush h) \hcomp (\uLpush h i \beta'))
      & \\ =
      & \tLpull f (\alpha \hcomp \beta')
      & \\ \permeq
      & (\tLpull f \alpha) \hcomp \beta'
      & \\ =
      & ((\tLpull f \alpha) \tRpush h) \hcomp (\uLpush h i \beta')
      & \\ =
      & ((\tLpull f \alpha) \tRpush h) \hcomp \beta
      &
    \end{array}
  \end{mathpar}

  If $\beta$ starts with a right rule, let us assume that it is
  a division $\beta = \beta' \uRpull i j$---the multiplication case is similar. Then we have
  \begin{mathpar}
    \begin{array}{cl@{\qquad}l}
      & (\tLpull f (\alpha \tRpush h)) \hcomp \beta
      & \\ =
      & (\tLpull f (\alpha \tRpush h)) \hcomp (\beta' \uRpull i j)
      & \\ \permeq
      & ((\tLpull f (\alpha \tRpush h)) \hcomp \beta') \uRpull {fghi} j
      & \\ \permeq
      & (((\tLpull f \alpha) \tRpush h) \hcomp \beta') \uRpull {fghi} j
      & \text{by induction hypothesis ($\beta' < \beta$)}
      \\ =
      & ((\tLpull f \alpha) \tRpush h) \hcomp (\beta' \uRpull i j)
      & \\ =
      & ((\tLpull f \alpha) \tRpush h) \hcomp \beta
    \end{array}
  \end{mathpar}

  \paragraph{Case $ (\uLpush f g \alpha) \tRpush {h}  \permeq \uLpush f {gh} (\alpha \tRpush {h})$}
  This proof is in fact identical to the previous one: the proof above only depends on the fact $\tLpull f \alpha \tRpush h$ has a multiplication on the right, not on the operation performed on the left.

  \paragraph{Case $ (\tLpull {f} \alpha) \uRpull {fg} h \permeq \tLpull {f} (\alpha \uRpull g h)$} To reason about $(\alpha \uRpull g h) \hcomp \beta$, we again reason by case analysis on $\beta$, depending on whether it is a multiplication by $h$ on the left or a right-rule.

  If $\beta = \tLpull h \beta'$, then we have
  \begin{mathpar}
    \begin{array}{cl@{\qquad}l}
      & ((\tLpull {f} \alpha) \uRpull {fg} h) \hcomp \beta
      & \\ =
      & ((\tLpull {f} \alpha) \uRpull {fg} h) \hcomp (\tLpull h \beta')
      & \\ =
      & (\tLpull {f} \alpha) \hcomp \beta'
      & \\ \permeq
      & \tLpull {f} (\alpha \hcomp \beta')
      & \\ =
      & \tLpull {f} ((\alpha \uRpull g h) \hcomp (\tLpull h \beta'))
      & \\ =
      & \tLpull {f} ((\alpha \uRpull g h) \hcomp \beta)
      & \\ \permeq
      & (\tLpull {f} (\alpha \uRpull g h)) \hcomp \beta
    \end{array}
  \end{mathpar}

  If $\beta$ starts with a right rule, for example $\beta = \beta' \uRpull i j$, then we have
  \begin{mathpar}
    \begin{array}{cl@{\qquad}l}
      & ((\tLpull {f} \alpha) \uRpull {fg} h) \hcomp \beta
      & \\ =
      & ((\tLpull {f} \alpha) \uRpull {fg} h) \hcomp (\beta' \uRpull i j)
      & \\ =
      & (((\tLpull {f} \alpha) \uRpull {fg} h) \hcomp \beta') \uRpull {fgi} j
      & \\ \permeq
      & (\tLpull {f} (\alpha \uRpull g h)) \hcomp \beta') \uRpull {fgi} j
      & \text{by induction hypothesis ($\beta' < \beta$)}
      \\ \permeq
      & (\tLpull {f} (\alpha \uRpull g h)) \hcomp (\beta' \uRpull i j)
      \\ =
      & (\tLpull {f} (\alpha \uRpull g h)) \hcomp \beta
    \end{array}
  \end{mathpar}

  \paragraph{Case $(\uLpush {f} {gh} \alpha) \uRpull g h \permeq \uLpush {f} {g} (\alpha \uRpull {fg} h)$} Again, this proof is identical to the previous case, as they both have a division on the right of $\alpha$.
\end{proof}

\etaexpansionpushpull*

The proof of this equations is included in the proof of the following lemma.

\idneutral*

\begin{proof}
  \label{prf:id-neutral}
  \label{prf:eta-expansion-pushpull}
  By induction on $S$.
  We prove the left neutrality case, right neutrality is symmetric.

  \paragraph{Case $\atom X$} If $\alpha$ has an atomic domain or
  codomain, it must be an atomic derivation $\atom \delta$:

  \begin{mathpar}
    \Id[\atom X] \hcomp \alpha
    =
    \atom {(\Id[X])} \hcomp \atom \delta
    =
    \atom {(\Id[X] \; \delta)}
    =
    \atom \delta
    =
    \alpha
  \end{mathpar}

  \paragraph{Case $\push f S$}

  \begin{mathpar}
    \begin{array}{cl@{\qquad}l}
      & \Id[\push f S] \hcomp \alpha
      & \\ =
      & (\uLpush f {\Id[B]} \opcart f S)
        \hcomp
        \alpha
      & \\ \permeq
      & \uLpush f g
        (\opcart f S \hcomp \alpha)
    \end{array}
  \end{mathpar}

  We now prove that any $\alpha$ is permutation equivalent to
  $\uLpush f g
        (\opcart f S \hcomp \alpha)
 $
 by induction on $\alpha$.

  If $\alpha$ starts with a left rule, we must have
  $\alpha = \uLpush f g {\alpha'}$, and then:

  \begin{mathpar}
    \begin{array}{cl@{\qquad}l}
      & \uLpush f g
        (\opcart f S \hcomp \alpha)
      & \\ =
      & \uLpush f g
        ((\Id[S] \tRpush f) \hcomp (\uLpush f g {\alpha'}))
      & \\ =
      & \uLpush f g (\Id[S] \hcomp \alpha')
      & \\ \permeq
      & \uLpush f g \alpha'
      & \text{by induction hypothesis on $S$}
      \\ =
      & \alpha
    \end{array}
  \end{mathpar}

  If $\alpha$ starts with a right rule, for example
  $\alpha = \alpha' \uRpull g h$, we have

  \begin{mathpar}
    \begin{array}{cl@{\qquad}l}
      & \uLpush f g
        (\opcart f S \hcomp \alpha)
      & \\ =
      & \uLpush f g
        (\opcart f S \hcomp (\alpha' \uRpull g h)
      & \\ =
      & \uLpush f g
        ((\opcart f S \hcomp \alpha') \uRpull g h)
      & \\ \permeq
      & (\uLpush f g
        (\opcart f S \hcomp \alpha')) \uRpull g h
      & \\ \permeq
      & \alpha' \uRpull g h
      & \text{by induction hypothesis on $\alpha'$}
      \\ =
      & \alpha
    \end{array}
  \end{mathpar}

  \paragraph{Case $\pull g T$}

  \begin{mathpar}
    \begin{array}{cl@{\qquad}l}
      & \Id[\pull g T] \hcomp \alpha
      & \\ =
      & (\cart g T \uRpull {\Id[B]} g)
        \hcomp
        \alpha
    \end{array}
  \end{mathpar}

  We prove that any $\alpha$ is permutation equivalent to
  $(\cart g T \uRpull {\Id[B]} g)
  \hcomp
  \alpha$
  by induction on $\alpha$.

  If $\alpha$ starts with a left rule, we must have
  $\alpha = \tLpull g \alpha'$, and so
  \begin{mathpar}
    \begin{array}{cl@{\qquad}l}
      & (\cart g T \uRpull {\Id[B]} g)
        \hcomp
        \alpha
      & \\ =
      & ((\tLpull g \Id[T]) \uRpull {\Id[B]} g)
        \hcomp
        (\tLpull g \alpha')
      & \\ =
      & (\tLpull g \Id[T]) \hcomp \alpha'
      & \\ \permeq
      & \tLpull g (\Id[T] \hcomp \alpha')
      & \\ \permeq
      & \tLpull g \alpha'
      & \text{by induction hypothesis on $T$}
      \\ =
      & \alpha
    \end{array}
  \end{mathpar}

  If $\alpha$ starts with a right rule, for example
  $\alpha = \alpha' \uRpull h i$, we have
  \begin{mathpar}
    \begin{array}{cl@{\qquad}l}
      & (\cart g T \uRpull {\Id[B]} g)
        \hcomp
        \alpha
      & \\ =
      & (\cart g T \uRpull {\Id[B]} g)
        \hcomp
        (\alpha' \uRpull h i)
      & \\ =
      & ((\cart g T \uRpull {\Id[B]} g)
        \hcomp
        \alpha') \uRpull h i
      \\ \permeq
      & \alpha' \uRpull h i
      & \text{by induction on $\alpha'$}
      \\ =
      & \alpha
    \end{array}
  \end{mathpar}
\end{proof}

\cutassociative*

\begin{proof}
  \label{prf:cut-associative}
  The proof is by induction on $\alpha$, $\beta$, $\gamma$.

  If $\alpha$ is an atomic derivation, then $\beta$ and $\gamma$ must be as well---and conversely---and then associativity is immediate:
  \begin{mathpar}
    (\atom \delta \hcomp \atom \epsilon) \hcomp \atom \zeta
    =
    \atom {(\delta \epsilon \zeta)}
    =
    \atom \delta \hcomp (\atom \epsilon \hcomp \atom \zeta)
  \end{mathpar}

  If $\alpha$ starts with a left rule, then associativity is
  immediate. For example if $\alpha = \tLpull f \alpha'$, we have
  \begin{mathpar}
    \begin{array}{cl@{\qquad}l}
      & (\alpha \hcomp \beta) \hcomp \gamma
      & \\ =
      & ((\tLpull f \alpha') \hcomp \beta) \hcomp \gamma
      & \\ \permeq
      & (\tLpull f (\alpha' \hcomp \beta)) \hcomp \gamma
      & \\ \permeq
      & \tLpull f ((\alpha' \hcomp \beta) \hcomp \gamma)
      & \\ \permeq
      & \tLpull f (\alpha' \hcomp (\beta \hcomp \gamma))
      & \text{by induction hypothesis ($\alpha' < \alpha$)}
      \\ \permeq
      & (\tLpull f \alpha') \hcomp (\beta \hcomp \gamma)
      & \\ =
      & \alpha \hcomp (\beta \hcomp \gamma)
    \end{array}
  \end{mathpar}

  Symmetrically, associativy is immediate if $\gamma$ starts with a right rule. When neither of those two easy cases apply, we know that $\alpha$ starts with a right rule and $\gamma$ with a left rule. We then reason by case analysis on $\beta$.

  If $\beta$ starts with a multiplication on the left, $\beta = \tLpull g {\beta'}$, then the right rule of $\alpha$ must end a matching division, $\alpha = \alpha' \uRpull f g$, and we have
  \begin{mathpar}
    \begin{array}{cl@{\qquad}l}
      & (\alpha \hcomp \beta) \hcomp \gamma
      & \\ =
      & ((\alpha' \uRpull f g) \hcomp (\tLpull g {\beta'})) \hcomp \gamma
      & \\ \permeq
      & (\alpha' \hcomp \beta') \hcomp \gamma
      & \\ \permeq
      & \alpha' \hcomp (\beta' \hcomp \gamma)
      & \text{by induction hypothesis ($\alpha' < \alpha$, $\beta' < \beta$)}
      \\ \permeq
      & (\alpha' \uRpull f g) \hcomp (\tLpull g (\beta' \hcomp \gamma))
      \\ \permeq
      & (\alpha' \uRpull f g) \hcomp ((\tLpull g  \beta') \hcomp \gamma)
      \\ =
      & \alpha \hcomp (\beta \hcomp \gamma)
    \end{array}
  \end{mathpar}

  The proof is similar if $\beta$ starts with a division on the left,
  $\beta = \uLpush g h \beta'$ and $\alpha = \alpha' \tRpush g$:
  \begin{mathpar}
    \begin{array}{cl@{\qquad}l}
      & (\alpha \hcomp \beta) \hcomp \gamma
      & \\ =
      & ((\alpha' \tRpush g) \hcomp (\uLpush g h {\beta'})) \hcomp \gamma
      & \\ \permeq
      & \alpha' \hcomp (\beta' \hcomp \gamma)
      & \text{by induction hypothesis ($\alpha' < \alpha$, $\beta' < \beta$)}
      \\ \permeq
      & (\alpha' \tRpush g) \hcomp (\uLpush g {hi} (\beta' \hcomp \gamma))
      \\ \permeq
      & (\alpha' \tRpush g) \hcomp ((\uLpush g h {\beta'}) \hcomp \gamma)
      \\ =
      & \alpha \hcomp (\beta \hcomp \gamma)
    \end{array}
  \end{mathpar}

  The cases where $\beta$ starts with a right rule are symmetrical.
\end{proof}



\subsection{Section~\ref{sec:focusing}}

\invseqS*

\begin{proof}
\label{prf:invseqS}
By induction on the underlying judgment of $\alpha_w$.

\begin{itemize}
\item If $\alpha_w$ has a non-neutral judgment
  $\alpha_w : \push {\pi} N \vdashf{f} \pull {\rho} P$,
  we have
  \begin{align*}
    \Strengthen{\alpha_w}
    & \defeq
      \uLRpushpull \pi {} {\Strengthen{\opcart \pi N \hcomp \alpha_w \hcomp \cart \rho P}} \rho
  \end{align*}
  For any $\alpha'_w \in \InvSeq{\Strengthen{\alpha_w}}$, $\alpha'_w$ is one of the two sequentializations of the bi-inversion $\uLRpushpull \pi {} {\beta_w} \rho$ for some $\beta_w \in \InvSeq{\Strengthen{\opcart \pi N \hcomp \alpha_w \hcomp \cart \rho P : N \vdashf{\pi f \rho} P}}$. Let us show the case where $\alpha'_w \defeq \uLpush \pi {} {({\beta_w} \uRpull {} \rho)}$, the other case is similar. By induction hypothesis on the smaller judgment $N \vdashf{\pi f \rho} P$, we have $\beta_w \permeq (\opcart \pi N \hcomp \alpha_w \hcomp \cart \rho P)$, and therefore
  \begin{align*}
    \alpha'_w
    & =
    \uLpush \pi {} {({\beta_w} \uRpull {} \rho)}
    \\ & \permeq
    \uLpush \pi {} {((\opcart \pi N \hcomp \alpha_w \hcomp \cart \rho P) \uRpull {} \rho)}
    \\ & \permeq
    \uLpush \pi {} {(\opcart \pi N \hcomp \alpha_w)}
    \\ & \permeq \alpha_w
  \end{align*}
where the last two equivalences come from Lemma~\ref{lem:eta-expansion-pushpull}.

\item Otherwise $\alpha_w$ has a neutral judgment $\alpha_w : N \vdashf{f} P$, and we reason as in the definition of $\Strengthen{\alpha_w}$ by case analysis on $\alpha_w$, which can be of the form $\tLpull \sigma \beta_w$ or $\beta_w \tRpush \tau$. For example in the case $\tLpull \sigma \beta_w$ (the other case is symmetric), $N$ is of the form $\pull \sigma Q$ and we have
  \[
    \InvSeq{\Strengthen{\alpha_w}} = \InvSeq{\Strengthen{\tLpull \sigma \beta_w}} = \InvSeq{\tLpull \sigma {\Strengthen{\beta_w}}}
  \]
  For any $\alpha'_w \in \InvSeq{\Strengthen{\alpha_w}}$, $\alpha'_w$ is of the form $\tLpull \sigma {\beta'_w}$ for some $\beta'_w \in \InvSeq{\Strengthen{\beta_w}}$. We have $\beta'_w \permeq \beta_w$ by induction hypothesis on the smaller judgment $Q \vdashf{} P$, and therefore
  \begin{align*}
    \alpha'_w
     =
    \tLpull \sigma {\beta'_w}
    \permeq
    \tLpull \sigma {\beta_w}
    = \alpha_w
  \end{align*}
\end{itemize}
\end{proof}

\Sinvseq*

\begin{proof}
  \label{prf:Sinvseq}
  The statement can in fact be strengthened into the following:
  \[ \forall \alpha_w \in \InvSeq{\alpha_m},
     \ \alpha_m = \Strengthen{\alpha_w} \]
  This is proved by induction on the judgment of $\alpha_w$ and $\alpha_m$:
  \begin{itemize}
  \item At a non-neutral judgment $\push {\pi} N \vdashf{f} \pull {\rho} P$,
    $\alpha_m$ must be of the form $
    \uLRpushpull \pi {} {\beta_m} \rho
    $ with $\beta_m : N \vdashf{\pi f\rho} P$, and $\alpha_w \in \InvSeq{\alpha_m}$ must be a sequentialization of the bi-inversion $\uLRpushpull \pi {} {\beta_w} \rho$ for some $\beta_w \in \InvSeq{\beta_m}$. By induction hypothesis we have $\Strengthen{\beta_w} = \beta_m$. If $\alpha_w$ is $\uLpush \pi {} {(\beta_w \uRpull {} \rho)}$ (the other case is similar), we have
\begin{align*}
  & \Strengthen{\alpha_w}
  \\ =\ & \uLRpushpull \pi {} {\Strengthen{\opcart \pi N \hcomp \alpha_w \hcomp \cart \rho P}} \rho
  \\ =\ & \uLRpushpull \pi {} {\Strengthen{\opcart \pi N \hcomp {\uLpush \pi {} {(\beta_w \uRpull {} \rho)}} \hcomp \cart \rho P}} \rho
  \\ =\ & \uLRpushpull \pi {} {\Strengthen{{\beta_w \uRpull {} \rho} \hcomp \cart \rho P}} \rho
  \\ =\ & \uLRpushpull \pi {} {\Strengthen{\beta_w}} \rho
  \\ =\ & \uLRpushpull \pi {} {\beta_m} \rho
  \\ =\ & \alpha_m
\end{align*}
\item At a neutral judgment $N \vdashf{f} P$, we reason by case analysis on $\alpha_m$. For example if it is $\tLpull \sigma {\beta_m}$ (the other case $\beta_m \tRpush \tau$ is similar), we have $\alpha_w = \tLpull \sigma {\beta_w}$ for $\beta_w \in \InvSeq{\beta_m}$. We have $\Strengthen{\beta_w} = \beta_m$ by induction hypothesis, so \[
  S(\alpha_w)
= S(\tLpull \sigma {\beta_w})
= \tLpull \sigma S({\beta_w})
= \tLpull \sigma {\beta_m}
= \alpha_m
\]
  \end{itemize}
\end{proof}

\localconfluence*

\begin{proof}
  \label{prf:local-confluence}
  The source of each rewrite rule must be a stack of two bipoles. Thus
  all critical pairs (where the source of two rewrite rules overlap)
  can be found by considering stacks of two or three bipoles.
  Let us first enumerate all such stack shapes to determine which may contain critical pairs.

  \paragraph{Enumerating all possible critical pairs}

  We exploit a symmetry that permutes $\Lsym$ and $\Rsym$ rules to reduce the number of cases: for example, the reduction behavior of $\biL ~ \biR$ stacks is symmetric to the behavior of $\biR ~ \biL$ stacks.

  Let us consider all stacks of two bipoles. They have one of the following shapes:
  \begin{enumerate}
  \item $\biL ~ \biL$: no reduction. 
  \item $\biL ~ \biR$: \textit{\textbf{a $\parsym^\Lsym_\Rsym$ rule may be applied.}}
  \item $\biL ~ \biLR$: no reduction. 
  \item $\biR ~ \biL$: symmetric to (2).
  \item $\biR ~ \biR$: symmetric to (1), no reduction.
  \item $\biR ~ \biLR$: symmetric to (3), no reduction.
  \item $\biLR ~ \biL$: \textit{\textbf{a $\grasym^\Lsym$ rule may be applied.}}
  \item $\biLR ~ \biR$: symmetric to (7).
  \item $\biLR ~ \biLR$: no reduction.
  \end{enumerate}

  We next enumerate the stacks of three bipoles.
  Notice that if a 2-bipole stack has no reduction, then a 3-bipole stack formed by extending the 2-stack with a new bipole above or below cannot form a new critical pair not already in the above enumeration.
  For example, there is no reduction from $\biL ~ \biL$, so extending it into $\biL ~ \biL ~ \biR$ (or $\biL ~ \biL ~ \biLR$, or $\biR ~ \biL ~ \biL$ for example) does not produce critical pairs involving all three bipoles. Any critical pair in an extension is a critical pair of 2-bipole stacks, which we have already considered.

  In other words, we only need to consider extensions of the reducible 2-bipole stacks, modulo symmetry, to wit (2) $\biL ~ \biR$ and (7) $\biLR ~ \biL$ above:

  \begin{enumerate}
  \item[(2.a)] $\biL ~ \biR ~ \biL$: \textit{\textbf{two overlapping $\parsym$ rules may be applied.}}
  \item[(2.b)] $\biL ~ \biR ~ \biR$: extension of no-reduction (5).
  \item[(2.c)] $\biL ~ \biR ~ \biLR$: extension of no-reduction (6).
  \item[(2.d)] $\biL ~ \biL ~ \biR$: extension of no-reduction (1).
  \item[(2.e)] $\biR ~ \biL ~ \biR$: symmetric to (2.a).
  \item[(2.f)] $\biLR ~ \biL ~ \biR$: \textit{\textbf{a $\grasym_\Lsym$ rule may be applied above, and a $\parsym^\Lsym_\Rsym$ below.}}
  \\
  \item[(7.a)] $\biLR ~ \biL ~ \biL$: extension of no-reduction (1).
  \item[(7.b)] $\biLR ~ \biL ~ \biR$: same as (2.f).
  \item[(7.c)] $\biLR ~ \biL ~ \biLR$: extension of no-reduction (3).
  \item[(7.d)] $\biL ~ \biLR ~ \biL$: extension of no-reduction (3).
  \item[(7.e)] $\biR ~ \biLR ~ \biL$: extension of no-reduction (6).
  \item[(7.f)] $\biLR ~ \biLR ~ \biL$: extension of no-reduction (9).
  \end{enumerate}

  \paragraph{Resolving the critical pairs.} We now consider each of the above cases in turn.

  \begin{enumerate}
  \item[(2)] [$\biL ~ \biR$] Two reductions from stacks of this shape must be two applications of the rule $\parsym^\Lsym_\Rsym$, which is deterministic under the \FP condition, so the two right-hand-sides are equal.
  \item[(7)] [$\biLR ~ \biL$] As above, the critical pair is trivially resolved because $\grasym_\Lsym$ is deterministic, under no assumptions.
  \item[(2.a)] [$\biL ~ \biR ~ \biL$] This is exactly the critical pair \eqref{eq:criticalLRL} that we resolved by introducing the gravity rules, in this case $\grasym_\Lsym$:
\[
\begin{tikzpicture}
  \node (B) {\usebox\myboxb};
  \node[below left=-2cm and 6em of B] (A) {  \usebox\myboxa };
  \node[below right=-2cm and 6em of B] (C) { \usebox\myboxc };
  \pgfsetfillpattern{north east lines}{blue}
  \draw[->] ([yshift=6.8em]B.south west) -- (A.east) node[midway,above left,fill,nearly transparent,text opacity=1]{~$\parsym^\Rsym_\Lsym$};
  \pgfsetfillpattern{north west lines}{red}
  \draw[->] ([yshift=-6.8em]B.north east) -- (C.west) node[midway,above right,fill,nearly transparent,text opacity=1]{$\parsym^\Lsym_\Rsym~$};
  \pgfsetfillpattern{crosshatch dots}{green}
  \draw[->] ([xshift=-1em]C.west) -- ([xshift=1em]A.east) node[midway, below=2pt,fill,nearly transparent,text opacity=1]{$\grasym_\Lsym$};
  \pgfsetfillpattern{north west lines}{red}
  \fill[nearly transparent] ([yshift=-0.25em]B.north west) rectangle ([yshift=-3.5em]B.east);
  \pgfsetfillpattern{north east lines}{blue}
  \fill[nearly transparent] ([yshift=3.5em]B.west) rectangle ([yshift=.25em]B.south east);
  \pgfsetfillpattern{crosshatch dots}{green}
  \fill[nearly transparent] ([yshift=-0.25em]C.north west) rectangle ([yshift=.25em]C.south east);
\end{tikzpicture}
\]
\item[(2.f)] [$\biLR ~ \biL ~ \biR$] This critical pair between $\parsym^\Lsym_\Rsym$ and $\grasym_\Lsym$ is resolved by composing the latter with $\grasym_\Rsym$ followed by $\parsym^\Lsym_\Rsym$:

\[ \criticalLRLR \]

  \end{enumerate}
\end{proof}

\section{Cut on multi-focused derivations, directly}
\label{app:multi-focused-cut}

In Section~\ref{subsubsec:multi-focused-cut} we defined horizontal composition of strongly multifocused proofs, and this definition is modulo permutation equivalence. We can also give a second, more syntactic definition of cut, directly on strongly multifocused derivations, and show that it coincides (modulo permutation equivalence) with the first definition.

Let us first define the operations $(\opcartcomp \pi N {\alpha_m})$, $(\cartcomp {\alpha_m} \rho P)$, which are equivalent to compositions $\opcart \pi N \hcomp \alpha_m$, $\alpha_m \hcomp \cart \rho P$ but can be defined by a simple case analysis---we use a colon to emphasize that this is a more primitive operation:
\begin{mathpar}
  \begin{array}{lll}
  (\cartcomp {-} \rho P)
    & :
    & (S \vdashf{f} \pull \rho P) \to (S \vdashf{f\rho} P)
  \\
  \cartcomp {(\beta \uRpull f \rho)} \rho P
    & \defeq
    & \beta
  \\
  \cartcomp {(\uLRpushpull \pi f \beta \rho)} \rho P
    & \defeq
    & \uLpush \pi {f\rho} \beta
  \\[2em]

  (\opcartcomp \pi N {-})
    & :
    & (\push \pi N \vdashf{f} T) \to (N \vdashf{\pi f} T)
  \\
  \opcartcomp \pi N {(\uLpush \pi f \beta)}
    & \defeq
    & \beta
  \\
  \opcartcomp \pi N {(\uLRpushpull \pi f \beta \rho)}
    & \defeq
    & \beta \uRpull {\pi f} \rho
  \end{array}
\end{mathpar}

In the case of $(\cartcomp {\alpha_m} \rho P)$ for example, the definition is by case analysis, but we know by inspection of the judgment that a strongly multifocused derivation of $\alpha_m : S \vdashf{f} \pull \rho P$ can only start with an inversion or bi-inversion rule.

We can now define the general cut of two strongly multifocused formulas, from $S \vdashf{f} U$ and $U \vdashf{f} T$ to $S \vdashf{fg} T$.
The definition is first by case analysis on $S$ and $T$, with a single case if $S$ is positive or $T$ is negative, and then when $S \vdashf{fg} T$ is a neutral judgment $N \vdashf{fg} P$ we reason by case analysis on $\alpha_m$ and $\beta_m$.

\begin{mathpar}
  \begin{array}{lll}
  (- \hcomp -)
    & :
    & (\push {\pi} N \vdashf{f} U)
      \times (U \vdashf{g} \pull {\rho} P)
      \to (\push {\pi} N \vdashf{fg} \pull {\rho} P)
  \\
  \alpha_m \hcomp \beta_m
    & \defeq
    & \uLRpushpull \pi {fg} {((\opcartcomp \pi N {\alpha_m}) \hcomp (\cartcomp {\beta_m} \rho P))} \rho
  \\[2em]

  (- \hcomp -)
    & :
    & (N \vdashf{f} U)
      \times (U \vdashf{g} P)
      \to (N \vdashf{fg} P)
  \\[1em]
  (\tLpull \sigma {\alpha'_m}) \hcomp (\beta'_m \tRpush \tau)
  & \defeq
  & \tLpull \sigma (\alpha'_m \hcomp \beta'm) \tRpush \tau
  \\
  (\tLpull \sigma {\alpha'_m}) \hcomp \beta_m
    & \defeq
    & \tLpull \sigma {(\alpha'_m \hcomp \beta_m)}
      \qquad\qquad{\beta_m \neq \beta'_m \tRpush \tau}
  \\
  \alpha_m \hcomp (\beta'_m \tRpush \tau)
    & \defeq
    & (\alpha_m \hcomp \beta'_m) \tRpush \tau
      \qquad\qquad{\alpha_m \neq \tLpull \sigma \alpha'_m}
  \\[1em]
  (\alpha'_m \tRpush \pi) \hcomp (\uLpush \pi g \beta'_m)
    & \defeq
    & \alpha'_m \hcomp \beta'_m
  \\
  (\tLRpullpush \sigma {\alpha'_m} \pi) \hcomp (\uLpush \pi g \beta'_m)
    & \defeq
    & (\tLpull \sigma \alpha'_m) \hcomp \beta'_m
  \\[1em]
  (\alpha'_m \uRpull f \rho) \hcomp (\tLpull \rho \beta'_m)
    & \defeq
    & \alpha'_m \hcomp \beta'_m
  \\
  (\alpha'_m \uRpull f \rho) \hcomp (\tLRpullpush \rho {\beta'_m} \tau)
    & \defeq
    & \alpha'_m \hcomp (\beta'_m \tRpush \tau)
    \\[1em]
  \atom\delta\hcomp \atom \epsilon
    &\defeq
    & \atom{\delta\, \epsilon}
  \end{array}
\end{mathpar}
For example, if $\alpha_m$ starts with a left-focusing rule, we propagate the left-focusing rule in the output and recursively compute a cut at a strictly simpler judgment. If $\alpha_m$ is a right-focusing or bi-focusing rule, then necessarily $\beta_m$ starts with a matching invertible rule (thanks to strong focusing) and we have a principal cut. If $\alpha_m$ starts with a left-inversion or bi-inversion rule, then its source formula $S$ is positive and we are in the first case of a non-neutral judgment. Finally, if $\alpha_m$ starts with a right-inversion rule, then necessarily $\beta_m$ uses a focusing rule.

\begin{prop}
  The explicit/intensional definition of $\alpha_m \hcomp \beta_m$ above agrees with the extensional definition we gave earlier modulo $(\permeqseq)$ as $\Strengthen{\Seq{\alpha_m} \hcomp \Seq{\beta_m}}$.
\end{prop}
\begin{proof}
  \label{prf:definitions-hcomp-agree}
  For each equation that defines $\alpha_m \hcomp \beta_m$, we check that there is a choice of sequentialization of $\alpha_m$, $\beta_m$ that gives the expected result.

  For example, in the non-neutral case, $\alpha_m : \push \pi N \vdashf{f} U$ must be of the form $\uLpush \pi f {\alpha'_m}$ or $\uLRpushpull \pi f {\alpha'_m} \tau$. If we have the latter form, we choose to sequentialize it into $\uLpush \pi f {(\alpha' \uRpull {\pi f} \tau)}$, so that both cases can be sequentialized into a term of the form $\uLpush \pi f {\alpha'_w}$, where $\alpha'_w$ is a sequentialization of $\opcartcomp \pi N {\alpha_m}$. By the same reasoning $\beta_m$ is sequentialized into a term of the form $\beta'_w \uRpull g \rho$, with $\beta'_w$ a sequentialization of $\cartcomp {\beta_m} \rho P$. Their composition as weakly focused proofs is $(\uLpush \pi f {\alpha'_w}) \hcomp ({\beta'_w} \uRpull g \rho)$, which is equal to either of the two permutation-equivalent proofs $\uLpush \pi {fg} {((\alpha'_w \hcomp \beta'_w) \uRpull {\pi fg} \rho)}$ or ${(\uLpush {\pi} {fg\rho} (\alpha'_w \hcomp \beta'_w))} \uRpull {fg} \rho$. Both choices strengthen into $\uLRpushpull \pi {fg} {\Strengthen{\alpha'_w \hcomp \beta'_w}} \rho$, which by induction hypothesis is $\uLRpushpull \pi {fg} {(\opcartcomp \pi N {\alpha_m}) \hcomp {(\cartcomp {\beta_m} \rho P)}} \rho$ as desired.

  In a representative neutral case,
  $
  (\tLRpullpush \sigma {\alpha'_m} \pi) \hcomp (\uLpush \pi g \beta'_m)
  $, $\alpha'_m$ starts with a bi-inversion rule, let us assume that it is of the form $\uLRpushpull {\tau} {h} {\alpha''_m} {\rho}$. We sequentialize $\alpha_m$ into \[
  ({(\tLpull \sigma {(\uLpush {\tau} {h\rho} {\alpha''_w})})} \uRpull {\sigma h} {\rho}) \tRpush \pi
\] where $\alpha''_w$ is a sequentialization of $\alpha''_m$, and we have
\begin{align*}
&
(({(\tLpull \sigma {(\uLpush {\tau} {h\rho} {\alpha''_w})})} \uRpull {\sigma h} {\rho}) \tRpush \pi)
\hcomp (\uLpush \pi h {\beta'_w})
\\ =\ &
({(\tLpull \sigma {(\uLpush {\tau} {h\rho} {\alpha''_w})})} \uRpull {\sigma h} {\rho})
\hcomp {\beta'_w}
\\ \permeq\ &
(\tLpull \sigma ({(\uLpush {\tau} {h\rho} {\alpha''_w}) \uRpull {h} {\rho}}))
\hcomp {\beta'_w}
\\ \permeq\ &
\tLpull \sigma (
  ({(\uLpush {\tau} {h\rho} {\alpha''_w}) \uRpull {h} {\rho}})
  \hcomp {\beta'_w}
)
\end{align*}
where $({(\uLpush {\tau} {h\rho} {\alpha''_w}) \uRpull {h} {\rho}})$ is a valid sequentialization of $\alpha'_m$, and so this is permutation equivalent to a sequentialization of $\tLpull \sigma (\alpha_m \hcomp \beta'_m)$ as desired.
\end{proof}

\newpage
\tableofcontents

\end{document}